\pdfoutput=1
\documentclass[11pt]{amsart}

\usepackage{graphicx}
\usepackage[english]{babel}
\usepackage[T1]{fontenc}
\usepackage[latin1]{inputenc}
\usepackage{amsfonts}
\usepackage{amssymb}
\usepackage{amsthm}
\usepackage{amsmath}
\usepackage{tikz}
\usepackage{float}
\usepackage{enumerate}
\usepackage{accents}
\usepackage{mathtools}
\usepackage{mathrsfs}
\usepackage{comment}
\usepackage{afterpage}
\usepackage{bbm}
\usepackage{dsfont}
\usepackage{stmaryrd}
\usepackage{hyperref}
\usepackage{mleftright}
\usepackage{subfig}
\usepackage{soul}
\usepackage{tikz-cd} 
\usepackage{fullpage}
\usepackage{cleveref}

\numberwithin{equation}{section}

\theoremstyle{plain}
\newtheorem{thm}{Theorem}[section]
\newtheorem*{thm*}{Theorem}
\newtheorem{prop}[thm]{Proposition}
\newtheorem*{prop*}{Proposition}
\newtheorem{cor}[thm]{Corollary}
\newtheorem*{cor*}{Corollary}
\newtheorem{lem}[thm]{Lemma}

\newtheorem{thmintro}{Theorem}

\newtheorem{propintro}[thmintro]{Proposition}

\theoremstyle{definition}
\newtheorem{defn}[thm]{Definition}
\newtheorem*{defn*}{Definition}
\newtheorem{ex}[thm]{Example}
\newtheorem{rmk}[thm]{Remark}
\newtheorem*{rmk*}{Remarks}

\newtheorem*{conj*}{Conjecture}
\newtheorem*{quest*}{Question}

\newtheorem{ass}[thm]{Assumption}

\newtheoremstyle{blue-environment}{}{}{}{}{\color{blue}\bfseries}{.}{ }{}
\theoremstyle{blue-environment}

\newcommand{\acts}{\curvearrowright}
\newcommand{\ra}{\rightarrow}
\newcommand{\Ra}{\Rightarrow}

\newcommand{\sq}{\subseteq}

\newcommand{\wt}{\widetilde}
\newcommand{\wh}{\widehat}
\newcommand{\x}{\times}
\renewcommand{\o}{\circ}
\newcommand{\id}{\mathrm{id}}

\newcommand{\mc}{\mathcal}
\newcommand{\mf}{\mathfrak}
\newcommand{\mscr}{\mathscr}

\newcommand{\R}{\mathbb{R}}
\newcommand{\Z}{\mathbb{Z}}
\newcommand{\N}{\mathbb{N}}
\newcommand{\Q}{\mathbb{Q}}

\newcommand{\s}{\sigma}
\newcommand{\eps}{\epsilon}
\newcommand{\Om}{\Omega}
\newcommand{\om}{\omega}
\newcommand{\g}{\gamma}
\newcommand{\G}{\Gamma}
\newcommand{\CAT}{{\rm CAT(0)}}
\newcommand{\Out}{{\rm Out}}
\newcommand{\Aut}{{\rm Aut}}

\newcommand{\SVP}{\mc{SVP}}

\DeclareMathOperator{\SL}{SL}

\DeclareMathOperator{\Elem}{Elem}
\DeclareMathOperator{\Sal}{Sal}

\DeclareMathOperator{\xSal}{xSal}
\DeclareMathOperator{\ot}{at}
\DeclareMathOperator{\pt}{pt}

\newcommand{\X}{\mc{X}}

\DeclareMathOperator{\diam}{diam} 
\DeclareMathOperator{\isom}{Isom}
\DeclareMathOperator{\lk}{lk}
\DeclareMathOperator{\St}{st}
\DeclareMathOperator{\Fix}{Fix}
\DeclareMathOperator{\Min}{Min}
\DeclareMathOperator{\Hom}{Hom}

\DeclareMathOperator{\Gr}{Gr}
\DeclareMathOperator{\rk}{rk}

\DeclareMathOperator{\Perp}{\perp}

\begin{document}

\title{Generators for automorphisms of special groups} 

\author[E.\,Fioravanti]{Elia Fioravanti}\address{Institute of Algebra and Geometry, Karlsruhe Institute of Technology}\email{elia.fioravanti@kit.edu} 
\thanks{The author is supported by Emmy Noether grant 515507199 of the Deutsche Forschungsgemeinschaft (DFG)}

\begin{abstract}
    Let $G$ be a (compact) special group in the sense of Haglund and Wise. We show that $\Out(G)$ is finitely generated, and provide a virtual generating set consisting of Dehn twists and ``pseudo-twists''. We exhibit instances where Dehn twists alone do not suffice and completely characterise this phenomenon: it is caused by certain abelian subgroups of $G$, called ``poison subgroups'', which can be removed by replacing $G$ with a finite-index subgroup.

    Similar results hold for coarse-median preserving automorphisms, without the pathologies: For every special group $G$, the coarse-median preserving subgroups $\Out(G,[\mu])\leq\Out(G)$ are virtually generated by finitely many Dehn twists with respect to splittings of $G$ over centralisers.

    Proofs are based on a novel, hierarchical version of Rips and Sela's shortening argument.
\end{abstract}

\maketitle

\section{Introduction}

Given an infinite, finitely generated group $G$, a fundamental problem is to describe the outer automorphism group $\Out(G)$.
One cannot expect any structure in $\Out(G)$ in complete generality, but three classical examples serve as a guiding light: non-abelian free groups $F_n$, closed surface groups $\pi_1(S)$, and free abelian groups $\Z^n$. The corresponding outer automorphism groups are $\Out(F_n)$, the extended mapping class group ${\rm Mod}^{\pm}(S)$, and the arithmetic group ${\rm GL}_n(\Z)$, whose intricate structures have animated a great deal of research in the last decades. 

In these three examples, it has been known since the work of Dehn and Nielsen in the early 20th century that $\Out(G)$ is finitely generated \cite{Nielsen,Dehn}. In fact, up to passing to finite index, these automorphism groups even admit finite classifying spaces. At the same time, exotic examples also abound. For the Baumslag--Solitar group $G={\rm BS}(2,4)$, the group $\Out(G)$ is not finitely generated \cite{Collins-Levin}: it is locally virtually free, with torsion of unbounded order \cite{Clay-BS}. Furthermore, every countable group arises as $\Out(G)$ for some finitely generated group $G$ \cite{Bumagin-Wise}.

A more complex problem is to understand, given a family $\mf{F}$ of finitely generated groups, what properties are shared by the groups $\Out(G)$ as $G$ varies in $\mf{F}$. General results of this kind are rare, and this is the setting of our article. Here two typical questions arise, as vague as they are natural:
\begin{itemize}
    \item Is the group $\Out(G)$ ``structured'' for all groups $G\in\mf{F}$, or can it be ``wild''?
    \item What ``explicit'' features of the group $G$ cause it to have many outer automorphisms?
\end{itemize}
One of the few families $\mf{F}$ for which we have good answers to the above are Gromov-hyperbolic groups \cite{Gromov-hyp}, due to breakthroughs of Rips and Sela in the 90s. Here $\Out(G)$ is always finitely generated \cite{RS94}, and it virtually admits a finite classifying space \cite{Sela-canonical,Levitt-GD}. Moreover, $\Out(G)$ can only be infinite if $G$ admits an ``obvious'' infinite-order automorphism: a \emph{Dehn twist}. 

Dehn twists provide a completely general construction of automorphisms for groups $G$ that split as an amalgamated product or HNN extension \cite{Bass-Jiang,Levitt-GD}.
For instance, an amalgam $G=A\ast_C B$ admits automorphisms that fix $A$ pointwise and conjugate $B$ by each element of the centraliser $Z_B(C)$ (see \Cref{sub:DTs} for details). When $G$ is hyperbolic, Dehn twists with respect to cyclic splittings of $G$ virtually generate $\Out(G)$ \cite{RS94}. Toral relatively hyperbolic groups \cite{Groves-II,GL-relhyp} and closed $3$--manifold groups \cite{Johannson} behave similarly, but this phenomenon fails in greater generality. There are many groups that have infinitely many outer automorphisms, but do not admit a single (non-identity) Dehn twist, and not even a single splitting: for instance Thompson's group $V$ \cite{Bleak-etal,Farley}, 
or the Kazhdan groups constructed in \cite{Ollivier-Wise(T)}. Even in the class of polycyclic groups --- where $\Out(G)$ is by all means ``structured'', as it is arithmetic over $\Q$ \cite{Baues-Grunewald} --- Dehn twists typically only generate 
a small portion of $\Out(G)$; see \Cref{ex:nilpotent}.

Beyond hyperbolic groups, understanding automorphisms of \emph{non-positively curved} groups remains a fundamental open problem. All the wild examples mentioned above display serious failures of non-positive curvature, but there is also little evidence that non-positive curvature alone should suffice to coerce automorphisms into tameness.
Swarup asked whether $\Out(G)$ is virtually generated by Dehn twists for every $\CAT$ group $G$ \cite[Q2.1]{BesQ}. Rips asked about the structure of $\Out(G)$ under the stronger hypothesis that $G$ be cocompactly cubulated \cite[p.\,826]{Sela-new1}. 

In the present article, we address Swarup's and Rips' questions for the smaller class of (compact) \emph{special groups}, in the sense of Haglund and Wise \cite{HW08,Wise-hierarchy}. Special groups are arguably the most important class of cocompactly cubulated groups --- marrying a wealth of examples with a powerful toolkit to study them. 
In particular, the three classical groups with structured automorphisms --- free groups, surface groups and free abelian groups --- are all special. Among special groups, we also mention many free-by-cyclic and $3$--manifold groups (including all hyperbolic ones, up to finite index \cite{Hagen-Wise,Agol}), right-angled Artin groups, commutator subgroups of right-angled Coxeter groups \cite{Droms}, and graph braid groups \cite{Crisp-Wiest}. Our main result is the following.

\begin{thmintro}\label{thmnewintro:fg}
    For all special groups $G$, the outer automorphism group $\Out(G)$ is finitely generated.
\end{thmintro}

While the true power of \Cref{thmnewintro:fg} lies in its generality, there are also several concrete instances where it is new and interesting. For right-angled Artin groups $\mc{A}_{\G}$, finite generation of $\Out(\mc{A}_{\G})$ is due to Laurence and Servatius \cite{Laurence,Servatius}, but nothing was seemingly known about $\Out(G)$ for \emph{finite-index} subgroups $G\leq\mc{A}_{\G}$ (except for those that are themselves Artin groups). Additionally, finite generation of $\Out(W_{\G}')$ is new for commutator subgroups $W_{\G}'$ of right-angled Coxeter groups.

We also prove finite generation of the \emph{relative} automorphism groups $\Out(G;\mscr{F}^t)$, where $G$ is special and $\mscr{F}$ is an arbitrary finite set of finitely generated subgroups of $G$ on which the automorphisms are required to act trivially (up to inner automorphisms of $G$); see \Cref{thm:main}. This was known when $G$ is a right-angled Artin group and $\mscr{F}$ is a set of parabolic subgroups \cite{Day-Wade}, and when $G$ is toral relatively hyperbolic with $\mscr{F}$ arbitrary \cite{GL-McCool}.

The proof of \Cref{thmnewintro:fg} is highly non-constructive, but it does show that $\Out(G)$ is virtually generated by automorphisms of the following three concrete kinds (see \Cref{thm:main}(2)).
\begin{enumerate}
    \item \emph{Centraliser twists}: Dehn twists with respect to $1$--edge splittings of $G$ over centralisers.
    \item \emph{Ascetic twists}: Dehn twists with respect to HNN splittings of $G$ over co-cyclic subgroups of centralisers, with edge groups normalised by a stable letter (\Cref{defn:ascetic_twist}).
    \item \emph{Pseudo-twists}: automorphisms of certain canonical abelian subgroups $A\leq G$ that can be extended to automorphisms of $G$ by means of splittings having the centraliser $Z_G(A)$ as a vertex group (\Cref{defn:pseudo-twists}).
\end{enumerate}
For a right-angled Artin group $\mc{A}_{\G}$, we recover the Laurence--Servatius generators of the finite-index \emph{pure} subgroup $\Out^0(\mc{A}_{\G})\leq\Out(\mc{A}_{\G})$, namely transvections and partial conjugations.

In general, our generators of the third kind --- pseudo-twists --- are not Dehn twists. This cannot be avoided, as there are surprisingly simple examples where $\Out(G)$ is not virtually generated by Dehn twists alone (regardless of the class of splittings used to define them):

\begin{propintro}\label{propnewintro:bad_examples}
    There exists $G$ special with $\Out(G)$ not virtually generated by Dehn twists, and:
    \begin{itemize}
        \item we can take $G$ to be a finite-index subgroup of a right-angled Artin group;
        \item we can alternatively take $G$ to virtually split as $\Z^2\x H$ with $H$ hyperbolic.
    \end{itemize}
\end{propintro}

Despite its occurrence, the failure of virtual generation of $\Out(G)$ by Dehn twists remains a rare phenomenon and there is a simple and explicit description of exactly when it occurs. It is regulated by the presence of certain canonical abelian subgroups of $G$, which we term \emph{poison subgroups} (see \Cref{thmnewintro:poison} below).
While poison subgroups can have arbitrarily large rank, they are all manifestations of the rank--$2$ free abelian group $\Z^2$. Ultimately, this comes down to the fact that $\SL_2(\Z)$ has non-trivial infinite-index normal subgroups, unlike $\SL_n(\Z)$ for $n\geq 3$.

Importantly, poison subgroups can always be eliminated by passing to finite index, and hence:

\begin{thmintro}\label{thmnewintro:fi_no_poison}
    Every special group $G$ admits a characteristic finite-index subgroup $G_0\lhd G$ such that $\Out(G_0)$ is virtually generated by Dehn twists.
\end{thmintro}

Note that \Cref{thmnewintro:fi_no_poison} badly fails for fundamental groups of non-compact special cube complexes,
and it also fails for general $\CAT$ groups. In both of these classes, there are groups $G$ such that every finite-index subgroup $G_0\leq G$ has infinite $\Out(G_0)$ and not a single (non-identity) Dehn twist. These examples can be easily extracted from recent work of Italiano--Martelli--Migliorini \cite{IMM,Groves-Manning-IMM} and Martelli \cite{Martelli-closing}, as we explain in \Cref{ex:IMM}.
In particular, the most general form of Swarup's question \cite[Q2.1]{BesQ} has a negative answer.

In a different direction, our results can be extended to groups of \emph{coarse-median preserving} automorphisms, which we introduced in \cite{Fio10a}. Examples include all automorphisms of hyperbolic groups, untwisted automorphisms of right-angled Artin groups (in the sense of \cite{CSV}), and arbitrary automorphisms of right-angled Coxeter groups. 

Roughly, an automorphism of a special group $G$ is coarse-median preserving if it has no ``skewing'' behaviours along abelian subgroups. More precisely, this notion depends on the choice of a convex-cocompact embedding $\iota\colon G\hookrightarrow\mc{A}_{\G}$, where $\mc{A}_{\G}$ is a right-angled Artin group. The median operator on $\mc{A}_{\G}$ can then be pulled back via $\iota$ to a coarse median operator $\mu\colon G^3\ra G$. The subgroup $\Out(G,[\mu])\leq\Out(G)$ consists of those outer automorphisms that coarsely preserve $\mu$.

Since the pathologies requiring pseudo-twists (and even ascetic twists) are uniquely supported on abelian subgroups, those generators are not needed in the coarse-median preserving case, leading to the following pleasantly simple result in this setting:

\begin{thmintro}\label{thmintro:cmp}
    Let $G$ be a special group, equipped with the coarse median structure $[\mu]$ induced by some convex-cocompact embedding $G\hookrightarrow\mc{A}_{\G}$ into a right-angled Artin group. Then the coarse-median preserving group $\Out(G,[\mu])\leq\Out(G)$ is virtually generated by finitely many centraliser twists.
\end{thmintro}

We conclude this introduction by saying more about poison subgroups, and then sketching the hierarchical shortening argument underlying the proof of all our main results.

\smallskip
{\bf Poison subgroups.}
As hinted at above, one can completely characterise whether $\Out(G)$ is virtually generated by Dehn twists in terms of the absence of ``poison subgroups'' of $G$, which are to be searched for among a finite canonical collection of abelian subgroups of $G$. 

The definition of poison subgroups requires combining many different elements, but it is sufficiently important that we choose to sketch it here. We proceed in five steps.
\begin{enumerate}
    \setlength\itemsep{.25em}
    \item Every special group $G$ has finitely many (conjugacy classes of) particularly natural abelian subgroups $A_1,\dots,A_k$, which we call the \emph{salient abelians} of $G$. These are the only candidates to be poison subgroups. They are defined as follows: start with a maximal subgroup $\Pi\leq G$ splitting as a direct product, then repeatedly pick an irreducible factor of $\Pi$ and replace it with a maximal product within itself. This yields a canonical finite collection of products within $G$, which we call the \emph{standard products} (\Cref{defn:SVP}). Salient abelian subgroups are the centres of (some of) the standard products (\Cref{defn:salient}).
    \item For each salient abelian subgroup $A_i\leq G$, we consider all simplicial trees $G\acts T$ in which the centraliser $Z_G(A_i)$ is an entire vertex group, and in which all incident edge groups are the same subgroup $L\lhd Z_G(A_i)$ with a free abelian quotient $Z_G(A_i)/L$. 
    Let $\mc{L}_i\lhd Z_G(A_i)$ be the smallest possible such subgroup $L$ as we vary the $G$--tree $T$. 
    \item Let $\overline A_i$ be the projection of $A_i$ to the free abelian quotient $Z_G(A_i)/\mc{L}_i$. Then let $\wh A_i$ be the direct summand of $Z_G(A_i)/\mc{L}_i$ that contains $\overline A_i$ as a finite-index subgroup. For $A_i$ to be a poison subgroup, it is necessary that $\overline A_i\neq\wh A_i$ and $\wh A_i\cong\Z^2$ (though not sufficient).
    \item Finally, we define a matrix group $\mf{D}(A_i)\leq\Aut(\wh A_i)\cong{\rm GL}_{n_i}(\Z)$, where $n_i:=\rk(\wh A_i)$. This is the group generated by the elementary automorphisms of $\wh A_i$ whose multiplier lies in the finite-index subgroup $\overline A_i$. A typical example of $\mf{D}(A_i)$ is the group normally generated by the $N$--th powers of the elementary generators of $\SL_{n_i}(\Z)$, for some large $N\geq 1$. The subgroup $\mf{D}(A_i)\leq{\rm GL}_{n_i}(\Z)$ is always finite-index if $n_i\neq 2$, but it need not be for $n_i=2$.
    \item A \emph{poison subgroup} of $G$ is any of the salient abelian subgroups $A_1,\dots,A_k\leq G$ for which the matrix group $\mf{D}(A_i)\leq {\rm GL}_{n_i}(\Z)$ has infinite index (\Cref{defn:poison}).
\end{enumerate}
With this definition, we prove the following equivalence:

\begin{thmintro}\label{thmnewintro:poison}
    For a special group $G$, the outer automorphism group $\Out(G)$ is virtually generated by Dehn twists if and only if none of the salient abelian subgroups $A_1,\dots,A_k$ is a poison subgroup.
\end{thmintro}

In fact, when there are no poison subgroups, centraliser twists and ascetic twists suffice in order to virtually generate $G$; see \Cref{thm:main}(3). As already mentioned, poison subgroups can always be removed by passing to a finite index subgroup of $G$ (\Cref{thmnewintro:fi_no_poison}).

\smallskip
{\bf The hierarchical shortening argument.} 
The proofs of Theorems~\ref{thmnewintro:fg}, \ref{thmintro:cmp} and~\ref{thmnewintro:poison} are all based on a \emph{shortening argument}, just like Rips and Sela's original work \cite{RS94}. The lack of hyperbolicity, however, forces us to consider finitely many trees at the same time, which leads to the main difficulty: shortening at least one of these trees without lengthening any of the others. Our main contribution is a hierarchical procedure to carry out the shortening process, moving up along a canonical family $\mf{H}(G)$ of larger and larger subgroups of $G$, which ensures that none of the relevant trees gets lengthened. We now explain this in greater detail.

Fix a proper cocompact action on a metric space $G\acts X$ and choose a finite generating set $S\sq G$. Let $\mc{DT}\leq\Aut(G)$ be the subgroup generated by some family of Dehn twists. Given a sequence of automorphisms $\varphi_n\in\Aut(G)$ (in distinct outer classes), our goal is to find elements $\tau_n\in\mc{DT}(G)$ that ``shorten'' the generating sets $\varphi_n(S)$, in the sense that the inequality
\[ \inf_{x\in X}\max_{s\in S} d(\varphi_n\tau_n(s)x,x) < \inf_{x\in X}\max_{s\in S} d(\varphi_n(s)x,x) \]
holds for infinitely many values of $n$. If this is possible for all sequences $(\varphi_n)_n$, then the argument of Rips and Sela shows that $\mc{DT}$ projects to a finite-index subgroup of $\Out(G)$. If, in addition, we can find all the $\tau_n$ within a \emph{finitely generated} subgroup $\Delta\leq\mc{DT}$ (which is allowed to depend on the sequence $(\varphi_n)_n$, but not on the specific index $n$), then $\Out(G)$ is finitely generated.

In order to find the $\tau_n$, we need a procedure to construct splittings of $G$, as well as Dehn twists with respect to these splittings. When $G$ is a special group, we developed such a procedure in \cite{Fio10e}. Roughly, we choose a convex-cocompact embedding into a right-angled Artin group $G\hookrightarrow\mc{A}_{\G}$. Each vertex $v\in\G$ corresponds to an HNN splitting $\mc{A}_{\G}\acts T^v$, and $G$ acts properly and cocompactly on a (non-convex) subcomplex $X\sq\prod_vT^v$. The induced splittings $G\acts T^v$ are useless: they typically do not admit any Dehn twists. However, we can apply the Bestvina--Paulin construction \cite{Bes88,Pau91}: take an ultralimit of copies of the action $G\acts X$, twisted by the automorphisms $\varphi_n$ and rescaled by a diverging sequence of factors. We thus obtain a \emph{degeneration} $G\acts X_{\om}$, which is a non-elliptic action on a finite-rank median space, and which equivariantly embeds in a finite product of $\R$--trees $G\acts T^v_{\om}$ (namely the ultralimits of twisted and rescaled copies of the simplicial trees $T^v$). The $\R$--trees $T^v_{\om}$ are neither small nor acylindrical, and they often have infinitely-generated arc-stabilisers. At the same time, we showed in \cite{Fio10e} that these $\R$--trees are \emph{stable} in the sense of \cite{BF-stable}, and that their arc-stabilisers are co-abelian subgroups of centralisers in $G$. We can then exploit the Rips machine to extract from each $T^v_{\om}$ a splitting of $G$ with some Dehn twists.

The Dehn twists constructed from the $\R$--tree $T^v_{\om}$ shorten the action of the generating sets $\varphi_n(S)$ on the simplicial tree $T^v$. However, in general, they lengthen the action on the trees $T^w$ for $w\in\G\setminus\{v\}$, and they can also lengthen the action on $X$.

A fundamental new insight of this article is that these issues do not arise if ``sufficiently many'' of the subgroups of $G$ splitting as a direct product are elliptic in the degeneration $X_{\om}$ (see in particular Propositions~\ref{prop:shortening_inert} and~\ref{prop:shortening_motile}). Thus, if some product subgroup $\Pi\leq G$ happens to be non-elliptic in $X_{\om}$, we should prioritise shortening a generating set of $\Pi$, even if this comes at the cost of lengthening $S\sq G$. In turn, this amounts to shortening a generating set of an irreducible factor $F\leq\Pi$ (while not spoiling the factors of $\Pi$ that are already elliptic in $X_{\om}$). Finally, if all maximal products within $F$ are elliptic in $X_{\om}$, then we can shorten immediately; otherwise, we restrict to a maximal product within $F$ and repeat.

Multiple issues aside, this leads to the existence of an $\Aut(G)$--invariant family of subgroups $\mf{H}(G)$, which we refer to as the \emph{hierarchy} of $G$ (\Cref{defn:hierarchy}). Fixing $X$ and $X_{\om}$ as above, the core of the article lies in showing that, if we consider a ``lowest'' subgroup $H\in\mf{H}(G)$ that is non-elliptic in $X_{\om}$ and a finite generating set $S_H\sq H$, then we can shorten the generating sets $\varphi_n(S_H)$ (with respect to $X$) by employing Dehn twists of $G$. Moreover, these twists are inner on the elements of $\mf{H}(G)$ that lie ``below'' $H$, and so they do not ruin whatever was previously arranged. 

The core of the hierarchical shortening argument lies in \Cref{thm:shortening}, whose statement is fairly accessible. Once that is obtained, the proof of the main theorems is short and elementary; we describe it in \Cref{thm:main}.

In truth, the above is a faithful account of our proof only in the coarse-median preserving case. In full generality, additional issues arise which may prevent us from finding Dehn twists that shorten even a single tree $T^v$. A typical example is that of an axial component $\alpha\sq T^v_{\om}$ whose stabiliser $G_{\alpha}$ virtually splits as $\Z^m\x K$ with $m\geq 2$, where $\Z^m$ acts freely on the line $\alpha$ and $K$ is the kernel of the $G_{\alpha}$--action. Since $G_{\alpha}$ contains $\Z^m\x K$ only as a \emph{finite-index} subgroup, it is possible that only very high powers of the Dehn twists of $\Z^m$ can be extended to Dehn twists of $G_{\alpha}$ and $G$. 
For $m=2$, this can mean that all Dehn twists of $G$ actually lengthen the actions on $T^v_{\om}$ and $T^v$, and this corresponds precisely to the existence of poison subgroups. Luckily, the particular structure of axial components ensures that ``most'' automorphisms of $\Z^m\leq G_{\alpha}$ can be extended to automorphisms of $G$, producing what we call ``pseudo-twists''. It turns out that interactions between pseudo-twists arising from distinct abelian subgroups of $G$ are remarkably limited, and this allows us to reduce to automorphisms of $G$ that only cause ``tame'' axial components in their degenerations. The key results in this direction are \Cref{prop:reduction_to_preserving_complements} and \Cref{prop:salient_vs_degenerations}.

\smallskip
{\bf Structure of the paper.} 
\Cref{sect:special_groups} is concerned with general properties of special groups. After collecting some terminology and results from the literature, we describe some new families of subgroups in \Cref{sub:staunchly_parabolic}; the most important ones are \emph{(extra-)salient abelians} (\Cref{defn:salient}).

\Cref{sect:DTs} regards Dehn twists. In \Cref{sub:shunning+ascetic}, we introduce and study the class of \emph{shunning splittings} (\Cref{defn:isolating}), which plays an important role in the definition of poison subgroups. Then \Cref{sub:more_DTs} defines \emph{centraliser twists}, \emph{ascetic twists} and \emph{pseudo-twists}.

\Cref{sect:poison} is about poison subgroups. It contains the proofs of \Cref{propnewintro:bad_examples} (\Cref{ex:poisonous_centre}) and of the forward implication in \Cref{thmnewintro:poison} (\Cref{cor:poison_obstructs_DT_generation}). It also reduces \Cref{thmnewintro:fi_no_poison} to the backward implication of \Cref{thmnewintro:poison} (\Cref{prop:no_poison_in_fi}), which is proven later in \Cref{sect:proofs}.

In \Cref{sect:hierarchy}, we construct and study the \emph{hierarchy} $\mf{H}(G)$ (\Cref{defn:hierarchy}). We then discuss \emph{reference systems} (\Cref{defn:reference_system}), which are non-canonical ways of ordering $\mf{H}(G)$ so as to produce a well-ordering on $\Out(G)$. The latter formalises the exact concept of ``shortening'' that we rely on in the hierarchical shortening argument.

\Cref{sec:degenerations} is concerned with the structure of degenerations $G\acts X_{\om}$ of special groups $G$. We review some results of \cite{Fio10e} (collected in \Cref{thm:inert_vs_motile}), and then analyse certain pathological features of degenerations (called \emph{IOM-lines}), showing that these are uniquely caused by (extra-)salient abelian subgroups (see \Cref{lem:line_stab_SVP} and \Cref{prop:salient_vs_degenerations}).

\Cref{sect:shortening} contains the technical core of the article, as it deals with the actual shortening argument. First, in \Cref{sub:shortening_median}, we develop a general procedure to shorten actions on median spaces, provided that there is an equivariant projection to an $\R$--tree with features of one of two kinds: \emph{shrinkable arcs} (\Cref{defn:shrinkable_arcs}) or \emph{shrinkable lines} (\Cref{defn:shrinkable_lines}). Then, in \Cref{sub:shortening_special}, we show how to find such features in degenerations of special groups and, finally, in \Cref{thm:shortening} we prove our hierarchical shortening theorem.

\Cref{sect:proofs} deduces Theorems~\ref{thmnewintro:fg}, \ref{thmintro:cmp} and~\ref{thmnewintro:poison} from \Cref{thm:shortening} and \Cref{prop:reduction_to_preserving_complements}. It also discusses failures of \Cref{thmnewintro:fi_no_poison} for non-compact special groups, CAT(0) groups, and nilpotent groups.

Finally, \Cref{app:A} proves a shortening theorem for automorphism groups of free products (\Cref{cor:shortening_free_products}), which is needed in one of the cases of the proof of the shortening theorem. We expect this to be known to experts, but it does not seem to appear in the literature.

\smallskip
{\bf Acknowledgements.} 
I am grateful to Adrien Abgrall, Benjamin Br\"uck, Sami Douba, Vincent Guirardel, Camille Horbez and Ric Wade for helpful conversations related to this paper.

\setcounter{tocdepth}{1}
\tableofcontents

\section{Special groups}\label{sect:special_groups}

\Cref{sub:prelims} mainly fixes terminology and collects various needed results on special groups that already appear in the literature. Then, in \Cref{sub:staunchly_parabolic}, we discuss new facts about products within special groups, and we define salient abelian subgroups (\Cref{defn:salient}).

\subsection{Preliminaries}\label{sub:prelims}

Throughout the article, a \emph{special group} $G$ is the fundamental group of a compact special cube complex in the sense of \cite{HW08,Wise-hierarchy}. Equivalently, $G$ embeds in a right-angled Artin group (RAAG) $\mc{A}_{\G}$ as a \emph{convex-cocompact} subgroup, meaning that there is a $G$--invariant convex subcomplex of the universal cover of the Salvetti complex of $\mc{A}_{\G}$ on which $G$ acts cocompactly.

We now review the most important families of subgroups of a special group $G$. For some of these families, the definition depends on the choice of a particular convex-cocompact embedding $\iota\colon G\hookrightarrow\mc{A}_{\G}$. Such embeddings are never canonical, and in general they can affect the properties of the families of subgroups under consideration. 

\subsubsection{Root-closed subgroups}

For a subgroup $H\leq G$, we define the \emph{root-closure} of $H$ as the set $\{g\in G\mid \langle g\rangle\cap H\neq\{1\}\}$. A subgroup is \emph{root-closed} if it equals its root-closure. In general, the root-closure of a subgroup need not be a subgroup, but it will be whenever we consider this notion. The typical situation is when $H_0$ is a finite-index subgroup of some root-closed subgroup $H\leq G$: then the root-closure of $H_0$ is precisely $H$. 

\subsubsection{Centralisers}

For a subgroup $H\leq G$ and a subset $U\sq G$, we employ the notation:
\begin{align*}
    Z_H(U)&=\{h\in H\mid hu=uh,\ \forall u\in U\}, & N_H(U)&=\{h\in H\mid hUh^{-1}=U\}. 
\end{align*}
A subgroup $H\leq G$ is a \emph{centraliser} in $G$ if it is of the form $H=Z_G(U)$ for some subset $U\sq G$. For every subgroup $H\leq G$, we have the inclusion $H\leq Z_GZ_G(H)$ and this inclusion is an equality precisely when $H$ is a centraliser in $G$. The group $G$ is always a centraliser in itself.

Centralisers are always root-closed and convex-cocompact, regardless of the chosen embedding $\iota\colon G\hookrightarrow\mc{A}_{\G}$. This is because these properties hold for centralisers in RAAGs \cite[Section~III]{Servatius}. 

We denote by $\mc{Z}(G)$ the family of centralisers in $G$. This family is $\Aut(G)$--invariant and closed under intersections. Chains of subgroups in $\mc{Z}(G)$ have uniformly bounded length. If $G$ is non-abelian, then the family $\mc{Z}(G)$ contains \emph{infinitely many} distinct $G$--conjugacy classes of subgroups. 

\subsubsection{$G$--parabolic subgroups}

Parabolic subgroups of a RAAG $\mc{A}_{\G}$ are those of the form $h\mc{A}_{\Delta}h^{-1}$ with $\Delta\sq\G$ and $h\in\mc{A}_{\G}$. Given a convex-cocompact embedding $\iota\colon G\hookrightarrow\mc{A}_{\G}$, a subgroup $H\leq G$ is \emph{$G$--parabolic} with respect to $\iota$ if it is of the form $H=\iota^{-1}(P)$ for a parabolic subgroup $P\leq G$; in other words, viewing $G\leq\mc{A}_{\G}$, we have $H=G\cap P$. Note that the whole group $G$ is always a $G$--parabolic subgroup of itself.

Parabolic subgroups do not admit a definition purely in terms of the algebra of $G$. As such, this notion depends on the choice of the embedding $\iota$ and it is {\bf not} preserved by automorphisms of $G$, not even by coarse-median preserving ones.
Still, in \Cref{sub:staunchly_parabolic} we will describe many subgroups of $G$ that are $G$--parabolic regardless of the choice of the embedding $\iota$.

The family of $G$--parabolic subgroups is closed under taking intersections, and all $G$--parabolic subgroups are root-closed. They are also convex-cocompact with respect to the chosen embedding $\iota\colon G\hookrightarrow\mc{A}_{\G}$. The following lemma summarises other important properties; see \cite[Corollary~3.21]{Fio10e} and \cite[Lemma~2.5]{Fio11} for proofs.

\begin{lem}\label{lem:parabolics_basics}
    Let $G\leq\mc{A}_{\G}$ be convex-cocompact.
    \begin{enumerate}
        \item There are only finitely many $G$--conjugacy classes of $G$--parabolic subgroups of $G$.
        \item For every $G$--parabolic subgroup $P\leq G$, the normaliser $N_G(P)$, the centraliser $Z_G(P)$ and the centre $Z_P(P)$ are all $G$--parabolic.
    \end{enumerate}
\end{lem}

\subsubsection{Convex-cocompact subgroups}

Fixing a convex-cocompact embedding $\iota\colon G\hookrightarrow\mc{A}_{\G}$ allows us to speak of \emph{convex-cocompact} subgroups of $G$ with respect to $\iota$: these are the subgroups $H\leq G$ such that there exists an $\iota(H)$--invariant convex subcomplex of the universal cover $\X_{\G}$ of the Salvetti complex of $\mc{A}_{\G}$ on which $\iota(H)$ acts cocompactly. We say that an element $g\in G$ is \emph{convex-cocompact} if the cyclic subgroup $\langle g\rangle$ is.

Changing the embedding $\iota$ can affect which subgroups of $G$ qualify as convex-cocompact. In fact, convex-cocompactness is completely determined by a coarse median structure: each convex-cocompact embedding $\iota\colon G\hookrightarrow\mc{A}_{\G}$ allows us to pull back the cubical median operator $m\colon\X_{\G}^3\ra\X_{\G}$ to a \emph{coarse median operator} $\mu\colon G^3\ra G$ in the sense of Bowditch \cite{Bow13}. A \emph{coarse median structure} is an equivalence class $[\mu]$ of coarse median operators on $G$ pairwise at bounded distance from each other. Convex-cocompact subgroups are then precisely those that are quasi-convex with respect to the coarse median structure induced by $\iota$; see \cite[Section~2.4]{FLS}.

For a coarse median structure $[\mu]$ arising from a convex-cocompact embedding $\iota\colon G\hookrightarrow\mc{A}_{\G}$, we are also interested in the \emph{coarse-median preserving} subgroup $\Aut(G,[\mu])\leq\Aut(G)$. An automorphism $\varphi\in\Aut(G)$ lies in $\Aut(G,[\mu])$ when the elements $\varphi(\mu(a,b,c))$ and $\mu(\varphi(a),\varphi(b),\varphi(c))$ are at uniformly bounded distance with respect to a word metric on $G$, as we vary the elements $a,b,c\in G$. All inner automorphisms satisfy this property, giving rise to a well-defined projection $\Out(G,[\mu])\leq\Out(G)$ of $\Aut(G,[\mu])$. Equivalently, coarse-median preserving automorphisms are precisely those taking convex-cocompact subgroups to convex-cocompact subgroups \cite[Theorem~2.17]{FLS}. All automorphisms of hyperbolic groups are coarse-median preserving, as there is only one possible coarse median structure in this case. Additionally, coarse-median preserving automorphisms of RAAGs are precisely those that are \emph{untwisted} in the sense of \cite{CSV}.

The following lemma summarises the key properties of convex-cocompact subgroups.

\begin{lem}\label{lem:cc_basics}
    Let $G\leq\mc{A}_{\G}$ and $H,K\leq G$ all be convex-cocompact.
    \begin{enumerate}
        \item The intersection $H\cap K$ is convex-cocompact.
        \item The normaliser $N_G(H)$ is convex-cocompact and virtually splits as $H\x P$ for a $G$--parabolic subgroup $P\leq G$.
        \item If $gHg^{-1}\leq H$ for some $g\in G$, then $gHg^{-1}=H$.
        \item If $H\leq K$ and $K=K_1\x\dots\x K_m$ with all $K_i$ convex-cocompact, then $H$ contains the product $(H\cap K_1)\x\dots\x (H\cap K_m)$ as a finite-index subgroup.
        \item There exists an integer $q=q(H)$ with the following property. For each $g\in G\setminus\{1\}$, there exist an integer $1\leq n\leq q$ and pairwise-commuting convex-cocompact elements $g_1,\dots,g_k\in G$ such that $g^n=g_1\cdot\ldots\cdot g_k$ and $Z_G(g)=Z_G(g_1)\cap\dots\cap Z_G(g_k)$.
        \item All elements of $\mc{Z}(G)$ and all $G$--parabolic subgroups are convex-cocompact.
    \end{enumerate}
\end{lem}
\begin{proof}
    Items~(1),~(2) and~(5) are respectively Lemma~2.8, Lemma~3.22(1) and Remark~3.7(6) in \cite{Fio10e}. Item~(3) is \cite[Lemma~2.4(3)]{Fio11}. Item~(4) is clear, thinking of convex-cocompactness as coarse median quasi-convexity. Finally, Item~(6) is clear for $G$--parabolics, and it follows from \cite[Section~III]{Servatius} for centralisers.
\end{proof}

\subsubsection{Singular subgroups}\label{subsub:singular}

Let $\mc{VP}(G)$ be the family of \emph{virtual products} in $G$, that is, the subgroups of $G$ that have a finite-index subgroup splitting as a nontrivial direct product (necessarily with infinite factors, since $G$ is torsion-free). A subgroup $H\leq G$ is \emph{strongly irreducible} if $H\not\in\mc{VP}(G)$.

\emph{Singular subgroups} are the maximal elements of the family $\mc{VP}(G)$ ordered by inclusion. The family of singular subgroups is denoted by $\mc{S}(G)$. Every element of $\mc{VP}(G)$ is contained in at least one element of $\mc{S}(G)$. Moreover, there are only finitely many $G$--conjugacy classes of subgroups in $\mc{S}(G)$, and they are always $G$--parabolic, regardless of the chosen convex-cocompact embedding $\iota\colon G\hookrightarrow\mc{A}_{\G}$. The family $\mc{S}(G)$ is empty precisely when $G$ is Gromov-hyperbolic. See \cite[Proposition~4.5]{Fio11} for proofs.

\subsubsection{Orthogonals}\label{subsub:orthogonals}

Consider for a moment a RAAG $\mc{A}_{\G}$. Let $\X_{\G}$ be the universal cover of the Salvetti complex of $\mc{A}_{\G}$, and let $\mscr{W}(\X_{\G})$ be its set of hyperplanes. For any convex subcomplex $C\sq\X_{\G}$, we denote by $\mscr{W}(C)\sq\mscr{W}(\X_{\G})$ the set of hyperplanes that it crosses. 

A \emph{parabolic stratum} is a subcomplex $\mc{P}\sq\X_{\G}$ such that, under the standard identification between the $0$--skeleton of $\X_{\G}$ and the group $\mc{A}_{\G}$, the vertex set of $\mc{P}$ becomes a left coset of a parabolic subgroup of $\mc{A}_{\G}$. We say that a subset of $\mscr{W}(\X_{\G})$ is \emph{parabolic} if it is of the form $\mscr{W}(\mc{P})$ for some parabolic stratum $\mc{P}$. Note that, conversely, if $C\sq\X_{\G}$ is a convex subcomplex such that the set of hyperplanes $\mscr{W}(C)$ is parabolic, then $C$ is a parabolic stratum. Each subset $\mc{W}\sq\mscr{W}(\X_{\G})$ has an \emph{orthogonal} $\mc{W}^{\perp}\sq\mscr{W}(\X_{\G})$, which is the set of hyperplanes transverse to all hyperplanes in $\mc{W}$. The orthogonal $\mc{W}^{\perp}$ is always a parabolic subset, even if $\mc{W}$ was not. 

Let $C\sq\X_{\G}$ be a convex subcomplex. For any vertex $x\in C$, we can consider the unique parabolic stratum $\mc{P}_x\sq\X_{\G}$ such that $x\in\mc{P}_x$ and $\mscr{W}(\mc{P}_x)=\mscr{W}(C)^{\perp}$; we refer to $\mc{P}_x$ as the \emph{orthogonal subcomplex} for $C$ through $x$. The isomorphism type of the cube complex $\mc{P}_x$ is independent of the choice of the vertex $x$, and we denote it by $C^{\perp}$. The union $\bigcup_{x\in C}\mc{P}_x$ is a convex subcomplex of $\X_{\G}$ isomorphic to the product $C\x C^{\perp}$.

\begin{rmk}\label{rmk:orthogonal_preserves_hyperplanes}
    If an element $g\in\mc{A}_{\G}$ leaves invariant a convex subcomplex $C\sq\X_{\G}$, then we have $g\mf{w}=\mf{w}$ for every hyperplane $\mf{w}\in\mscr{W}(C)^{\perp}\sq\mscr{W}(\X_{\G})$.
\end{rmk}

Let now $H\leq\mc{A}_{\G}$ be a convex-cocompact subgroup, and let $\mf{E}_H\sq\X_{\G}$ be some convex $H$--essential subcomplex \cite[Section~3.4]{CS11}. All the orthogonal subcomplexes for $\mf{E}_H$ through the vertices of $\mf{E}_H$ have the same $\mc{A}_{\G}$--stabiliser. This is a parabolic subgroup of $\mc{A}_{\G}$ and we denote it by $H^{\perp}$. If $G,H\leq\mc{A}_{\G}$ are convex-cocompact subgroups (not necessarily contained in each other), we write 
\[ \Perp_G(H):=H^{\perp}\cap G , \] 
and we refer to $\Perp_G(H)$ as the \emph{orthogonal} of $H$ in $G$. Alternatively, $\Perp_G(H)$ can be described as the (unique) maximal convex-cocompact subgroup of the centraliser $Z_G(H)$ that intersects $H$ trivially. 
When $H$ has trivial centre, we simply have $\Perp_G(H)=Z_G(H)$, and so orthogonal subgroups are a ``new'' notion only when $H$ has infinite centre.

When $H\leq G$, the subgroup $\Perp_G(H)$ is precisely the $G$--parabolic subgroup $P$ appearing in \Cref{lem:cc_basics}(2). Also note that we always have the inclusion $H\leq\Perp_G\Perp_G(H)$.

\begin{rmk}\label{rmk:cmp_preserves_Perp}
    Let $H\leq G\leq\mc{A}_{\G}$ be convex-cocompact subgroups and let $[\mu]$ be the induced coarse median structure on $G$. Then each coarse-median preserving automorphism $\varphi\in\Aut(G,[\mu])$ satisfies $\varphi(\Perp_G(H))=\Perp_G(\varphi(H))$, since $\varphi$ preserves convex-cocompactness of subgroups. At the same time, the previous equality can fail for general automorphisms of $G$, and the definition of the orthogonal $\Perp_G(H)$ depends on the choice of the convex-cocompact embedding $G\hookrightarrow\mc{A}_{\G}$ in general.
\end{rmk}

Since $\Perp_G(H)$ is always a $G$--parabolic subgroup of $G$, there are only finitely many $G$--conjugacy classes of subgroups of the form $\Perp_G(H)$, even though there can be infinitely many conjugacy classes of convex-cocompact subgroups $H\leq G$. In particular, there is a constant $C=C(G)$ such that, independently of the convex-cocompact subgroup $H\leq\mc{A}_{\G}$, the orthogonal $\Perp_G(H)$ acts on a convex subcomplex of $\mf{E}_G$ with at most $C$ orbits of vertices.

\subsection{Staunchly parabolic subgroups}\label{sub:staunchly_parabolic}

Let $G$ be a special group. Fixing some convex-cocompact embedding $\iota\colon G\hookrightarrow\mc{A}_{\G}$, the induced notion of $G$--parabolic subgroup can be extremely useful, but unfortunately it depends on the choice of $\iota$ and is not preserved by $\Aut(G)$ in general.

Still, certain ``particularly natural'' subgroups will always end up being $G$--parabolic:

\begin{defn}\label{defn:staunchly_parabolic}
    A subgroup $H\leq G$ is \emph{staunchly $G$--parabolic} if $H$ is $G$--parabolic with respect to {\bf every} convex-cocompact embedding of $G$ into a RAAG.
\end{defn}

Note that this property is invariant under automorphisms of $G$, and there are only finitely many conjugacy classes of staunchly parabolic subgroups. Thus, if $P\leq G$ is staunchly $G$--parabolic, then a finite-index subgroup of $\Out(G)$ will preserve the $G$--conjugacy class of $P$.

\begin{rmk}\label{rmk:SP_of_SP}
    If $P\leq G$ is staunchly $G$--parabolic and $H\leq P$ is staunchly $P$--parabolic, then $H$ is also staunchly $G$--parabolic.
\end{rmk}

Singular subgroups of $G$ are staunchly $G$--parabolic (see \Cref{subsub:singular}). In Sections~\ref{subsub:extended_factors}--\ref{subsub:salient} below, we identify further classes of subgroups with this property. The most important are the \emph{standard virtual products} and \emph{salient abelian} subgroups already mentioned in the Introduction. The elements of the hierarchy defined in \Cref{sect:hierarchy} will also be staunchly $G$--parabolic.

\subsubsection{Extended factors}\label{subsub:extended_factors}

There always exists a finite-index subgroup of $G$ that splits as a product $G_1\x\dots\x G_k\x A$ for some $k\geq 0$, where $A$ is the centre of $G$ and the $G_i$ are strongly irreducible. 
While the centre $A$ is uniquely determined, the factors $G_i$ are not: in general, there is no way of algebraically distinguishing $G_i$ from the graph of a homomorphism $G_i\ra A$. However, this is the only source of ambiguity. To show this, we need the following observation.

\begin{lem}\label{lem:special_modulo_centre}
    The quotient $G/Z_G(G)$ is special and has trivial centre.
\end{lem}
\begin{proof}
    Realise $G$ as a convex-cocompact subgroup of a RAAG $\mc{A}_{\G}$; without loss of generality $G$ is not contained in any proper parabolic subgroup of $\mc{A}_{\G}$. If the centre $Z_G(G)$ is nontrivial, then the structure of centralisers in $\mc{A}_{\G}$ \cite{Servatius} implies that we have a splitting $\mc{A}_{\G}=\mc{A}_{\Delta}\x\mc{A}_{\Lambda_1}\x\dots\x\mc{A}_{\Lambda_m}$ such that $G_0:=G\cap\mc{A}_{\Delta}$ is centreless and $G\cap\mc{A}_{\Lambda_i}\cong\Z$ for all $i$. Setting $\Lambda:=\bigcup\Lambda_i$ and $A:=G\cap\mc{A}_{\Lambda}$, we have that the subgroup $G_0\x A$ has finite index in $G$, by \Cref{lem:cc_basics}(4). Since $G_0$ is centreless and not contained in any proper parabolic subgroup of $\mc{A}_{\Delta}$, the centraliser of $G_0$ in $\mc{A}_{\Delta}$ is trivial, and hence $Z_G(G)\leq A$. Conversely, since $Z_G(G)$ is root-closed in $G$ and since $A$ is virtually generated by the subgroups $G\cap\mc{A}_{\Lambda_i}$, we must have $Z_G(G)=A=G\cap\mc{A}_{\Lambda}$. 

    Now, consider the factor projection $\pi\colon\mc{A}_{\Delta}\x\mc{A}_{\Lambda}\ra\mc{A}_{\Delta}$. The kernel of $\pi|_G$ equals $Z_G(G)$, and its image contains $G_0$ as a finite-index subgroup. Thus, $\pi(G)$ is a convex-cocompact subgroup of $\mc{A}_{\Delta}$ with trivial centre, and it is isomorphic to the quotient $G/Z_G(G)$.
\end{proof}

With the lemma in mind, we can give the following:

\begin{defn}[Extended factors]\label{defn:extended_factors}
    Let $G$ be a special group.
    \begin{itemize}
        \item If $G$ has trivial centre, we say that subgroups $E_1,\dots,E_k\leq G$ are \emph{extended factors} of $G$ if:
            \begin{enumerate}
                \item the subgroup $\langle E_1,\dots,E_k\rangle$ is the product $E_1\x\dots\x E_k$ and has finite index in $G$;
                \item each $E_i$ is strongly irreducible and root-closed in $G$.
            \end{enumerate}
        \item In general, subgroups $E_1,\dots,E_k\leq G$ are \emph{extended factors} of $G$ if they all contain the centre $Z_G(G)$ and they project to extended factors of the centreless special group $G/Z_G(G)$.
    \end{itemize}
\end{defn}

If the quotient $G/Z_G(G)$ is strongly irreducible (e.g.\ if $G$ is abelian), then $G$ is the only extended factor of itself. We now prove existence and uniqueness of extended factors in general.

\begin{lem}\label{lem:extended_factors}
    For a special group $G$, all the following hold.
    \begin{enumerate}
        \item There exists a unique collection of extended factors $E_1,\dots,E_k$ of $G$.
        \item Extended factors are staunchly $G$--parabolic.
        \item Isomorphisms between special groups take extended factors to extended factors.
        \item For each convex-cocompact embedding $\iota\colon G\hookrightarrow\mc{A}_{\G}$, there is a finite-index subgroup of $G$ of the form $G_1\x\dots\x G_k\x Z_G(G)$, with the $G_i$ strongly irreducible and $G$--parabolic with respect to $\iota$. The extended factors $E_i$ then equal the root-closures in $G$ of the subgroups $G_i\x Z_G(G)$.
    \end{enumerate}
\end{lem}
\begin{proof}
    We begin by simultaneously showing the existence of extended factors and proving Item~(4). Let $\iota\colon G\hookrightarrow\mc{A}_{\G}$ be some convex-cocompact embedding into a RAAG; assume without loss of generality that $G$ is not contained in a proper parabolic subgroup of $\mc{A}_{\G}$. 
    
    Let $\mc{A}_{\Delta_1}\x\dots\x\mc{A}_{\Delta_k}\x\mc{A}_{\Lambda_1}\x\dots\x\mc{A}_{\Lambda_m}$ be the decomposition of $\mc{A}_{\G}$ into directly indecomposable factors, grouped up so that $G_i:=G\cap\mc{A}_{\Delta_i}$ is non-cyclic and $G\cap\mc{A}_{\Lambda_j}\cong\Z$ for all $i,j$. Set $\Lambda:=\bigcup_j\Lambda_j$ and $A:=G\cap\mc{A}_{\Lambda}$. The product $G_1\x\dots\x G_k\x A$ has finite index in $G$ by \Cref{lem:cc_basics}(4), and the subgroups $G_i$ are $G$--parabolic. No factor $G_i$ is contained in a proper parabolic subgroup of $\mc{A}_{\Delta_i}$, which implies that all $G_i$ have trivial centre and are strongly irreducible (otherwise this would cause $\mc{A}_{\Delta_i}$ to split as a product). As in the proof of \Cref{lem:special_modulo_centre}, we have $A=Z_G(G)$.
    
    Now, the root closures of the subgroups $G_i\x Z_G(G)$ are precisely the $G$--parabolic groups $R_i:=G\cap (\mc{A}_{\Delta_i}\x\mc{A}_{\Lambda})$, and it is immediate to check that these satisfy the definition of extended factors. Conversely, if $E_1,\dots,E_m$ are some other collection of extended factors of $G$, then $m=k$ and each subgroup $E_i/Z_G(G)$ is commensurable to $R_i/Z_G(G)$ within $G/Z_G(G)$ after a permutation (since both families give virtual splittings of $G/Z_G(G)$ with strongly irreducible factors). Commensurable root-closed subgroups are equal, and so $E_i/Z_G(G)=R_i/Z_G(G)$ and $E_i=R_i$. 
    
    This proves Items~(1) and~(4). Finally, Item~(2) is an immediate consequence of Item~(4), while Item~(3) follows from Item~(1), given that the definition of extended factors is purely algebraic.
\end{proof}

\subsubsection{Standard virtual products}\label{subsub:standard_products}

Recall that $\mc{VP}(G)$ is the family of virtual products in $G$. Its maximal elements are singular subgroups, which are staunchly $G$--parabolic.

\begin{defn}[Children]\label{defn:children}
    Let $H\leq G$ be a staunchly $G$--parabolic subgroup. Choose a writing of some finite-index subgroup of $H$ as $H_1\x\dots\x H_k\x Z_H(H)$ with $H_i\not\in\mc{VP}(G)$ and $k\geq 0$. The \emph{children} of $H$ are all of the following subgroups:
    \begin{itemize}
        \item the centre $Z_G(H)$, provided that it has rank $\geq 2$ and that it is distinct from $H$;
        \item the root-closure of the product $H_1\x\dots\x H_{i-1}\x S\x H_{i+1}\x\dots\x H_k\x Z_H(H)$, for each index $1\leq i\leq k$ and each singular subgroup $S\in\mc{S}(H_i)$.
    \end{itemize}
\end{defn}

Children are always proper subgroups of $H$, and they always lie in $\mc{VP}(G)$. If $H$ is childless, then $H$ is either abelian or a direct product of non-elementary hyperbolic groups (with $\geq 1$ factors).

\begin{lem}\label{lem:children}
    Let $H\leq G$ be a staunchly $G$--parabolic subgroup.
    \begin{enumerate}
        \item The definition of the children of $H$ is independent of the choice of the virtual factors $H_i$.
        \item The children of $H$ are staunchly $G$--parabolic elements of $\mc{VP}(G)$.
    \end{enumerate}
\end{lem}
\begin{proof}
    In view of \Cref{rmk:SP_of_SP} is not restrictive to assume that $H=G$. Let $G_1\x\dots\x G_k\x Z_G(G)$ be a product splitting of a finite-index subgroup of $G$, with $G_i\not\in\mc{VP}(G)$ and $k\geq 0$. The centre of $G$ is staunchly $G$--parabolic by \Cref{lem:parabolics_basics}(2), so it suffices to consider a child $H$ that is the root-closure of $S\x G_2\x\dots\x G_k\x Z_G(G)$ for some $S\in\mc{S}(G_1)$. 
    
    Let $\iota\colon G\hookrightarrow\mc{A}_{\G}$ be a convex-cocompact embedding into a RAAG, and assume without loss of generality that $G$ is not contained in any proper parabolic subgroup of $\mc{A}_{\G}$. As in the proof of \Cref{lem:extended_factors}, we can write $\mc{A}_{\G}=\mc{A}_{\Delta_1}\x\dots\x\mc{A}_{\Delta_k}\x\mc{A}_{\Lambda}$ with $\mc{A}_{\Delta_i}$ irreducible and $G\cap\mc{A}_{\Lambda}=Z_G(G)$. Consider the $G$--parabolic subgroups $G_i':=G\cap\mc{A}_{\Delta_i}$, which give a finite-index subgroup of $G$ splitting as $G_1'\x\dots\x G_k'\x Z_G(G)$. Up to permuting the $G_i'$, we have that each product $G_i\x Z_G(G)$ is commensurable to the product $G_i'\x Z_G(G)$ (as their root-closures are the same extended factor of $G$). Thus, there exists a singular subgroup $S'\in\mc{S}(G_1')$ such that $S\x Z_G(G)$ and $S'\x Z_G(G)$ are commensurable subgroups of $G$. 
    
    Note that, being a singular subgroup, $S'$ is $G_1'$--parabolic, and hence also $G$--parabolic. Let $U$ be the smallest parabolic subgroup of $\mc{A}_{\Delta_1}$ containing $S'$, and let $H'$ be the root-closure in $G$ of the product $S'\x G_2\x\dots\x G_k\x Z_G(H)$. We then have $H'=G\cap (U\x\mc{A}_{\Delta_2}\x\dots\mc{A}_{\Delta_k}\x\mc{A}_{\Lambda})$, and so $H'$ is a well-defined subgroup (rather than just a subset) and it is $G$--parabolic. Now, since the child $H$ was defined as the root-closure of $S\x G_2\x\dots\x G_k\x Z_G(G)$, which is commensurable to $H'$, it follows that we have $H=H'$. In conclusion, children are staunchly $G$--parabolic and independent of the chosen factors $G_i$ (as the definition of $H'$ only depended on the choice of $\iota$).
\end{proof} 

We can now define \emph{standard virtual products} as those subgroups of $G$ obtained by starting with a maximal virtual product in $G$ and then repeatedly replacing a strongly irreducible virtual factor with a maximal virtual product within it. More precisely:

\begin{defn}\label{defn:SVP}
    The family $\SVP(G)$ of \emph{standard virtual products} is the smallest child-closed family of subgroups of $G$ containing the family of singular subgroups $\mc{S}(G)$.
\end{defn}

By \Cref{lem:children}, the elements of $\SVP(G)$ are staunchly $G$--parabolic and lie in $\mc{VP}(G)$. In particular, there are only finitely many $G$--conjugacy classes of subgroups in $\SVP(G)$. Any isomorphism $\varphi\colon G\ra H$ between special groups maps the elements of $\SVP(G)$ to those of $\SVP(H)$.

\begin{ex}\label{ex:stars_SVP}
    For a RAAG $\mc{A}_{\G}$ and a vertex $v\in\G$ with $\lk(v)\neq\emptyset$, we have:
    \[ \mc{A}_{\St(v)}\in\SVP(\mc{A}_{\G}) \quad \Longleftrightarrow \quad \mc{A}_{\lk(v)}\not\in\mc{Z}(\mc{A}_{\G}) \quad \Longleftrightarrow \quad \text{$\not\exists x\in\G\setminus\{v\}$ with $\lk(v)\sq\lk(x)$.} \]
    The equivalence of (2) and (3) is straightforward, so we focus on the equivalence of (1) and (3). Regarding $(1)\Ra(3)$, note that a vertex $x\in\G\setminus\{v\}$ with $\lk(v)\sq\lk(x)$ would yield a transvection $\tau\in\Aut(\mc{A}_{\G})$ mapping $v\mapsto vx$ and fixing $\Gamma\setminus\{v\}$ pointwise. Then, the subgroups $\tau^n(\mc{A}_{\St(v)})$ would be pairwise non-conjugate, which cannot happen if $\mc{A}_{\St(v)}\in\SVP(\mc{A}_{\G})$. 
    
    Finally, regarding $(3)\Ra(1)$, suppose that there is no $x\in\G\setminus\{v\}$ with $\lk(v)\sq\lk(x)$. Since $\mc{A}_{\St(v)}$ splits as a product, there exists a singular subgroup $\mc{A}_{\Delta}\in\mc{S}(\mc{A}_{\G})$ containing $\mc{A}_{\St(v)}$. Write $\Delta$ as a join $\Delta_1\ast\dots\ast\Delta_k\ast\kappa$, with $\Delta_i$ irreducible and $\kappa$ a clique. Without loss of generality, we have either $v\in\kappa$ or $v\in\Delta_1$. In the former case, we have $\St(v)=\Delta$ and so $\mc{A}_{\St(v)}\in\mc{S}(\mc{A}_{\G})\sq\SVP(\mc{A}_{\G})$. In the latter, we have $\Delta_2\cup\dots\cup\Delta_k\cup\kappa\sq\lk(v)$. If $\Delta_1=\{v\}$, then we again have $\St(v)=\Delta$ and we are done. Otherwise, we must have $\lk(v)\cap\Delta_1\neq\emptyset$ by the assumption of (3),
    and so $\mc{A}_{\St(v)\cap\Delta_1}$ splits as a product and we can repeat the previous argument with $\Gamma$ replaced by $\Delta_1$. After finitely many steps, this shows that $\mc{A}_{\St(v)}$ is a standard virtual product, as desired.
\end{ex}

We will also need the following notion related to standard virtual products.

\begin{defn}
    The \emph{cofactors} of an element $P\in\SVP(G)$ are defined as follows:
    \begin{itemize}
        \item if $P$ has at least two extended factors, then its cofactors are the root-closures of the subgroups generated by all extended factors of $P$ but one;
        \item if $P$ is non-abelian and has a unique extended factor (i.e.\ itself), then its only cofactor is the centre $Z_P(P)$ (which is necessarily nontrivial, since $P$ is a virtual product);
        \item if $P$ is abelian, then it has no cofactors.
    \end{itemize}
\end{defn}

Note that cofactors are always nontrivial, proper subgroups of $P$ when they exist.

\begin{lem}\label{lem:cofactors}
    Let $Q$ be a cofactor of a virtual standard product $P\in\SVP(G)$. Then $Q$ is staunchly $G$--parabolic and the quotient $N_G(Q)/Q$ is non-abelian.
\end{lem}
\begin{proof}
    Let $Q$ be a cofactor of $P\in\SVP(G)$. A finite-index subgroup of $P$ splits as $P_1\x\dots\x P_k\x Z_P(P)$ with strongly irreducible factors $P_i$. Up to permuting the $P_i$, the group $Q$ is the root-closure of the product $P_2\x\dots\x P_k\x Z_P(P)$. We have $k\geq 1$, as otherwise $P$ would have no cofactors.

    The fact that $Q$ is a staunchly $G$--parabolic subgroup is clear, arguing as in the proof of \Cref{lem:extended_factors}. As to the other statement, note that the extended factors of $P$ are all normal in $P$, and hence $Q\lhd P$. It follows that the non-abelian group $P_1$ is contained in $N_G(Q)$. Observing that $P_1\cap Q=\{1\}$, it follows that $P_1$ embeds in the quotient $N_G(Q)/Q$, showing that this is non-abelian.
\end{proof}

\subsubsection{Algebraically convex elements}\label{subsub:algconvex}

As always, let $G$ be a special group.

\begin{defn}
    An element $g\in G\setminus\{1\}$ is \emph{algebraically convex} if, for every standard virtual product $P\in\SVP(G)$ containing $g$, there exists an extended factor $E\leq P$ containing $g$.
\end{defn}

The typical source of algebraically convex elements are convex-cocompact ones, with respect to any fixed embedding in a RAAG. However, unlike convex-cocompactness, algebraic convexity is a purely algebraic notion, and so it is preserved by all automorphisms of $G$.

\begin{lem}\label{lem:cc->ac}
    Consider some $g\in G\setminus\{1\}$. If the subgroup $\langle g\rangle$ is convex-cocompact with respect to at least one convex-cocompact embedding $\iota\colon G\hookrightarrow\mc{A}_{\G}$, then $g$ is algebraically convex in $G$.
\end{lem}
\begin{proof}
    Consider some standard virtual product $P\in\SVP(G)$ containing $g$. Since $P$ is staunchly $G$--parabolic, it is in particular $G$--parabolic with respect to the embedding $\iota$. We can then write a finite-index subgroup $P'\leq P$ as $P'=P_1\x\dots\x P_k\x Z_P(P)$, where the $P_i$ are $G$--parabolic and strongly irreducible (see again the proof of \Cref{lem:extended_factors}). Denoting by $E_i$ the extended factors of $P$, we have that each $E_i$ contains the product $P_i\x Z_P(P)$ as a finite-index subgroup.

    Now, since $g$ lies in $P$, there exists some $n\geq 1$ such that $g^n\in P'$. Since $\langle g\rangle$ is convex-cocompact with respect to $\iota$, a further power $g^N$ must lie either in one of the $P_i$ or in the centre $Z_P(P)$, by \Cref{lem:cc_basics}(4). Since $Z_P(P)$ and the $P_i$ are $G$--parabolic and hence root-closed, this implies that either $g\in P_i$ or $g\in Z_P(P)$. Either way, $g$ lies in an extended factor $E_i$, proving the lemma.
\end{proof}

The main reason for our interest in algebraically convex elements (besides invariance under automorphisms) is that their centralisers have the following pleasantly rigid structure. 

\begin{prop}\label{prop:centraliser_g_algconvex}
    If $g\in G$ is algebraically convex and $Z_G(g)\not\cong\Z$, then:
    \begin{itemize}
        \item either $Z_G(g)\in\SVP(G)$;
        \item or there exists a subgroup $Q\lhd Z_G(g)$ with $Z_G(g)/Q\cong\Z$ such that $Q$ is a cofactor of an element of $\SVP(G)$ (and hence $Q$ is staunchly $G$--parabolic).
    \end{itemize}
\end{prop}
\begin{proof}
    Let $g\in G$ be algebraically convex. Assuming that $Z_G(g)\not\cong\Z$, we have that $Z_G(g)$ is a virtual product, and so $Z_G(g)$ is contained in at least one singular subgroup of $G$. Let $P\in\SVP(G)$ be a minimal standard virtual product containing $Z_G(g)$. A finite-index subgroup of $P$ splits as $P_1\x\dots\x P_k\x Z_P(P)$ with $P_i\not\in\mc{VP}(G)$. Since $g\in Z_G(g)\leq P$ and $g$ is algebraically convex, there exists an extended factor $E\leq P$ containing $g$. Say that $E$ is the root-closure of $P_1\x Z_P(P)$.
    
    If $g\in Z_P(P)$, then $Z_G(g)=Z_P(g)=P$, and so $Z_G(g)\in\SVP(G)$ and we are in the first case of the proposition. Otherwise, a power of $g$ is a product $p_1z$ for some $p_1\in P_1\setminus\{1\}$ and some $z\in Z_P(P)$. In this case, we have $Z_G(g)=Z_P(g)=Z_P(p_1)$ and this subgroup is the root-closure of the product
    \[ Z_{P_1}(p_1)\x P_2\x\dots\x P_k\x Z_P(P) .\]
    We must have $Z_{P_1}(p_1)\cong\Z$, as otherwise $Z_{P_1}(p_1)$ would be contained in a singular subgroup of $P_1$, which would give rise to a child of $P$ containing $Z_G(g)$; this would violate the fact that $P$ is a minimal element of $\SVP(G)$ containing $Z_G(g)$.

    Now, let $Q$ be the root-closure of the product $\Pi:=P_2\x\dots\x P_k\x Z_P(P)$. Thus, $Q$ is a cofactor of $P$, which implies that $Q$ is staunchly $G$--parabolic, by \Cref{lem:cofactors}. Since $\Pi$ is contained in the root-closed subgroup $Z_G(g)$, we have $Q\leq Z_G(g)$. Since $Q$ is a cofactor of $P$, we have $Q\lhd P$ and hence $Q\lhd Z_G(g)$. The quotient $Z_G(g)/Q$ is virtually cyclic because $Z_{P_1}(p_1)\cong\Z$. At the same time, this quotient is torsion-free because $Q$ is root-closed, and hence $Z_G(g)/Q\cong\Z$, as required. In conclusion, we are in the second case of the proposition.
\end{proof}

\subsubsection{Salient abelians}\label{subsub:salient}

Finally, we introduce two last families of subgroups of a special group $G$.

\begin{defn}\label{defn:salient}
    An abelian subgroup $A\leq G$ is:
    \begin{enumerate}
        \item \emph{salient} if $A\neq\{1\}$ and $A$ is the centre of some $P\in\SVP(G)$ such that $N_G(P)=P$;
        \item \emph{extra-salient} if, in addition, there do not exist any centralisers $Z\in\mc{Z}(G)$ with $A/A\cap Z\cong\Z$.
    \end{enumerate}
    We denote by $\xSal(G)\sq\Sal(G)$ the families of extra-salient and salient abelian subgroups of $G$.
\end{defn}

The main motivation for considering salient abelian subgroups is that they are the only subgroups of $G$ that can arise as centres of certain pathological line-stabilisers within the degenerations of $G$ (see \Cref{lem:line_stab_SVP}). Consequently, it is only salient (in fact, \emph{extra}-salient) abelians that can become poison subgroups of $G$ (\Cref{defn:poison}).
 
\begin{rmk}
    We have $\Sal(G)\sq\mc{Z}(G)$. Indeed, if an element $P\in\SVP(G)$ satisfies $N_G(P)=P$, then it also satisfies $Z_G(P)\leq P$ and hence $A:=Z_P(P)=Z_G(P)\in\mc{Z}(G)$.
\end{rmk}

Summing up, salient abelian subgroups are staunchly $G$--parabolic abelian centralisers, and there are at most finitely many $G$--conjugacy classes of them. 

\begin{rmk}\label{rmk:N_G(A)=Z_G(A)}
    For any $A\in\Sal(G)$, we have $N_G(A)=Z_G(A)$. (This holds more generally for any abelian subgroup that is convex-cocompact with respect to some embedding in a RAAG.) Indeed, \Cref{lem:cc_basics}(2) shows that $N_G(A)$ is virtually generated by $A$ and $Z_G(A)$. The fact that $Z_G(A)$ is root-closed and contains $A$ then implies the equality $N_G(A)=Z_G(A)$.
\end{rmk}

\section{Dehn twists}\label{sect:DTs}

\Cref{sub:DTs} covers generalities on Dehn twists, while \Cref{sub:more_DTs} is concerned with two particular kinds: \emph{centraliser twists} and \emph{ascetic twists}. There, we also introduce another type of automorphism, called a \emph{pseudo-twist}, which is not always a Dehn twist. We will show in \Cref{sect:proofs} that, for a special group $G$, pseudo-twists, centraliser twists, and ascetic twists suffice in order to virtually generate $\Out(G)$. Along the way, \Cref{sub:shunning+ascetic} discusses \emph{shunning splittings}, which are needed here to introduce pseudo-twists, and also later on in the definition of poison subgroups.

\subsection{Basic notions}\label{sub:DTs}

Let $G$ be a finitely generated group. A \emph{$G$--tree} is an action on a simplicial tree $G\acts T$ without edge-inversions. Following \cite{GL-JSJ}, we speak of an \emph{$(\mscr{A},\mscr{E})$--tree}, where  $\mscr{A}$ and $\mscr{E}$ are families of subgroups of $G$, if the $G$--stabiliser of each edge of $T$ lies in $\mscr{A}$, and if each subgroup in $\mscr{E}$ is elliptic in $T$. If $G$ is not elliptic in $T$, there is a unique smallest $G$--invariant subtree of $T$, which we denote by $\Min(G;T)$; we say that the $G$--tree $T$ is \emph{minimal} if $\Min(G;T)=T$. A \emph{splitting} is a non-elliptic minimal $G$--tree. We speak of a \emph{$1$--edge splitting} if $G$ acts edge-transitively, and of a \emph{free splitting} if all edge-stabilisers are trivial (even if the $G$--action is not free).

Consider a $1$--edge splitting $G\acts T$, pick an oriented edge $e\sq T$, and call $C$ its $G$--stabiliser. We obtain a writing of $G$ of one of two kinds, depending on whether $T$ has two or one orbits of vertices:
\begin{itemize}
    \setlength\itemsep{.2em}
    \item an \emph{amalgamated-product} $G=A\ast_C B$, where $A,B$ are the stabilisers of the vertices of $e$;
    \item an \emph{HNN extension} $G=A\ast_{\g}:=\langle A,t\mid t^{-1}ct=\gamma(c),\ \forall c\in C\rangle$, where $t\in G$ is a loxodromic element such that there exists a vertex $v\in T$ with $e=[v,tv]$, the subgroup $A$ is the stabiliser of $v$, and $\g$ is the isomorphism between the stabilisers of $e$ and $t^{-1}e$ given by conjugation by $t^{-1}$. We refer to $t$ as a \emph{stable letter} for the splitting $T$ and the choice of oriented edge $e$.
\end{itemize}
In each case, we have the following simple recipe \cite{Bass-Jiang,Levitt-GD} for constructing automorphisms of $G$. In the amalgam case, we choose an element $z\in Z_B(C)$ and define
\begin{align}\label{eq:DT_1}
    \varphi(a)&=a,\quad\forall a\in A, & \varphi(b)&=zbz^{-1},\quad\forall b\in B.
\end{align}
One can of course also pick an element $z\in Z_A(C)$ and swap the roles of the subgroups $A$ and $B$. In the HNN case, we can choose an element $z\in Z_A(C)$ and define
\begin{align}\label{eq:DT_2}
    \varphi(a)&=a,\quad\forall a\in A, & \varphi(t)&=zt\phantom{'},
\end{align}
or choose an element $z'\in Z_A(\gamma(C))$ and define
\begin{align}\label{eq:DT_3}
    \varphi(a)&=a,\quad\forall a\in A, & \varphi(t)&=tz'.
\end{align}
Note that modifying the choice of the unoriented edge underlying $e\sq T$, or that of the stable letter $t\in G$, simply amounts to composing the automorphism $\varphi$ with an inner automorphism of $G$.

\begin{defn}\label{defn:DT}
    A \emph{Dehn twist}\footnote{We called these ``DLS automorphisms'' in \cite{Fio10e}, but we now prefer the more standard term ``Dehn twist''.} with respect to the $1$--edge splitting $G\acts T$ is any automorphism $\varphi\in\Aut(G)$ of one of the forms \ref{eq:DT_1}, \ref{eq:DT_2}, \ref{eq:DT_3}. An outer class $\phi\in\Out(G)$ is a \emph{Dehn twist} if it is represented by at least one Dehn twist in the previous sense. 
\end{defn}

We refer to the element $z$ as the \emph{multiplier} of the Dehn twist $\varphi$. Note that, fixing the choice of the oriented edge $e\sq T$, the multiplier $z$ is uniquely determined in Cases~\ref{eq:DT_2} and~\ref{eq:DT_3}, while it is defined up to right multiplication by elements of the centre of $B$ in Case~\ref{eq:DT_1}.

\begin{rmk}
    In the literature, Dehn twists are often defined with the stronger requirement that $z$ be in the \emph{centre} of $C$ or $\gamma(C)$. That definition is in fact equivalent to ours, provided that we allow ourselves to modify the $1$--edge splitting $T$: we can fold edges to produce a new $1$--edge splitting in which $z$ fixes an edge. For instance, in Case~\ref{eq:DT_1}, we also have $G=(A\ast_C Z_B(z))\ast_{Z_B(z)}B$ and $\varphi$ still has Form~\ref{eq:DT_1} with respect to this new splitting. We prefer the added flexibility of \Cref{defn:DT} because it usually allows us to work with more natural splittings of $G$: this is particularly evident in the HNN case, for instance when we wish to view transvections of RAAGs
    as Dehn twists.
\end{rmk}

When discussing coarse-median preserving automorphisms, we will occasionally need criteria to determine whether a Dehn twist preserves a coarse median structure on $G$. This requires distinguishing between three types of Dehn twists, but these are unrelated to the ``centraliser twists'' and ``ascetic twists'' mentioned in the Introduction, and they are of little relevance when we are not interested in preserving a coarse median. The following terminology comes from \cite{Fio10e}, where it was chosen to parallel the naming convention for elementary automorphisms of RAAGs in \cite[Section~2.2]{CSV}. We have chosen to rename ``twists'' to ``skews'' here, in order to avoid confusion with the all-encompassing concept of a Dehn twist\footnote{We nevertheless stick to speaking of ``untwisted'' automorphisms of RAAGs, rather than ``unskewed''.}. 

Let $G$ be special and fix a convex-cocompact embedding $\iota\colon G\hookrightarrow\mc{A}_{\G}$, in order to be able to talk about convex-cocompactness and orthogonals as in \Cref{sub:prelims}.

\begin{defn}\label{defn:DT_types}
    Suppose that the $1$--edge splitting $G\acts T$ has convex-cocompact edge-stabilisers. A Dehn twist with respect to $T$ is a:
    \begin{itemize}
        \item \emph{partial conjugation} if it has Form~\ref{eq:DT_1};
        \item \emph{fold} if has Form~\ref{eq:DT_2} (resp.\ \ref{eq:DT_3}) and the multiplier lies in $\Perp_G(C)$ (resp.\ in $\Perp_G(\g(C))$);
        \item \emph{skew} if has Form~\ref{eq:DT_2} (resp.\ \ref{eq:DT_3}) and the multiplier lies in the centre of $C$ (resp.\ of $\g(C)$).
    \end{itemize}
\end{defn}

Note that the above definition depends on the coarse median structure induced on $G$ by the embedding $\iota$. One can characterise rather precisely when the three types of Dehn twists preserve this coarse median structure: this is the subject of \cite[Theorem~E]{Fio10e}, which we reproduce here.

\begin{thm}[\cite{Fio10e}]\label{thm:cmp_DT}
    Let $G\leq\mc{A}_{\G}$ be a convex-cocompact subgroup. Consider a $1$--edge splitting $G\acts T$ with convex-cocompact edge groups, and a Dehn twist $\varphi\in\Aut(G)$ with respect to $T$.
    \begin{enumerate}
        \item If $\varphi$ is a fold or partial conjugation, then $\varphi$ is coarse-median preserving.
        \item If $\varphi$ is a skew with multiplier $z$, where $\langle z\rangle$ is convex-cocompact in $G$ and $Z_G(z)$ is elliptic in $T$, then $\varphi$ is coarse-median preserving.
    \end{enumerate}
\end{thm}

On the other hand, it is easy to see that a skew $\varphi$ can never be coarse-median preserving if its multiplier $z$ has centraliser $Z_G(z)$ that is non-elliptic in $T$.

\begin{ex}
    Let $G=\pi_1(S_g)$ for the closed orientable surface $S_g$. Here all Dehn twists with respect to cyclic splittings are skews, according to \Cref{defn:DT_types}. These skews are all coarse-median preserving for $g\geq 2$, while none of them is for $g=1$. (Here, both facts happen to be independent of the chosen embedding of $G$ into a RAAG.)
\end{ex}

\subsection{Ascetic and shunning splittings}\label{sub:shunning+ascetic}

When $G$ is special, it turns out that two rather particular kinds of Dehn twists play a special role in virtually generating $\Out(G)$: these are the ``centraliser twists'' and ``ascetic twists'' mentioned in the Introduction. In order to talk about them, we first need to develop two related classes of splittings of $G$. This is the goal of the present subsection, which is also important in the definition of poison subgroups in \Cref{sub:poison}.

We argue in full generality: let $H\leq G$ be finitely generated groups and let $\mscr{E}$ be any family of subgroups of $G$. Given an action on an $\R$--tree $G\acts T$, we say that a subtree $\s\sq T$ is \emph{ascetic} if distinct $G$--translates of $\s$ always share at most one point. Ascetic lines frequently occur in degenerations of special groups, see \Cref{lem:line_addenda}(3) below. This motivates the following concept:

\begin{defn}\label{defn:ascetic_splitting}
    An HNN splitting $G\acts T$ is \emph{$(H,\mscr{E})$--ascetic} if the following hold.
    \begin{enumerate}
        \item The group $H$ is non-elliptic in $T$ and the subtree $\Min(H;T)$ is a line.
        \item The action $H\acts\Min(H;T)$ is edge-transitive and does not swap the two ends of $\Min(H;T)$.
        \item The $G$--stabiliser of each edge of $\Min(H;T)$ is contained in $H$.
        \item All subgroups in $\mscr{E}$ are elliptic in $T$.
    \end{enumerate}
\end{defn}

Items~(1)--(3) of \Cref{defn:ascetic_splitting} imply that the line $\alpha:=\Min(H;T)$ is ascetic in $T$, motivating our terminology. They also imply that the $G$--stabiliser of the line $\alpha$ is precisely the subgroup $H$, and that the $G$--stabiliser of each edge of $\alpha$ equals the kernel of the action $H\acts\alpha$. 

\begin{defn}\label{defn:isolating}
A simplicial tree $G\acts T$ 
is \emph{$(H,\mscr{E})$--isolating} if the following hold.
\begin{enumerate}
    \item There exists a vertex $v\in T$ such that $H$ equals the $G$--stabiliser of $v$.
    \item There is a proper subgroup $L\lhd H$ such that all the following hold:
    \begin{enumerate}
        \item $H/L$ is free abelian,
        \item $L$ contains all $G$--stabilisers of edges of $T$ incident to $v$,
        \item $L$ contains all elements of $\mscr{E}$ contained in $H$.
    \end{enumerate}
    \item All subgroups in $\mscr{E}$ are elliptic in $T$.
\end{enumerate}
We refer to the smallest subgroup $L$ satisfying Item~(2) as the \emph{locus} of the tree $T$. 

The tree $G\acts T$ is \emph{$(H,\mscr{E})$--shunning} if it is $(H,\mscr{E})$--isolating and its locus is minimal among loci of $(H,\mscr{E})$--isolating trees of $G$.
\end{defn}

We allow isolating trees that are not splittings, but this is only relevant when $H=G$, where $T$ can be taken to be a single point. The locus of $(G,\emptyset)$--shunning trees is then simply the root-closure of the commutator subgroup of $G$. For $H\neq G$, the group $G$ cannot be elliptic in $T$,
and we can always replace $T$ by $\Min(G;T)$ without affecting the fact that this is $(H,\mscr{E})$--isolating.

\begin{rmk}\label{rmk:folding_shunning}
    Let $G\acts T$ be an $(H,\mscr{E})$--isolating splitting with locus $L$. Then there exists another $(H,\mscr{E})$--isolating splitting $G\acts T'$ such that all edges of $T'$ incident to the $H$--fixed vertex have $G$--stabiliser \emph{equal} to $L$ (rather than simply \emph{contained} in $L$). The splitting $T'$ is obtained from $T$ by folding into a single edge each $L$--orbit of edges incident to the $H$--fixed vertex of $T$. 
    In fact, after a further folding and a collapse, we can even take $T'$ to be the Bass--Serre tree of an amalgamated-product splitting of the form $G=H\ast_LH'$ for some subgroup $H'\leq G$.
\end{rmk}

Existence of $(H,\mscr{E})$--isolating trees always implies existence of $(H,\mscr{E})$--shunning ones, since the abelianisation of $H$ is finitely generated by hypothesis, and so descending chains of loci must eventually stabilise. Note that the definition of shunning trees does not require, a priori, that loci of shunning trees be contained in \emph{all} loci of isolating trees. However, this stronger property does follow from \Cref{cor:same_locus} below.

The next two lemmas and remark clarify the relationship between ascetic and isolating splittings. We denote by $\ell(g;T)$ the translation length of an element $g\in G$ in a $G$--tree $T$.

\begin{lem}\label{lem:isolating->ascetic}
    Let $\eta\colon H\twoheadrightarrow\Z$ be an epimorphism that vanishes on the locus of some $(H,\mscr{E})$--isolating tree $G\acts T$. Then there exists an $(H,\mscr{E})$--ascetic HNN splitting $G\acts T'$ such that:
    \begin{enumerate}
        \item $\ell(h;T')=|\eta(h)|$ for all elements $h\in H$;
        \item if a subgroup $K\leq G$ is elliptic in $T$ and does not intersect any $G$--conjugate of $H\setminus\ker(\eta)$, then $K$ is also elliptic in $T'$.
    \end{enumerate}
\end{lem}
\begin{proof}
    The homomorphism $\eta$ gives an HNN splitting $H\acts R$, where $R$ is a line being acted upon by translations according to $|\eta|$. Since $H$ is the $G$--stabiliser of a vertex $v\in T$, we can refine $T$ into a splitting $G\acts T''$ by blowing up $v$ to a copy of $R$. Then, let $G\acts T'$ be obtained from $T''$ by collapsing all the ``old'' edges, that is, by collapsing the edges of $T''$ on which the collapse map $T''\ra T$ is injective. It is immediate that $T'$ is an HNN splitting of $G$ satisfying Items~(1)--(3) of \Cref{defn:ascetic_splitting}, as well as the equality $\ell(h;T')=|\eta(h)|$ for all $h\in H$. Finally, if a subgroup $K\leq G$ is elliptic in $T$ and disjoint from all conjugates of $H\setminus\ker(\eta)$, then $K$ stays elliptic in the blow-up $T''$, and hence also in its collapse $T'$. All elements of $\mscr{E}$ satisfy this property, by the definition of $(H,\mscr{E})$--isolating splitting, and so they are elliptic in $T'$, yielding Item~(4) of \Cref{defn:ascetic_splitting}.
\end{proof}

\begin{rmk}\label{rmk:ascetic->isolating}
    There is a close relationship between ascetic splittings and isolating ones. Given an $(H,\mscr{E})$--isolating splitting, \Cref{lem:isolating->ascetic} shows how to construct $(H,\mscr{E})$--ascetic ones. There is also a standard construction of $(H,\mscr{E})$--isolating splitting from $(H,\mscr{E})$--ascetic ones, as we now explain. 

    Let $G\acts T$ be an $(H,\mscr{E})$--ascetic HNN splitting and set $\alpha:=\Min(H;T)$. Since $T$ is edge-transitive and the line $\alpha$ is ascetic, the $G$--translates of $\alpha$ form a transverse covering of $T$, in the sense of \cite[Definition~1.4]{Guir-Fourier}. Assuming that $H\neq G$, this yields a bipartite simplicial splitting $G\acts\mc{S}$ constructed as follows: $\mc{S}$ has a white vertex for each $G$--translate of $\alpha$, and a black vertex for each vertex of $T$; edges of $\mc{S}$ connect incident line-vertex pairs. The splitting $G\acts\mc{S}$ is $(H,\mscr{E})$--isolating with locus equal to the kernel of the action $H\acts\alpha$.

    In fact, a similar argument constructs an isolating splitting of $G$ from any minimal geometric $\R$--tree $G\acts T$ with an ascetic subtree $\s\sq T$. (We need the tree to be geometric to ensure that each arc of $T$ intersects only finitely many $G$--translates of $\s$ \cite{LP97}.)
\end{rmk}

\begin{lem}\label{lem:ell/lox_trick}
    If $G\acts T$ is an $(H,\mscr{E})$--ascetic HNN splitting, then the kernel of $H\acts\Min(H;T)$ contains the locus of each $(H,\mscr{E})$--shunning tree of $G$.
\end{lem}
\begin{proof}
    Assume without loss of generality $H\neq G$. Let $G\acts T_1$ be an $(H,\mscr{E})$--shunning splitting with locus $L_1\lhd H$. Let $G\acts T_2$ be an $(H,\mscr{E})$--ascetic HNN splitting, set $\alpha_2:=\Min(H;T)$ and let $K_2\lhd H$ be the kernel of the action $H\acts\alpha_2$. We will show that $K_2\geq L_1$. 
    
    The subgroup $K_2$ is the $G$--stabiliser of each edge of the line $\alpha_2$. Thus, since $K_2\leq H$, the HNN splitting $T_2$ is elliptic in $T_1$, and it follows that there exists a splitting $G\acts T'$ such that:
    \begin{itemize}
        \item $T'$ refines $T_2$ and dominates $T_1$;
        \item subgroups of $G$ that are elliptic in both $T_1$ and $T_2$ are also elliptic in $T'$.
    \end{itemize}
    (See e.g.\ \cite[Proposition~2.2]{GL-JSJ}.) Let $\mu_2\colon T'\ra T_2$ and $\mu_1\colon T'\ra T_1$ denote, respectively, the collapse map to $T_2$ and a $G$--equivariant morphism to $T_1$ (where we can assume that, for each edge $e\sq T'$, either $\mu_1|_e$ is injective or $\mu_1(e)$ is a single point). Without loss of generality, there is no edge $e\sq T'$ such that both $\mu_1(e)$ and $\mu_2(e)$ are single points (otherwise simply replace $T'$ with the result of collapsing all such edges).
    
    Set $\alpha:=\Min(H;T')$. Since the subgroup $K_2$ is elliptic in both $T_1$ and $T_2$, it is also elliptic in $T'$ and, since $K_2\lhd H$, we have that the subtree $\alpha:=\Min(H;T')$ is fixed pointwise by $K_2$. In fact, since $H/K_2\cong\Z$, the tree $\alpha$ is a line and $K_2$ equals the kernel of the action $H\acts\alpha$. 
    
    The image $\mu_1(\alpha)$ is connected and $H$--invariant, so it contains the unique $H$--fixed point $v_1\in T_1$. If $\mu_1(\alpha)\neq\{v_1\}$, then $K_2$ fixes an edge of $T_1$ incident to $v_1$, which forces $L_1=K_2$ since $H/K_2\cong\Z$. In the rest of the proof, we can thus assume that $\mu_1(\alpha)=\{v_1\}$.
    
    Note that the edges $e\sq T'$ such that $\mu_1(e)$ is a single point are precisely the edges contained in the $G$--translates of $\alpha$. In one direction, we have seen that $\mu_1(\alpha)=\{v_1\}$. In the other, if $\mu_1(e)$ is a point, then $\mu_2|_e$ is injective (by assumption) and we can appeal to the fact that $G\acts T_2$ has a unique orbit of edges (being an HNN extension).
    Another important observation is that distinct $G$--translates of $\alpha$ are disjoint: this is because they get collapsed by $\mu_1$ to distinct points of $T_1$, since the $G$--stabiliser of $v_1$ is no bigger than $H$. 
    
    Now, let $G\acts T''$ be the splitting obtained from $T'$ by collapsing to a point each edge of $T'$ that gets collapsed by the morphism $\mu_1$. In view of the previous paragraph, the fibres of the collapse map $T'\ra T''$ are precisely the $G$--translates of $\alpha$.
     Let $v''\in T''$ be the point to which $\alpha$ gets collapsed, and let $\mu_1''\colon T''\ra T_1$ be the morphism to which $\mu_1$ descends. 
     
     We claim that $T''$ is $(H,\mscr{E})$--isolating with locus contained in $L_1\cap K_2$. Since $\mu_1''(v'')=v_1$, we immediately see that the $G$--stabiliser of $v''$ equals $H$. Moreover, the subgroups in $\mscr{E}$ are elliptic in both $T_1$ and $T_2$, and hence also in $T'$ and $T''$. We are only left to check Item~(2) of \Cref{defn:isolating} for the subgroup $L_1\cap K_2$. The quotient $H/(L_1\cap K_2)$ is free abelian, because $H$ is finitely generated and both $H/L_1$ and $H/K_2$ are free abelian. Every subgroup of $H$ lying in $\mscr{E}$ is contained in both $L_1$ and $K_2$, and hence in $L_1\cap K_2$. Finally, if $e''\sq T''$ is an edge incident to $v''$, then the stabiliser $G_{e''}$ fixes $\mu_1''(e'')$, which is an arc in $T_1$ based at $v_1$; it follows that $G_{e''}\leq L_1$. In fact, $G_{e''}$ also fixes the lift of $e''$ to an edge $\overline e''\sq T'$. Note that $\overline e''$ intersects $\alpha$ (here it is essential that the fibres of the collapse map $T'\ra T''$ are precisely the $G$--translates of $\alpha$ and not larger). Since $\overline e''$ intersects $\alpha$ and we have $G_{e''}\leq H$, it follows that $G_{e''}$ fixes $\alpha$ pointwise, and hence $G_{e''}\leq K_2$. (The last three sentences are the essence of the proof, see \Cref{rmk:no_ell_ell}.)
    
    In conclusion, $G\acts T''$ is an $(H,\mscr{E})$--isolating splitting with locus contained in $L_1\cap K_2$. Since $T_1$ is $(H,\mscr{E})$--shunning with locus $L_1$, we must have $L_1=L_1\cap K_2$, and hence $L_1\leq K_2$ as desired. 
\end{proof}

\begin{rmk}\label{rmk:no_ell_ell}
    If we had run the proof of \Cref{lem:ell/lox_trick} for a pair of isolating splittings $T_1,T_2$ with loci $L_1,L_2$, we would \emph{not} have found an isolating splitting with locus $\leq L_1\cap L_2$ in general. The difficulty lies in ensuring that the splitting $T''$ satisfies Item~(2b) of \Cref{defn:isolating}, which is why we had to pass through an ascetic splitting, where $H$ has an axis.
\end{rmk}

Combining the previous lemmas, we obtain an important property of shunning trees:

\begin{cor}\label{cor:same_locus}
    All $(H,\mscr{E})$--shunning trees have the same locus.
\end{cor}
\begin{proof}
    Let $L,L'$ be loci of $(H,\mscr{E})$--shunning trees. Since $H/L$ is free abelian, we can choose subgroups $K_1,\dots,K_k\lhd H$ with $H/K_i\cong\Z$ and $L=K_1\cap\dots\cap K_k$. The combination of \Cref{lem:isolating->ascetic} and \Cref{lem:ell/lox_trick} shows that $L'\leq K_i$ for all $i$, and hence $L'\leq L$. The opposite inclusion is obtained in the same way.
\end{proof}

\subsection{The elementary generators}\label{sub:more_DTs}

Let $G$ be special and let $\mscr{E}$ be a family of subgroups. Here we discuss the three kinds of ``elementary automorphisms'' that, as mentioned in the Introduction, we will use to virtually generate $\Out(G)$. Only two of these are actual Dehn twists.

We also wish to generate \emph{relative} automorphism groups. Let\footnote{This is sometimes also called a Fuchs-Rabinowitz group or a McCool group in the literature, depending on the context. The notation we use, with the superscript $t$ denoting triviality, seems to have first appeared in \cite{GL-McCool}.} $\Aut(G;\mscr{E}^t)\leq\Aut(G)$ be the group of automorphisms that coincide with inner automorphisms of $G$ on all subgroups in $\mscr{E}$ (with possibly different inner automorphisms on different subgroups). Let $\Out(G;\mscr{E}^t)$ be the projection to $\Out(G)$. We adopt the convention that arbitrary families of arbitrary subgroups of $G$ are denoted by the letter $\mscr{E}$, whereas finite families of finitely generated subgroups are denoted by $\mscr{F}$.

\subsubsection{Centraliser twists and ascetic twists}\label{subsub:centraliser+ascetic}

The first two kinds of elementary generators are Dehn twists with respect to suitable splittings of $G$.

\begin{defn}\label{defn:DT_main_types}
    A \emph{centraliser twist} relative to $\mscr{E}$ is a Dehn twist $\tau\in\Aut(G;\mscr{E}^t)$ with respect to a $1$--edge $(\mc{Z}(G),\mscr{E})$--splitting of $G$.
\end{defn}

\begin{defn}\label{defn:ascetic_twist}
    Let $A\in\Sal(G)$. An \emph{$(A,\mscr{E})$--ascetic twist} is a Dehn twist $\tau\in\Aut(G;\mscr{E}^t)$ with respect to a $(Z_G(A),\mscr{E})$--ascetic HNN splitting $G\acts T$ such that the following conditions hold:
    \begin{enumerate}
        \item if $B\in\Sal(G)$ does not contain a $G$--conjugate of $A$, then $B$ is elliptic in $T$;
        \item we have $\tau(A)=A$ and there exists $a_0\in A$ such that $\tau(z)\in z\langle a_0\rangle$ for all $z\in Z_G(A)$.
    \end{enumerate}
    We speak of \emph{ascetic twists} relative to $\mscr{E}$ if we do not wish to specify $A$. Let $\ot_{\mscr{E}}(A)\leq\Aut(G;\mscr{E}^t)$ be the group generated by $(A,\mscr{E})$--ascetic twists, and $\overline\ot_{\mscr{E}}(A)\leq\Aut(A)$ its restriction to $A$.
\end{defn}

Note that the HNN splittings giving rise to ascetic twists are rather particular: we can choose the stable letter so that it normalises an edge group. The edge groups may be infinitely generated, but they are still co-cyclic subgroups of centralisers of salient abelians. 

Item~(1) of \Cref{defn:ascetic_twist} does not play much of a role in our proofs, but we found interesting that we only use twists and splittings with this feature (see \Cref{prop:salient_vs_degenerations}). When $B\in\xSal(G)$, this property is actually automatic (see the proof of \Cref{lem:E^A_replacement} below).

In the case of RAAGs, each of the Laurence--Servatius generators of $\Out(\mc{A}_{\G})$ \cite{Servatius,Laurence,CSV} is either a centraliser twist or an ascetic twist, with the exception of graph automorphisms and inversions, which are not needed if we only wish to \emph{virtually} generate $\Out(\mc{A}_{\G})$:

\begin{ex}\label{ex:type_1}
    For a RAAG $\mc{A}_{\G}$, the Laurence--Servatius generators known as partial conjugations are always centraliser twists. The transvections $\tau_{v,w}\in\Aut(\mc{A}_{\G})$ with $\lk(v)\sq\lk(w)$ are also centraliser twists, as they are Dehn twists with respect to an HNN splitting of $\mc{A}_{\G}$ with edge group $\mc{A}_{\lk(v)}$, which is the centraliser of $\langle v,w\rangle$ since $\St(v)\cap\St(w)=\lk(v)$. 
    
    Finally, the transvections $\tau_{v,w}$ with $\St(v)\sq\St(w)$ are centraliser twists if and only if $\lk(v)$ is an intersection of stars, that is, if and only if $\mc{A}_{\lk(v)}\in\mc{Z}(\mc{A}_{\G})$. In particular, no transvections of $\Z^n$ are centraliser twists. If a transvection of $\mc{A}_{\G}$ is not a centraliser twist, then it is an ascetic twist, as we now explain. Thus, consider vertices $v,w\in\G$ such that $\St(v)\sq\St(w)$ and such that $\lk(v)$ is not an intersection of stars. Consider the clique $\kappa(v):=\{w\in\G\mid\St(v)\sq\St(w)\}$ and set $A:=\mc{A}_{\kappa(v)}$. Note that $A$ is the centre of $\mc{A}_{\St(v)}$ and, by \Cref{ex:stars_SVP}, we have $\mc{A}_{\St(v)}\in\SVP(\mc{A}_{\G})$. It is clear that $\mc{A}_{\St(v)}$ equals its own normaliser, and so $A\in\Sal(\mc{A}_{\G})$. Now, let $\mc{A}_{\G}\acts T$ be the HNN splitting with $v$ as a stable letter, $\mc{A}_{\lk(v)}$ as an edge group, and $\mc{A}_{\G\setminus\{v\}}$ as a vertex group. This is a $(Z_G(A),\emptyset)$--ascetic splitting inducing the transvection $\tau_{v,w}$ as a Dehn twist. If a parabolic subgroup $\mc{A}_{\Delta}\leq\mc{A}_{\G}$ is abelian and not elliptic in $T$, then we have $v\in\Delta$. If in addition $\mc{A}_{\Delta}$ is staunchly parabolic, then $\kappa(v)\sq\Delta$, because of the transvections $\tau_{v,w}$ with $w\in\kappa(v)$. In conclusion, the elements of $\Sal(\mc{A}_{\G})$ that do not contain a conjugate of $A$ are elliptic in $T$, showing that $\tau_{v,w}$ is an $(A,\emptyset)$--ascetic twist.
\end{ex}

\begin{rmk}
    In general, there are infinitely many elements of $\Out(G)$ that are represented by centraliser twists and ascetic twists. We will show that $\Out(G)$ is virtually generated by \emph{finitely many} centraliser twists, ascetic twists and pseudo-twists, but we cannot exhibit an explicit such finite set. Our proofs are by a shortening argument, and so they heavily rely on the axiom of choice and are non-constructive.
\end{rmk}

\subsubsection{Pseudo-twists}\label{subsub:pseudo-twists}

The elementary generators of the third kind are not always Dehn twists, though they are closely related automorphisms. They only arise from \emph{extra}-salient abelians.

Given $A\in\xSal(G)$, we denote by $\mc{L}_{\mscr{E}}(A)\lhd Z_G(A)$ the shared locus of $(Z_G(A),\mscr{E})$--shunning trees of $G$. This is well-defined by \Cref{cor:same_locus}. 

\begin{rmk}\label{rmk:degenerate_situations}
    The following two degenerate situations require further comment.
    \begin{itemize}
        \item If $G$ has no $(Z_G(A),\mscr{E})$--shunning trees (or equivalently no isolating ones), then we simply set $\mc{L}_{\mscr{E}}(A):=Z_G(A)$.
        \item If $G$ has no $(Z_G(A),\mscr{E})$--shunning splittings, then there may still be some $(Z_G(A),\mscr{E})$--shunning trees, though this can only occur for $Z_G(A)=G$. Hence $A$ must equal the centre of $G$ (since we have $A\in\mc{Z}(G)$ by definition). In this case, $\mc{L}_{\mscr{E}}(A)$ is the intersection of all root-closed co-abelian subgroups of $G$ containing the elements of $\mscr{E}$. In particular, $\mc{L}_{\emptyset}(A)$ is the root-closure of the commutator subgroup of $G$.
    \end{itemize}
\end{rmk}

If $\mc{L}_{\mscr{E}}(A)\neq Z_G(A)$ and $Z_G(A)\neq G$, then \Cref{rmk:folding_shunning} shows that $G$ admits amalgamated-product splittings of the form $G=Z_G(A)\ast_{\mc{L}_{\mscr{E}}(A)} W$ for subgroups $W\leq G$. This allows us to extend automorphisms of $Z_G(A)$ fixing $\mc{L}_{\mscr{E}}(A)$ to automorphisms of $G$, by setting them to be the identity on $W$. In general, such extensions are non-unique, as they depend on the choice of $W$. Pseudo-twists are automorphisms obtained by this extension procedure, with the additional requirement that their non-identity behaviour is ``uniquely supported'' on the abelian subgroup $A$:

\begin{defn}\label{defn:pseudo-twists}
    Let $A\in\xSal(G)$. An \emph{$(A,\mscr{E})$--pseudo-twist} is an automorphism $\pi\in\Aut(G;\mscr{E}^t)$ that satisfies $\pi(A)=A$, 
    $\pi|_N=\id_N$ and $\pi|_W=\id_W$ for two subgroups $N$ and $W$ chosen as follows:
    \begin{enumerate}
        \item $W$ fits in an amalgamated-product writing $G=Z_G(A)\ast_{\mc{L}_{\mscr{E}}(A)} W$ relative to $\mscr{E}$, where $\mc{L}_{\mscr{E}}(A)\lhd Z_G(A)$ is the locus of $(Z_G(A),\mscr{E})$--shunning trees of $G$;
        \item $N$ is a subgroup with $\mc{L}_{\mscr{E}}(A)\leq N\lhd Z_G(A)$ such that $Z_G(A)/N$ is free abelian and $A\cdot N$ has finite index in $Z_G(A)$.
    \end{enumerate} 
    We speak of \emph{pseudo-twists} relative to $\mscr{E}$ if we do not wish to specify $A$. Let $\pt_{\mscr{E}}(A)\leq\Aut(G;\mscr{E}^t)$ be the group generated by $(A,\mscr{E})$--pseudo-twists, and $\overline\pt_{\mscr{E}}(A)\leq\Aut(A)$ its restriction to $A$.
\end{defn}

\begin{rmk}
    We speak of a ``writing'' rather than a ``splitting'' in Item~(1) of \Cref{defn:pseudo-twists} because we wish to cover the degenerate situations discussed in \Cref{rmk:degenerate_situations}:
    \begin{itemize}
        \item If $\mc{L}_{\mscr{E}}(A)=Z_G(A)$, then the only $(A,\mscr{E})$--pseudo-twist is the identity of $G$.
        \item If $Z_G(A)=G$, then $(A,\mscr{E})$--pseudo-twists are those automorphisms of $G$ that fix pointwise a subgroup $N\lhd G$ containing all elements of $\mscr{E}$, with $G/N$ free abelian, and with $A\cdot N\leq G$ of finite index.
    \end{itemize}
    Outside these two exceptions, $\mc{L}_{\mscr{E}}(A)$ is required to be a \emph{proper} subgroup of both $Z_G(A)$ and $W$.
\end{rmk}

\begin{rmk}
    If the subgroup $N$ in \Cref{defn:pseudo-twists} satisfies $Z_G(A)/N\cong\Z$, then we recover \emph{precisely} the definition of $(A,\mscr{E})$--ascetic twists; this follows e.g.\ from the proof of \Cref{prop:ot<pt}(1) below. (Note, however, that ascetic twists are defined for all $A\in\Sal(G)$, while we only consider pseudo-twists when $A\in\xSal(G)$.) Except for ascetic twists, all pseudo-twists needed in the article satisfy $Z_G(A)/N\cong\Z^2$, so this requirement can safely be added to \Cref{defn:pseudo-twists} if the reader so wishes. We prefer to avoid this as it makes statements a little more awkward.
\end{rmk}

\begin{rmk}
    The reason why ascetic twists are defined for all $A\in\Sal(G)$ while pseudo-twists are only defined for $A\in\xSal(G)$ is not meant to be transparent at this point in the article. The true reason are certain pathologies of degenerations of special groups (see \Cref{lem:line_stab_SVP}), and the issues that these may cause during the shortening argument (see \Cref{rmk:shortening_motile+}). The subgroups in $\Sal(G)\setminus\xSal(G)$ are not subject to these pathologies, and so we do not need pseudo-twists to deal with them (although \Cref{defn:pseudo-twists} would make perfect sense for all $A\in\Sal(G)$ as well).
\end{rmk}

A useful property is that interactions between pseudo-twists with respect to non-conjugate elements of $\xSal(G)$ are fairly limited:

\begin{lem}\label{lem:E^A_replacement}
    Consider $A,B\in\xSal(G)$ and an $(A,\mscr{E})$--pseudo-twist $\pi\in\Aut(G;\mscr{E}^t)$. If $B$ does not contain a conjugate of $A$, then $\pi$ coincides with an inner automorphism of $G$ on $B$.
\end{lem}
\begin{proof}
    The lemma is void if $A$ is the centre of $G$, 
    and trivial if $\mc{L}_{\mscr{E}}(A)=Z_G(A)$. 
    Thus, we can assume that the amalgam $G=Z_G(A)\ast_{\mc{L}_{\mscr{E}}(A)} W$ appearing in the definition of $\pi$ has $\mc{L}_{\mscr{E}}(A)$ properly contained in both $Z_G(A)$ and $W$. Let $G\acts T$ be the Bass--Serre tree of this amalgam. 

    If $B$ were non-elliptic in $T$ then, up to conjugating $B$, the line $\Min(B;T)$ would contain the $Z_G(A)$--fixed vertex of $T$. Hence $Z_G(A)\cap B$ would lie in $\mc{Z}(G)$ and contain the kernel of the action $B\acts\Min(B;T)$, which is co-cyclic in $B$. At the same time, $Z_G(A)\cap B$ is elliptic in $T$, and hence a proper subgroup of $B$, contradicting the fact that $B$ is extra-salient. Thus, $B$ is elliptic in $T$. If $B$ is conjugate into $W$, the lemma is clear. Otherwise, up to conjugating $B$, we have $B\leq Z_G(A)$ and $B\not\leq\mc{L}_{\mscr{E}}(A)$. However, in this case the $Z_G(A)$--fixed vertex of $T$ is the \emph{unique} $B$--fixed vertex of $T$, and hence $Z_G(B)\leq Z_G(A)$. Since $A$ and $B$ are centralisers, it follows that $A=Z_GZ_G(A)\leq Z_GZ_G(B)\leq B$, contradicting our hypotheses and proving the lemma.
\end{proof}


For $N$ as in \Cref{defn:pseudo-twists}, the product $A\cdot N$ has finite index in $Z_G(A)$ and splits as $A_0\x N$ for a subgroup $A_0\leq A$. Each $\varphi\in\pt_{\mscr{E}}(A)$ preserves $A\cdot N$, fixes $N$ pointwise, and on $A_0$ it is of the form $a\mapsto\alpha(a)\eta(a)$ for some $\alpha\in\Aut(A_0)$ and some $\eta\in{\rm Hom}(A_0,A\cap N)$. Not all automorphisms of $A_0\x N$ of the latter form extend to automorphisms of $Z_G(A)$, but ``most'' of them do. To check this (\Cref{prop:ot<pt}), we need the following general observation.

\begin{lem}\label{lem:fi_image_centre}
    Let $H$ be a group and let $C$ be a subgroup of its centre. Consider $N\lhd H$ with $H/N$ free abelian, $C\cap N=\{1\}$, and $C\cdot N$ of finite index in $H$. 
    Then there exists a finite-index subgroup of $\Aut(C)$ whose elements extend to automorphisms of $H$ preserving $C$ and fixing $N$ pointwise.
\end{lem}
\begin{proof}
    Consider the quotient projections $\pi_N\colon H\ra H/N$ and $\pi_C\colon H\ra H/C$. 
    Since $C\cap N$ is trivial, the product homomorphism $\eta:=(\pi_N,\pi_C)$ is an embedding $H\hookrightarrow H/N\x H/C=:\Pi$. We have $\eta(C\cdot N)=\pi_N(C)\x\pi_C(N)\cong C\x N$, which has finite index in $\Pi$, so in particular $\eta(H)$ has finite index in $\Pi$. From now on, we identify the group $H$ with the subgroup $\eta(H)\leq\Pi$, and we identify the subgroups $C$ and $N$ with $\pi_N(C)\x\{1\}$ and $\{1\}\x\pi_C(N)$, respectively.
    
    Consider the free abelian group $\wh C:=H/N\x\{1\}\leq\Pi$, of which $C$ is a finite-index subgroup. Let $\mc{O}\leq\Aut(\Pi)$ be the group of automorphisms of the form $\alpha\times\id_{H/C}$, where $\alpha\in\Aut(\wh C)$. Note that the elements of $\mc{O}$ fix $N$ pointwise. Let $\mc{O}_C,\mc{O}_H\leq\mc{O}$ be the subgroups of automorphisms leaving $C$ and $H$ invariant, respectively. Since $C$ and $H$ have finite index in $\wh C$ and $\Pi$, respectively, it follows that $\mc{O}_C$ and $\mc{O}_H$ have finite index in $\mc{O}$. 
    Also note that the restriction homomorphism $\mc{O}_C\ra\Aut(C)$ has finite-index image: indeed $C$ contains the power subgroup $\wh C^m$ 
    for $m=[\wh C:C]$; a finite-index subgroup of $\Aut(C)$ then leaves $\wh C^m$ invariant, and thus extends to $\wh C$.
    
    Combining these observations, we obtain that the restriction homomorphism $\mc{O}_C\cap\mc{O}_H\ra\Aut(C)$ has image of finite index in $\Aut(C)$. This restriction factors through a subgroup of $\Aut(H)$ fixing $N$ pointwise, proving the lemma.
\end{proof}

The following relates the actions on $A$ by $(A,\mscr{E})$--pseudo-twists and $(A,\mscr{E})$--ascetic twists. Note that, for a subgroup $A'\leq A$, the notation $\Aut(A;(A')^t)$ simply stands for the group of automorphisms of $A$ that fix $A'$ pointwise, because $A$ is abelian.

\begin{prop}\label{prop:ot<pt}
    The following hold for all $A\in\xSal(G)$, setting $A'_{\mscr{E}}:=A\cap\mc{L}_{\mscr{E}}(A)$.
    \begin{enumerate}
        \setlength\itemsep{.2em}
        \item We have $\overline\ot_{\mscr{E}}(A)\leq\overline{\rm pt}_\mscr{E}(A)$.
        \item The subgroup $\overline{\rm pt}_\mscr{E}(A)\leq\Aut(A;(A'_{\mscr{E}})^t)$ has finite index.
     \end{enumerate}
\end{prop}
\begin{proof}
    For part~(1), consider an $(A,\mscr{E})$--ascetic twist $\tau\in\Aut(G;\mscr{E}^t)$, arising from a $(Z_G(A),\mscr{E})$--ascetic HNN splitting $G\acts T$. Let $N\lhd Z_G(A)$ be the kernel of the $Z_G(A)$--action on the line $\Min(Z_G(A);T)$, which satisfies $Z_G(A)/N\cong\Z$. Assuming without loss of generality that $\tau|_A$ is not the identity, we have that $A$ is non-elliptic in $T$, and hence $A\cdot N$ has finite index in $Z_G(A)$. Moreover, by \Cref{lem:ell/lox_trick}, we have $\mc{L}_{\mscr{E}}(A)\leq N$. Now, consider a subgroup $W\leq G$ fitting in a writing of the form $G=Z_G(A)\ast_{\mc{L}_{\mscr{E}}(A)} W$, relative to $\mscr{E}$. Since $\tau$ is the identity on $N$, which contains $\mc{L}_{\mscr{E}}(A)$, there is an automorphism $\pi\in\Aut(G;\mscr{E}^t)$ that is the identity on $W$ and coincides with $\tau$ on $Z_G(A)$. Note that $\pi$ is a pseudo-twist, and it coincides with $\tau$ on $Z_G(A)$. This shows that every $(A,\mscr{E})$--ascetic twist coincides on $A$ with an $(A,\mscr{E})$--pseudo-twist, proving part~(1).

    We now prove part~(2). Assume without loss of generality that $\mc{L}_{\mscr{E}}(A)\neq Z_G(A)$, and note that we have $\overline\pt_{\mscr{E}}(A)\leq\Aut(A;(A'_{\mscr{E}})^t)$ by definition. Since $\mc{L}_{\mscr{E}}(A)$ is root-closed by definition, the subgroup $A'_{\mscr{E}}$ is a direct summand of $A$. Choose $V\leq A$ such that $A=A'_{\mscr{E}}\oplus V$. Denote by $p\colon Z_G(A)\ra Z_G(A)/\mc{L}_{\mscr{E}}(A)$ the quotient projection, and let $\wh A_{\mscr{E}}$ be the direct summand of the free abelian group $Z_G(A)/\mc{L}_{\mscr{E}}(A)$ that contains $p(A)$ as a finite-index subgroup. Note that $p|_V$ is an isomorphism between $V$ and $p(A)\cong A/A'_{\mscr{E}}$.

    Every homomorphism $\xi\colon V\ra A'_{\mscr{E}}$ determines an automorphism $\wh\xi\in\Aut(A;(A'_{\mscr{E}})^t)$ with $\wh\xi(v)=\xi(v)v$ for all $v\in V$. Let $\Xi\leq\Aut(A;(A'_{\mscr{E}})^t)$ be the subgroup formed by all the automorphisms of the form $\wh\xi$, and let $\Lambda\leq\Aut(A;(A'_{\mscr{E}})^t)$ be the subgroup of automorphisms preserving $V$. We then have a decomposition $\Aut(A;(A'_{\mscr{E}})^t)=\Xi\rtimes\Lambda$,
    where $\Xi$ is identified with the free abelian group $\Hom(V,A'_{\mscr{E}})$, and $\Lambda$ is identified with $\Aut(V)$.
    
    \smallskip
    {\bf Claim.} \emph{There is a finite-index subgroup $\Xi_0\leq\Xi$ such that $\Xi_0\leq\overline\ot_{\mscr{E}}(A)$.}

    \smallskip\noindent
    \emph{Proof of claim.}
    Observe that any homomorphism $\eta\colon \wh A_{\mscr{E}}\ra A'_{\mscr{E}}$ gives rise to an automorphism $\varphi\in\ot_{\mscr{E}}(A)\leq\Aut(G;\mscr{E}^t)$ such that $\varphi(v)=v(\eta\o p)(v)$ for all $v\in V$. Indeed, we can write $\eta$ as the sum of finitely many homomorphisms $\eta_i\colon \wh A_{\mscr{E}}\ra A'_{\mscr{E}}$, each with infinite cyclic image. Since $\wh A_{\mscr{E}}$ is a direct summand of $Z_G(A)/\mc{L}_{\mscr{E}}(A)$, each $\eta_i$ extends to a homomorphism $\eta_i'\colon Z_G(A)\ra A'_{\mscr{E}}$ that vanishes on $\mc{L}_{\mscr{E}}(A)$, coincides with $\eta_i\o p$ on $p^{-1}(\wh A_{\mscr{E}})$, and has the same image as $\eta_i$. Each $\eta_i'$ yields a $(Z_G(A),\mscr{E})$--ascetic splitting $G\acts T_i$, using \Cref{lem:isolating->ascetic}, and this splitting produces an $(A,\mscr{E})$--ascetic Dehn twist $\tau_i\in\Aut(G;\mscr{E}^t)$ such that $\tau_i(v)=v(\eta_i\o p)(v)$ for all $v\in V$. Finally, it suffices to define $\varphi\in\ot_{\mscr{E}}(A)$ as the composition of the twists $\tau_i$.
    
    Now, since the injection $p|_V\colon V\hookrightarrow \wh A_{\mscr{E}}$ has finite-index image, the pull-back map 
    \[ p^*\colon \Hom(\wh A_{\mscr{E}},A'_{\mscr{E}})\ra\Hom(V,A'_{\mscr{E}}) \] 
    is injective with finite-index image.
    Identifying $\Xi$ with $\Hom(V,A'_{\mscr{E}})$, we define $\Xi_0$ as the image of $p^*$. The previous paragraph then shows that $\Xi_0\leq\overline\ot_{\mscr{E}}(A)$ as desired.
    \hfill$\blacksquare$

    \smallskip
    In view of part~(1) and the claim, we are only left to find a finite-index subgroup of $\Lambda\cong\Aut(V)$ that is contained in $\overline\pt_{\mscr{E}}(A)$. Write $Z_G(A)/\mc{L}_{\mscr{E}}(A)=\wh A_{\mscr{E}}\oplus U$ for a subgroup $U$. Consider the normal subgroup $N:=p^{-1}(U)\lhd Z_G(A)$, and observe that $V\cap N=\{1\}$ and that $V\cdot N$ has finite index in $Z_G(A)$. \Cref{lem:fi_image_centre} then yields a finite-index subgroup of $\Lambda^0\leq\Lambda$ whose elements extend to automorphisms of $Z_G(A)$ fixing $N$ pointwise. We can then extend these automorphisms of $Z_G(A)$ to automorphisms of $G$ as in \Cref{defn:pseudo-twists} to produce some pseudo-twists. In conclusion we have $\Lambda^0\leq\overline\pt_{\mscr{E}}(A)$, proving part~(2).
\end{proof}

As we will see in \Cref{sect:poison}, the inclusion $\overline\ot_{\mscr{E}}(A)\leq\overline{\rm pt}_\mscr{E}(A)$ can be of infinite index. In the notation of the proof of \Cref{prop:ot<pt}, this can only occur when $p(A)\neq\wh A_{\mscr{E}}$, that is, when $A$ projects to a subgroup of $Z_G(A)/\mc{L}_{\mscr{E}}(A)$ that is not root-closed.

\section{Poison subgroups}\label{sect:poison}

This section is concerned with poison subgroups (\Cref{defn:poison}). We show that poison subgroups obstruct virtual generation of $\Out(G)$ by Dehn twists (\Cref{cor:poison_obstructs_DT_generation}), and we exhibit examples of special groups whose centre is a poison subgroup (\Cref{ex:poisonous_centre}). We then show that poison subgroups can be removed by passing to finite index (\Cref{prop:no_poison_in_fi}). Thus, this section contains the proofs of \Cref{propnewintro:bad_examples}, of one arrow of \Cref{thmnewintro:poison}, and of a reduction of \Cref{thmnewintro:fi_no_poison} to \Cref{thmnewintro:poison}. The rest of the results in the Introduction are proven in \Cref{sect:proofs}.

At the end, in \Cref{sub:SOCs}, we show that we can use pseudo-twists to reduce to studying automorphisms ``behaving tamely'' on extra-salient abelians (\Cref{prop:reduction_to_preserving_complements}). This is a key result in order to prove finite generation of $\Out(G)$ even when Dehn twists do not virtually generate.

\subsection{Elementary automorphisms of abelian groups}\label{sub:elem_abelian}

As sketched in the Introduction, only salient abelian subgroups $A\in\Sal(G)$ can be declared to be poison subgroups (in fact, only \emph{extra}-salient ones). For this to happen, however, $A$ needs to pass two additional tests. The first is the existence of $(Z_G(A),\emptyset)$--shunning trees. The second requires that certain subgroups of $\SL_n(\Z)$ have infinite index. In this subsection, we discuss these latter matrix groups.

Consider an (abstract) free abelian group $A\cong\Z^k$ of rank $k\geq 1$. Let $\SL(A)\leq\Aut(A)$ be the index--$2$ subgroup of orientation-preserving automorphisms. 

\begin{defn}
    An \emph{elementary automorphism} is any $\tau_{c,C}\in\SL(A)$ defined as follows, starting from a codimension--$1$ direct summand $C\leq A$ and an element $c\in C$. Pick some $x\in A$ such that $A=\langle x\rangle\oplus C$, then set $\tau_{c,C}|_C:=\id$ and $\tau_{c,C}(x):=xc$. The element $c$ is the \emph{multiplier} of $\tau_{c,C}$. 
\end{defn}

Note that the choice of complementary element $x$ is inconsequential: any other choice lies in the cosets $x^{\pm 1}C$ and yields the exact same automorphism of $A$ (or its inverse). If we fix a basis $\mc{B}$ for $A$ and only consider $c\in\mc{B}$ and $C$ generated by basis elements, then the elementary automorphisms $\tau_{c,C}$ correspond precisely to classical ``elementary matrices'' with respect to $\mc{B}$. In particular, $\SL(A)$ is generated by finitely many elementary automorphisms.

Consider now a finite-index subgroup $A_0\leq A$. 
Let $\SL_{A_0}(A)\leq\SL(A)$ be the finite-index subgroup of automorphisms leaving $A_0$ invariant.

\begin{defn}
    An elementary automorphism $\tau_{c,C}\in\SL(A)$ is \emph{$A_0$--local} if $c\in A_0\cap C$. Denote by ${\rm Elem}^{A_0}(A)\leq\SL_{A_0}(A)$ the subgroup generated by $A_0$--local elementary automorphisms of $A$.
\end{defn}

It is important to remember that being $A_0$--local is a significantly stronger requirement on elementary automorphism than simply preserving the finite-index subgroup $A_0$ (see \Cref{ex:elem}). We also note that $\SL_{A_0}(A)$--conjugates of $A_0$--local elementary automorphisms are again $A_0$--local and elementary, and hence we have $\Elem^{A_0}(A)\lhd\SL_{A_0}(A)$.

\begin{ex}\label{ex:elem}
    Suppose that $A_0=A^m$ for some integer $m\geq 2$, that is, $A_0$ the subgroup of $m$--th powers (we use multiplicative notation even though $A$ is abelian). In this case, we have $\SL_{A_0}(A)=\SL(A)$, while the group $\Elem^{A_0}(A)$ is the normal closure in $\SL(A)$ of the group generated by $m$--th powers of elementary matrices (with respect to any choice of a basis for $A$).
\end{ex}

As we now observe, the subgroup $\Elem^{A_0}(A)$ can have infinite index in $\SL(A)$ or, equivalently, it can fail to be finitely generated.

\begin{lem}\label{lem:infinitely_generated_mth_power}
    Consider a finite-index subgroup $A_0\leq A$.
    \begin{enumerate}
        \item If $\rk(A)\neq 2$, then $\Elem^{A_0}(A)$ has finite-index in $\SL(A)$.
        \item If $\rk(A)=2$ and $A_0=A^{24m}$ with $m\geq 8000$, then $\Elem^{A_0}(A)$ has infinite index in $\SL(A)$.
    \end{enumerate}
\end{lem}
\begin{proof}
    The case when $\rk(A)\leq 1$ can be safely ignored, as $\SL(A)$ is finite there. For $k\geq 3$ and any $m\geq 1$, the group $\SL_k(\Z)$ is virtually normally generated by $m$--th powers of elementary matrices: for instance, this follows from the normal subgroup theorem \cite{Margulis_book}. (In fact, even without taking a normal closure, powers of elementary matrices virtually generate $\SL_k(\Z)$ by \cite{Tits-generators}.) 
    In other words, if $\rk(A)\neq 2$, then the subgroup $\Elem^{A^m}(A)$ has finite index in $\SL(A)$ for all $m\geq 1$. The finite-index subgroup $A_0\leq A$ satisfies $A^m\leq A_0$ for $m=[A:A_0]$, and hence $\Elem^{A^m}(A)\leq\Elem^{A_0}(A)$.

    Suppose now that $A\cong\Z^2$. The principal congruence subgroup $\mc{P}_3:=\ker(\SL_2(\Z)\ra\SL_2(\Z/3\Z))$ is torsion-free, so it is a finitely generated free group of index $24$ in $\SL_2(\Z)$. 
    All $\SL_2(\Z)$--conjugates of $24m$--th powers of elementary matrices are $m$--th powers of elements of $\mc{P}_3$, and so $\Elem^{A^{24m}}(A)$ is contained in the $m$--th power subgroup of $\mc{P}_3$. The solution to the Burnside problem \cite{Adian_Burnside,Lysenok_Burnside} then implies that $\Elem^{A^{24m}}(A)$ has infinite index in $\mc{P}_3$ for $m\geq 8000$, and so $\Elem^{A^{24m}}(A)$ is an infinitely generated free group.
\end{proof}

While the bound in \Cref{lem:infinitely_generated_mth_power}(2) suffices for our purposes, it is far from optimal. After reading a version of this article, Adrien Abgrall communicated to me that, for $A\cong\Z^2$ and $A_0=A^m$, the subgroup $\Elem^{A_0}(A)\leq\SL(A)$ has infinite index if and only if $m\geq 6$ \cite{Abgrall}.

\begin{rmk}\label{rmk:elem_fg_fi}
    For all finite-index subgroups $A_0\leq A'\leq A$, the groups $\Elem^{A_0}(A)$ and $\Elem^{A_0}(A')$ are isomorphic. 
    This is because all $A_0$--local automorphisms of $A$ leave $A'$ invariant and, conversely, all $A_0$--local automorphisms of $A'$ extend to $A_0$--local automorphisms of $A$.
\end{rmk}

\subsection{Poison and its antidotes}\label{sub:poison}

We are now ready to define poison subgroups of a special group $G$, relative to any family of subgroups $\mscr{E}$. These regulate virtual generation by Dehn twists for the relative automorphism group $\Out(G;\mscr{E}^t)\leq\Out(G)$, that is, for the group of outer classes of automorphisms that coincide with inner automorphisms of $G$ on all subgroups in $\mscr{E}$.

\subsubsection{The poison}\label{subsub:poison}

As in \Cref{subsub:pseudo-twists}, consider an extra-salient abelian subgroup $A\in\xSal(G)$, and let $\mc{L}_{\mscr{E}}(A)\lhd Z_G(A)$ be the shared locus of $(Z_G(A),\mscr{E})$--shunning trees of $G$. Denote by $\overline A_{\mscr{E}}$ the projection of $A$ to $Z_G(A)/\mc{L}_{\mscr{E}}(A)$, and let $\wh A_{\mscr{E}}$ be the root-closure of $\overline A_{\mscr{E}}$ within this quotient. Thus, $\overline A_{\mscr{E}}$ has finite index in $\wh A_{\mscr{E}}$, and $\wh A_{\mscr{E}}$ is a direct summand of the free abelian group $Z_G(A)/\mc{L}_{\mscr{E}}(A)$. In the notation of \Cref{sub:elem_abelian}, we define the \emph{$(A,\mscr{E})$--danger} group of $A$ as:
\[ \mf{D}_{\mscr{E}}(A):=\Elem^{\overline A_{\mscr{E}}}(\wh A_{\mscr{E}}) .\]

\begin{defn}\label{defn:poison}
    An \emph{$\mscr{E}$--poison subgroup} for $G$ is an element $A\in\xSal(G)$ whose danger group $\mf{D}_{\mscr{E}}(A)$ has infinite index in $\Aut(\wh A_{\mscr{E}})$. Equivalently, $\mf{D}_{\mscr{E}}(A)$ is not finitely generated.
\end{defn}

Poison subgroups as meant in the Introduction are simply $\emptyset$--poison ones in the above sense. Note that, for an extra-salient abelian subgroup $A\in\xSal(G)$ to be an $\mscr{E}$--poison subgroup, we need two fairly coincidental properties to occur simultaneously: by \Cref{lem:infinitely_generated_mth_power}, it is necessary that $\overline A_{\mscr{E}}\cong\Z^2$, and also that $\overline A_{\mscr{E}}$ is not root-closed within $Z_G(A)/\mc{L}_{\mscr{E}}(A)$. In particular, we need $\mc{L}_{\mscr{E}}(A)\neq Z_G(A)$, which requires either a $(Z_G(A),\mscr{E})$--isolating splitting, or the equality $Z_G(A)=G$.

Now, denote by $\Aut_A(G;\mscr{E}^t)\leq\Aut(G;\mscr{E}^t)$ the subgroup of automorphisms taking $A$ to itself. Its projection $\Out_A(G;\mscr{E}^t)$ is a finite-index subgroup of $\Out(G;\mscr{E}^t)$, because $\xSal(G)$ is finite up to $G$--conjugacy. Observe that we have a well-defined restriction homomorphism:
\[ \rho_{A,\mscr{E}}\colon\Out_A(G;\mscr{E}^t)\ra\Aut(\wh A_{\mscr{E}}) . \]
Indeed, any automorphism $\varphi\in\Aut_A(G;\mscr{E}^t)$ preserves both $A$ and the family $\mscr{E}$ (since, without loss of generality, $\mscr{E}$ is closed under taking $G$--conjugates). It follows that $\varphi$ preserves $Z_G(A)$ and $\mc{L}_{\mscr{E}}(A)$, and so it descends to an automorphism of $Z_G(A)/\mc{L}_{\mscr{E}}(A)$ leaving $\wh A_{\mscr{E}}$ invariant. If $\psi\in\Aut_A(G;\mscr{E}^t)$ lies in the same outer class as $\varphi$, then we have $\psi|_A=\varphi|_A$ by \Cref{rmk:N_G(A)=Z_G(A)}, and hence $\psi$ and $\varphi$ induce the same automorphism of $\wh A_{\mscr{E}}$. This shows that the homomorphism $\rho_{A,\mscr{E}}$ is well-defined.

The image of $\rho_{A,\mscr{E}}$ is contained in the finite-index subgroup of $\Aut(\wh A_{\mscr{E}})$ consisting of automorphisms preserving $\overline A_{\mscr{E}}\leq\wh A_{\mscr{E}}$. By \Cref{prop:ot<pt}(2), the images of $(A,\mscr{E})$--pseudo-twists generate a finite-index subgroup of $\Aut(\wh A_{\mscr{E}})$. As we now show, images of Dehn twists instead always lie in the $(A,\mscr{E})$--danger subgroup. This is true of \emph{completely arbitrary} Dehn twists of $G$.

\begin{lem}\label{lem:arbitrary_Dehn_twists_in_danger}
    Let $\tau\in\Aut(G;\mscr{E}^t)$ be a Dehn twist with respect to an arbitrary $1$--edge splitting $G\acts T$ relative to $\mscr{E}$. Then, for each $A\in\xSal(G)$, there exists an inner automorphism $\kappa\in{\rm Inn}(G)$ such that $\kappa\tau(A)=A$ and $\rho_{A,\mscr{E}}(\kappa\tau)\in\mf{D}_{\mscr{E}}(A)\leq\Aut(\wh A_{\mscr{E}})$.
\end{lem}
\begin{proof}
    We can assume that $A$ is non-elliptic in $T$, as otherwise $\tau$ coincides with an inner automorphism of $G$ on $A$, and the lemma is clear.
    
    Consider the line $\alpha:=\Min(A;T)$. Since $\tau$ is a Dehn twist with respect to $T$, the splitting $G\acts T$ is $\tau$--invariant: this means that there exists an isometry $\Phi\colon T\ra T$ (taking vertices to vertices) such that $\Phi(gx)=\tau(g)\Phi(x)$ for all $x\in T$ and $g\in G$. Moreover, there exist an edge $e\sq T$ and a vertex $v\in e$, such that $\Phi$ fixes $e$ and $\tau$ is the identity on the stabiliser $G_v$. Up to composing $\tau$ with an inner automorphism of $G$, we can assume that $e\sq\alpha$.
    
    Now, let $K\lhd Z_G(A)$ be the kernel of the action $Z_G(A)\acts\alpha$.
    We have $K\leq G_v$, so the twist $\tau$ is the identity on $K$, and it follows that $A\cap K\leq A\cap\tau(A)$. Combined with the fact that $A/K\cap A\cong\Z$, that $A\cap\tau(A)\in\mc{Z}(G)$ and that $A$ is extra-salient, this implies that $\tau(A)=A$.

    We are left to check that $\rho_{A,\mscr{E}}(\tau)$ lies $\mf{D}_{\mscr{E}}(A)$, for which we need to analyse the restriction of $\tau$ to $Z_G(A)$. We have seen that $\tau$ is the identity on the co-cyclic subgroup $K\lhd Z_G(A)$. Moreover, since $\tau(A)=A$ and $\Phi|_e=\id_e$, the isometry $\Phi$ fixes the line $\alpha$ pointwise. Thus, we have $gx=\tau(g)x$ for all $x\in\alpha$ and $g\in Z_G(A)$, which implies that $g^{-1}\tau(g)\in K$. For each $g\in Z_G(A)$ and $k\in K$, we also have $gkg^{-1}=\tau(gkg^{-1})=\tau(g)k\tau(g)^{-1}$, meaning that $g^{-1}\tau(g)$ lies in the centre of $K$.

    We claim that the centre of $K$ is contained in the centre of $Z_G(A)$. Otherwise, there would exist an element $u\in G$ such that $K\leq Z_G(A)\cap Z_G(u)\neq Z_G(A)$. Since centralisers are root-closed, it would follow that $Z_G(A)\cap Z_G(u)$ has infinite index in $Z_G(A)$ and, since $K$ is co-cyclic in $Z_G(A)$, we would have $K=Z_G(A)\cap Z_G(u)$ and hence $K\in\mc{Z}(G)$. Since $K\cap A\in\mc{Z}(G)$ is co-cyclic in $A$, this would again violate the fact that $A$ is extra-salient.
    
    In conclusion, the centre of $K$ is contained in the centre of $Z_G(A)$, and so it equals $K\cap A$. 
    Now, the map $Z_G(A)\ra K\cap A$ sending $g\mapsto g^{-1}\tau(g)$ is in fact a homomorphism:
    \[ (g_1g_2)^{-1}\tau(g_1g_2)= g_2^{-1}\cdot g_1^{-1}\tau(g_1)\cdot \tau(g_2) = g_1^{-1}\tau(g_1)\cdot g_2^{-1}\tau(g_2) . \]
    Summing up, we have shown that there exists a homomorphism $\eta\colon Z_G(A)\ra K\cap A$ with kernel equal to $K$ and such that $\tau(g)=g\eta(g)$ for all $g\in Z_G(A)$. 
    
    Since $T$ is relative to $\mscr{E}$, we have $\tau\in\Aut(G;\mscr{E}^t)$ and hence $\tau$ takes $\mc{L}_{\mscr{E}}(A)$ to itself. Thus, if $\mc{L}_{\mscr{E}}(A)$ is not contained in the kernel of $\eta$, then the image of $\eta$ has nontrivial intersection with $\mc{L}_{\mscr{E}}(A)$. Since the quotient $Z_G(A)/\mc{L}_{\mscr{E}}(A)$ is free abelian, the locus $\mc{L}_{\mscr{E}}(A)$ is root-closed and hence the image of $\eta$ is entirely contained in $\mc{L}_{\mscr{E}}(A)$. In this case, $\tau$ descends to the identity automorphism on the quotient $Z_G(A)/\mc{L}_{\mscr{E}}(A)$, and hence $\rho_{A,\mscr{E}}(\tau)=\id$.

    We are left to consider the case when $\mc{L}_{\mscr{E}}(A)\leq\ker(\eta)$. Here $\eta$ descends to a homomorphism $Z_G(A)/\mc{L}_{\mscr{E}}(A)\ra K\cap A$. Restricting the latter to $\wh A_{\mscr{E}}$ and composing it with the projection $A\cong\overline A_{\mscr{E}}\leq\wh A_{\mscr{E}}$, we obtain a homomorphism $\overline\eta\colon\wh A_{\mscr{E}}\ra \overline A_{\mscr{E}}$ whose image is cyclic and contained in $\ker(\overline\eta)$. The automorphism $\rho_{A,\mscr{E}}(\tau)\in\Aut(\wh A_{\mscr{E}})$ is now of the form $a\mapsto a\overline\eta(a)$, which means precisely that $\rho_{A,\mscr{E}}(\tau)\in\Elem^{\overline A_{\mscr{E}}}(\wh A_{\mscr{E}})=\mf{D}_{\mscr{E}}(A)$. This concludes the proof of the lemma.  
\end{proof} 

\begin{rmk}\label{rmk:res(ot)=D}
    \Cref{lem:arbitrary_Dehn_twists_in_danger} shows that we have $\rho_{A,\mscr{E}}(\ot_{\mscr{E}}(A))\leq\mf{D}_{\mscr{E}}(A)$ for all $A\in\xSal(G)$. In fact, we have an equality $\rho_{A,\mscr{E}}(\ot_{\mscr{E}}(A))=\mf{D}_{\mscr{E}}(A)$: this is easily shown arguing as in the claim during the proof of \Cref{prop:ot<pt}.
\end{rmk}

The following proves one implication of \Cref{thmnewintro:poison} from the Introduction.

\begin{cor}\label{cor:poison_obstructs_DT_generation}
    The following statements hold for each $A\in\xSal(G)$.
    \begin{enumerate}
        \item The subgroup $\overline\ot_{\mscr{E}}(A)\leq\overline\pt_{\mscr{E}}(A)$ has infinite index if and only if $A$ is $\mscr{E}$--poison.
        \item If $A$ is $\mscr{E}$--poison, then $\Out(G;\mscr{E}^t)$ is not virtually generated by Dehn twists relative to $\mscr{E}$.
    \end{enumerate}
\end{cor}
\begin{proof}
    Set $A'_{\mscr{E}}:=A\cap\mc{L}_{\mscr{E}}(A)$ and choose a subgroup $V\leq A$ such that $A=A'_{\mscr{E}}\oplus V$. \Cref{prop:ot<pt} shows that we have $\overline\ot_{\mscr{E}}(A)\leq\overline\pt_{\mscr{E}}(A)$, and that $\overline\pt_{\mscr{E}}(A)\leq\Aut(A;(A'_{\mscr{E}})^t)$ has finite index. Moreover, as in the proof of \Cref{prop:ot<pt}, we have a decomposition $\Aut(A;(A'_{\mscr{E}})^t)=\Xi\rtimes\Lambda$, where:
    \begin{itemize}
        \item $\Lambda\leq\Aut(A;(A'_{\mscr{E}})^t)$ is the subgroup of automorphisms preserving $V$;
        \item a finite-index subgroup $\Xi_0\leq\Xi$ is contained in $\overline\ot_{\mscr{E}}(A)$.
    \end{itemize}
    The projection $Z_G(A)\ra Z_G(A)/\mc{L}_{\mscr{E}}(A)$ gives an isomorphism between $V$ and $\overline A_{\mscr{E}}$, and so we can identify $\Lambda\cong\Aut(V)=\Aut(\overline A_{\mscr{E}})$. We can also identify a finite-index subgroup of $\Aut(\overline A_{\mscr{E}})$ with the finite-index subgroup of $\Aut(\wh A_{\mscr{E}})$ consisting of automorphisms preserving $\overline A_{\mscr{E}}$. Under these identifications, we can view the danger group $\mf{D}_{\mscr{E}}(A)\leq\Aut(\wh A_{\mscr{E}})$ as a subgroup of $\Lambda$.
    
    Now, \Cref{lem:arbitrary_Dehn_twists_in_danger} and \Cref{rmk:res(ot)=D} show that the intersection between $\overline\ot_{\mscr{E}}(A)$ and $\Lambda$ equals the danger group $\mf{D}_{\mscr{E}}(A)$. Thus, $\overline\ot_{\mscr{E}}(A)$ is commensurable to $\Xi\rtimes\mf{D}_{\mscr{E}}(A)$. It is then clear that $\overline\ot_{\mscr{E}}(A)$ has infinite index in $\overline\pt_{\mscr{E}}(A)$ exactly when $A$ is an $\mscr{E}$--poison subgroup. This proves part~(1).

    Regarding part~(2), suppose that a finite-index subgroup $\Delta\leq\Out(G;\mscr{E}^t)$ is generated by Dehn twists with respect to (arbitrary) splittings of $G$ relative to $\mscr{E}$. \Cref{lem:arbitrary_Dehn_twists_in_danger} then shows that we have $\Delta\leq\Out_A(G;\mscr{E}^t)$ for each $A\in\xSal(G)$ and that $\rho_{A,\mscr{E}}(\Delta)\leq\mf{D}_{\mscr{E}}(A)$. At the same time, the subgroup $\rho_{A,\mscr{E}}(\Out_A(G;\mscr{E}^t))$ has finite index in $\Aut(\wh A_{\mscr{E}})$ by \Cref{prop:ot<pt}(2). As a consequence, $\mf{D}_{\mscr{E}}(A)$ must have finite index in $\Aut(\wh A_{\mscr{E}})$, showing that no $A\in\xSal(G)$ can be an $\mscr{E}$--poison subgroup.
\end{proof}

We still have not shown that poison subgroups can truly arise, but we are about to rectify this. Combined with \Cref{cor:poison_obstructs_DT_generation}(2), the following example proves \Cref{propnewintro:bad_examples} from the Introduction.

\begin{ex}\label{ex:poisonous_centre}
    There exist special groups $G$ whose centre is an $\emptyset$--poison subgroup. In fact, such groups $G$ can already be found among finite-index subgroups of RAAGs. This follows by combining the following two items.
    \begin{enumerate}
    \setlength\itemsep{.2em}
        \item Let $m\geq 2$ be a (large) integer. Let $H$ be a special group with trivial centre such that the abelianisation of $H$ splits as $H_{\rm ab}=D_m\oplus V$ for some subgroup $V$ and a subgroup $D_m\cong\Z/m\Z\oplus\Z/m\Z$. (The next item explains how to construct such a group $H$.) Let $\eta\colon H\ra D_m$ be the quotient projection and let $x,y\in H$ be elements such that $\eta(x)$ and $\eta(y)$ are generators of the two cyclic summands of $D_m$. Finally, let $\langle a,b\rangle\cong\Z^2$ be a rank--$2$ free abelian group; we will write its group operation in multiplicative notation.
        
        Consider the subgroup $G\leq\langle a,b\rangle\x H$ defined by $G:=\langle a^m,b^m,ax,by,\ker(\eta)\rangle$. Note that $G$ contains the subgroup $G_0:=\langle a^m,b^m\rangle\x\ker(\eta)$ as a normal, finite-index subgroup. In particular, $G$ has finite index in the special group $\langle a,b\rangle\x H$, and so $G$ is itself special. Also note that $G\cap(\langle a,b\rangle\x\{1\})=\langle a^m,b^m\rangle$
        and hence $A:=\langle a^m,b^m\rangle$ is precisely the centre of $G$ (since $H$ has trivial centre). 
        
        Now, let $\pi_1\colon G\ra\langle a,b\rangle$ and $\pi_2\colon G\ra H$ be the two factor projections arising from the fact that $G\leq\langle a,b\rangle\x H$. The second projection realises $G$ as a (non-split) central extension:
        \[ 1\to A\to G\xrightarrow{\pi_2} H\to 1 .\]
        However, we are rather interested in the first projection: since $\pi_1(ax)=a$ and $\pi_1(by)=b$, the homomorphism $\pi_1\colon G\ra\langle a,b\rangle$ is surjective. This implies that, denoting by $\alpha$ and $\beta$ the projections of $ax$ and $by$ to the abelianisation $G_{\rm ab}$, respectively, we have $\langle\alpha,\beta\rangle\cong\Z^2$ and
        \[ G_{\rm ab}=\langle\alpha,\beta\rangle\oplus W \]
        for some subgroup $W\leq G_{\rm ab}$ (namely the projection of $\ker(\pi_1)$ to $G_{\rm ab}$). At the same time, the projection of the centre $A=\langle a^m,b^m\rangle$ to $G_{\rm ab}$ is the $m$--th power subgroup $\langle\alpha^m,\beta^m\rangle$. In other words, $\wh A_{\emptyset}=\langle\alpha,\beta\rangle$ and $\overline A_{\emptyset}=\langle\alpha^m,\beta^m\rangle$.

        In conclusion, the $(A,\emptyset)$--danger group $\mf{D}(A)$ is isomorphic to $\Elem^{(m\Z)^2}(\Z^2)\leq\SL_2(\Z)$. \Cref{lem:infinitely_generated_mth_power} then shows that $\mf{D}(A)$ is not finitely generated for large $m$, and so the centre $A=Z_G(G)$ is an $\emptyset$--poison subgroup for $G$.
        \item We are left to construct, for each integer $m\geq 2$, a special group $H$ as required by Item~(1). Let $M$ be the fundamental group of a closed hyperbolic $3$--manifold. By the combination of \cite{Kahn-Markovic,Bergeron-Wise,Agol}, we can pass to finite index and assume that $M$ is special. By \cite[Theorem~1.5]{Sun} (also see \cite{Chu-Groves}), there exists a finite-index subgroup $M_m\leq M$ whose abelianisation has $D_m$ as a direct summand, for all $m\geq 2$. We can then take $H:=M_m$.

        In order to find $H$ of finite index within a RAAG, realise instead $M_m$ as a convex-cocompact subgroup of a RAAG $\mc{A}_{\G}$. By \cite{HW08}, there is a finite-index subgroup $H\leq\mc{A}_{\G}$ that contains $M_m$ as a retract. This implies that the abelianisation $(M_m)_{\rm ab}$ is a retract of $H_{\rm ab}$, and hence $(M_m)_{\rm ab}$ is a direct summand of $H_{\rm ab}$.
    \end{enumerate}
\end{ex}

\subsubsection{The antidote}

Poison subgroups require a fair amount of coincidence in order to arise. In particular, $G$ has no $\mscr{E}$--poison subgroups if each $A\in\xSal(G)$ projects to a root-closed subgroup of $Z_G(A)/\mc{L}_{\mscr{E}}(A)$, or if this projection has rank $\neq 2$. Unfortunately, checking these conditions requires a perfect understanding of the locus $\mc{L}_{\mscr{E}}(A)$, which is only possible when one has a near-complete understanding of all possible $(Z_G(A),\mscr{E})$--isolating splittings of $G$. 

Nevertheless, poison subgroups can always be avoided by passing to finite index, as we now discuss. For the next proposition, assume without loss of generality that the family $\mscr{E}$ is closed under passing to subgroups and $G$--conjugates. Given a subgroup $H\leq G$, we denote by $\mscr{E}|_H$ the set of subgroups in $\mscr{E}$ that are contained in $H$. We say that a subgroup of $G$ is \emph{$\mscr{E}$--characteristic} if it is preserved by all elements of $\Aut(G;\mscr{E}^t)$; in particular, $\emptyset$--characteristic subgroups are precisely characteristic subgroups in the classical sense.

\begin{prop}\label{prop:no_poison_in_fi}
    There exists an $\mscr{E}$--characteristic, finite-index subgroup $G_0\lhd G$ such that $G_0$ has no $\mscr{E}|_{G_0}$--poison subgroups.
\end{prop}
\begin{proof}
    We start with several preliminary observations. Denote by $\mscr{S}(G;\mscr{E})$ the set of $\Aut(G;\mscr{E}^t)$--orbit of subgroups in $\xSal(G)$. Each such orbit is a finite union of $G$--conjugacy classes of subgroups.
    Moreover, the set $\mscr{S}(G;\mscr{E})$ is finite. There is a natural partial ordering on $\mscr{S}(G;\mscr{E})$, namely $[A]\preceq[A']$ if there exists $\varphi\in\Aut(G;\mscr{E}^t)$ such that $A\leq \varphi(A')$. We can then define the \emph{height} $h([A])$ of a class $[A]\in\mscr{S}(G;\mscr{E})$ as the length of a longest $\preceq$--chain in $\mscr{S}(G;\mscr{E})$ having $[A]$ as its maximum; height--$1$ elements are precisely the $\preceq$--minimal ones. 

    If $G'\lhd G$ is an $\mscr{E}$--characteristic finite-index subgroup then, for each $A'\in\xSal(G')$, there exists a unique element $A\in\xSal(G)$ such that $A'=A\cap G'$.
    For any $\varphi\in\Aut(G;\mscr{E}^t)$, the restriction $\varphi|_{G'}$ lies in $\Aut(G';(\mscr{E}|_{G'})^t)$ and maps $A\cap G'$ to $\varphi(A)\cap G'$. As a consequence, we have $\#\mscr{S}(G';\mscr{E}|_{G'})\leq\#\mscr{S}(G;\mscr{E})$ and,
    for $A,A'$ as above, we also have $h([A'])\leq h([A])$.

    For each $A\in\xSal(G)$, let $L(A;G)\lhd Z_G(A)$ be the locus of $(Z_G(A),\mscr{E})$--shunning trees of $G$. Consider again a finite-index $\mscr{E}$--characteristic subgroup $G'\lhd G$, and some $A'\in\xSal(G')$ with $A'=A\cap G'$ for some $A\in\xSal(G)$. If we restrict to $G'$ any shunning tree of $G$, we see that $L(A;G)\cap G'\geq L(A';G')$ (this is only an inclusion because $G'$ might have isolating splittings that do not extend to $G$). Moreover, if this inclusion is strict, then the free abelian group $Z_G(A)/L(A;G)$ has strictly lower rank than the free abelian group $Z_{G'}(A')/L(A';G')$.

    We now begin with the proof of the proposition proper. The above observations guarantee that, up to replacing $G$ with a finite-index $\mscr{E}$--characteristic subgroup, we can assume that the following conditions hold for all finite-index $\mscr{E}$--characteristic subgroups $G'\lhd G$:
    \begin{itemize}
        \item there is an $\Aut(G;\mscr{E}^t)$--equivariant bijection $\iota\colon\xSal(G')\ra\xSal(G)$ such that $A'=\iota(A')\cap G'$ for all $A'\in\xSal(G')$;
        \item $\Aut(G;\mscr{E}^t)$--orbits and $\Aut(G';(\mscr{E}|_{G'})^t)$--orbits coincide in $\xSal(G')$;
        \item for all $A'\in\xSal(G')$, we have $L(A';G')=L(\iota(A');G)\cap G'$.
    \end{itemize}
    We then define the \emph{$\mscr{E}$--poison height} $ph(G;\mscr{E})$ as the least height of a class $[A]\in\mscr{S}(G;\mscr{E})$ consisting of $\mscr{E}$--poison subgroups of $G$.
    We set $ph(G;\mscr{E})=+\infty$ if no $\mscr{E}$--poison subgroups exist. Assuming that $ph(G;\mscr{E})<+\infty$, we will construct an $\mscr{E}$--characteristic finite-index subgroup $G_0\lhd G$ such that $ph(G_0;\mscr{E}|_{G_0})>ph(G;\mscr{E})$. An iterated application of this procedure then proves the proposition.

    Let $A_1,\dots,A_k$ be representatives of the $\mscr{E}$--poison classes of height $ph(G;\mscr{E})$ in $\mscr{S}(G;\mscr{E})$. Let $G\acts T_i$ be a $(Z_G(A_i),\mscr{E})$--shunning tree of $G$; by \Cref{rmk:folding_shunning}, we can take $T_i$ to be an amalgam with $Z_G(A_i)$ as a vertex group. For simplicity, we set $L_i:=L(A_i;G)$ and $\overline A_i:=(\overline A_i)_{\mscr{E}}$, $\wh A_i:=(\wh A_i)_{\mscr{E}}$. Recall that $\overline A_i$ has finite index in $\wh A_i$, which is a direct summand of the free abelian group $Z_G(A_i)/L_i$. Since $A_i$ is $\mscr{E}$--poison, the group $\overline A_i$ is a proper subgroup of $\wh A_i$. Choose for each index $i$ a finite-index subgroup $F_i\leq Z_G(A_i)/L_i$ that contains $\overline A_i$ as a direct summand. Consider the quotient projection
    \[ \eta_i\colon Z_G(A_i)\twoheadrightarrow (Z_G(A_i)/L_i)/F_i ;\]
    thus, $\ker(\eta_i)$ is a finite-index subgroup of $Z_G(A_i)$ containing $A_i$. By means of the tree $T_i$, we can extend $\eta_i$ to a homomorphism $\wt\eta_i\colon G\ra{\rm im}(\eta_i)$ that vanishes on all subgroups of $G$ that fix at least one vertex of $T_i$ not fixed by any conjugate of $Z_G(A_i)$. All $A\in\xSal(G)$ with $[A_i]\not\preceq[A]$ are of the latter kind, as shown in the proof of \Cref{lem:E^A_replacement}, and so $\wt\eta_i$ vanishes on all of them.
    
    Define $G_i:=\ker\widetilde\eta_i$. Let $G_0$ be the intersection of $G_1,\dots,G_k$ and all the subgroups in their $\Aut(G;\mscr{E}^t)$--orbit. It is clear that $G_0$ is a finite-index, $\mscr{E}$--characteristic subgroup of $G$. 
    
    For each index $1\leq i\leq k$, we have $A_i\leq G_i$ and $L_i\leq G_i$, and the group $A_i$ projects to a root-closed subgroup of the quotient $Z_{G_i}(A_i)/L_i\cong F_i$. Moreover, we have $A\leq G_i$ for all $A\in\xSal(G)$ with $[A_i]\not\preceq [A]$.

    Now, consider some $A\in\xSal(G)$ with $h([A])\leq ph(G;\mscr{E})$, which implies that $A\leq G_0$. The itemised properties above imply that $L(A;G_0)=L(A;G)\cap G_0$, while it is clear that $L(A;G)\cap G_0=L(A;G)\cap Z_{G_0}(A)$. Thus, the quotient $Z_{G_0}(A)/L(A;G_0)$ is naturally identified with the image of the finite-index subgroup $Z_{G_0}(A)\leq Z_G(A)$ in the quotient $Z_G(A)/L(G;A)$. Using \Cref{rmk:elem_fg_fi}, this implies that, if $A$ was not $\mscr{E}$--poison for $G$, then $A$ is still not $\mscr{E}|_{G_0}$--poison for $G_0$. If instead $A=A_i$, then the projection of $A_i$ to $Z_{G_0}(A_i)/L(A_i;G_0)$ has become root-closed: this is true in $G_i$ by construction, and the previous argument shows that this property is not lost when passing from $G_i$ to $G_0$. In particular, $A_i$ is no longer $\mscr{E}|_{G_0}$--poison for $G_0$.

    Summing up, each $A\in\xSal(G)$ with $h([A])\leq ph(G;\mscr{E})$ is contained in $G_0$ and not $\mscr{E}|_{G_0}$--poison for $G_0$. In other words, we have $ph(G_0;\mscr{E}|_{G_0})>ph(G;\mscr{E})$. A finite iteration of this procedure eventually yields an $\mscr{E}$--characteristic finite-index subgroup $G'\lhd G$ with $ph(G';\mscr{E}|_{G'})>\#\mscr{S}(G;\mscr{E})$, which must mean that $ph(G';\mscr{E}|_{G'})=+\infty$, that is, that $G'$ has no $\mscr{E}|_{G'}$--poison subgroups.
\end{proof}

\subsection{Systems of complements}\label{sub:SOCs}

Extra-salient abelian subgroups are often problematic when studying automorphisms of special groups. We have seen above that they can be $\mscr{E}$--poison subgroups, but they can also cause other issues in \Cref{sect:shortening} during the shortening argument. For this reason, it is usually desirable to reduce to automorphisms ``behaving tamely'' on extra-salient abelians, and this is precisely the goal of this subsection.

Let $G$ be a special group with a family of subgroups $\mscr{E}$. As above, it is convenient to set $A'_{\mscr{E}}:=A\cap\mc{L}_{\mscr{E}}(A)$ for all $A\in\xSal(G)$, where $\mc{L}_{\mscr{E}}(A)$ is the locus of $(Z_G(A),\mscr{E})$--shunning trees of $G$.

\begin{defn}\label{defn:SOC}
    An \emph{$\mscr{E}$--system of complements} is the data $\mf{C}$ of a subgroup $\mf{C}_A\leq A$ for each extra-salient abelian subgroup $A\in\xSal(G)$, such that the following two conditions hold:
    \begin{enumerate}
        \item $A=A'_{\mscr{E}}\oplus\mf{C}_A$ for all $A\in\xSal(G)$;
        \item $\mf{C}_{gAg^{-1}}=g\mf{C}_Ag^{-1}$ for all $g\in G$ and $A\in\xSal(G)$.
    \end{enumerate}
\end{defn}

Note that $\mscr{E}$--system of complements always exist. Indeed, recalling that $Z_G(A)/\mc{L}_{\mscr{E}}(A)$ is free abelian by definition, it follows that $A'_{\mscr{E}}$ is a root-closed subgroup of $A$, and hence a direct summand. This ensures that we can find a decomposition as in Item~(1) of \Cref{defn:SOC} for each extra-salient abelian subgroup. The fact that the conjugation action $N_G(A)\acts A$ is trivial (\Cref{rmk:N_G(A)=Z_G(A)}) then implies that Item~(2) can also be ensured.

Note that, viewing $\mf{C}=\{\mf{C}_A\}_A$ as a family of subgroups of $G$, we can consider the relative automorphism group $\Aut(G;\mf{C}^t)\leq\Aut(G)$, as defined at the start of \Cref{sub:more_DTs}.

\begin{prop}\label{prop:reduction_to_preserving_complements}
    Given an $\mscr{E}$--system of complements $\mf{C}$, we can write $\Aut(G;\mscr{E}^t)=\mc{U}_{\mf{C}}\cdot\mc{V}_{\mf{C}}$ for a subset $\mc{U}_{\mf{C}}\sq\Aut(G;\mscr{E}^t)$ and a subgroup $\mc{V}_{\mf{C}}\leq\Aut(G;\mscr{E}^t)$ satisfying all of the following.
    \begin{enumerate}
        \item The set $\mc{U}_{\mf{C}}$ is a finite union of left cosets 
        of the subgroup $\Aut(G;(\mf{C}\cup\mscr{E})^t)\leq\Aut(G;\mscr{E}^t)$. 
        \item The subgroup $\mc{V}_{\mf{C}}$ is generated by finitely many pseudo-twists of $G$ relative to $\mscr{E}$.
        \item If $G$ has no $\mscr{E}$--poison subgroups, then $\mc{V}_{\mf{C}}$ is generated by ascetic twists of $G$ relative to $\mscr{E}$.
    \end{enumerate}
\end{prop}
\begin{proof}
    Let $A_1,\dots,A_m$ be representatives of the finitely many conjugacy classes of subgroups in $\xSal(G)$. By \Cref{lem:cc_basics}(3), we can place a partial order on $\{1,\dots,m\}$ for which we have $i\preceq j$ exactly when $A_j$ contains a conjugate of $A_i$. We then arbitrarily extend this partial order to a total order and assume, without loss of generality, that it is precisely the standard order on the integers $1,\dots,m$. Set for simplicity $\mf{C}_i:=\mf{C}_{A_i}$, $\mscr{C}_i:=\mscr{E}\cup\{\mf{C}_1,\dots,\mf{C}_i\}$ and $A_i':=(A_i)_{\mscr{E}}'$.

    \Cref{lem:E^A_replacement} shows that each $(A_i,\mscr{E})$--pseudo-twist coincides with inner automorphisms of $G$ on $A_1,\dots,A_{i-1}$ so, in particular, we have $\pt_{\mscr{E}}(A_i)\leq\Aut(G;\mscr{C}_{i-1}^t)$. By \Cref{prop:ot<pt}(2), the subgroup of restrictions of $(A_i,\mscr{E})$--pseudo-twists $\overline\pt_{\mscr{E}}(A_i)\leq\Aut(A_i;(A_i')^t)$ has finite index. Since $\Aut(A_i;(A_i')^t)$ is finitely generated (it is a parabolic subgroup of ${\rm GL}_m(\Z)$ for some $m\geq 0$), the group $\overline\pt_{\mscr{E}}(A_i)$ is also finitely generated. In conclusion, there exists a finitely generated subgroup $\mc{O}_i\leq\pt_{\mscr{E}}(A_i)\leq\Aut(G;\mscr{C}_{i-1}^t)$ mapping onto a finite-index subgroup of $\Aut(A_i;(A_i')^t)$ under the restriction to $A_i$. If $A_i$ is not an $\mscr{E}$--poison subgroup, then \Cref{cor:poison_obstructs_DT_generation}(1) shows that $\mc{O}_i$ can even be found within the group $\ot_{\mscr{E}}(A_i)$ generated by $(A_i,\mscr{E})$--ascetic twists.

    Now, note that if $\psi,\psi'$ are any two automorphisms of $A_i$ taking $A_i'$ to itself, then there exists an element $\zeta\in\Aut(A_i;(A_i')^t)$ such that $\psi'=\psi\o\zeta$ on $\mf{C}_i$. Also recall that $\Aut(G;\mscr{E}^t)$ permutes the $G$--conjugacy classes of the subgroups $A_1,\dots,A_k$, and it permutes in the same way the loci $\mc{L}_{\mscr{E}}(A_i)$ and consequently the subgroups $A_i'$. It follows that there exists a finite set $\mc{H}_i$ of homomorphisms $\mf{C}_i\ra G$ such that, for each automorphism $\varphi\in\Aut(G;\mscr{E}^t)$, there exists an automorphism $\zeta\in\mc{O}_i$ such that the restriction $\varphi\zeta|_{\mf{C}_i}$ lies in $\mc{H}_i$ up to composing with an inner automorphism of $G$.

    Arguing by induction on $k$, this shows the following: Each automorphism $\varphi\in\Aut(G;\mscr{E}^t)$ admits automorphisms $\zeta_i\in\mc{O}_i\leq\Aut(G;\mscr{C}_{i-1}^t)$ such that, for all indices $1\leq j\leq k\leq m$, the composition $\varphi\zeta_1\zeta_2\dots\zeta_k$ has restriction to $\mf{C}_j$ lying in $\mc{H}_j$ up to composing with inner automorphisms of $G$.

    To conclude, let $\mc{V}_{\mf{C}}\leq\Aut(G;\mscr{E}^t)$ be the subgroup generated by $\mc{O}_1\cup\dots\cup\mc{O}_m$, and note that it satisfies Items~(2) and~(3) of the proposition. Let $\mc{U}_{\mf{C}}\sq\Aut(G;\mscr{E}^t)$ be the subset of automorphisms whose restriction to each $\mf{C}_i$ lies in $\mc{H}_i$ up to composition with inner automorphisms of $G$. The previous paragraph with $k=m$ shows that $\Aut(G;\mscr{E}^t)=\mc{U}_{\mf{C}}\cdot\mc{V}_{\mf{C}}$, so we are only left to check Item~(1). Note that there is a map $\rho\colon\mc{U}_{\mf{C}}\ra\prod_i\mc{H}_i$ keeping track of the restrictions to the $\mf{C}_i$. If $\psi,\psi'$ are automorphisms in the same fibre of $\rho$, then $\psi^{-1}\psi'$ coincides with an inner automorphism of $G$ on each $\mf{C}_i$, 
    that is, we have $\psi^{-1}\psi'\in\Aut(G;(\mf{C}\cup\mscr{E})^t)$. This shows that each fibre of $\rho$ is a left coset of $\Aut(G;(\mf{C}\cup\mscr{E})^t)$ and, since the product $\prod_i\mc{H}_i$ is finite, $\mc{U}_{\mf{C}}$ is the union of finitely many such cosets. This yields Item~(1) of the proposition, concluding the proof.
\end{proof}

\Cref{prop:reduction_to_preserving_complements} reduces the proof of Theorems~\ref{thmnewintro:fg} and~\ref{thmnewintro:poison} to the case of automorphisms acting trivially on some system of complements $\mf{C}$. More precisely, in order to show that $\Out(G;\mscr{E}^t)$ is finitely generated, it suffices to prove that $\Out(G;(\mscr{E}\cup\mf{C})^t)$ is finitely generated. Similarly, in order to show that $\Out(G;\mscr{E}^t)$ is virtually generated by Dehn twists (in the absence of $\mscr{E}$--poison subgroups), it is enough to prove that $\Out(G;(\mscr{E}\cup\mf{C})^t)$ is virtually generated by Dehn twists.

\section{The hierarchy}\label{sect:hierarchy}

The shortening argument in the proof of the main theorems (\Cref{thm:shortening}) works its way up a canonical family $\mf{H}(G)$ of subgroups of $G$. This section is devoted to the construction of this \emph{hierarchy} $\mf{H}(G)$ (\Cref{defn:hierarchy}) and to its properties. At the end, in \Cref{sub:reference_systems}, we discuss \emph{reference systems}, which are ways of ordering $\mf{H}(G)$ so as to produce a well-ordering on $\Out(G)$.

\subsection{Junctures}\label{sub:junctures}

Let $G$ be a special group. Before introducing the hierarchy $\mf{H}(G)$, we need to discuss one more family of subgroups of $G$. Given a subgroup $H\leq G$, we first define the family
\[ \mc{J}_G'(H):=\{H\cap Z_G(g)\mid g\in G\setminus N_G(H)\} .\]
We then denote by $\mc{J}_G(H)\sq\mc{J}_G'(H)$ the sub-family of maximal elements of $\mc{J}_G'(H)$ (under inclusion). We refer to the subgroups in $\mc{J}_G(H)$ as the \emph{junctures} of $H$ within $G$.

\begin{lem}\label{lem:junctures}
    Fix a convex-cocompact embedding $\iota\colon G\hookrightarrow\mc{A}_{\G}$. Let $H\leq G$ be a convex-cocompact subgroup such that the normaliser $N_G(H)$ is root-closed in $G$. Then all the following hold.
    \begin{enumerate}
        \item The elements of $\mc{J}_G(H)$ are $H$--parabolic.
        \item For $J\in\mc{J}_G(H)$, there exists a convex-cocompact element $g\in G\setminus N_G(H)$ with $J=H\cap Z_G(g)$.
        \item For all $\varphi\in\Aut(G)$, we have $\varphi(\mc{J}_G(H))=\mc{J}_G(\varphi(H))$.
    \end{enumerate}
\end{lem}
\begin{proof}
    Part~(3) is clear, since the definition of the family $\mc{J}_G(\cdot)$ is purely algebraic. Regarding part~(2), consider a subgroup $J\in\mc{J}_G(H)$ and an element $g\in G\setminus N_G(H)$ such that $J=H\cap Z_G(g)$. By \Cref{lem:cc_basics}(5), there exist $m\geq 1$ and convex-cocompact elements $g_1,\dots,g_k\in G$ such that $g^m=g_1\dots g_k$ and $Z_G(g)=Z_G(g_1)\cap\dots\cap Z_G(g_k)$. Since $N_G(H)$ is root-closed, we have $g^m\not\in N_G(H)$ and thus $g_1\not\in N_G(H)$ after possibly permuting the $g_i$. Since $J\leq H\cap Z_G(g_1)$, the maximality of $J$ within $\mc{J}_G'(H)$ implies that $J=H\cap Z_G(g_1)$, yielding part~(2).

    Finally, we show part~(1). Since $g_1$ is convex-cocompact, \Cref{lem:cc_basics}(2) implies that the centraliser $Z_G(g_1)$ virtually splits as $P\x\langle g_1\rangle$, for a $G$--parabolic subgroup $P$. Since $H$ is convex-cocompact and $H\cap\langle g_1\rangle=\{1\}$, 
    it follows that the $H$--parabolic subgroup $H\cap P$ has finite index in $H\cap Z_G(g_1)=J$. In fact, since $H$--parabolics are root-closed in $H$, the subgroup $J$ must itself be $H$--parabolic.
\end{proof}

Thus, if $H\leq G$ is staunchly $G$--parabolic, then all elements of $\mc{J}_G(H)$ are staunchly $G$--parabolic.

\begin{rmk}
    The important property of junctures is part~(2) of \Cref{lem:junctures}. However, we could not have simply defined junctures by this property, because convex-cocompactness of elements is not automorphism-invariant. Hence the more roundabout definition through the family $\mc{J}_G'(H)$.
\end{rmk}

\subsection{Definition of the hierarchy}\label{sub:hierarchy}

Let $G$ be a special group. The following definition relies on terminology and notation from \Cref{defn:extended_factors}, \Cref{defn:children} and \Cref{sub:junctures}.

\begin{defn}\label{defn:hierarchy}
The \emph{hierarchy} $\mf{H}(G)$ is the smallest family of subgroups of $G$ that satisfies the following closure properties:
\begin{enumerate}
    \item[$(H0)$] we have $G\in\mf{H}(G)$;
    \item[$(H1)$] if $H\in\mf{H}(G)$, then the centre $Z_H(H)$ and all extended factors of $H$ lie in $\mf{H}(G)$; 
    \item[$(H2)$] if $H\in\mf{H}(G)$ and $H/Z_H(H)$ is strongly irreducible, then $\mf{H}(G)$ contains the children of $H$;
    \item[$(H3)$] if $H\in\mf{H}(G)$ and $H/Z_H(H)$ is strongly irreducible, then $\mc{J}_G(H)\sq\mf{H}(G)$.
\end{enumerate}
\end{defn}

\begin{rmk}
    Property~$(H2)$ only applies to subgroups $H$ such that $H/Z_H(H)$ is strongly irreducible and, as a consequence, we have $\SVP(G)\not\sq\mf{H}(G)$ in general. This is important, as standard virtual products do not always satisfy part~(2) of \Cref{prop:hierarchy_properties} below.
\end{rmk}

\begin{rmk}
    The need for Property~$(H3)$ in \Cref{defn:hierarchy} may appear a little mysterious. It is there uniquely to ensure that Step~4 in the proof of \Cref{thm:shortening} works. Roughly, it allows us to reduce to $\R$--trees with trivial arc-stabilisers during the shortening argument.
\end{rmk}

The following summarises the most important properties of the hierarchy.

\begin{prop}\label{prop:hierarchy_properties}
    For all special groups $G$:
    \begin{enumerate}
        \item the hierarchy $\mf{H}(G)$ is $\Aut(G)$--invariant and consists of staunchly $G$--parabolic subgroups;
        \item for each $H\in\mf{H}(G)$, the extended factors of $H$ all lie in $\mc{Z}(G)$.
    \end{enumerate}
\end{prop}
\begin{proof}
    For a staunchly $G$--parabolic subgroup $H\leq G$, let $\Delta(H)$ be the family of subgroups of $H$ defined by the following list of requirements: 
    \begin{itemize}
        \item we always have $Z_H(H)\in\Delta(H)$;
        \item if $H/Z_H(H)$ is not strongly irreducible, then $\Delta(H)$ contains the extended factors of $H$;
        \item if $H/Z_H(H)$ is strongly irreducible, then $\Delta(H)$ contains the children and junctures of $H$.
    \end{itemize}
    We then inductively define $\Delta^0(G):=\{G\}$ and $\Delta^{n+1}(G)=\bigcup_{H\in\Delta^n(G)}\Delta(H)$ for $n\geq 0$. Straight from definitions, we have $\mf{H}(G)=\bigcup_{n\geq 0}\Delta^n(G)$.

    For all $\varphi\in\Aut(G)$, we have $\Delta(\varphi(H))=\varphi(\Delta(H))$: this follows from \Cref{lem:extended_factors}(3), \Cref{lem:children}(1) and \Cref{lem:junctures}(3). If $H$ is staunchly $G$--parabolic, then so is its normaliser $N_G(H)$ by \Cref{lem:parabolics_basics}(2), and hence $N_G(H)$ is root-closed. It then follows that all elements of $\Delta(H)$ are staunchly $G$--parabolic, by combining \Cref{lem:parabolics_basics}(2), \Cref{lem:extended_factors}(2), \Cref{lem:children}(2) and \Cref{lem:junctures}(1) with \Cref{rmk:SP_of_SP}. Finally, arguing by induction on $n$, we obtain that each family $\Delta^n(G)$ is $\Aut(G)$--invariant and consists of staunchly $G$--parabolic subgroups. This proves part~(1).

    We are left to prove part~(2). It will be useful to remember that, for each $Z\in\mc{Z}(G)$, the extended factors of $Z$ and the centre $Z_Z(Z)$ all lie in $\mc{Z}(G)$. Also note that any special group has the same centre as each of its extended factors.
    
    Now, consider $H\in\Delta^n(G)$; we will prove that the extended factors of $H$ lie in $\mc{Z}(G)$ arguing by induction on $n$. The base case $n=0$ is clear. Thus, pick $K\in\Delta^{n-1}(G)$ such that $H\in\Delta(K)$ and suppose that the statement has been proven for $K$. If $H$ is the centre or an extended factor of $K$, then it follows from the previous observations that $H\in\mc{Z}(G)$. In the remaining cases, we can assume that $K/Z_K(K)$ is strongly irreducible, and so $K$ is an extended factor of itself, and we have $K\in\mc{Z}(G)$ by the inductive assumption. If $H\in\mc{J}_G(K)$, it is immediate from definitions that $H\in\mc{Z}(G)$, and so the extended factors of $H$ also lie in $\mc{Z}(G)$. 
    
    We are left to consider the case when $H$ is a child of $K$. A finite-index subgroup of $K$ splits as $K'\x Z_K(K)$ with $K'$ strongly irreducible,
    and $H$ virtually splits as $S\x Z_K(K)$ with $S\in\mc{S}(K')$. Let $S_1\x\dots\x S_k\x Z_S(S)$ be a finite-index subgroup of $S$ with strongly irreducible factors $S_i$. If $k\leq 1$, then $Z_S(S)\neq\{1\}$ and each $s\in Z_S(S)\setminus\{1\}$ satisfies $S=Z_{K'}(s)$. In this case, the group $H=K\cap Z_G(s)$ lies in $\mc{Z}(G)$, and so do its extended factors. If instead $k\geq 2$, consider an extended factor $H_1$ of $H$, without loss of generality the root-closure of the product $S_1\x Z_S(S)\x Z_K(K)$. We then have $H_1=K\cap Z_G(S_2\x\dots\x S_k)$, which again lies in $\mc{Z}(G)$. 
    
    In conclusion, we have shown that the extended factors of $H$ lie in $\mc{Z}(G)$, regardless of how $H$ originates from $K$, completing the proof of the proposition.
\end{proof}

Another important property of the members of the hierarchy $\mf{H}(G)$ involves actions on $\R$--trees (\Cref{prop:no_indecomposable_crossing}). We briefly recall here some of the relevant concepts.

An action on an $\R$--tree $G\acts T$ is \emph{geometric} if it arises as the dual to a (tame) measured foliation on a compact $2$--dimensional cell complex with fundamental group $G$; see \cite{LP97} for a more precise definition, keeping in mind that we are only interested in the case when $G$ is finitely presented. Any minimal geometric tree $G\acts T$ decomposes as a tree of ``edge-like'' and ``indecomposable'' pieces \cite[Proposition~1.25]{Guir-Fourier}. This means that we have a $G$--equivariant decomposition $T=\bigcup_i U_i$ for a collection of subtrees $U_i$ that pairwise share at most one point; each $U_i$ is either a maximal arc with no branch points of $T$ in its interior, or an \emph{indecomposable} subtree in the sense of \cite[Definition~1.17]{Guir-Fourier}; additionally, each arc of $T$ intersects only finitely many subtrees $U_i$, and there are only finitely many $U_i$ up to the $G$--action. The indecomposable subtrees $U_i\sq T$ are uniquely defined, and we refer to them as the \emph{indecomposable components} of $G\acts T$. 

A subtree $S\sq T$ is said to be \emph{stable} if all its arcs have the same $G$--stabiliser. Suppose that $G\acts T$ is geometric and $U\sq T$ is an indecomposable component that happens to be stable. Let $G_U\leq G$ be the stabiliser of the set $U$, and let $K_U\lhd G_U$ be the shared $G$--stabiliser of the arcs of $U$; suppose that $K_U$ is finitely generated. We then have that the action $G_U\acts U$ factors through the action $G_U/K_U\acts U$, which now has trivial arc-stabilisers, while remaining geometric and indecomposable. The Rips' machine then shows that the action $G_U/K_U\acts U$ is always of one of three types: \emph{axial}, \emph{exotic}, or \emph{surface-type}; see \cite[Proposition~A.6]{Guir-Fourier}.

\begin{prop}\label{prop:no_indecomposable_crossing}
    Let $H\in\mf{H}(G)$. Let $G\acts T$ be a minimal geometric $\R$--tree and let $U\sq T$ be an indecomposable component. Let $G_U$ and $H_U$ be the $G$-- and $H$--stabiliser of $U$. Suppose that
    \begin{enumerate}
        \item $H$ is non-elliptic in $T$ and $\Min(H;T)$ shares an arc with $U$;
        \item $U$ is stable 
        and the shared $G$--stabiliser $K_U$ of the arcs of $U$ is convex-cocompact with respect to some convex-cocompact embedding $\iota\colon G\hookrightarrow\mc{A}_{\G}$.
    \end{enumerate}
   Then, the image of the composition $H_U\hookrightarrow G_U\twoheadrightarrow G_U/K_U$ has finite index in $G_U/K_U$.
\end{prop}
\begin{proof}
    A simple explanation of why the proposition holds is that, if $\Min(H;T)$ crosses $U$ without containing $U$, then the many Dehn twists of $G$ arising from $U$ take $H$ to infinitely many non-conjugate subgroups of $G$, violating the fact that $H$ is staunchly $G$--parabolic. The reader happy with this intuitive argument can skip the rest of the proof. Unfortunately, formalising the previous idea is fiddly and requires more background on $\R$--trees, so we prefer to give the following, rather different argument, which exploits directly the algebraic definition of the hierarchy $\mf{H}(G)$.

    Fix throughout the convex-cocompact embedding $\iota\colon G\hookrightarrow\mc{A}_{\G}$. Since the kernel $K_U$ is convex-cocompact, the normaliser $N_G(K_U)$ virtually splits as $R\x K_U$ for a $G$--parabolic subgroup $R$, by \Cref{lem:cc_basics}(2). We have $K_U\leq G_U\leq N_G(K_U)$ and so, denoting by $R_U$ the $R$--stabiliser of $U$, we have that $R_U$ projects to a finite-index subgroup of $G_U/K_U$.
    The subgroups $G_U$ and $R_U$ will not be convex-cocompact in general, but we do have $R_U\leq\Perp_G(K_U)$ (in the notation of \Cref{subsub:orthogonals}).

    We before we continue, we need the following observation.

    \smallskip
    {\bf Claim.} \emph{If a convex-cocompact subgroup $B\leq G$ leaves invariant a line $\alpha\sq T$ such that $\alpha\cap U$ contains an arc and $U\neq\alpha$, then there exists a convex-cocompact subgroup $B_0\lhd B$ fixing $\alpha$ pointwise such that $B/B_0$ is either $\{1\}$, $\Z$ or $\Z\rtimes\Z/2\Z$.}

    \smallskip\noindent
    \emph{Proof of claim.}
    Letting $\beta\sq\alpha\cap U$ be any arc, we set $B_0:=G_{\beta}\cap B$ and note that this is the kernel of the action $B\acts\alpha$; in particular, $B_0\lhd B$ and $B/B_0$ is virtually abelian. We have $G_{\beta}=K_U$, which is convex-cocompact, and so $B_0$ is convex-cocompact by \Cref{lem:cc_basics}(1). It follows that $B$ virtually splits as $B_0\x A$ with $A$ abelian, by \Cref{lem:cc_basics}(2). If $A$ were to have rank $\geq 2$, then $B$ would contain elements with arbitrarily small translation length along $\alpha$ and, since $U$ is indecomposable, this would imply that $\alpha\sq U$. However, the latter is impossible: since $U\neq\alpha$, the component $U$ is of either surface or exotic type, and neither of these can contain an indiscrete line (this is clear in the surface case and follows e.g.\ from \cite[Proposition~7.2]{Guir-CMH} in the exotic case). 
    In conclusion, $A$ is cyclic and the claim follows.
    \hfill$\blacksquare$

    \smallskip
    Let again the operator $\Delta(\cdot)$ and the families $\Delta^n(G)$ be as defined in the proof of \Cref{prop:hierarchy_properties}. Let $\beta\sq U$ be an arc. Our next claim is that, if $H\in\mf{H}(G)$ is non-elliptic in $T$ with $\beta\sq\Min(H;T)$, then $R_U\leq H$. The proposition will immediately follow from this. 
    
    This claim holds trivially for $H\in\Delta^0(G)$. Thus, arguing by induction, suppose the claim has been shown for all $H\in\Delta^n(G)$, and consider some $H\in\Delta^{n+1}(G)$ that is non-elliptic in $T$ with $\beta\sq\Min(H;T)$. Pick an element $J\in\Delta^n(G)$ such that $H\in\Delta(J)$. Since $H\leq J$, we have $\beta\sq\Min(J;T)$ and the inductive hypothesis guarantees that $R_U\leq J$. To conclude, we need to distinguish four cases, depending on how $H$ originates from $J$ (recall the definition of $\Delta(\cdot)$). Throughout, we let $J_1\x\dots\x J_k\x Z_J(J)$ be the finite-index subgroup of $J$ with strongly irreducible, $G$--parabolic factors $J_i$, as provided by \Cref{lem:extended_factors}(4).
    \begin{enumerate}
    \setlength\itemsep{.2em}
        \item \emph{The subgroup $H$ is the centre of $J$.} \\
        Since $H$ is not elliptic in $T$, it follows that $\Min(J;T)=\Min(H;T)$ is a line and, since $R_U\leq J$, the component $U$ must be axial. Each intersection $J_i\cap K_U$ is the kernel of the $J_i$--action on the line $U=\Min(J;T)$; in particular, $J_i\cap K_U$ is convex-cocompact and $J_i/J_i\cap K_U$ is virtually abelian. Since $J_i$ is strongly irreducible, this implies that $J_i\leq K_U$. In conclusion, 
        \[ R_U\leq \Perp_G(K_U)\cap J\leq \Perp_J(J_1\x\dots\x J_k) = Z_J(J)= H . \]
        \item \emph{We have $H\in\mc{J}_G(J)$.} \\
        Here we have $H=J\cap Z_G(g)$ for a convex-cocompact element $g\in G\setminus N_G(J)$, by \Cref{lem:junctures}(2). The centraliser $Z_G(g)$ virtually splits as $\langle g\rangle\x\Perp_G(\langle g\rangle)$ and we have $H=J\cap\Perp_G(\langle g\rangle)$. If $g$ were loxodromic in $T$, then its axis would equal $\Min(H;T)\supseteq\beta$ and the Claim (applied to $B=Z_G(g)$) would imply that this line equals $U$; 
        since $g$ is convex-cocompact, we would then have $g\in R_U\leq J$ contradicting the fact that $g\not\in N_G(J)$. Thus, $g$ must be elliptic in $T$, and it must fix $\Min(H;T)\supseteq\beta$ pointwise, so we have $g\in K_U$. In particular, $R_U\leq\Perp_G(K_U)\leq\Perp_G(\langle g\rangle)$ and, recalling that $R_U\leq J$, we obtain $R_U\leq H$ as desired.
        \item \emph{We have $k\geq 2$ and $H$ is the root-closure of $J_1\x Z_J(J)$.} \\
        If a factor $J_i$ were to contain a loxodromic element for $i\neq 1$, then the axis of this element would be $H$--invariant and it would equal $\Min(H;T)$; since $J_i$ and $H$ commute, this line would also equal $\Min(J_i;T)$. As in Case~1, this would violate the fact that $J_i$ is strongly irreducible. Thus, the factors $J_2,\dots,J_k$ are elliptic in $T$, and they must fix $\Min(H;T)\supseteq\beta$ pointwise. It follows that the subgroup $J_2\x\dots\x J_k$ is contained in $K_U$, and hence 
        \[ R_U\leq J\cap\Perp_G(K_U)\leq\Perp_J(J_2\x\dots\x J_k)=H. \]
        \item \emph{We have $k=1$ and $H$ is the root-closure of $S\x Z_J(J)$ for some $S\in\mc{S}(J_1)$.} \\
        We can assume that the centre $Z_J(J)$ is elliptic in $T$, otherwise, arguing as in Case~1, we have $R_U\leq Z_J(J)\leq H$. Let $S_1\x\dots\x S_s\x Z_S(S)$ be the finite-index subgroup of $S$ with strongly irreducible, $G$--parabolic factors $S_i$. 
        If one of the $S_i$ is non-elliptic, say $S_1$, then $\Min(H;T)=\Min(S_1;T)$ is fixed pointwise by $S_2\x\dots\x S_s\x Z_S(S)\x Z_J(J)$ (arguing as in Case~1) and we have
        \[ R_U\leq J\cap\Perp_G(K_U)\leq\Perp_J(S_2\x\dots\x S_s\x Z_S(S)\x Z_J(J))=S_1\leq H, \]
        as desired. Suppose instead that all $S_i$ are elliptic, so that $\Min(H;T)=\Min(Z_S(S);T)$. If $s\geq 1$, we again have
        \[ R_U\leq J\cap\Perp_G(K_U)\leq\Perp_J(S_1\x\dots\x S_s\x Z_J(J))=Z_S(S)\leq H . \]
        If instead $s=0$, then $S$ is free abelian of rank $\geq 2$ and, using \cite[Proposition~4.5(5)]{Fio11}, we have $S=Z_{J_1}(g)$ for all $g\in S\setminus\{1\}$. In particular, we have $H=Z_J(g)$ for all $g\in S\setminus\{1\}$. If the action $S\acts\Min(S;T)$ is free, then the Claim implies that $U=\Min(S;T)$; the subgroup $R_U\leq J$ is then virtually abelian 
        and contains $S$,
        so $R_U\leq N_J(S)=H$. Finally, if there exists a nontrivial element $g_0\in S$ fixing $\Min(S;T)$, then $g_0\in K_U$ and consequently $R_U\leq Z_J(K_U)\leq Z_J(g_0)=H$, as we wanted.
    \end{enumerate}
    This concludes the proof of the proposition
\end{proof}

We also need a larger version of the hierarchy that takes into account free factors. If $H\leq G$ is staunchly $G$--parabolic, we denote by $\mc{F}(H)$ the collection of subgroups of $H$ that contain the centre $Z_H(H)$ and project to $1$--ended free factors 
of the quotient $H/Z_H(H)$.

\begin{defn}\label{defn:enlarged_hierarchy}
    The \emph{enlarged hierarchy} $\mf{H}^*(G)$ is the family of subgroups
    \[ \mf{H}^*(G):=\mf{H}(G)\cup\bigcup_{H\in\mf{H}(G)}\mc{F}(H) .\]
\end{defn}

The enlarged hierarchy $\mf{H}^*(G)$ is $\Aut(G)$--invariant and consists of finitely many $G$--conjugacy classes of convex-cocompact subgroups (regardless of the embedding $G\hookrightarrow\mc{A}_{\G}$, see \cite[Proposition~2.30]{Fio11}). However, not all elements of $\mf{H}^*(G)$ are $G$--parabolic.

Importantly, we do not require the enlarged hierarchy $\mf{H}^*(G)$ to satisfy property $(H3)$ from \Cref{defn:hierarchy}, as this would yield an unwieldy family of subgroups. At the same time, it it is easy to see (and unimportant) that the enlarged hierarchy does satisfy properties $(H0)-(H2)$. 

\subsection{Reference systems}\label{sub:reference_systems}

Let $G$ be a special group. As sketched in the Introduction, the core of the article will seek to ``shorten'' a sequence of automorphisms $\phi_n\in\Out(G)$ by replacing it with a sequence $\phi_n\tau_n$, where each $\tau_n$ is a suitable product of Dehn twists. For this to be possible, the shortening will need to be carried out hierarchically: choosing a smallest element $H\in\mf{H}^*(G)$ on which the sequence $\phi_n$ grows as fast as on the entire group $G$, and aiming to shorten just the restrictions $\phi_n|_H$, ignoring larger elements of $\mf{H}^*(G)$. In this subsection, we set up some terminology that is useful for this hierarchical shortening process.

Denote by $[\mf{H}(G)]$ and $[\mf{H}^*(G)]$ the (finite) sets of $G$--conjugacy classes of subgroups in $\mf{H}(G)$ and $\mf{H}^*(G)$, respectively. Both sets are equipped with a natural poset structure: we write $[H_1]\preceq [H_2]$ if some $G$--conjugate of $H_1$ is contained in a $G$--conjugate of $H_2$. Note that $[H_1]\preceq [H_2]\preceq [H_1]$ can indeed only occur when $[H_1]=[H_2]$, by \Cref{lem:cc_basics}(3). 

\begin{defn}\label{defn:reference_system}
    A \emph{reference system} is a triple $\mf{R}=(\mc{H},\ll,\{S_H\}_{H\in\mc{H}})$ where:
    \begin{itemize}
    \item $\mc{H}\sq\mf{H}^*(G)$ is a finite subset projecting bijectively to the quotient $[\mf{H}^*(G)]$;
    \item $S_H\sq H$ is a finite generating set with $S_H=S_H^{-1}$ and $1\in S_H$ for each subgroup $H\in\mc{H}$;
    \item $\ll$ is a total order on $\mc{H}$ satisfying the following conditions:
    \begin{itemize}
        \item[(a)] on $\mc{H}\cap\mf{H}(G)$, the total order $\ll$ extends the partial order $\preceq$ inherited from $[\mf{H}(G)]$;
        \item[(b)] if $H\in\mf{H}(G)$ and $H'\in\mc{F}(H)$, then we have $H'\ll H$ and all chains $H'\ll K\ll H$ satisfy $K\in\mc{F}(H)$.
    \end{itemize}
\end{itemize}
\end{defn}

Note that reference systems always exist: we can first arbitrarily extend $\preceq$ to a total order on $[\mf{H}(G)]$, and then extend this to a total order on $[\mf{H}^*(G)]$ by squeezing the elements of each set $[\mc{F}(H)]$ into the gap just below $[H]$.

Given an action on a metric space $G\acts X$, an element $g\in G$ and a finite subset $\Om\sq G$, write:
\begin{align*}
    \ell(g;X)&:=\inf_{x\in X} d(x,gx), & \ell(\Om;X)&:=\sum_{s\in\Om}\ell(s;\Om), \\
    \Om^2&:=\{ss'\mid s,s'\in\Om\}, & \mf{t}(\Om;X)&:=\inf_{x\in X}\max_{s\in\Om}d(x,sx) .
\end{align*}
Note that $\ell(\Om;X)=\ell(g\Om g^{-1};X)$ and $\mf{t}(\Om;X)=\mf{t}(g\Om g^{-1};X)$ for all $g\in G$.

It is customary to measure the ``length'' of a subset $\Om\sq G$ by the displacement parameter $\mf{t}(\Om;X)$. In our setting, it is more convenient to work with lengths of squares, that is, with the quantity $\ell(\Om^2;X)$. This is entirely equivalent in view of the following observation. (Note that we will be interested in sequences of generating sets obtained from a fixed one by applying automorphisms, so their cardinality will be uniformly bounded.)

\begin{lem}\label{lem:t_vs_ell2}
    Let $\X_{\G}$ be the universal cover of the Salvetti complex of a RAAG $\mc{A}_{\G}$. For all finite subsets $\Om\sq\mc{A}_{\G}$ with $1\in\Om$, we have:
    \[ \frac{1}{2|\Om^2|}\cdot\ell(\Om^2;\X_{\G}) \leq \mf{t}(\Om;\X_{\G})\leq \tfrac{3}{2}|\G|\cdot\ell(\Om^2;\X_{\G}) .\]
\end{lem}
\begin{proof}
    The first inequality holds for any action $G\acts X$. It suffices to observe that
    \[ \ell(\Om^2;X)\leq|\Om^2|\cdot\max_{s\in\Om}\ell(s;\Om)\leq |\Om^2|\cdot\mf{t}(\Om^2;X)\leq 2|\Om^2|\cdot \mf{t}(\Om;X) . \]

    We now discuss the second inequality. For each $v\in\G$, let $T^v$ be the tree dual to the set of hyperplanes of $\X_{\G}$ labelled by $v$. There are natural collapse maps $\pi^v\colon\X_{\G}\ra T^v$ giving an $\mc{A}_{\G}$--equivariant isometric embedding $\prod_v\pi^v\colon\X_{\G}\hookrightarrow\prod_vT^v$. For all $x,y\in\X_{\G}$, set for simplicity $d_v(x,y):=d(\pi_v(x),\pi_v(y))$. We always have $d(x,y)=\sum_{v\in\G}d_v(x,y)$.

    In \cite[Definition~6.8]{Fio10a}, we introduced convex subcomplexes $\mc{B}(\Om;\X_{\G})\sq\X_{\G}$ and $\mc{B}(\Om;T^v)\sq T^v$, which we called \emph{multibridges}; this is a canonical construction that works in any median algebra, yielding a set of points roughly realising the parameter $\mf{t}(\Om;\cdot)$. From the definition of multibridges, it is immediate that $\pi^v(\mc{B}(\Om;\X_{\G}))=\mc{B}(\Om;T^v)$. For all points $x\in\mc{B}(\Om;T^v)$ and elements $s\in\Om$, we have $d(x,sx)\leq 3\mf{t}(\Om;T^v)$ by \cite[Proposition~6.11(3)]{Fio10a}. Finally, since $T^v$ is a tree, the value $\mf{t}(\Om;T^v)$ is either the translation length of an element of $\Om$, or it is double the distance between the axes or fixed sets of two elements of $\Om$. This shows that $2\mf{t}(\Om;T^v)\leq\ell(\Om^2;T^v)$. 
    
    Combining all the above observations and considering a point $x\in\mc{B}(\Om;\X_{\G})$ realising $\mf{t}(\Om;\X_{\G})$ (which exists by \cite[Proposition~6.9(1)]{Fio10a}), we obtain the desired inequality:
    \[ \mf{t}(\Om;\X_{\G})=\max_{s\in \Om}\sum_v d_v(x,sx)\leq \sum_v \max_{s\in \Om} d_v(x,sx)\leq\sum_v\tfrac{3}{2}\cdot\ell(\Om^2;T^v)\leq \tfrac{3}{2}|\G|\cdot\ell(\Om^2;\X_{\G}) .\]
\end{proof}

Now, fix a convex-cocompact embedding $\iota\colon G\hookrightarrow\mc{A}_{\G}$ and let $\mf{E}_G\sq\X_{\G}$ be a $G$--essential convex subcomplex, in the sense of \cite[Section~3.4]{CS11}. We have $\ell(g;\X_{\G})=\ell(G;\mf{E}_G)$ for all $g\in G$. Each reference system $\mf{R}=(\mc{H},\ll,\{S_H\}_{H\in\mc{H}})$ determines an order-like relation $\sqsubseteq$ on $\Out(G)$ as follows. Consider $\varphi,\psi\in\Aut(G)$ and their outer classes $[\varphi],[\psi]\in\Out(G)$. For each $H\in\mc{H}$, we write 
\begin{align*}
    &[\varphi]\sqsubset_H[\psi] \quad\text{if}\quad \ell(\varphi(S_H^2);\mf{E}_G)<\ell(\psi(S_H^2);\mf{E}_G), \\
    &[\varphi]\asymp_H[\psi] \quad\text{if}\quad \ell(\varphi(S_H^2);\mf{E}_G)=\ell(\psi(S_H^2);\mf{E}_G).
\end{align*}
Note the squaring of $S_H$ in these definitions. We then write
\begin{align*}
    &[\varphi]\sqsubset[\psi] \quad\text{if $[\varphi]\sqsubset_H[\psi]$ for some $H\in\mc{H}$ such that $[\varphi]\asymp_K[\psi]$ for all $K\in\mc{H}$ with $K\ll H$,} \\
    &[\varphi]\asymp[\psi] \quad\text{if $[\varphi]\asymp_H[\psi]$ for all $H\in\mc{H}$.}
\end{align*}
The symbols $\asymp_H$ and $\asymp$ define equivalence relations on $\Out(G)$. Each equivalence class of $\asymp_G$ is finite,
by \Cref{lem:t_vs_ell2}, and so the equivalence classes of $\asymp$ are also finite. Using the symbol $\sqsubseteq$ to denote either $\sqsubset$ or $\asymp$, we have that $\sqsubseteq$ is a total order on the set of equivalence classes $\Out(G)/\asymp$. 

Since $\ell(g;\mf{E}_G)\in\N$ for all $g\in G$, and since the set $\mc{H}$ is finite, it is immediate to check that every (possibly infinite) subset of $\Out(G)$ has a $\sqsubseteq$--minimum. In other words:

\begin{lem}\label{lem:well-ordering}
    The total order $\sqsubseteq$ on $\Out(G)/\asymp$ is a well-ordering. 
\end{lem}

Note that $\asymp$ and $\sqsubseteq$ depend on the choice of the embedding $\iota$ and of reference system $\mf{R}$, but not on the choice of the subcomplex $\mf{E}_G$.

\section{Structure of degenerations}\label{sec:degenerations}

Let $G$ be a special group. An infinite sequence in $\Out(G)$ and the choice of a non-principal ultrafilter $\om$ give rise to an isometric action on a median space $G\acts\X_{\om}$, which equivariantly embeds in a finite product of $\R$--trees $T^v_{\om}$. We refer to such actions as \emph{degenerations}. We briefly recall this construction and its properties shown in \cite{Fio10e}, then prove a few more.

Realise $G$ as a convex-cocompact subgroup $G\leq\mc{A}_{\G}$, for a RAAG $\mc{A}_{\G}$, and endow $G$ with the resulting coarse median structure. Let $\X_{\G}$ be the universal cover of the Salvetti complex of $\mc{A}_{\G}$. For each vertex $v\in\G$, let $\mc{A}_{\G}\acts T^v$ be the Bass--Serre tree of the HNN splitting of $\mc{A}_{\G}$ with vertex group $\mc{A}_{\G\setminus\{v\}}$, edge group $\mc{A}_{\lk(v)}$ and stable letter $v$. Equivalently, $T^v$ is the tree dual to the collection of hyperplanes of $\X_{\G}$ labelled by $v$, and so there is a natural projection $\pi^v\colon\X_{\G}\ra T^v$, which is a \emph{restriction quotient} in the sense of \cite{CS11} (also see \cite[Section~2.5.1]{Fio10a}). The maps $\pi^v$ give an $\mc{A}_{\G}$--equivariant isometric embedding $\X_{\G}\hookrightarrow\prod_{v\in\G}T^v$, using $\ell_1$--metrics.

Let $\mf{E}_G\sq\X_{\G}$ be a $G$--essential convex subcomplex. Assuming without loss of generality that $G$ is not contained in any proper parabolic subgroup of $\mc{A}_{\G}$, we have that the subcomplex $\mf{E}_G\sq\X_{\G}$ is unique and $\pi^v(\mf{E}_G)=\Min(G;T^v)$ for all $v\in\G$; in particular, $G$ is non-elliptic in all $T^v$. Convex-cocompactness of $G$ in $\mc{A}_{\G}$ means that the action $G\acts\mf{E}_G$ is cocompact \cite[Lemma~3.2]{Fio10a}. 

Fix a non-principal ultrafilter $\om$, and consider a sequence of automorphisms $\varphi_n\in\Aut(G)$ projecting to an infinite sequence $\phi_n\in\Out(G)$. We now apply the Bestvina--Paulin construction \cite{Bes88,Pau91} to the action $G\acts\mf{E}_G$. This means that we pick a finite generating set $S\sq G$, define $\lambda_n:=\mf{t}(\varphi_n(S);\mf{E}_G)$ (in the notation of \Cref{sub:reference_systems}), and choose basepoints $o_n\in\mf{E}_G$ realising these displacement parameters $\lambda_n$. We then denote by $G\acts\X_{\om}$ the $\om$--limit of countably many copies of the action $G\acts\X_{\G}$, with the $n$--th copy twisted by $\varphi_n$, 
rescaled by $\lambda_n$, and based at $o_n$. For each $v\in\G$, we similarly obtain an action $G\acts T^v_{\om}$ as the $\om$--limit of copies of the action $G\acts T^v$, rescaling the $n$--th copy by the same parameter $\lambda_n$ and basing it at the point $\pi^v(o_n)$.

The degeneration $G\acts\X_{\om}$ is an isometric action on a median space with unbounded orbits. The spaces $T^v_{\om}$ are $\R$--trees; the group $G$ can be elliptic in $T^v_{\om}$ for some vertices $v\in\G$, though never for all of them. The maps $\pi^v\colon\X_{\G}\ra T^v$ induce $G$--equivariant, $1$--Lipschitz, median-preserving maps $\pi^v_{\om}\colon\X_{\om}\ra T^v_{\om}$, which together give an isometric embedding $\X_{\om}\hookrightarrow\prod_{v\in\G}T^v_{\om}$. Although $\pi^v\colon\X_{\G}\ra T^v$ is surjective, the maps $\pi^v_{\om}\colon\X_{\om}\ra T^v_{\om}$ need not be.

We refer to the action $G\acts\X_{\om}$ as the \emph{degeneration} determined by $\om$ and $[\varphi_n]\in\Out(G)$ (as well as the chosen embedding $G\hookrightarrow\mc{A}_{\G}$). The choice of representatives $\varphi_n\in\Aut(G)$ only affects $\X_{\om}$ by a $G$--equivariant isometry. A priori, the degeneration can depend on the choice of basepoints $o_n$, but this is inconsequential. For an arc $\beta$ in an $\R$--tree $T$, we denote by $G_{\beta}$ its pointwise $G$--stabiliser; instead, for a line $\alpha$, we denote by $G_{\alpha}$ the setwise stabiliser.

\begin{thm}\label{thm:inert_vs_motile}
    Let $G\acts\X_{\om}$ be the degeneration determined by a sequence $\varphi_n\in\Aut(G)$. Pick $v\in\G$ with $G$ non-elliptic in $T^v_{\om}$. Then each arc $\beta\sq\Min(G;T^v_{\om})$ is of one of the following two types.
    \begin{enumerate}
    \setlength\itemsep{.2em}
        \item An ``inert'' arc: for $\om$--all $n$, the subgroup $\varphi_n(G_{\beta})$ is elliptic in $T^v$. In this case, either we have $G_{\beta}\in\mc{Z}(G)$, or there exists a line $\beta\sq\alpha\sq T^v_{\om}$ such that $G_{\beta}\lhd G_{\alpha}$, $G_{\alpha}\in\mc{Z}(G)$, and the quotient $G_{\alpha}/G_{\beta}$ is free abelian and acts freely on $\alpha$ with dense orbits.
        \item A ``motile'' arc: for $\om$--all $n$, the subgroup $\varphi_n(G_{\beta})$ is non-elliptic in $T^v$. In this case, $\varphi_n(G_{\beta})$ leaves invariant a line $\alpha_n\sq T^v$. The line $\alpha:=\lim_{\om}\alpha_n$ satisfies $\beta\sq\alpha$, $G_{\beta}\lhd G_{\alpha}$, and $G_{\alpha}\in\mc{Z}(G)$. The quotient $G_{\alpha}/G_{\beta}$ is free abelian and acts freely on $\alpha$.
    \end{enumerate}
    Furthermore, if the automorphisms $\varphi_n$ are coarse-median preserving, then all inert arcs $\beta$ have $G_{\beta}\in\mc{Z}(G)$, and all motile arcs $\beta$ have $G_{\beta}=G_{\alpha}$ (i.e.\ the action $G_{\alpha}\acts\alpha$ is trivial).
\end{thm}
\begin{proof}
    Part of the theorem is \cite[Proposition~5.12]{Fio10e} and the rest is implicit at various points in \cite{Fio10e}, though the terminology ``inert/motile arc'' is new. We explain how to recover a full proof, omitting details. We will repeatedly use the following fact (see e.g.\ \cite[Lemma~3.16(2)]{Fioravanti-Kerr}):
    \begin{enumerate}
        \item[$(*)$] There is a constant $C\geq 1$ such that, if two elements $g,h\in G$ are loxodromic in $T^v$ with distinct axes $A_g,A_h\sq T^v$, then the intersection $A_g\cap A_h$ has length $\leq C\cdot\max\{\ell(g;T^v),\ell(h;T^v)\}$. Moreover, if $g,h\in G$ are convex-cocompact and $A_g=A_h$, then $\langle g\rangle=\langle h\rangle$.
    \end{enumerate}
    Throughout, we consider an ``unadulterated'' copy of the $G$--tree $T^v$: unscaled and untwisted.
    
    Consider an arc $\beta\sq\Min(G;T^v_{\om})$. We distinguish two cases, based on the behaviour of $G_{\beta}$:
    \begin{enumerate}
        \item[(i)] for each $h\in G_{\beta}$, the element $\varphi_n(h)$ is elliptic in $T^v$ for $\om$--all $n$; 
        \item[(ii)] there exists an element $h_0\in G_{\beta}$ such that $\varphi_n(h_0)$ is loxodromic in $T^v$ for $\om$--all $n$.
    \end{enumerate}
    Unsurprisingly, Case~(i) will yield that $\beta$ is inert, and Case~(ii) that $\beta$ is motile.

    In Case~(ii), let $\alpha_n$ be the axis of $\varphi_n(h_0)$ in $T^v$, and set $\alpha:=\lim_{\om}\alpha_n$. Property~$(*)$ implies that we have $G_{\alpha_n}=Z_G(g_n)$ for some elements $g_n\in G$ that are convex-cocompact in $G$ and loxodromic in $T^v$ with $\ell(g_n;T^v)\leq\ell(\varphi_n(h_0);T^v)$. Property~$(*)$ also implies that $G_{\alpha}=\bigcap_{\om}\varphi_n^{-1}(G_{\alpha_n})$ and that the line $\alpha_n$ is $\varphi_n(G_{\beta})$--invariant for $\om$--all $n$. It follows that $\beta\sq\alpha$, $G_{\beta}\leq G_{\alpha}$ and $G_{\alpha}\in\mc{Z}(G)$. Thus, $G_{\beta}$ is the kernel of the action $G_{\alpha}\acts\alpha$, and the quotient $G_{\alpha}/G_{\beta}$ acts freely. In conclusion, $\beta$ is ``motile'' and Item~(2) of the theorem holds.

    In Case~(i), we need to further distinguish two cases. Indeed, \cite[Theorem~4.2]{Fio10e} shows that we can find arcs $\beta_n\sq\Min(G;T^v)$ such that $\beta=\lim_{\om}\beta_n$ and:
    \begin{enumerate}
        \item[(A)] either $G_{\beta_n}\in\mc{Z}(G)$ for $\om$--all $n$;
        \item[(B)] or $\beta_n\sq\gamma_n$ for $\om$--all $n$, where the lines $\gamma_n\sq T^v$ are the axes of elements $g_n\in G$ that are convex-cocompact in $G$ and loxodromic in $T^v$ with $\sup_{\om}\ell(g_n;T^v)<+\infty$.
    \end{enumerate}
    In Case~(A), we can argue as in the proof of \cite[Proposition~5.12]{Fio10e} to conclude that, up to perturbing the arcs $\beta_n$, we have $G_{\beta}=\bigcap_{\om}\varphi_n^{-1}(G_{\beta_n})$. 
    In particular, we have $G_{\beta}\in\mc{Z}(G)$ and hence, since $G_{\beta}$ is finitely generated, the subgroup $\varphi_n(G_{\beta})$ is elliptic in $T^v$ for $\om$--all $n$.

    In the rest of the proof, we consider Case~(iB). The line $\gamma:=\lim_{\om}\gamma_n$ contains $\beta$. Denote by $K_n\lhd Z_G(g_n)$ the kernel of the action $Z_G(g_n)\acts\gamma_n$. Property~$(*)$ implies that we have $G_{\gamma_n}=Z_G(g_n)$ and $G_{\beta_n}=K_n$ for $\om$--all $n$, 
    and also that $G_{\beta}=\bigcap_{\om}\varphi_n^{-1}(K_n)$
    and $G_{\gamma}=\bigcap_{\om}Z_G(\varphi_n^{-1}(g_n))$. This shows that $G_{\gamma}\in\mc{Z}(G)$, that $G_{\beta}\leq G_{\gamma}$ and that the action $G_{\gamma}\acts\gamma$ is orientation-preserving. In particular, $G_{\beta}$ is the kernel of the action $G_{\gamma}\acts\gamma$, and the quotient $G_{\gamma}/G_{\beta}$ acts freely on $\gamma$. Also note that the group $\varphi_n(G_{\beta})$ is elliptic in $T^v$ for $\om$--all $n$, even when $G_{\beta}$ is infinitely generated: this is because the action $G_{\beta}\cong\varphi_n(G_{\beta})\acts\gamma_n$
    factors through the projection of $G_{\beta}$ to the abelianisation of $G_{\gamma}$, which is finitely generated since $G_{\gamma}\in\mc{Z}(G)$.
    
    Now, if $G_{\beta}=G_{\gamma}$, we have $G_{\beta}\in\mc{Z}(G)$ as in Case~(A). Supposing instead that $G_{\beta}\neq G_{\gamma}$, we are only left to show that $G_{\gamma}\acts\gamma$ has dense orbits, that is, that we have $G_{\gamma}/G_{\beta}\not\cong\Z$. 

    Since $G_{\gamma}\neq G_{\beta}$, the group $\varphi_n(G_{\gamma})$ is non-elliptic in $T^v$ for $\om$--all $n$. Since $\varphi_n(G_{\gamma})\leq Z_G(g_n)$ is convex-cocompact and root-closed, being a centraliser, we must have $g_n\in\varphi_n(G_{\gamma})$ for $\om$--all $n$. Thus, the elements $\varphi_n^{-1}(g_n)$ all lie in the centre of $G_{\gamma}$, which implies that the centraliser $Z_G(\varphi_n^{-1}(g_n))$ is $\om$--constant (as in the proof of \cite[Proposition~5.15(a)]{Fio10e}). This shows that we have $\varphi_n(G_{\gamma})=Z_G(g_n)$ for $\om$--all $n$. If we had $G_{\gamma}/G_{\beta}\cong\Z$, it would follow that $\varphi_n(G_{\beta})=K_n$ for $\om$--all $n$. 
    Picking an element $g\in G_{\gamma}$ projecting to a generator of $G_{\gamma}/G_{\beta}$, we would then have that $\varphi_n(g)$ projects to a generator of $Z_G(g_n)/K_n$. Recalling that $\sup_{\om}\ell(g_n;T^v)<+\infty$, this would yield that $\sup_{\om}\ell(\varphi_n(g);T^v)<+\infty$ and so $g$ would be elliptic in $T^v_{\om}$, contradicting the fact that $g\not\in G_{\beta}$. 
    
    This completes the proof of the theorem.
\end{proof}

\Cref{thm:inert_vs_motile} should also be taken as the definition of \emph{inert} and \emph{motile} arcs, which will play an important role in \Cref{sect:shortening}. If $\beta$ is a motile arc, we refer to the line $\alpha\supseteq\beta$ described in the theorem as the \emph{characteristic line} of $\beta$. Note that the action $G_{\alpha}\acts\alpha$ can be discrete or trivial, unlike for lines arising from inert arcs; when the action is trivial, we will typically have $\alpha\not\sq\Min(G;T^v_{\om})$.

\begin{rmk}\label{rmk:BF-stable}
    \Cref{thm:inert_vs_motile} shows that all arc-stabilisers of $\Min(G;T^v_{\om})$ are co-(free abelian) subgroups of elements of $\mc{Z}(G)$. In particular, chains of arc-stabilisers have bounded length (see \cite[Lemma~3.36]{Fio10e}), and so every arc contains a stable sub-arc, that is, the action $G\acts\Min(G;T^v_{\om})$ is \emph{stable} in the sense of \cite{BF-stable}. In fact, it is even \emph{piecewise stable} in the sense of \cite{Guir-Fourier}.
\end{rmk}

We call a line $\alpha\sq T^v_{\om}$ \emph{motile} if all its arcs are motile,
and \emph{indiscrete} if the action $G_{\alpha}\acts\alpha$ has dense orbits. For simplicity, we speak of \emph{IOM-lines} when referring to lines that are either indiscrete or motile. The following are a few useful addenda to the statement of \Cref{thm:inert_vs_motile}.

\begin{lem}\label{lem:line_addenda}
    Consider $v\in\G$ such that $G$ is non-elliptic in $T^v_{\om}$.
    \begin{enumerate}
        \item For any line $\alpha\sq T^v_{\om}$ with $G_{\alpha}\acts\alpha$ nontrivial, we have that $G_{\alpha}=Z_G(g)$ for an algebraically convex element $g\in G\setminus\{1\}$. Moreover, $\Min(\varphi_n(G_{\alpha});T^v)$ is a line for $\om$--all $n$ and, if $C$ denotes the centre of $G_{\alpha}$, we have $Z_G(C)=G_{\alpha}$.
        \item A line $\alpha\sq T^v_{\om}$ is motile if and only if it is the characteristic line of some motile arc.
        \item If $\alpha,\beta\sq T^v_{\om}$ are distinct IOM-lines, then $\alpha\cap\beta$ contains at most one point.
        \item If the automorphisms $\varphi_n$ are coarse-median preserving, there are no IOM-lines in $T^v_{\om}$.
    \end{enumerate}
\end{lem}
\begin{proof}
    Regarding part~(1), the fact that $\alpha_n:=\Min(\varphi_n(G_{\alpha});T^v)$ is $\om$--always a line follows from the Property~$(*)$ mentioned at the start of the proof of \Cref{thm:inert_vs_motile}. We then have $G_{\alpha}=\bigcap_{\om}\varphi_n^{-1}(G_{\alpha_n})$ and $G_{\alpha_n}=Z_G(g_n)$ for convex-cocompact elements $g_n\in G$ that are loxodromic in $T^v$. Now, the argument used at the end of Case~(iB) in the proof of \Cref{thm:inert_vs_motile} shows that we actually have $\varphi_n(G_{\alpha})=Z_G(g_n)$ for $\om$--all $n$ (the specific hypotheses of Case~(iB) were irrelevant in that part of the argument). In particular, the equality $G_{\alpha}=Z_G(\varphi_n^{-1}(g_n))$ holds for at least one index $n$, and the element $g:=\varphi_n^{-1}(g_n)$ is algebraically convex by \Cref{lem:cc->ac}. Finally, if $C$ is the centre of $G_{\alpha}$, then $\varphi_n(C)$ is the centre of $Z_G(g_n)$, which contains $g_n$. Thus $\varphi_n(C)$ is non-elliptic in $T^v$ with axis $\alpha_n$ for $\om$--all $n$. It follows that the centraliser $Z_G(C)$ preserves the lines $\alpha_n$, and so it preserves their $\om$--limit $\alpha$. Hence $Z_G(C)\leq G_{\alpha}$, and the other inclusion is clear. 
    This proves part~(1).

    Part~(2) easily follows from part~(3). 
    In turn, part~(3) is another straightforward consequence of Property~$(*)$: just approximate $\alpha,\beta\sq T^v_{\om}$ by long arcs $\alpha_n,\beta_n\sq T^v$, and observe that $\varphi_n(G_{\alpha})$ and $\varphi_n(G_{\beta})$ contain elements whose translation lengths in $T^v$ are much smaller than the lengths of $\alpha_n$ and $\beta_n$. Finally, part~(4) follows from \Cref{thm:inert_vs_motile} and \cite[Proposition~5.15(c1)]{Fio10e}.
\end{proof}

IOM-lines are problematic because, in view of \Cref{thm:inert_vs_motile}, they are the ``cause'' of arc-stabilisers outside $\mc{Z}(G)$. Fortunately, the next result shows that such lines are extremely restricted.

\begin{lem}\label{lem:line_stab_SVP}
    Let $\alpha\sq T^v_{\om}$ be an IOM-line and let $A$ denote the centre of $G_{\alpha}$.
    \begin{enumerate}
        \item If $G_{\alpha}\acts\alpha$ is nontrivial, then we have $G_{\alpha}\in\SVP(G)$ and $A\in\Sal(G)$ (\Cref{defn:salient}).
        \item If $A\acts\alpha$ is nontrivial, then $A\in\xSal(G)$.
    \end{enumerate}
\end{lem}
\begin{proof}
    We begin with part~(1). By \Cref{lem:line_addenda}(1), we have $G_{\alpha}=Z_G(g)$ for an algebraically convex element $g$. Since $\alpha$ is an IOM-line, we have $G_{\alpha}\not\cong\Z$. Thus, supposing for the sake of contradiction that $G_{\alpha}\not\in\SVP(G)$, \Cref{prop:centraliser_g_algconvex} and \Cref{lem:cofactors} yield a staunchly $G$--parabolic subgroup $Q\lhd G_{\alpha}$ such that $G_{\alpha}/Q\cong\Z$ and $N_G(Q)/Q$ is non-abelian. Observe that $\varphi_n(Q)$ is elliptic in $T^v$ for $\om$--all $n$: otherwise, $N_G(Q)$ would preserve the $\varphi_n(Q)$--minimal subtrees of $T^v$, which would be lines (since the $\varphi_n(G_{\alpha})$--minimal subtrees are lines by \Cref{lem:line_addenda}(1)); these lines would then converge to $\alpha$, giving $N_G(Q)\leq G_{\alpha}$ and violating the fact that $N_G(Q)/Q$ is non-abelian.
    
    Now, since $\varphi_n(Q)$ is $\om$--always elliptic in $T^v$, we have that $Q$ is elliptic in $T^v_{\om}$, and so $Q$ fixes $\alpha$ pointwise. This prevents $\alpha$ from being indiscrete, since $G_{\alpha}/Q\cong\Z$. It also implies that $Q$ is the kernel of the action $G_{\alpha}\acts\alpha$, since this action is nontrivial by hypothesis and $G_{\alpha}/Q\cong\Z$. In turn, the fact that $Q$ is the kernel prevents $\alpha$ from being motile, which is a contradiction.

    In conclusion, we have shown that $G_{\alpha}\in\SVP(G)$. Since $G_{\alpha}\acts\alpha$ is nontrivial, the line $\alpha$ is the $G_{\alpha}$--minimal subtree of $T^v_{\om}$ and so we have $N_G(G_{\alpha})=G_{\alpha}$. Finally, the centre $A\leq G_{\alpha}$ is nontrivial because $G_{\alpha}=Z_G(g)$ with $g\neq 1$, and so $A\in\Sal(G)$, proving part~(1).

    As to part~(2), assume now that $A$ is non-elliptic in $T^v_{\om}$. Note that $Z_G(A)=G_{\alpha}$ by \Cref{lem:line_addenda}(1).
    Suppose for the sake of contradiction that there is a centraliser $Z\in\mc{Z}(G)$ with $A/A\cap Z\cong\Z$. 
    
    Suppose for a moment that the abelian group $\varphi_n(A\cap Z)$ is non-elliptic in $T^v$ for $\om$--all $n$ and let $\gamma_n$ be its axis. This must coincide with the line $\Min(\varphi_n(G_{\alpha});T^v)$, and so the lines $\gamma_n\sq T^v$ converge to $\alpha\sq T^v_{\om}$ modulo $\om$. Since the centraliser $\varphi_n(Z_G(Z))$ commutes with $\varphi_n(A\cap Z)$, it preserves the line $\gamma_n$. Consequently, $Z_G(Z)$ preserves the line $\alpha$ and hence $Z_G(Z)\leq G_{\alpha}=Z_G(A)$. This implies that $A=Z_GZ_G(A)\leq Z_GZ_G(Z)=Z$, contradicting the fact that $A/A\cap Z\cong\Z$.

    Thus, the subgroup $\varphi_n(A\cap Z)$ must be elliptic in $T^v$ for $\om$--all $n$, and hence $A\cap Z$ is elliptic in $T^v_{\om}$. Since $A/A\cap Z\cong\Z$ and $A$ is non-elliptic in $T^v_{\om}$, it follows that $A\cap Z$ is precisely the kernel of the action $A\acts\alpha$. In conclusion, the action $A\acts\alpha$ is discrete and its kernel is elliptic in $T^v$ when twisted by $\om$--all of the $\varphi_n$. This means that $\alpha$ is \emph{not} an IOM-line for the action $A\acts T^v_{\om}$.

    We can finally reach a contradiction by applying \Cref{thm:inert_vs_motile} to the restrictions $\varphi_n|_A$. More precisely, since $A$ lies in $\Sal(G)$ by part~(1), it is staunchly $G$--parabolic, and so the subgroups $\varphi_n(A)$ fall into finitely many $G$--conjugacy classes. Thus, up to composing each $\varphi_n\in\Aut(G)$ with an inner automorphism of $G$ (which only modifies $\X_{\om}$ and $T^v_{\om}$ by a $G$--equivariant isometry), we can suppose that the subgroup $\varphi_n(A)$ is $\om$--constant. Moreover, up to right-composing each $\varphi_n$ with some automorphism $\psi\in\Aut(G)$ (which simply twists the action on $T^v_{\om}$ by $\psi$), we can assume that $\varphi_n(A)=A$ for $\om$--all $n$. Now, we can consider the degeneration $A\acts\mc{Y}_{\om}$ and the $\R$--tree $A\acts\mscr{T}^v_{\om}$ that arise from the sequence $\varphi_n|_A\in\Aut(A)$. Since $A$ is non-elliptic in $T^v_{\om}$, we can use the same scaling factors to define $\X_{\om}$ and $\mc{Y}_{\om}$, and so the actions $A\acts\Min(A;T^v_{\om})$ and $A\acts\Min(A;\mscr{T}^v_{\om})$ are equivariantly isometric. Since $A$ is abelian, we have $\mc{Z}(A)=\{A\}$ and so \Cref{thm:inert_vs_motile} guarantees that $\Min(A;\mscr{T}^v_{\om})$ is either indiscrete or motile, that is, an IOM-line. 

    The previous two paragraphs stand in contradiction to each other, and so no centraliser $Z\in\mc{Z}(G)$ as above can exist, showing that $A\in\xSal(G)$ and concluding the proof.
\end{proof}

\begin{rmk}\label{rmk:abelians_give_IOM}
    It is worth recording the following fact, which we deduced from \Cref{thm:inert_vs_motile} in the second-last paragraph of the proof of \Cref{lem:line_stab_SVP}: \emph{If $B\leq G$ is an abelian, staunchly $G$--parabolic subgroup that is non-elliptic in $T^v_{\om}$ (for some degeneration of $G$), then the subtree $\Min(B;T^v_{\om})$ is an IOM-line (with respect to both the $B$-- and $G$--actions on $T^v_{\om}$).}
\end{rmk}

We will also need the following two observations.

\begin{rmk}\label{rmk:short_central_elements}
    If $\alpha\sq T^v_{\om}$ is an IOM-line, then we have $G_{\alpha}=Z_G(h_n)$ for elements $h_n$ such that $\varphi_n(h_n)$ is convex-cocompact in $G$ and loxodromic in $T^v$ with $\lim_{\om}\frac{1}{\lambda_n}\ell(\varphi_n(h_n);T^v)=0$. Indeed, recall from the proof of \Cref{lem:line_addenda}(1), that we have $\varphi_n(G_{\alpha})=Z_G(g_n)$ for $\om$--all $n$, where $g_n\in G$ is convex-cocompact in $G$ and loxodromic in $T^v$. By \Cref{lem:line_stab_SVP}(1), the groups $\varphi_n(G_{\alpha})$ are all $G$--parabolic and, by \Cref{lem:cc_basics}(2), the cyclic subgroups $\langle g_n\rangle$ are virtual direct factors of $\varphi_n(G_{\alpha})$. Thus, assuming that we picked the $g_n$ so that they are not proper powers, all the subgroups $\langle g_n\rangle$ are $G$--parabolic and hence $\sup_{\om}\ell(g_n;T^v)<+\infty$. Thus, we can simply take $h_n:=\varphi_n^{-1}(g_n)$.
\end{rmk}

\begin{rmk}\label{rmk:elliptic_in_degeneration}
    If the degeneration $G\acts\X_{\om}$ arises from a sequence $\varphi_n\in\Aut(G;\mscr{E}^t)$, where $\mscr{E}$ is an arbitrary family of arbitrary subgroups of $G$, then all $E\in\mscr{E}$ are elliptic in $\X_{\om}$. Indeed, each element $g\in E$ satisfies $\ell(\varphi_n(g);\X_{\G})=\ell(g;\X_{\G})$ for all $n$, and hence also $\ell(g;\X_{\om})=0$. It follows that each finitely generated subgroup $E_0\leq E$ is elliptic in all the $\R$--trees $T^v_{\om}$ (by Serre's lemma \cite[p.\,64]{Serre}), and hence it is elliptic in $\X_{\om}$. Finally, the entire $E$ is itself elliptic because chains of arc-stabilisers in the trees $T^v_{\om}$ have uniformly bounded length (using e.g.\ \cite[Lemma~2.18]{Fioravanti-Kerr}).
\end{rmk}

We can now relate IOM-lines to the isolating splittings considered in Sections~\ref{sub:more_DTs} and~\ref{sub:poison}, as well as to the systems of complements from \Cref{defn:SOC}.

\begin{prop}\label{prop:salient_vs_degenerations}
    Consider a degeneration $G\acts\X_{\om}$ and some $v\in\G$. Let $\mscr{F}$ be a finite set of finitely generated subgroups of $G$ that are elliptic in $T^v_{\om}$. If $\alpha\sq\Min(G;T^v_{\om})$ is an IOM-line with $G_{\alpha}\acts\alpha$ nontrivial,
    then the following hold.
    \begin{enumerate}
        \item The group $G$ has a $(G_{\alpha},\mscr{F})$--isolating splitting $T$ with locus contained in the kernel of $G_{\alpha}\acts\alpha$.
        If $B\in\Sal(G)$ does not contain a conjugate of the centre of $G_{\alpha}$, then $B$ is elliptic in $T$.
        \item If $\X_{\om}$ arises from a sequence $\varphi_n\in\Aut(G;\mf{C}^t)$ for an $\mscr{F}$--system of complements $\mf{C}$, then the centre of $G_{\alpha}$ is elliptic in $T^v_{\om}$.
    \end{enumerate}
\end{prop}
\begin{proof}
    For part~(1), consider a geometric approximation $G\acts\mc{G}$ of $\Min(G;T^v_{\om})$ and set $\mc{M}:=\Min(G_{\alpha};\mc{G})$. Choosing $\mc{G}$ suitably, we can ensure that all subgroups in $\mscr{F}$ are elliptic in $\mc{G}$,
    and that the approximating morphism maps $\mc{M}$ into $\alpha$. 
    Now, \Cref{lem:line_addenda}(3) shows that the line $\alpha$ is ascetic, and hence distinct $G$--translates of the subtree $\mc{M}\sq\mc{G}$ share at most one point; moreover, the $G$--stabiliser of $\mc{M}$ is precisely $G_{\alpha}$. As observed in \Cref{rmk:ascetic->isolating} (also see Steps~1--2 in \Cref{subsub:shrinkable_lines} below for details), this allows us to construct a splitting $G\acts\mc{S}$ with the following properties:
    \begin{itemize}
        \item there is a vertex $[\mc{M}]\in\mc{S}$ whose $G$--stabiliser is $G_{\alpha}$;
        \item the $G$--stabiliser of each edge incident to $[\mc{M}]$ is the $G_{\alpha}$--stabiliser of a point of $\mc{M}\sq\mc{G}$;
        \item if a subgroup $H\leq G$ leaves invariant a subtree $\mc{H}\sq\mc{G}$ sharing at most one point with each $G$--translate of $\mc{M}$, then $H$ is elliptic in $\mc{S}$ (in particular, this is the case if $H$ is elliptic in $\mc{G}$).
    \end{itemize}
    These properties imply that $G\acts\mc{S}$ is $(G_{\alpha},\mscr{F})$--isolating with locus contained in the kernel of the action $G_{\alpha}\acts\alpha$. 

    Now, let $A$ be the centre of $G_{\alpha}$. We are left to check that, if $B\in\Sal(G)$ does not contain any conjugates of $A$, then $B$ is elliptic in $\mc{S}$. This is clear if $B$ is elliptic in $T^v_{\om}$, or if $\Min(B;T^v_{\om})$ does not share arcs with any $G$--translate of $\alpha$, provided that we chose the geometric approximation $\mc{G}$ suitably. Suppose instead that $\Min(B;T^v_{\om})$ shares an arc with $\alpha$ (after possibly replacing $B$ with a $G$--conjugate). By \Cref{rmk:abelians_give_IOM}, the line $\Min(B;T^v_{\om})$ is IOM and so \Cref{lem:line_addenda}(3) implies that $\Min(B;T^v_{\om})=\alpha$. Now, the centraliser $Z_G(B)$ preserves the axis of $B$ and hence $Z_G(B)\leq G_{\alpha}=Z_G(A)$, using \Cref{lem:line_addenda}(1). Since $A$ and $B$ are centralisers, it follows that $A=Z_GZ_G(A)\leq Z_GZ_G(B)\leq B$, violating our assumptions on $B$. This proves part~(1).

    Regarding part~(2), if the centre $A$ does not lie in $\xSal(G)$, then $A$ is certainly elliptic in $T^v_{\om}$ by \Cref{lem:line_stab_SVP}(2). Thus, suppose that $A\in\xSal(G)$.
    Denoting by $K_{\alpha}\lhd Z_G(A)$ the kernel of the action $Z_G(A)\acts\alpha$, the splitting constructed in part~(1) shows that we have $\mc{L}_{\mscr{F}}(A)\leq K_{\alpha}$, where $\mc{L}_{\mscr{F}}(A)$ is the locus of $(Z_G(A),\mscr{F})$--shunning trees of $G$. By the definition of systems of complements we have $A=(A\cap\mc{L}_{\mscr{F}}(A))\oplus\mf{C}_A$, and we have just seen that the first direct summand is elliptic in $T^v_{\om}$. Now, if $\X_{\om}$ arises from a sequence $\varphi_n\in\Aut(G;\mf{C}^t)$, then $\mf{C}_A$ is elliptic in $\X_{\om}$ and $T^v_{\om}$ by \Cref{rmk:elliptic_in_degeneration}. This shows that $A$ is elliptic in $T^v_{\om}$ as desired.
\end{proof}

\section{Shortening a degeneration}\label{sect:shortening}

This section develops our hierarchical shortening argument and it constitutes the technical core of the article. First, in \Cref{sub:shortening_median}, we develop a general framework to shorten an action on a median space, provided that there is an equivariant projection to an $\R$--tree with arcs or lines of a certain kind (\Cref{defn:shrinkable_arcs} and \Cref{defn:shrinkable_lines}). We then apply this framework to the degenerations of a special group $G$ in \Cref{sub:shortening_special}, where we prove two results producing Dehn twists shortening finite subsets of $G$ (Propositions~\ref{prop:shortening_inert} and~\ref{prop:shortening_motile}). Finally, \Cref{sub:shortening_theorem} is uniquely concerned with the proof of the shortening theorem (\Cref{thm:shortening}). Once this is obtained, the proof of all the main theorems will follow quickly and easily in \Cref{sect:proofs}. 

\subsection{Shortening an action on a median space}\label{sub:shortening_median}

This subsection is concerned with the problem of shortening a general group action on a median space. After reviewing some general terminology, we introduce shortening isometries in \Cref{subsub:shortening_isometries}. We then discuss a shortening procedure arising from two different features: \emph{shrinkable arcs} in \Cref{subsub:shrinkable_arcs}, and \emph{shrinkable lines} in \Cref{subsub:shrinkable_lines}.

\subsubsection{Generalities}

For definitions and background on median metric spaces, we refer to \cite[Section~2]{CDH}, \cite[Sections~4--7]{Bow13}, \cite[Section~2.1]{Fio1}, and to the excellent monograph \cite{Bow-book}.

Let $X$ be a geodesic median space with compact intervals. Note that we do not assume completeness of $X$, and indeed our examples of interest will almost never be complete. For instance, an $\R$--tree $T\not\cong\R$ admitting a minimal action of a countable group with dense orbits cannot be complete, as explained in \cite[Example~II.6]{GL95}.

A \emph{halfspace} of $X$ is a nonempty convex subset $\mf{h}\sq X$ such that the complement $\mf{h}^*:=X\setminus\mf{h}$ is also nonempty and convex. An \emph{(abstract) wall} $\mf{w}$ is an unordered pair $\{\mf{h},\mf{h}^*\}$ for some halfspace $\mf{h}$; we say that the halfspaces $\mf{h}$ and $\mf{h}^*$ are \emph{bounded} by the wall $\mf{w}$. Each wall $\mf{w}=\{\mf{h},\mf{h}^*\}$ determines a \emph{geometric wall} $\mf{W}\sq X$, namely the intersection of the closures of the two corresponding halfspaces; in symbols, $\mf{W}=\overline{\mf{h}}\cap\overline{\mf{h}^*}$. Two subsets $A,B\sq X$ are \emph{separated} by a wall $\mf{w}=\{\mf{h},\mf{h}^*\}$ if $A\sq\mf{h}$ and $B\sq\mf{h}^*$, or vice versa. A subset $A$ is \emph{crossed} by $\mf{w}$ if two points of $A$ are separated by $\mf{w}$.

\begin{rmk}
    Geometric walls are gate-convex subsets of $X$: despite the lack of completeness, this can be deduced from compactness of intervals; see \cite[Lemma~2.6]{Fio1}.
    
    Additionally, all median spaces of relevance to this article have finite rank. This implies that geometric walls are nowhere-dense \cite[Corollary~2.23]{Fio1} and have strictly lower rank than $X$ \cite[Lemma~7.5]{Bow13}, but these latter properties are not used in the discussion below.
\end{rmk}

A \emph{wallpair} is a pair $(\mf{w},\mf{w}')$ of walls of $X$ such that either $\mf{w}=\mf{w}'$, or the corresponding geometric walls $\mf{W},\mf{W}'$ are disjoint. The wallpair is \emph{degenerate} in the former case and \emph{non-degenerate} in the latter; note that a non-degenerate wallpair always has $d(\mf{W},\mf{W}')>0$. We say that a tuple of walls $(\mf{w}_1,\dots,\mf{w}_k)$ is a \emph{chain} if we can choose halfspaces $\mf{h}_i$ bounded by $\mf{w}_i$ so that $\mf{h}_1\sq\mf{h}_2\sq\dots\sq\mf{h}_k$. Wallpairs are a very particular kind of chain of length $2$. A chain of walls \emph{separates} two subsets $A,B\sq X$ if all walls in the chain separate the two subsets; in other words, we can choose the $\mf{h}_i$ so that $A\sq\mf{h}_1$ and $B\sq\mf{h}_k^*$ or vice versa.

If $C_1,C_2\sq X$ are gate-convex subsets with gate-projections $\pi_i\colon X\ra C_i$, then the sets $\mf{s}_1:=\pi_1(C_2)\sq C_1$ and $\mf{s}_2:=\pi_2(C_1)\sq C_2$ are again gate-convex within $X$, and they are isometric to each other: the restriction of $\pi_2$ gives an isometry $\mf{s}_1\ra\mf{s}_2$ whose inverse is the restriction of $\pi_1$. We refer to $\mf{s}_1,\mf{s}_2$ as the \emph{shores} for the pair of gate-convex sets $C_1,C_2$.
Note that a pair of points $(x_1,x_2)\in C_1\x C_2$ realises the distance between $C_1$ and $C_2$ if and only if $(x_1,x_2)\in\mf{s}_1\x\mf{s}_2$ and $x_2=\pi_2(x_1)$. See \cite[Section~2.2]{Fio3} for further details.

\subsubsection{Shortening isometries}\label{subsub:shortening_isometries}

We are interested in isometries that move a geometric wall towards another without shearing along the wall. They are the most important concept for this subsection:

\begin{defn}\label{defn:shortening_isom}
    A \emph{shortening isometry} for a wallpair $(\mf{w},\mf{w}')$ is an element $\Phi\in\isom(X)$ such that $\Phi|_{\mf{s}'}=\pi|_{\mf{s}'}$, where $\mf{s}\sq\mf{W}$, $\mf{s}'\sq\mf{W}'$ are the shores for the pair of geometric walls and $\pi\colon X\ra\mf{s}$ denotes the gate-projection to $\mf{s}$.	
\end{defn}

Note that a shortening isometry $\Phi$ is not required to take the wall $\mf{w}'$ to $\mf{w}$, although the geometric walls $\mf{W}$ and $\Phi(\mf{W}')$ must intersect, as they both contain the shore $\mf{s}=\Phi(\mf{s}')$. Also note that we do not require the walls $\mf{w}'$ and $\Phi(\mf{w}')$ to be on the same side of $\mf{w}$.

If $\Phi$ is a shortening isometry for a wallpair $(\mf{w},\mf{w}')$, then $\Phi(\mf{u})=\mf{u}$ for all walls $\mf{u}$ crossing both $\mf{W}$ and $\mf{W}'$. Moreover, $\Phi^{-1}$ is a shortening isometry for the reverse wallpair $(\mf{w}',\mf{w})$. In the case of a degenerate wallpair, shortening isometries are those that fix the entire geometric wall $\mf{W}$ pointwise; the identity is always a shortening isometry in that case.

The following remark contains the basic idea at the core of the shortening process for median spaces. We say that a sequence of points $x_1,\dots,x_m\in X$ is a \emph{geodesic sequence} if there exists a geodesic in $X$ meeting the points in this order.

\begin{rmk}\label{rmk:basic_shortening}
    Consider two points $x,y\in X$, a wallpair $(\mf{w},\mf{w}')$ separating $x$ from $y$, and a shortening isometry $\Phi$ for $(\mf{w},\mf{w}')$. We then have 
    \[ d(x,\Phi y)\leq d(x,y)-d(\mf{W},\mf{W}') .\] 
    In order to see this, let $z,z'$ be the gate-projections of $x$ to $\mf{W},\mf{W}'$, respectively, and let $w$ be the gate-projection of $z'$ to $\mf{W}$. Note that $x,z,w,z',y$ is a geodesic sequence, and that $d(w,z')=d(\mf{W},\mf{W}')$. Since we have $\Phi(z')=w$, this yields the claimed inequality.
\end{rmk}

We can apply the previous observation to shorten an isometric action $G\acts X$, provided that the group $G$ has a suitable algebraic splitting. We will use the two parts of the following lemma to formalise this idea for HNN and amalgamated-product splittings, respectively. 

\begin{lem}\label{lem:basic_shortening}
    Consider isometries $a_1,\dots,a_m,b_1,\dots,b_m\in\isom(X)$ for some $m\geq 1$, and set $g_i:=(a_1b_1)\dots (a_ib_i)$ for $0\leq i\leq m$. Consider wallpairs $(\mf{w}_i,\mf{w}_i')$ for $1\leq i\leq m$, and shortening isometries $\Phi_i$ for them. Finally, consider a basepoint $x\in X$.
    \begin{enumerate}
    \setlength\itemsep{.2em}
    \item Suppose that the wallpair $(\mf{w}_i,\mf{w}_i')$ separates $x$ from $a_ix$ for each $1\leq i\leq m$, and that the tuple $(\mf{w}_1,\mf{w}_1',\dots,g_{i-1}\mf{w}_i,g_{i-1}\mf{w}_i',\dots,g_{m-1}\mf{w}_m,g_{m-1}\mf{w}_m')$ is a chain separating $x$ from $g_m x$. Then, setting $g_m':=(\Phi_1a_1b_1)\dots (\Phi_ma_mb_m)$, we have $d(x,g_m' x)\leq d(x,g_m x)$ and the inequality is strict if at least one of the wallpairs is non-degenerate.
    \item Suppose that $(\mf{w}_i,\mf{w}_i',a_i\mf{w}_i',a_i\mf{w}_i)$ separates $x$ from $a_ix$ for each $1\leq i\leq m$, and that the tuple obtained by concatenating the $4$--tuples $g_{i-1}(\mf{w}_i,\mf{w}_i',a_i\mf{w}_i',a_i\mf{w}_i)$ is a chain separating $x$ from $g_m x$. Then, setting $g_m':=(\Phi_1a_1\Phi_1^{-1}b_1)\dots (\Phi_ma_m\Phi_m^{-1}b_m)$, we have $d(x,g_m'x)\leq d(x,gx)$ and the inequality is strict if at least one of the wallpairs is non-degenerate.
    \end{enumerate}
\end{lem}
\begin{proof}
    Consider first part~(1). Pick geodesics $\alpha_i,\beta_i\sq X$ from $x$ to $a_ix$ and $b_ix$, respectively. Let $\gamma$ be the path from $x$ to $g_mx$ obtained by concatenating the paths $g_{i-1}\alpha_i$ and $g_{i-1}a_i\beta_i$; note that this will not be a geodesic from $x$ to $g_mx$ in general. Nevertheless, part~(1) immediately follows from a repeated application of \Cref{rmk:basic_shortening}, replacing the elements $a_i$ with $\Phi_ia_i$ one at a time, in any order: with each replacement, the distance $d(x,g_mx)$ does not increase, and it decreases by at least the distance $d(\mf{W}_i,\mf{W}_i')$ if the wallpair $(\mf{w}_i,\mf{w}_i')$ is non-degenerate.
	
    Part~(2) then follows by a double application of part~(1). Indeed, both wallpairs $(\mf{w}_i,\mf{w}_i')$ and $a_i(\mf{w}_i',\mf{w}_i)$ separate $x$ from $a_ix$, and we have $\Phi_ia_i\Phi_i^{-1}=\Phi_i(a_i\Phi_i^{-1}a_i^{-1})a_i$, where $\Phi_i$ is a shortening isometry for the wallpair $(\mf{w}_i,\mf{w}_i')$, and $a_i\Phi_i^{-1}a_i^{-1}$ is a shortening isometry for $a_i(\mf{w}_i',\mf{w}_i)$.
\end{proof}

\subsubsection{Shrinkable arcs}\label{subsub:shrinkable_arcs}

In this subsection, we consider the following setting:
\begin{itemize}
    \item $G$ is a finitely presented group;
    \item $X$ is a finite-rank, geodesic median space with compact intervals;
    \item $\rho\colon G\ra\isom(X)$ is a homomorphism inducing an action $G\acts X$;
    \item $G\acts\mscr{T}$ is a minimal action on an $\R$--tree such that there exists a $G$--equivariant, Lipschitz, median-preserving map $\Pi\colon X\ra\mscr{T}$. 
\end{itemize}

Let $[x,y]\sq\mscr{T}$ be an arc
and let $\pi_{x,y}\colon \mscr{T}\ra [x,y]$ be its nearest-point projection. For each point $z\in [x,y)$, we obtain partition of $X$ into two halfspaces, namely 
\[ X=(\Pi\o\pi_{x,y})^{-1}\big([x,z]\big)\sqcup (\Pi\o\pi_{x,y})^{-1}\big((z,y]\big) ;\] 
we denote by $\mf{w}_{[x,y]}(z)$ the corresponding wall of $X$. We say that a wall of $X$ \emph{arises} from the arc $[x,y]$ if it is of the form $\mf{w}_{[x,y]}(z)$ for a point $z\in [x,y)$.  We also say that a wallpair \emph{arises} from $[x,y]$ if it is of the form $(\mf{w}_{[x,y]}(z),\mf{w}_{[x,y]}(z'))$ for two points $z,z'\in [x,y)$.
Walls and wallpairs arising from the reverse arc $[y,x]$ are different from those arising from $[x,y]$, but this will not cause issues.

For each isometry $g\in\isom(X)$, we can consider the \emph{minset}
\[ \mc{L}(g;X):=\{x\in X\mid d(x,gx)=\ell(g;X) \}. \]
Since $X$ is finite-rank and connected, we always have $\mc{L}(g;X)\neq\emptyset$ by \cite[Corollary~A]{Fio10b}. Note that, for all elements $g\in G$, we have the inclusion $\Pi(\mc{L}(\rho(g);X))\sq\mc{L}(g;\mscr{T})$ 
and, if $g$ is loxodromic in $\mscr{T}$, this is actually an equality $\Pi(\mc{L}(\rho(g);X))={\rm Ax}(g;\mscr{T})$ (we prefer the latter notation for minsets of loxodromics in trees, namely their axes).

We have briefly considered geometric $\R$--trees in \Cref{sub:hierarchy}, and we refer to the discussion there for relevant terminology. Now, it additionally becomes important that each minimal $\R$--tree $G\acts T$ can be approximated by geometric $\R$--trees in the topology of strong convergence \cite{LP97}. Since $G$ is finitely presented, every sufficiently close geometric $\R$--tree $\mc{G}$ is equipped with an action of the group $G$ itself, and the approximating morphism $\mc{G}\ra T$ is equivariant: we will only consider geometric approximations of this kind, without mentioning this again explicitly.

The following is the first kind of feature that allows us to shorten the $G$--action on the median space $X$. We say that an arc $\beta$ in an $\R$--tree $T$ is \emph{edge-like} if the interior of $\beta$ does not contain any branch points of $T$.

\begin{defn}[Shrinkable arcs]\label{defn:shrinkable_arcs}
    Consider a geometric approximation $f\colon\mc{G}\ra\mscr{T}$ and an arc $\beta\sq\mscr{T}$. We say that the pair $(\beta,\mc{G})$ is \emph{shrinkable} if both the following hold:
    \begin{enumerate}
        \item $\beta\sq\mscr{T}$ isometrically lifts to an edge-like arc $\tilde\beta\sq\mc{G}$ with the same $G$--stabiliser;
        \item there exists a non-degenerate wallpair of $X$ arising from $\beta$ that admits a shortening isometry $\Phi\in\isom(X)$ commuting with the subgroup $\rho(G_{\beta})\leq\isom(X)$.
    \end{enumerate}
\end{defn}

\begin{rmk}
    Given an arc $\beta\sq\mscr{T}$, the typical situation in which we will be able to find a geometric approximation $\mc{G}$ such that $(\beta,\mc{G})$ is shrinkable (in the setting of degenerations of special groups) is when $\beta$ is stable, $G_{\beta}$ is finitely generated, and the normaliser $N_G(G_{\beta})$ is elliptic in $\mscr{T}$.
\end{rmk}

In the setting of \Cref{defn:shrinkable_arcs}, we can define a simplicial tree $\mc{G}(\tilde\beta)$ by collapsing all arcs of $\mc{G}$ that are disjoint from the interiors of the $G$--translates of $\tilde\beta$. More formally, vertices of $\mc{G}(\tilde\beta)$ are connected components of the set $\mc{G}\setminus\bigcup_{g\in G}g\cdot{\rm int}(\tilde\beta)$, with two vertices spanning an edge when the corresponding subsets of $\mc{G}$ are separated by a single $G$--translate of $\tilde\beta$. The natural action $G\acts\mc{G}(\tilde\beta)$ is minimal and edge-transitive. The arc $\tilde\beta\sq\mc{G}$ determines a distinguished edge $e(\tilde\beta)\sq\mc{G}(\tilde\beta)$, whose stabiliser coincides with $G_{\tilde\beta}=G_{\beta}$. Thus, the action $G\acts\mc{G}(\tilde\beta)$ determines a splitting of $G$ as an amalgamated product or HNN extension over the subgroup $G_{\beta}$. 

Let $u,v$ denote the vertices of the edge $e(\tilde\beta)$, and let $U,V$ be their $G$--stabiliser; in the HNN case, also choose an element $t\in G$ with $tu=v$. By definition, there is a shortening isometry $\Phi\in\isom(X)$ for a non-degenerate wallpair $(\mf{w},\mf{w}')$, which arises from distinct points $z,z'\in\beta$. By lifting to $\mc{G}$ and then projecting to $\mc{G}(\tilde\beta)$, the points $z,z'$ determine points $\overline z,\overline z'$ on the edge $e(\tilde\beta)$. We should be careful to have labeled the vertices of $e(\tilde\beta)$ so that $u$ is closer to $\overline z$, while $v$ is closer to $\overline z'$ (otherwise $\Phi$ needs to be replaced by $\Phi^{-1}$ in the following definition).

Since the shortening isometry $\Phi$ commutes with $\rho(G_{\beta})$ by definition, we can use it to define a new ``shrunk'' representation $\rho'\colon G\ra\isom(X)$.

\begin{defn}\label{defn:shrunk_homo_arcs}
    Let $(\beta,\mc{G})$ be a shrinkable arc, and let the shortening isometry $\Phi$ and the $1$--edge splitting $G\acts \mc{G}(\tilde\beta)$ be as above. The \emph{shrunk homomorphism} $\rho'\colon G\ra\isom(X)$ is defined as follows:
    \begin{itemize}
        \item if $G\acts\mc{G}(\tilde\beta)$ is an amalgam, we set $\rho'(g):=\rho(g)$ for $g\in U$ and $\rho'(h)=\Phi\rho(h)\Phi^{-1}$ for $h\in V$;
        \item if $G\acts\mc{G}(\tilde\beta)$ is an HNN extension, we set $\rho'(g):=\rho(g)$ for $g\in U$ and $\rho'(t):=\Phi\rho(t)$.
    \end{itemize}
\end{defn}

The shrunk homomorphism is ``shorter'' in the following sense.

\begin{prop}\label{prop:shrinkable_arc}
    Let $(\beta,\mc{G})$ be a shrinkable arc and let $\rho'\colon G\ra\isom(X)$ be the resulting shrunk homomorphism. The following statements hold for all elements $g\in G$.
    \begin{enumerate}
        \item If $g$ is elliptic in $\mc{G}(\tilde\beta)$, then $\rho'(g)$ and $\rho(g)$ are conjugate, so $\ell(\rho'(g);X)=\ell(\rho(g);X)$. 
        \item Let $g$ be loxodromic in $\mscr{T}$ with $\beta\sq{\rm Ax}(g;\mscr{T})$. 
        Suppose that the morphism $f\colon\mc{G}\ra\mscr{T}$ restricts to an isometry on the axis ${\rm Ax}(g;\mc{G})$.
        Then we have $\ell(\rho'(g);X)<\ell(\rho(g);X)$. 
    \end{enumerate}
\end{prop}
\begin{proof}
    Part~(1) is clear, so we concentrate on part~(2). Our hypotheses imply that $\ell(g;\mc{G})=\ell(g;\mscr{T})=:\ell$; choose an arc $\eta\sq{\rm Ax}(g;\mc{G})$ of length $\ell$. Let $g_1\tilde\beta,\dots,g_k\tilde\beta$ be the $G$--translates of the arc $\tilde\beta\sq\mc{G}$ that share an arc with $\eta$, in the order in which they are met along $\eta$. In fact, up to slightly sliding $\eta$ along ${\rm Ax}(g;\mc{G})$, we can assume that the $g_i\tilde\beta$ are all contained in $\eta$ and, up to conjugating the element $g$, it is also not restrictive to assume that $g_1\tilde\beta=\tilde\beta$. We also have $g_1\beta=\beta$, and the arc $f(\eta)\sq{\rm Ax}(g;\mscr{T})$ has length $\ell$ and contains the arcs $\beta,g_2\beta,\dots,g_k\beta$, which have pairwise disjoint interiors. (However, it is important to remember that the picture in $\mscr{T}$ can be rather different from the one in $\mc{G}$: there can be infinitely many other $G$--translates of $\beta$ sharing an arc with $f(\eta)$.) 

    Let $(\mf{w},\mf{w}')$ be the non-degenerate wallpair of $X$ for which $\Phi$ is a shortening isometry, and let $z,z'\in\beta\sq\mscr{T}$ be the distinct points from which this wallpair arises. We set $\Phi_i:=\rho(g_i)\Phi \rho(g_i)^{-1}$, which is a shortening isometry for the wallpair $\rho(g_i)(\mf{w},\mf{w}')$, which arises from the arc $g_i\beta\sq\mscr{T}$. Picking any point $x\in X$ such that $\Pi(x)$ is the initial endpoint of the arc $f(\eta)$, we have that $\Pi(\rho(g)x)$ is the terminal endpoint of $f(\eta)$. In particular, the concatenation of the wall pairs $\rho(g_i)(\mf{w},\mf{w}')$ is a chain of walls of $X$ separating $x$ from $\rho(g)x$. In fact, recalling that $\Pi(\mc{L}(\rho(g);X))={\rm Ax}(g;\mscr{T})$, we can ensure that $x\in\mc{L}(\rho(g);X)$ when we choose the point $x$.
    
    Now, writing $g$ in normal form with respect to the $1$--edge splitting $G\acts\mc{G}(\tilde\beta)$ and invoking \Cref{lem:basic_shortening} with respect to the wallpairs  $\rho(g_i)(\mf{w},\mf{w}')$ and shortening isometries $\Phi_i$, we obtain that
    \[ d(x,\rho'(g)x)< d(x,\rho(g)x)=\ell(\rho(g);X) ,\]
    and therefore $\ell(\rho'(g);X)<\ell(\rho(g);X)$ as required.
\end{proof}

\subsubsection{Shrinkable lines}\label{subsub:shrinkable_lines}

It will not always be possible to find shrinkable arcs in our cases of interest, so there is a second feature of our actions that we will need to consider in order to shorten them. Throughout, we retain the general setup described at the start of \Cref{subsub:shrinkable_arcs}.

Recall that a subtree in a $G$--tree is \emph{ascetic} if it shares at most one point with each of its $G$--translates. For a line $\alpha$, we denote by $G_{\alpha}$ its $G$--stabiliser, and by $K_{\alpha}$ the kernel of $G_{\alpha}\acts\alpha$.

\begin{defn}[Shrinkable lines]\label{defn:shrinkable_lines}
    A line $\alpha\sq\mscr{T}$ is \emph{shrinkable} if $G_{\alpha}\neq K_{\alpha}$ and the following hold:
    \begin{enumerate}
        \item the line $\alpha$ is ascetic, the stabiliser $G_{\alpha}$ is finitely generated, and $G_{\alpha}/K_{\alpha}$ is free abelian;
        \item for arbitrarily small numbers $\eps>0$, there exists an element $\Phi_{\eps}\in\isom(X)$ commuting with $\rho(G_{\alpha})$ such that either $\Phi_{\eps}$ or $\Phi_{\eps}^{-1}$ is a shortening isometry for each wallpair of the form $(\mf{w}_{\beta}(z),\mf{w}_{\beta}(z'))$, where $\beta\sq\alpha$ is an arc and $z,z'\in{\rm int}(\beta)$ are points with $d(z,z')=\eps$.
    \end{enumerate}
\end{defn}

Note that our definition of shrinkability is a bit stronger for lines than it is for arcs: we want the shortening isometries to simultaneously work for \emph{all} wallpairs arising from the line. This is also what allows us to make the definition independent of the choice of a geometric approximation.

\begin{rmk}
    In degenerations of special groups, shrinkable lines will originate from IOM-lines $\alpha$ such that the centre of $G_{\alpha}$ is elliptic. The case with non-elliptic centre is more delicate, and it is precisely what can lead to the existence of poison subgroups and infinite generation of $\Out(G)$.
\end{rmk}

The definition of the shrunk homomorphism and the shortening procedure are significantly more involved in the case of lines. They rely on the following construction, which can be applied to any sufficiently close geometric approximation of $\mscr{T}$, and which we explain in several steps.

Let $\alpha$ be a shrinkable line. Choose a finite generating set $S\sq G_{\alpha}$ with $S=S^{-1}$ and a basepoint $x\in\alpha$. Let $f\colon\mc{G}\ra\mscr{T}$ be a geometric approximation admitting a point $\tilde x\in f^{-1}(x)$ such that $f$ is isometric on each of the arcs $[\tilde x,s\tilde x]$ with $s\in S$. Let $\mc{M}$ be the $G_{\alpha}$--minimal subtree of $\mc{G}$.
Note that, in general, $\mc{M}$ is not a line.

\smallskip
{\bf Step~1.} \emph{Properties of the $G_{\alpha}$--minimal subtree $\mc{M}\sq\mc{G}$.}

\smallskip\noindent
We claim that $\mc{M}\sq\mc{G}$ is ascetic and has $G$--stabiliser precisely $G_{\alpha}$. Indeed, each arc of $\mc{M}$ shares an arc with $h[\tilde x,s\tilde x]$ for some $s\in S$ and some $h\in G_{\alpha}$. Supposing that $\delta$ is an arc contained in $g\mc{M}\cap\mc{M}$ for some $g\in G$, we have that a sub-arc $\delta'\sq\delta$ is contained in an intersection $h[\tilde x,s\tilde x]\cap gh'[\tilde x,s'\tilde x]$ with $s,s'\in S$ and $h,h'\in G_{\alpha}$. By our choice of $\mc{G}$, it follows that the morphism $f$ maps $\delta'$ isometrically onto an arc of $\alpha\cap g\alpha$, and so $g\in G_{\alpha}$ since $\alpha$ is ascetic.

By the claim, $\mc{M}$ contains each indecomposable component of $\mc{G}$ that it shares an arc with.
Thus, $\mc{M}$ is a union of edge-like arcs and indecomposable components of $\mc{G}$, which shows that each arc of $\mc{G}$ intersects only finitely many distinct $G$--translates of $\mc{M}$. 

\smallskip
{\bf Step~2.} \emph{The bipartite splitting $G\acts\mc{S}$.}

\smallskip\noindent
By Step~1, there is a transverse covering of $\mc{G}$ (in the sense of \cite[Definition~1.4]{Guir-Fourier}) formed by the $G$--translates of $\mc{M}$ and the closures of the connected components of $\mc{G}\setminus\bigcup_{g\in G}g\mc{M}$. We thus obtain a bipartite simplicial splitting $G\acts\mc{S}$ constructed as follows: $\mc{S}$ has a white vertex for each translate $g\mc{M}$ with $g\in G$; there is a black vertex of $\mc{S}$ for every maximal closed subset of $\mc{G}$ that is connected by arcs intersecting each translate $g\mc{M}$ in at most one point (we disregard maximal subsets if they are a single point and this point lies in a unique $G$--translate of $\mc{M}$); finally, edges of $\mc{S}$ connect a white vertex to a black one if the corresponding subsets of $\mc{G}$ intersect, which necessarily happens at a single point. The induced action $G\acts\mc{S}$ is minimal. (Note that $\mc{S}$ is a collapse of the tree associated to the transverse covering in \cite[Section~1.3]{Guir-Fourier}.) 

Let $\mc{M}_*\sq\mc{M}$ be the subset of \emph{exit points} for $\mc{M}$ within $\mc{G}$: these are the points $x\in\mc{M}$ for which there exists a point $y\in\mc{G}\setminus\{x\}$ such that the half-open arc $(x,y]$ is contained in $\mc{G}\setminus\mc{M}$.

The quotient $\mc{S}/G$ is a finite bipartite graph with a single white vertex and some number of black vertices. Denoting by $[\mc{M}]$ the unique white vertex of the tree $\mc{S}$ that has $G_{\alpha}$ as its stabiliser, we have that the $G$--stabiliser of each edge incident to $[\mc{M}]$ is contained in the normal subgroup $K_{\alpha}\lhd G_{\alpha}$.
These incident edges are naturally in $1$--to--$1$ correspondence with the points of $\mc{M}_*$. Similarly, black vertices of $\mc{S}/G$ are in $1$--to--$1$ correspondence with $G$--orbits of points in $\mc{M}_*$, while edges of $\mc{S}/G$ correspond to $G_{\alpha}$--orbits of such points. 

\smallskip
{\bf Step~3.} \emph{The amalgams $G\acts\mc{A}(\mc{E})$.}

\smallskip\noindent
Let $\mc{E}$ be a (finite) set of representatives for the $G_{\alpha}$--orbits of edges of $\mc{S}$ incident to the vertex $[\mc{M}]$ (equivalently, for the $G_{\alpha}$--orbits in the set $\mc{M}_*$). We will choose a specific set $\mc{E}$ after Step~6. 

The choice of $\mc{E}$ uniquely determines an amalgamated-product splitting of $G$ as follows. Starting from $\mc{S}$, we first collapse all edges not incident to any $G$--translate of $[\mc{M}]$. Then we fold, for each $g\in G$, the edges of the set $gK_{\alpha}\cdot\mc{E}$ into a single edge. The result $G\acts\mc{A}(\mc{E})$ is the Bass--Serre tree of a writing $G=G_{\alpha}\ast_{K_{\alpha}} H_{\mc{E}}$ for a subgroup $H_{\mc{E}}\leq G$. Any subgroup of $G$ that is elliptic in $\mc{G}$ or $\mc{S}$ has a $G$--conjugate contained in $H_{\mc{E}}$. In the terminology of \Cref{sub:shunning+ascetic}, we also have that $\mc{A}(\mc{E})$ is a $(G_{\alpha},\emptyset)$--isolating splitting of $G$ with locus $K_{\alpha}$.

\smallskip
{\bf Step~4.} \emph{The HNN splitting $G\acts\mc{H}(\mc{E},\eta)$.}

\smallskip\noindent
Let $\eta\colon G_{\alpha}\twoheadrightarrow\Z$ be an epimorphism with $K_{\alpha}\leq\ker(\eta)$. Applying \Cref{lem:isolating->ascetic} to $\mc{A}(\mc{E})$, we obtain an HNN splitting $G\acts\mc{H}(\mc{E},\eta)$ over the subgroup $\ker(\eta)$. The subgroup $H_{\mc{E}}$ is elliptic in $\mc{H}(\mc{E},\eta)$, as are all subgroups of $G$ that are elliptic in $\mc{A}(\mc{E})$ and do not intersect any $G$--conjugate of $G_{\alpha}\setminus\ker(\eta)$. Moreover, we can choose within $G_{\alpha}$ a stable letter for the HNN splitting $\mc{H}(\mc{E},\eta)$: it suffices to pick any element $t\in G_{\alpha}$ with $\eta(t)=\pm 1$.

(Note that $G\acts\mc{S}$ was also $(G_{\alpha},\emptyset)$--isolating, but applying \Cref{lem:isolating->ascetic} directly to $\mc{S}$ would have given us less control over many hidden choices. In particular, the subgroup $H_{\mc{E}}$ might not have stayed elliptic in such an HNN, which would be problematic in the discussion after Step~6.)

We are finally ready to define a shrunk homomorphism $\rho_{\mc{E},\eta,\delta}\colon G\ra\isom(X)$, which depends on the choices of $\mc{E},\eta$ and $\mc{G}$, as well as on a small parameter $\delta>0$.

\begin{defn}\label{defn:shrunk_homo_lines}
    Let $\alpha$ be a shrinkable line and let the HNN splitting $G\acts\mc{H}(\mc{E},\eta)$ be as above, for some choice of $\mc{E}$ and $\eta$. Pick a stable letter $t\in G_{\alpha}$, and let $U$ be the $G$--stabiliser of some vertex on the axis ${\rm Ax}(t;\mc{H}(\mc{E},\eta))$. Suppose that the shortening isometries $\Phi_{\eps}$ in \Cref{defn:shortening_isom} move walls along $\alpha$ in the direction opposite to the direction of translation of $t$ along $\alpha$ (otherwise replace $\Phi_{\eps}$ with its inverse). For each $\delta>0$, the \emph{shrunk homomorphism} $\rho_{\mc{E},\eta,\delta}\colon G\ra\isom(X)$ is defined by $\rho_{\mc{E},\eta,\delta}(g)=\rho(g)$ for all $g\in U$ and $\rho_{\mc{E},\eta,\delta}(t)=\Phi_{\delta}\cdot\rho(t)$.
\end{defn}

Note that the shrunk homomorphism is well-defined because $\Phi_{\delta}$ commutes with $\rho(G_{\alpha})$, and hence with $\rho(\ker(\eta))$. If we are not careful with the choice of the fundamental domain $\mc{E}\sq\mc{M}_*$, the shrunk homomorphism might be ``longer'' than $\rho$. We thus proceed to describe what is required of $\mc{E}$.

\smallskip
{\bf Step~5.} \emph{The $\mc{E}$--form of an element $g\in G$.}

\smallskip\noindent
Consider an element $\g\in G$ that is loxodromic in $\mc{G}$ with $\beta:={\rm Ax}(\g;\mc{G})\cap\mc{M}$ an arc. Let $x_{\g}$ be the initial endpoint of $\beta$, orienting ${\rm Ax}(\g;\mc{G})$ in the direction of translation of $\g$. Let $\beta_1=\beta,\beta_2,\dots,\beta_k$ be the (maximal) oriented arcs that $[x_{\g},\g x_{\g}]$ shares with the $G$--translates of the subtree $\mc{M}$. By Step~1, the $\beta_i$ have pairwise disjoint interiors, and there are only finitely many of them; moreover, for each $g\in G$, the intersection ${\rm Ax}(\g;\mc{G})\cap g\mc{M}$ either contains at most one point or is of the form $\g^i\beta_j$ for some $i\in\Z$ and $1\leq j\leq k$. Choose elements $g_i\in G$ such that $\beta_i\sq g_i\mc{M}$. Note that the endpoints of each $\beta_i$, which we denote by $\beta_i^{\pm}$, lie in the set of exit points $g_i\mc{M}_*$. Thus, there are elements $b_i^{\pm}\in G_{\alpha}$ such that $g_i^{-1}\beta_i^{\pm}\in b_i^{\pm}\cdot\mc{E}$. We then set $a_i:=b_i^+(b_i^-)^{-1}\in G_{\alpha}$.

We should briefly consider how the $a_i$ are affected by the choices involved in their definition. The elements $g_i$ are uniquely defined only up to right multiplication by elements of $G_{\alpha}$. If we replace $g_i$ with $g_ia$ for some $a\in G_{\alpha}$, then the elements $b_i^{\pm}$ get replaced by $a^{-1}b_i^{\pm}$, and the element $a_i$ gets replaced by $a^{-1}a_ia$. However, since $K_{\alpha}$ is normal in $G_{\alpha}$ and $G_{\alpha}/K_{\alpha}$ is abelian, there exists an element $u\in K_{\alpha}$ such that $a^{-1}a_ia=a_iu$. This shows that the elements $a_i$ are uniquely defined up to right multiplication by $K_{\alpha}$.

Now, it should be clear that we can write $\g$ with respect to the splitting $G_{\alpha}\ast_{K_{\alpha}} H_{\mc{E}}$ as
\[ \g=u(a_1h_1)(a_2h_2)\dots (a_kh_k), \]
where the elements $a_i\in G_{\alpha}$ are the ones defined above, $u\in K_{\alpha}$ is some other element, and finally $h_i\in H_{\mc{E}}\setminus K_{\alpha}$ is inductively defined as any element of $H_{\mc{E}}$ mapping the subtree $\mc{M}\sq\mc{G}$ to $(ua_1h_1\dots a_{i-1}h_{i-1}a_i)^{-1}g_i\mc{M}$. We refer to the above writing of $\g$ as the \emph{$\mc{E}$--form of $\g$}, recalling that all elements $a_i,h_i$ are defined only up to right multiplication by $K_{\alpha}$. Note that the $\mc{E}$--form of $\g$ can differ from the normal form of $\g$ with respect to the amalgam $G_{\alpha}\ast_{K_{\alpha}} H_{\mc{E}}$, though only because the elements $a_i$ are allowed to vanish. 
We refer to the $a_i$ as the \emph{$G_{\alpha}$--components} of the $\mc{E}$--form.

Note that we have defined the $\mc{E}$--form only for those elements $\g\in G$ that are loxodromic in $\mc{G}$ with ${\rm Ax}(\g;\mc{G})\cap\mc{M}$ an arc. However, every element of $G$ is either conjugate to such an element, or it is conjugate into $G_{\alpha}\cup H_{\mc{E}}$.

We now explain that $\mc{E}$--forms need to have additional properties for $\mc{E}$ to be useful in shortening the action. First, however, it is convenient to introduce the following shorthand.

\begin{defn}
    A geometric approximation $\mc{G}\ra\mscr{T}$ is \emph{well-adapted} to a finite subset $F\sq G$ if we have $\ell(\g;\mc{G})=\ell(\g;\mscr{T})$ for all elements $\g\in F$ (in general, only the inequality $\geq$ holds).
\end{defn}

{\bf Step~6.} \emph{Good fundamental domains $\mc{E}\sq\mc{M}_*$.}

\smallskip\noindent
Let $\g\in G$ again be loxodromic in $\mc{G}$ with ${\rm Ax}(\g;\mc{G})\cap\mc{M}$ an arc, and retain the notation from Step~5. Suppose in addition that $\mc{G}$ is well-adapted to $\g$, which implies that the morphism $f\colon\mc{G}\ra\mscr{T}$ is isometric on the entire axis ${\rm Ax}(\g;\mc{G})$. Thus, we have arcs $f(\beta_i)\sq {\rm Ax}(g;\mscr{T})\cap g_i\alpha$. 
(It is important to remember that the picture in $\mscr{T}$ is different from the one in $\mc{G}$ in general: $f(\beta_i)$ can be a \emph{proper} sub-arc of the intersection ${\rm Ax}(g;\mscr{T})\cap g_i\alpha$, and the arc $f([x_{\g},\g x_{\g}])$ can intersect infinitely many other $G$--translates of the line $\alpha$.)

For each index $i$ such that $a_i\not\in K_{\alpha}$, we can compare two orientations on the line $\alpha$. The first orientation is obtained by considering the arc $f(\beta_i)\sq g_i\alpha\cap{\rm Ax}(\g;\mscr{T})$, which we orient in the direction of translation of $\g$, then extending this orientation to the whole $g_i\alpha$, and transferring it back to $\alpha$ using the action of $g_i^{-1}$ (this is well-defined because the action $G_{\alpha}\acts\alpha$ is orientation-preserving, since $G_{\alpha}/K_{\alpha}$ is free abelian by assumption).
The second orientation is simply given by the direction of translation of the element $a_i\in G_{\alpha}$ along $\alpha$. If these two orientations coincide for each $1\leq i\leq k$ for which $a_i\not\in K_{\alpha}$, then we say that $\mc{E}$ is \emph{good} for the element $\g$. We say that $\mc{E}$ is \emph{very good} for $\g$ if, in addition, there exists at least one $G_{\alpha}$--component $a_i\not\in K_{\alpha}$.

For elements $\g\in G$ such that ${\rm Ax}(\g;\mc{G})\cap\mc{M}$ is not an arc, we define goodness as follows. First, if $\mc{E}$ is good (resp.\ very good) for some $G$--conjugate of $\g$, then we say that $\mc{E}$ is also good (resp.\ very good) for $\g$. For elements $\g\in G_{\alpha}$, we declare all choices of $\mc{E}$ to be very good. Finally, for elements $\g\in H_{\mc{E}}$, we instead declare all choices of $\mc{E}$ to be good, and none very good.

\begin{rmk}
    The following is the motivation for the definition of goodness. Pick an index $i_0$ such that $a_{i_0}\not\in K_{\alpha}$, and define $a_{i_0}'\in G_{\alpha}$ to be an element translating along $\alpha$ in the same direction as $a_{i_0}$ by an amount that is infinitesimally smaller than $\ell(a_{i_0};\mscr{T})$. If $\mc{E}$ is very good for $\gamma$, then replacing $a_{i_0}$ with $a_{i_0}'$ in the $\mc{E}$--form of $\g$ causes $\ell(\g;\mscr{T})$ to decrease. Instead, if $\mc{E}$ is not even good for $\gamma$ and $i_0$ is an index witnessing this, then replacing $a_{i_0}$ with $a_{i_0}'$ will cause $\ell(\g;\mscr{T})$ to increase.
\end{rmk}

Our goal is now to construct good fundamental domains $\mc{E}\sq\mc{M}_*$ for the $G_{\alpha}$--action.

\begin{lem}\label{lem:finding_good}
    Let the geometric approximation $\mc{G}$ be well-adapted to a finite subset $F\sq G$. Suppose that $F$ contains a loxodromic element $\g_0$ such that ${\rm Ax}(\g_0;\mc{G})\cap\mc{M}$ is an arc. Then there exists a fundamental domain $\mc{E}\sq\mc{M}_*$ that is good for all elements of $F$ and very good for at least one.
\end{lem}
\begin{proof}
    We first need to rephrase the problem purely in terms of a group of translations of the real line. In this direction, let $\tilde{\mf{B}}$ be the collection of oriented arcs of the form $g^{-1}{\rm Ax}(\g;\mc{G})\cap\mc{M}$, where $\g$ varies among loxodromics in $F$ while $g$ varies in $G$, and where we consider the orientations induced by the direction of translation of $g^{-1}\g g$ along $g^{-1}{\rm Ax}(\g;\mc{G})$. As discussed in Step~5, $\tilde{\mf{B}}$ is a finite collection of oriented arcs with endpoints in $\mc{M}_*$, and the morphism $f\colon\mc{G}\ra\mscr{T}$ is isometric on each of them; the existence of the element $\g_0$ also implies that $\tilde{\mf{B}}\neq\emptyset$. Let $\mf{B}$ be the finite collection of oriented arcs within $\alpha$ that are the image under $f$ of the arcs in $\tilde{\mf{B}}$. For each $\beta\in\mf{B}$, we denote by $\beta^-$ its initial endpoint and by $\beta^+$ its terminal one. Note that $\beta^{\pm}\in f(\mc{M}_*)\sq\alpha$. 

Consider now the free abelian group $A:=G_{\alpha}/K_{\alpha}$ and its (free) action by translations on $\alpha$. The action $A\acts f(\mc{M}_*)$ is cofinite, since so is the action $G_{\alpha}\acts\mc{M}_*$. Let $\overline{\mc{E}}\sq f(\mc{M}_*)$ be a finite set of representatives for the orbits of the $A$--action. The choice of $\overline{\mc{E}}$ determines, for each arc $\beta\in\mf{B}$, a unique element $a_{\beta}\in A$ such that the points $a_{\beta}\cdot\beta^-$ and $\beta^+$ lie in the same $A$--translate of $\overline{\mc{E}}$.

Our goal is thus to choose $\overline{\mc{E}}$ so that, for all $\beta\in\mf{B}$ for which $a_{\beta}\neq 0$, the element $a_{\beta}$ translates along $\alpha$ in the direction of the orientation of $\beta=[\beta^-,\beta^+]$. In addition, we wish to ensure that $a_{\beta_0}\neq 0$ for at least one arc $\beta_0\in\mf{B}$. Once we have such a set $\overline{\mc{E}}$, any choice of the fundamental domain $\mc{E}$ within the preimage $f^{-1}(\overline{\mc{E}})$ will satisfy the statement of the lemma.

We thus proceed to construct $\overline{\mc{E}}$, distinguishing two cases based on the rank of $A$. If $\rk(A)\geq 2$, then the action $A\acts\alpha$ is indiscrete, and there are fundamental domains of arbitrarily small diameter for $A\acts f(\mc{M}_*)$. It then suffices to choose $\overline{\mc{E}}$ with diameter less than the length of all arcs in $\mc{B}$.

Suppose instead that $A\cong\Z$ and let $a\in A$ be one of its two generators. Let us choose $\overline{\mc{E}}$ so that it has the smallest possible diameter. If $\beta\in\mf{B}$ were an arc such that $a_{\beta}\beta^-<\beta^-<\beta^+$ (identifying $\alpha$ with $\R$ according to the direction of translation of $a$), we would have $\{a_{\beta}\beta^-,\beta^+\}\sq\overline{\mc{E}}$ up to replacing $\overline{\mc{E}}$ with an $A$--translate. We could then decrease the diameter of $\overline{\mc{E}}$ by replacing it with 
\[ a_{\beta}^{-1}\big(\overline{\mc{E}}\cap (-\infty,a_{\beta}\beta^-]\big)\cup\big(\overline{\mc{E}}\cap (a_{\beta}\beta^-,+\infty)\big) ,\] 
which is a contradiction. Thus, $\overline{\mc{E}}$ satisfies our requirements, except for the possibility that $a_{\beta}=0$ for all $\beta\in\mf{B}$. If the latter occurs, we can assume that all arcs in $\mf{B}$ have both endpoints in $\overline{\mc{E}}$ (as we can always replace an element of $\mf{B}$ with one of its $A$--translates). Denoting by $\mu$ the maximum of $\overline{\mc{E}}$ (with respect to the identification $\alpha\cong\R$), we can then replace $\overline{\mc{E}}$ with $\overline{\mc{E}}\setminus\{\mu\}\cup\{a^{-1}\mu\}$, and it is straightforward to check that this new set satisfies all requirements.
\end{proof}

We can now focus on the choice of the epimorphism $\eta\colon G_{\alpha}\twoheadrightarrow\Z$ vanishing on $K_{\alpha}$. We say that the pair $(\mc{E},\eta)$ is \emph{very good} for an element $\g\in G$ if $\mc{E}$ is good for $\g$ and there exists at least one $G_{\alpha}$--component $a_i\in G_{\alpha}$ in the $\mc{E}$--form of $\g$ such that $\eta(a_i)\neq 0$ (in particular, if $\g\in G_{\alpha}$, we require that $\eta(\g)\neq 0$). We say that the pair $(\mc{E},\eta)$ is \emph{good} for $\g$ simply if $\mc{E}$ is good, regardless of $\eta$.

If $\mc{E}$ is very good for an element $\g\in G$, we can always find an epimorphism $\eta\colon G_{\alpha}\twoheadrightarrow\Z$ vanishing on $K_{\alpha}$ such that the pair $(\mc{E},\eta)$ is very good for $\g$. Indeed, since $\mc{E}$ is very good for $\g$, there exists at least one $G_{\alpha}$--component $a_i$ in the $\mc{E}$--form of $\g$ such that $a_i\not\in K_{\alpha}$. The fact that the quotient $G_{\alpha}/K_{\alpha}$ is free abelian then implies the existence of the required epimorphism $\eta$.

Summing up, we have essentially shown the following (using the notation from the above steps).

\begin{prop}\label{prop:shrinkable_line}
    Let $\alpha\sq\mscr{T}$ be a shrinkable line. Let $F\sq G$ be a finite subset containing at least one loxodromic element $\g_0$ such that ${\rm Ax}(\g_0;\mscr{T})\cap\alpha$ contains an arc.
    \begin{enumerate}
    \setlength\itemsep{.2em}
        \item There exists a geometric approximation $\mc{G}\ra\mscr{T}$ such that $\mc{G}$ is well-adapted to $F$ and ${\rm Ax}(\g_0;\mscr{T})\cap\mc{M}$ contains an arc. Moreover, there exist a fundamental domain $\mc{E}\sq\mc{M}_*$ and an epimorphism $\eta\colon G_{\alpha}\twoheadrightarrow\Z$ such that the pair $(\mc{E},\eta)$ is good for all elements of $F$ and very good for at least one of them (possibly different from $\g_0$).
        \item Let $(\mc{E},\eta)$ be any pair that is good for all elements of $F$ and very good for some element $\gamma_1\in F$. Then there exists $\delta_0>0$ such that, for all $0<\delta<\delta_0$, we have:
        \begin{align*}
            \ell(\rho_{\mc{E},\eta,\delta}(\g);X)&\leq\ell(\rho(\g);X),\quad\forall\g\in F, & \ell(\rho_{\mc{E},\eta,\delta}(\g_1);X)&<\ell(\rho(\g_1);X).
        \end{align*}
    \end{enumerate}
\end{prop}
\begin{proof}
    Any sufficiently close geometric approximation $f\colon\mc{G}\ra\mscr{T}$ is well-adapted to $F$, satisfies the assumptions in the paragraph before Step~1, and has the property that ${\rm Ax}(\g_0;\mc{G})\cap\mc{M}$ contains an arc. We can then invoke \Cref{lem:finding_good} to find a good fundamental domain $\mc{E}\sq\mc{M}_*$ that is good for all elements of $F$ and very good for at least one of them. Finally, arguing as in the paragraphs after \Cref{lem:finding_good}, we find a suitable epimorphism $\eta$. This proves part~(1).

    We now discuss part~(2). For an element $\g\in G_{\alpha}$, we have 
    \[ \rho_{\mc{E},\eta,\delta}(\g)=\Phi_{\delta}^{\eta(\g)}\rho(\g) .\] 
    If $\eta(\g)=0$, then the translation length $\ell(\rho_{\mc{E},\eta,\delta}(\g);X)$ is equal to $\ell(\rho(\g);X)$. If instead $\eta(\g)\neq 0$, then we have $\ell(\g;\mscr{T})>0$ (since $\eta$ vanishes on $K_{\alpha}$) and we have $\ell(\rho_{\mc{E},\eta,\delta}(\g);X)<\ell(\rho(\g);X)$ provided that $\delta|\eta(\g)|<\ell(\g;\mscr{T})$ (for instance, using \Cref{lem:basic_shortening}). This shows that, choosing $\delta_0$ small enough, it suffices to prove part~(2) under the assumption that no element of $F$ has a conjugate inside $G_{\alpha}$.
    
    If an element $\g\in G$ is elliptic in the splitting $\mc{S}$ (from Step~2), then it is elliptic in the amalgam $\mc{A}(\mc{E})$, and it is consequently either conjugate into $G_{\alpha}$ or elliptic in the HNN splitting $\mc{H}(\mc{E},\eta)$. In the latter case, $\rho_{\mc{E},\eta,\delta}(\g)$ and $\rho(\g)$ are conjugate within $\isom(X)$, so they have the same translation length in $X$. As a consequence, up to replacing each element of $F$ with a $G$--conjugate, it is not restrictive to assume that all $\g\in F$ are loxodromic in $\mc{G}$ with ${\rm Ax}(\g;\mc{G})\cap\mc{M}$ being an arc.

    We can thus consider, for each $\g\in F$, its $\mc{E}$--form
    \[ \g=u(a_1h_1)(a_2h_2)\dots (a_kh_k). \]
    Define the integer $n(\g):=\max_i|\eta(a_i)|$ and let $L(\g)>0$ be length of the shortest arc among the arcs $\beta_1,\dots,\beta_k$ associated with $\g$ in Step~5 (namely the arcs that ${\rm Ax}(\g;\mc{G})$ shares with the $G$--translates of the $G_{\alpha}$--minimal subtree $\mc{M}\sq\mc{G}$). If $n(\g)=0$, we again have that $\g$ is elliptic in $\mc{H}(\mc{E},\eta)$, and so $\rho_{\mc{E},\eta,\delta}(\g)$ is conjugate to $\rho(\g)$ within $\isom(X)$. If instead $n(\g)>0$, we have the strict inequality $\ell(\rho_{\mc{E},\eta,\delta}(\g);X)<\ell(\rho(\g);X)$ provided that $n(\g)\delta<L(\g)$: this can be deduced from \Cref{lem:basic_shortening} and the discussion in Steps~5 and~6, choosing wallpairs arising from small sub-arcs of the arcs $f(\beta_i)$ (and recalling that, by \Cref{defn:shrinkable_lines}, the shortening isometries $\Phi_{\delta}$ simultaneously shorten \emph{all} wallpairs along $\alpha$). In conclusion, it suffices to set $\delta_0:=\min_{\g\in F} L(\g)/n(\g)$.
\end{proof}

\subsection{Shortening degenerations of special groups}\label{sub:shortening_special}

We now apply the general discussion in \Cref{sub:shortening_median} to degenerations of special groups, searching for shrinkable arcs and shrinkable lines. We will obtain two tools to shorten sequences of automorphisms: Propositions~\ref{prop:shortening_inert} and~\ref{prop:shortening_motile} respectively. This will lead to the shortening theorem (\Cref{thm:shortening}).

Consider a special group $G$ and a sequence $\varphi_n\in\Aut(G)$ projecting to an infinite sequence $\phi_n\in\Out(G)$. Realise $G$ as a convex-cocompact subgroup of a RAAG $\mc{A}_{\G}$, and choose a non-principal ultrafilter $\om$. As in \Cref{sec:degenerations}, we obtain a degeneration $G\acts\X_{\om}$ and $\R$--trees $G\acts T^v_{\om}$ for $v\in\G$, with projections $\pi^v_{\om}\colon\X_{\om}\ra T^v_{\om}$. Since we will also work with other $G$--actions on $\X_{\om}$, we denote by $\rho\colon G\ra\isom(\X_{\om})$ the homomorphism given by $G\acts\X_{\om}$.

Let $\mscr{C}\sq\X_{\om}$ be a minimal $G$--invariant convex subset, which exists by \cite[Theorem~C]{Fio10b}. Note that $\mscr{C}$ is not closed in $\X_{\om}$ in general, and thus possibly incomplete. Nevertheless, $\mscr{C}$ is a finite-rank geodesic median space with compact intervals. For each $v\in\G$, the projection $\pi^v_{\om}(\mscr{C})$ is either the $G$--minimal subtree of $T^v_{\om}$ (if $G$ is not elliptic in $T^v_{\om}$), or else a single point of $T^v_{\om}$ fixed by $G$.

Picking a vertex $v\in\G$ such that $G$ is non-elliptic in $T^v_{\om}$, and setting $X:=\mscr{C}$, $\mscr{T}:=\Min(G;T^v_{\om})$ and $\Pi:=\pi^v_{\om}|_{\mscr{C}}$, we place ourselves in the situation described at the start of \Cref{subsub:shrinkable_arcs}.  

The next result involves a subgroup $H\leq G$. We will therefore speak of ``$G$--stable'' and ``$G$--inert'' arcs to emphasise that these notions are meant with respect to the $G$--action. Recall that $\mf{E}_G\sq\X_{\G}$ denotes a $G$--essential convex subcomplex; we have $\ell(g;\mf{E}_G)=\ell(g;\X_{\G})$ for all $g\in G$.

\begin{prop}\label{prop:shortening_inert}
    Consider a hierarchy member $H\in\mf{H}(G)$, 
    a vertex $v\in\G$ such that $H$ is not elliptic in $T^v_{\om}$, and a $G$--stable arc $\beta\sq\Min(H;T^v_{\om})$. Suppose that {\bf both} the following hold:
    \begin{enumerate}
        \item[(a)] $\beta$ is $G$--inert, and the normaliser $N_H(G_{\beta})$ is elliptic in $T^v_{\om}$;
        \item[(b)] $\beta$ is not contained in any indiscrete line for $G\acts T^v_{\om}$.
    \end{enumerate}
    Then, for every finite subset $F\sq G$, there exist a $1$--edge splitting $G\acts T$ and a sequence of Dehn twists $\tau_n\in\Aut(G)$ with respect to $T$ such that all the following hold.
    \begin{enumerate}
        \item For each $\g\in F$, we have $\ell(\varphi_n\tau_n(\g);\mf{E}_G) \leq \ell(\varphi_n(\g);\mf{E}_G)$ for $\om$--all $n$. The strict inequality holds if the element $\g\in F$ is loxodromic in $T^v_{\om}$ with axis sharing an arc with $\beta$.
        \item Edge groups of $G\acts T$ are conjugate to $G_{\beta}$, which lies in $\mc{Z}(G)$. If a subgroup of $G$ is not elliptic in $T$, then it contains an element that is loxodromic in $T^v_{\om}$ with axis sharing an arc with a $G$--translate of $\beta$. 
        \item If the $\varphi_n$ are coarse-median preserving, then the $\tau_n$ are folds or partial conjugations.
    \end{enumerate}
\end{prop}
\begin{proof}
    Since the arc $\beta$ is inert and not contained in an indiscrete line, \Cref{thm:inert_vs_motile} guarantees that $\varphi_n(G_{\beta})$ is elliptic in $T^v$ for $\om$--all $n$ and that $G_{\beta}\in\mc{Z}(G)$. In particular, both $G_{\beta}$ and the normaliser $N_H(G_{\beta})$ are finitely generated. Let $f\colon\mc{G}\ra\Min(G;T^v_{\om})$ be a geometric approximation such that:
    \begin{itemize}
        \item $\beta$ lifts isometrically to an arc $\tilde\beta\sq\mc{G}$ fixed by $G_{\beta}$ and contained in $\Min(H;\mc{G})$;
        \item the normaliser $N_H(G_{\beta})$ is elliptic in $\mc{G}$; 
        \item if $\g\in F$ is loxodromic in $T^v_{\om}$, then $f$ is isometric on ${\rm Ax}(\g;\mc{G})$;
        \item if a subgroup $K\leq G$ leaves invariant a subtree of $T^v_{\om}$ that does not share any arcs with the $G$--translates of $\beta$, then $K$ leaves invariant a subtree of $\mc{G}$ that does not share arcs with the $G$--translates of $\tilde\beta$.
    \end{itemize}
    Finding $\mc{G}$ satisfying the first three items is standard. As to the fourth one, consider the $\R$--tree $\mc{T}$ obtained from $T^v_{\om}$ by collapsing to a point the closure of each connected component of $T^v_{\om}\setminus\bigcup_{g\in G}g\beta$. The induced action $G\acts\mc{T}$ is minimal and its arc-stabilisers lie in $\mc{Z}(G)$, so \cite[Corollary~B.3]{Fio11} implies that point-stabilisers of $\mc{T}$ are finitely generated and fall into finitely many $G$--conjugacy classes of subgroups $[P_1],\dots,[P_s]$. We can then choose $\mc{G}$ so that each of the $P_i$ leaves invariant a subtree not sharing arcs with any $G$--translate of $\tilde\beta$,
    and the fourth item follows.

    Now that we have $\mc{G}$, suppose for the sake of contradiction that $\tilde\beta$ shares an arc with an indecomposable component $U\sq\mc{G}$. Since $\beta$ is $G$--stable by hypothesis, the kernel of the action $G_U\acts U$ coincides with $G_{\beta}$. \Cref{prop:no_indecomposable_crossing} then implies that the stabiliser $H_U$ projects to an infinite subgroup of the quotient $G_U/G_{\beta}$. However, indecomposability of $U$ implies that $H_U\leq N_H(G_{\beta})$, and the latter subgroup is elliptic in $\mc{G}$. This is a contradiction.

    In conclusion, up to shrinking $\beta$, we can assume that the arc $\tilde\beta\sq\mc{G}$ is edge-like. Define $G\acts T$ as the $1$--edge splitting over $G_{\beta}$ obtained by collapsing to a point the closure of each connected component of $\mc{G}\setminus\bigcup_{g\in G}g\tilde\beta$. Our choice of the geometric approximation $\mc{G}$ guarantees that every subgroup of $G$ that is elliptic in $T^v_{\om}$ is also elliptic in $T$. Additionally, loxodromic elements of $F$ are also elliptic in $T$ if their axis in $T^v_{\om}$ does not share arcs with any $G$--translate of $\beta$. 
    
    Let $e\sq T$ be the edge corresponding to the arc $\tilde\beta$. Let $p\in T$ be a vertex fixed by both $G_{\beta}$ and the normaliser $N_H(G_{\beta})$. Since $G_{\beta}$ is convex-cocompact, it is not properly contained in any of its $G$--conjugates (\Cref{lem:cc_basics}(3)); thus, up to replacing $\tilde\beta$ with a $G$--translate, we can assume that $p\in e$. Let $p'$ be the vertex of $e$ other than $p$. Let $A$ and $B$ be the $G$--stabilisers of $p$ and $p'$, respectively; if $T$ is an HNN extension, we also pick an element $t\in G$ with $tp=p'$. 

    Now, consider the minimal $\rho(G)$--invariant convex subset $\mscr{C}\sq\X_{\om}$ and its equivariant projection $\pi^v_{\om}\colon\mscr{C}\ra\Min(G;T^v_{\om})$. Let $(\mf{w},\mf{w}')$ be a non-degenerate wallpair of $\mscr{C}$ arising from the arc $\beta$, and let $\mf{W},\mf{W}'$ be the corresponding geometric walls. Identifying $\beta$ with $\tilde\beta\sq\mc{G}$ and $e\sq T$, we should name the walls so that the point $\pi^v_{\om}(\mf{W})$ is closer to $p$, while $\pi^v_{\om}(\mf{W}')$ is closer to $p'$. 

    Recall that the action $G\acts T^v_{\om}$ is the $\om$--limit of countably many copies of the action $G\acts T^v$, with the $n$--th copy rescaled by a factor and twisted by the automorphism $\varphi_n$. Moreover, $\pi^v(\mf{E}_G)=\Min(G;T^v)$. Since the arc $\beta$ is $G$--inert and contained in $\Min(H;T^v_{\om})$, there exist hyperplanes $\mf{w}_n,\mf{w}_n'\in\mscr{W}(\mf{E}_G)$ originating from edges of $\Min(\varphi_n(H);T^v)\cap\Fix(\varphi_n(G_{\beta});T^v)$ and converging to the walls $\mf{w},\mf{w}'\in\mscr{W}(\X_{\om})$. Let $\mf{s}_n\sq\mf{w}_n$ and $\mf{s}_n'\sq\mf{w}_n'$ be their shores within $\mf{E}_G$, and let $\pi_n\colon\mf{s}_n'\ra\mf{s}_n$ be the gate-projection (which is an isometry).

    \smallskip
    {\bf Claim.} \emph{There exist a constant $C$ (only depending on $G$, viewed as a subgroup of $\mc{A}_{\G}$) and elements $h_n\in Z_H(G_{\beta})$ such that the following holds for $\om$--all $n$. For every point $x\in\mf{s}_n'\sq\mf{E}_G$, we have 
    \[ d(\varphi_n(h_n)x,\pi_n(x))\leq C .\]
    If the automorphisms $\varphi_n$ are coarse-median preserving, then we additionally have $h_n\in\Perp_H(G_{\beta})$.
    }
    
    \smallskip
    \emph{Proof of claim.} Since $\mf{H}(G)$ is $\Aut(G)$--invariant and finite up to $G$--conjugacy, the conjugacy class of $\varphi_n(H)$ is $\om$--constant. Up to composing each $\varphi_n$ with an inner automorphism of $G$, we can thus assume that $\varphi_n(H)=H$ for $\om$--all $n$. Choose a convex $H$--essential subcomplex $\mf{E}_H\sq\mf{E}_G\sq\X_{\G}$. Let $C$ be a constant such that, for each parabolic stratum $\mc{P}\sq\X_{\G}$ with $\mc{P}\cap\mf{E}_H\neq\emptyset$, the intersection $\mc{P}\cap\mf{E}_H$ is acted upon by its $H$--stabiliser with at most $C$ orbits of vertices (see \Cref{subsub:orthogonals} for terminology). Note that the constant $C$ is bounded only depending on $G$.
    
    Now, the hyperplanes $\mf{w}_n,\mf{w}_n'\in\mscr{W}(\mf{E}_G)\sq\mscr{W}(\X_{\G})$ are $H$--essential for $\om$--all $n$, by construction, and so they cross the subcomplex $\mf{E}_H$. This implies that $\mf{s}_n\cap\mf{E}_H\neq\emptyset$ and $\mf{s}_n'\cap\mf{E}_H\neq\emptyset$. Pick some vertex $y_n\in\mf{s}_n'\cap\mf{E}_H$ and let $\mc{P}_n\sq\X_{\G}$ be the orthogonal subcomplex for $\mf{s}_n'$ through $y_n$ (see again \Cref{subsub:orthogonals}); it follows that $\mc{P}_n$ is also the orthogonal subcomplex for $\mf{s}_n$ through the point $\pi_n(y_n)$, which lies in $\mf{s}_n\cap\mf{E}_H$. By the choice of the constant $C$, there exists an element $g_n$ in the $H$--stabiliser of $\mc{P}_n$ such that $d(g_ny_n,\pi_n(y_n))\leq C$. At the same time, since $g_n\mc{P}_n=\mc{P}_n$, the element $g_n$ leaves invariant each hyperplane in the set $\mscr{W}(\mf{s}_n)=\mscr{W}(\mf{s}_n')$ by \Cref{rmk:orthogonal_preserves_hyperplanes}, and so we have $d(g_nx,\pi_n(x))=d(g_ny_n,\pi_n(y_n))\leq C$ for all points $x\in\mf{s}_n'$.

    Finally, since $\varphi_n(G_{\beta})$ leaves invariant the hyperplanes $\mf{w}_n,\mf{w}_n'$, it also leaves invariant the shores $\mf{s}_n,\mf{s}_n'$ and so it admits an essential core within either of them. This implies that $g_n\in\Perp_H(\varphi_n(G_{\beta})$ and, in particular, that $g_n\in Z_H(\varphi_n(G_{\beta}))$. Setting $h_n:=\varphi_n^{-1}(g_n)$, we obtain $h_n\in Z_H(G_{\beta})$ and, if $\varphi_n$ is coarse-median preserving, also $h_n\in\Perp_H(G_{\beta})$, as claimed.
    \hfill$\blacksquare$
    
    \smallskip
    We can now use the elements $h_n$ provided by the claim to define Dehn twists $\tau_n\in\Aut(G)$ with respect to the splitting $T$. We declare $\tau_n$ to be the identity on the subgroup $A\leq G$, and to either equal $b\mapsto h_nbh_n^{-1}$ on $B$ (in the amalgam case), or map $t\mapsto h_nt$ (in the HNN case).
    According to \Cref{defn:DT_types}, the $\tau_n$ are partial conjugations when $T$ is an amalgam, and they are folds if $T$ is an HNN splitting and the automorphisms $\varphi_n$ are coarse-median preserving, yielding Item~(3).

    Representing points of the degeneration $\X_{\om}$ by equivalence classes of sequences $(q_n)\in\X_{\G}^{\N}$, let $\Phi\colon\X_{\om}\ra\X_{\om}$ be the isometry defined by $(q_n)\mapsto (\varphi_n(h_n)q_n)$. 
    It is immediate to check that $\Phi$ is well-defined and a shortening isometry for the wallpair $(\mf{w},\mf{w}')$, in the sense of \Cref{defn:shortening_isom}. In particular, $\Phi$ witnesses the fact that the pair $(\mc{G},\beta)$ is a shrinkable arc (\Cref{defn:shrinkable_arcs}) and this yields a shrunk homomorphism $\rho'\colon G\ra\isom(\X_{\om})$ (\Cref{defn:shrunk_homo_arcs}). Finally, \Cref{prop:shrinkable_arc} guarantees that $\ell(\rho'(\g);\X_{\om})\leq\ell(\rho(\g);\X_{\om})$ for all $\g\in F$, with the strict inequality holding if $\g$ is loxodromic in $T^v_{\om}$ with axis sharing an arc with $\beta$.
    
    To conclude, observe that the degeneration associated to the sequence of automorphisms $\varphi_n\tau_n$ (using the same scaling factors as for the sequence $\varphi_n$) is precisely the $G$--action on $\X_{\om}$ induced by the shrunk homomorphism $\rho'$. The previous inequalities then amount precisely to Item~(1) in the statement of the proposition (for convergence of length functions as we take the $\om$--limit, see e.g.\ \cite[Lemma~7.9(2)]{Fio10a}). This concludes the proof.
\end{proof}

The following is needed to construct shortening isometries of $\X_{\om}$ in the next case we consider. The important point is that the right-hand side of the inequality in the lemma is independent of the integer $n$. Note that this badly fails for elements $g$ that are not convex-cocompact.

\begin{lem}\label{lem:cc_element_shortening}
    Let $Q$ be the number of vertex orbits for $G\acts\mf{E}_G$, and set $D:=\dim\mf{E}_G$. Consider a convex-compact element $g\in G$ and a hyperplane $\mf{w}\in\mscr{W}(\mf{E}_G)$ skewered by $g$. For $n\in\N$, let $\mf{s}_n\sq\mf{w}$ and $\mf{s}_n'\sq g^n\mf{w}$ be the shores for the pair of hyperplanes $\mf{w},g^n\mf{w}$, and denote by $\pi_n\colon\mf{s}_n'\ra\mf{s}_n$ the nearest-point projection. Then, for all $n\geq 4D$ and all vertices $x\in\mf{s}_n'$, we have
    \[ d(g^{-n}x,\pi_n(x))\leq 2Q+4D\cdot \ell(g;\mf{E}_G) .\]
\end{lem}
\begin{proof}
    We first reduce to the case when $Z_G(g)\cong\Z$, which will require replacing the group $G$. The centraliser $Z_{\mc{A}_{\G}}(g)$ splits as $C\x P$, where $C$ is the maximal cyclic subgroup of $\mc{A}_{\G}$ containing $g$, and where $P\leq\mc{A}_{\G}$ is a parabolic subgroup. As in the statement of the lemma, we consider some $n\in\N$ and a vertex $x\in\mf{s}_n'\sq\mf{E}_G$. Let $\Pi\sq\X_{\G}$ be the parabolic stratum containing $x$ such that the $\mc{A}_{\G}$--stabiliser of $\Pi$ is the orthogonal parabolic subgroup $P^{\perp}$.
    Since distinct $\mc{A}_{\G}$--translates of $\Pi$ are disjoint, the action $G\cap P^{\perp}\acts\Pi\cap\mf{E}_G$ has at most $Q$ orbits of vertices. Moreover, the element $g$ lies in $G\cap P^{\perp}$ and has cyclic centraliser within this subgroup. Finally, the hyperplanes $\mf{w}$ and $g^n\mf{w}$ cross the convex subcomplex $\Pi\cap\mf{E}_G$, and $x$ lies in the shore of $g^n\mf{w}$ with respect to $\mf{w}$ within $\Pi\cap\mf{E}_G$. We can thus replace the group $G$ with $G\cap P^{\perp}$, and the complex $\mf{E}_G$ with a $G\cap P^{\perp}$--essential subcomplex of $\Pi\cap\mf{E}_G$; it suffices to prove the lemma in this context.

    We thus assume that $Z_G(g)\cong\Z$. The fact that $g$ is convex-cocompact implies that, if a hyperplane $\mf{v}\in\mscr{W}(\X_{\G})$ is skewered by $g$ and if a hyperplane $\mf{u}\in\mscr{W}(\X_{\G})$ is transverse to both $\mf{v}$ and $g^{4D}\mf{v}$, then $g\mf{u}=\mf{u}$ (see e.g.\ \cite[Lemma~3.10]{Fio10a}).
    Since $Z_G(g)\cong\Z$, it follows that, for every hyperplane $\mf{v}\in\mscr{W}(\mf{E}_G)$ skewered by $g$ and for every integer $k\geq 4D$, each of the shores within $\mf{E}_G$ for the pair of hyperplanes $\mf{v},g^k\mf{v}$ has trivial $G$--stabiliser (indeed, any element preserving a shore would only skewer $\langle g\rangle$--invariant hyperplanes and thus commute with $g$). Since each of these shores is the intersection with $\mf{E}_G$ of a parabolic stratum of $\X_{\G}$, it follows that these shores each contain at most $Q$ vertices and so have diameter $\leq Q$.

    Now, consider $n\geq 4D$ and the shores $\mf{s}_n\sq\mf{w}$ and $\mf{s}_n'\sq g^n\mf{w}$. Note that $g^{-n}\mf{s}_n'$ is the shore within $\mf{w}$ for the pair $\mf{w},g^{-n}\mf{w}$. Let $\mc{U}$ be the set of hyperplanes separating $\mf{s}_n$ from $g^{-n}\mf{s}_n'$. Each $\mf{u}\in\mc{U}$ is transverse to $\mf{w}$; moreover, $\mf{u}$ separates $g^n\mf{w}$ from $g^{-n}\mf{w}$ and so $\mf{u}$ is skewered by $g$. For $k\geq 4D$, the hyperplane $g^k\mf{u}$ cannot be transverse to $\mf{w}$; otherwise, as observed in the previous paragraph, we would have $g\mf{w}=\mf{w}$, contradicting the fact that $\mf{w}$ is skewered by $g$. This shows that no two hyperplanes in $\mc{U}$ are in the same $\langle g^{4D}\rangle$--orbit, and hence $|\mc{U}|\leq 4D\cdot\ell(g;\mf{E}_G)$.

    In conclusion, the sets $\mf{s}_n$ and $g^{-n}\mf{s}_n'$ are at distance $\leq 4D\cdot\ell(g;\mf{E}_G)$, and they each have diameter $\leq Q$. This implies that $\mf{s}_n\cup g^{-n}\mf{s}_n'$ has diameter $\leq 2Q+4D\cdot\ell(g;\mf{E}_G)$, which implies the thesis.
\end{proof}

The next result is based on the observation that IOM-lines are shrinkable lines when the centre of their stabiliser is elliptic. We discuss in \Cref{rmk:shortening_motile+} what happens when the centre is non-elliptic.

\begin{prop}\label{prop:shortening_motile}
    Consider an IOM-line $\alpha\sq T^v_{\om}$ sharing an arc with $\Min(G;T^v_{\om})$, 
    and suppose that the centre of $G_{\alpha}$ is elliptic in $T^v_{\om}$. 
    Then, for every finite set $F\sq G$ and any finite collection $\mscr{F}$ of finitely generated subgroups of $G$ that are elliptic in $T^v_{\om}$, there are a $1$--edge splitting $G\acts T$ and a sequence of Dehn twists $\tau_n\in\Aut(G)$ with respect to $T$ such that the following hold.
    \begin{enumerate}
        \item For each $\g\in F$, we have $\ell(\varphi_n\tau_n(\g);\mf{E}_G) \leq \ell(\varphi_n(\g);\mf{E}_G)$ for $\om$--all $n$.
        \item The strict inequality $\ell(\varphi_n\tau_n(\g);\mf{E}_G) < \ell(\varphi_n(\g);\mf{E}_G)$ holds for $\om$--all $n$ for at least one element $\g\in F$, if at least one element of $F$ is loxodromic in $T^v_{\om}$ with axis sharing an arc with $\alpha$. 
        \item If the action $G_{\alpha}\acts\alpha$ is nontrivial, then the splitting $T$ is $(G_{\alpha},\mscr{F})$--ascetic and the twists $\tau_n$ are $(A,\mscr{F})$--ascetic, where $A\in\Sal(G)$ is the centre of $G_{\alpha}$. If instead $G_{\alpha}$ is el\-lip\-tic in $T^v_{\om}$, then $T$ is $(G_{\alpha},\mscr{F})$--isolating with edge groups conjugate to $G_{\alpha}\in\mc{Z}(G)$. 
        \item If the $\varphi_n$ are coarse-median preserving, then so are the $\tau_n$. Moreover, $G_{\alpha}$ is elliptic in $T^v_{\om}$, and the $\tau_n$ are skews or partial conjugations. 
    \end{enumerate}
\end{prop}
\begin{proof}
Our goal is to show that the line $\alpha$ is shrinkable (\Cref{defn:shrinkable_lines}), so that we can exploit the construction described in \Cref{subsub:shrinkable_lines}. (This will only be possible if $G_{\alpha}$ is not elliptic.)

Recall that the action $G\acts T^v_{\om}$ is the $\om$--limit of countably many copies of $G\acts T^v$, with the $n$--th copy twisted by the automorphism $\varphi_n$ and rescaled by a factor $\lambda_n>0$. As usual, let $K_{\alpha}$ be the kernel of the action $G_{\alpha}\acts\alpha$. \Cref{rmk:short_central_elements} yields elements $g_n$ with $Z_G(g_n)=G_{\alpha}$ and such that $\varphi_n(g_n)$ is convex-cocompact in $G$ and loxodromic in $T^v$ with $\lim_{\om}\frac{1}{\lambda_n}\ell(\varphi_n(g_n);T^v)=0$. Since the centre of $G_{\alpha}$ is elliptic by hypothesis, we have $g_n\in K_{\alpha}$. For each real number $\delta>0$, we can find a (diverging) sequence of natural numbers $a_n(\delta)$ such that 
\[ \lim_{\om} \tfrac{1}{\lambda_n}\ell(\varphi_n(g_n^{a_n(\delta)});T^v)=\lim_{\om} a_n(\delta)\cdot \tfrac{1}{\lambda_n} \ell(\varphi_n(g_n);T^v)=\delta . \]
Representing points of $\X_{\om}$ by sequences $(q_n)\in\X_{\G}^{\N}$, we can then consider the isometry $\Phi_{\delta}\colon\X_{\om}\ra\X_{\om}$ given by $(q_n)\mapsto (\varphi_n(g_n^{a_n(\delta)})q_n)$. Note that $\Phi_{\delta}$ commutes with $\rho(G_{\alpha})$ within $\isom(\X_{\om})$, since $g_n$ lies in the centre of $G_{\alpha}$. We now wish to show that $\Phi_{\delta}$ is a shortening isometry (and that it is well-defined). This requires the following observation.

\smallskip
{\bf Claim.} \emph{We have $\lim_{\om}\frac{1}{\lambda_n}\ell(\varphi_n(g_n);\mf{E}_G)=0$.}

\smallskip\noindent
\emph{Proof of claim.} 
We we will deduce this from the fact, seen above, that $\lim_{\om}\frac{1}{\lambda_n}\ell(\varphi_n(g_n);T^v)=0$. Suppose for the sake of contradiction that $\lim_{\om} \ell(\varphi_n(g_n);\mf{E}_G)/\ell(\varphi_n(g_n);T^v)=+\infty$.

The key point is that, by assumption, the line $\alpha$ shares an arc $\beta$ with $\Min(G;T^v_{\om})$.
Choose arcs $\beta_n\sq{\rm Ax}(\varphi_n(g_n);T^v)$ converging to $\beta$ and note that, denoting arc-lengths by $|\cdot|$, we have $\lim_{\om}|\beta_n|/\lambda_n=|\beta|>0$. We can choose hyperplanes $\mf{w}_n\in\mscr{W}(\mf{E}_G)$ such that both $\mf{w}_n$ and $\varphi_n(g_n^{c_n})\mf{w}_n$ are dual to edges of the arc $\beta_n\sq T^v$, for some integers $c_n$ such that $c_n\ell(\varphi_n(g_n);T^v)\geq |\beta_n|/2$ for $\om$--all $n$. By \cite[Lemma~3.10]{Fio10a}, the hyperplanes $\mf{w}_n$ and $\varphi_n(g_n^{4D})\mf{w}_n$ are separated by at least $\ell(\varphi_n(g_n);\mf{E}_G)$ hyperplanes of $\mf{E}_G$. In particular, we have
\[ \lim_{\om} \frac{1}{\lambda_n} d_{\mf{E}_G}(\mf{w}_n,\varphi_n(g_n^{c_n})\mf{w}_n) \geq \lim_{\om} \frac{1}{\lambda_n} \left\lfloor\frac{c_n}{4D}\right\rfloor \ell(\varphi_n(g_n);\mf{E}_G) \geq \frac{1}{8D} \lim_{\om} \frac{|\beta_n|}{\lambda_n }\frac{\ell(\varphi_n(g_n);\mf{E}_G)}{\ell(\varphi_n(g_n);T^v)}=+\infty . \]
This implies that the hyperplanes $\mf{w}_n$ and $\varphi_n(g_n^{c_n})\mf{w}_n$ cannot converge to two walls of $\X_{\om}$. But, at the same time, they also must converge to two such walls because the arcs $\beta_n\sq T^v$ converge to the arc $\beta\sq T^v_{\om}$ and we have $\beta\sq\Min(G;T^v_{\om})=\pi^v_{\om}(\mscr{C})\sq\pi^v_{\om}(\X_{\om})$. This is a contradiction.
\hfill$\blacksquare$

\smallskip
We now claim that, for all $\delta>0$, the isometry $\Phi_{\delta}$ is well-defined and that it is a shortening isometry for all wallpairs $(\mf{w},\mf{w}')$ of $\X_{\om}$ arising from pairs of points of $\alpha$ at distance $\delta$ (up to suitably replacing $a_n(\delta)$ with $-a_n(\delta)$). In order to see this, choose wallpairs $(\mf{w}_n,\varphi_n(g_n^{-a_n(\delta)})\mf{w}_n)$ of $\mf{E}_G$ that are skewered by $\varphi_n(g_n)$ and converge to $(\mf{w},\mf{w}')$. Consider the shores $\mf{s}_n\sq\mf{w}_n$, $\mf{s}_n'\sq\varphi_n(g_n^{-a_n(\delta)})\mf{w}_n$, and also let $\mf{s},\mf{s}'\sq\X_{\om}$ be the shores for the geometric walls of $\X_{\om}$ corresponding to $\mf{w},\mf{w}'$. Let $\pi_n\colon\mf{s}_n'\ra\mf{s}_n$ and $\pi\colon\mf{s}'\ra\mf{s}$ be the nearest-point projections. Representing each point $x\in\mf{s}'$ by a sequence $(x_n)\in\prod_n\mf{s}_n'$, and using \Cref{lem:cc_element_shortening} 
and the above claim, we obtain
\[ d(\Phi_{\delta}x,\pi(x))=\lim_{\om}\frac{d(\varphi_n(g_n^{a_n(\delta)})x_n,\pi_n(x_n))}{\lambda_n} \leq 4D\cdot \lim_{\om}\frac{\ell(\varphi_n(g_n);\mf{E}_G)}{\lambda_n}=0 ,\]
where $D=\dim\mf{E}_G$. This proves all our claims about $\Phi_{\delta}$.

We now complete the proof of the proposition by distinguishing two cases.

\smallskip
{\bf Case~1:} \emph{$G_{\alpha}\neq K_{\alpha}$.}

\noindent
In this case, the above discussion shows that the line $\alpha$ is shrinkable in the sense of \Cref{defn:shrinkable_lines}. Arguing as in Steps~1--6 of \Cref{subsub:shrinkable_lines} and invoking \Cref{prop:shrinkable_line}, we find a $(G_{\alpha},\emptyset)$--ascetic HNN splitting $G\acts\mc{H}(\mc{E},\eta)$. We can ensure that all subgroups in $\mscr{F}$ are elliptic in $\mc{H}(\mc{E},\eta)$ by considering geometric approximations in \Cref{subsub:shrinkable_lines} in which the elements of $\mscr{F}$ are elliptic. If $B\in\Sal(G)$ does not contain a $G$--conjugate of the centre $A$ of $G_{\alpha}$, then $B$ is elliptic in $\mc{H}(\mc{E},\eta)$: this is the second part of Item~(1) of \Cref{prop:salient_vs_degenerations}.

Considering the shrunk homomorphism $\rho':=\rho_{\mc{E},\eta,\delta}$ (\Cref{defn:shrunk_homo_lines}) for a sufficiently small parameter $\delta$, we have $\ell(\rho'(\g);\X_{\om})\leq\ell(\rho(\g);\X_{\om})$ for all $\g\in F$, and the strict inequality holds for at least one element of $F$, provided that $F$ contains loxodromic elements $\g'$ such that ${\rm Ax}(\g';T^v_{\om})$ shares an arc with a $G$--translate of the line $\alpha$. Choose a stable letter $t\in G_{\alpha}$ for the HNN splitting $G\acts\mc{H}(\mc{E},\eta)$ and let $U$ be the $G$--stabiliser of a vertex of $T$ on the axis of $t$. We define the Dehn twist $\tau_n\in\Aut(G)$ so that $\tau_n|_U$ is the identity and $\tau_n(t)=g_n^{a_n(\delta)}t$, where the elements $g_n\in Z_{K_{\alpha}}(G_{\alpha})$ and integers $a_n(\delta)$ were defined above. Note that $g_n$ fixes ${\rm Ax}(t;T)$ pointwise and commutes with $G_{\alpha}$, and so it commutes with the $G$--stabiliser of the edges of ${\rm Ax}(t;T)$. Thus $\tau_n$ is well-defined.
Moreover, $\tau_n$ is an $(A,\mscr{F})$--ascetic twist (we have $A\in\Sal(G)$ by \Cref{lem:line_stab_SVP}(1)).

Finally, observe that the degeneration determined by the automorphisms $\varphi_n\tau_n$ (still using the $\lambda_n$ as scaling factors) is precisely the $G$--action on $\X_{\om}$ given by the shrunk homomorphism $\rho'$ (up to replacing $a_n(\delta)$ with $-a_n(\delta)$ in the definition of $\tau_n$ to ensure that orientations match up). This yields Items~(1)--(3) of the proposition in Case~1. As to Item~(4), note that Case~1 does not arise if the automorphisms $\varphi_n$ are coarse-median preserving (see the last statement in \Cref{thm:inert_vs_motile}).

\smallskip
{\bf Case~2:} \emph{$G_{\alpha}=K_{\alpha}$.}

\noindent
Here the argument is similar to the proof of \Cref{prop:shortening_inert}, so we omit details. Choose a geometric approximation $\mc{G}$ of $\Min(G;T^v_{\om})$ so that an arc $\beta\sq\alpha\cap\Min(G;T^v_{\om})$ lifts to an arc $\tilde\beta\sq\Min(G;\mc{G})$ fixed by $G_{\alpha}$. Observe that $N_G(G_{\alpha})=G_{\alpha}$, for instance by \Cref{thm:inert_vs_motile}(2). 
Thus, since the normaliser $N_G(G_{\alpha})$ is elliptic in $\mc{G}$, the arc $\tilde\beta$ shares at most one point with each indecomposable component of $\mc{G}$ and, up to shrinking $\beta$ and $\tilde\beta$, we can assume that $\tilde\beta$ is edge-like. Fixing a sufficiently small $\delta>0$, the shortening isometry $\Phi_{\delta}$ constructed above witnesses that the pair $(\beta,\mc{G})$ is a shrinkable arc, and we obtain a shrunk homomorphism $\rho'$. Let $G\acts T$ be the $1$--edge splitting over $G_{\alpha}$ that we obtain from $\mc{G}$ by collapsing arcs disjoint from $G$--translates of $\tilde\beta$. There are Dehn twists $\tau_n\in\Aut(G)$ with respect to $T$, with multiplier $g_n^{a_n(\delta)}$, such that the degeneration given by the automorphisms $\varphi_n\tau_n$ is the $G$--action on $\X_{\om}$ given by the shrunk homomorphism $\rho'$. Items~(1) and~(2) of the proposition then follow from \Cref{prop:shrinkable_arc}, while Item~(3) is clear.

As to Item~(4), note that, according to \Cref{defn:DT_types}, the Dehn twists $\tau_n$ are either partial conjugations (if $T$ is an amalgam) or skews (if $T$ is an HNN). By \Cref{thm:cmp_DT}, the former are always coarse-median preserving, while the latter are if $g_n$ is convex-cocompact and $Z_G(g_n)$ is elliptic in $T$. As mentioned at the start, we have $Z_G(g_n)=G_{\alpha}$, so this subgroup is indeed elliptic in $T$. Finally, convex-cocompactness of $g_n$ follows from convex-cocompactness of $\varphi_n(g_n)$, provided that $\varphi_n$ is coarse-median preserving. This completes the proof of the proposition.
\end{proof}

\begin{rmk}\label{rmk:shortening_motile+}
    \Cref{prop:shortening_motile} can fail when the centre $A\leq G_{\alpha}$ is non-elliptic in $T^v_{\om}$. We still have elements $g_n\in A$ such that $\varphi_n(g_n)$ is convex-cocompact in $G$ and loxodromic in $T^v$ with $\lim_{\om}\frac{1}{\lambda_n}\ell(\varphi_n(g_n);T^v)=0$, and we can still define a shortening isometry $\Phi_{\delta}\colon\X_{\om}\ra\X_{\om}$ in the same way. However, there might not be any \emph{single} HNN splitting $G\acts\mc{H}(\mc{E},\eta)$ in which all the $g_n$ are elliptic. Thus, the $g_n$ only define Dehn twists $\tau_n$ with respect to a sequence of \emph{distinct} HNN splittings $\mc{H}(\mc{E},\eta_n)$, 
    and hence the degeneration associated with the sequence $\varphi_n\tau_n$ is \emph{not} the shortened $G$--action on $\X_{\om}$ determined by $\Phi_{\delta}$ (because the $\tau_n$ do not pairwise commute).

    This explains why the discussion in \Cref{sub:shortening_median} does not apply when $A$ is non-elliptic. In fact, we know that in general it is impossible to only use Dehn twists to shorten: for the groups $G$ constructed in \Cref{ex:poisonous_centre}, the groups $\Out(G)$ are not virtually generated by Dehn twists. In their degenerations, all trees $\Min(G;T^v_{\om})$ will typically be $G$--equivariantly isometric inert lines, with $A=Z_G(G)\cong\Z^2$ acting freely and indiscretely. Every Dehn twist of $A$ has a uniform power that extends to a Dehn twist of $G$, but this does not always suffice to shorten the action.

    We will avoid these issues by using \Cref{prop:reduction_to_preserving_complements} to reduce to the situation where all IOM-lines have elliptic centre, so that \Cref{prop:shortening_motile} suffices for our purposes.
\end{rmk}

Finally, we record here the following simple observation, which is useful to check that Propositions~\ref{prop:shortening_inert} and~\ref{prop:shortening_motile} shorten a generating set that is given a priori.

\begin{lem}\label{lem:axis_in_S^2}
    Let $G$ be a group with a finite generating set $S$ with $1\in S$. Consider a minimal $\R$--tree $G\acts T$ and an arc $\beta\sq T$. Then there exists an element $g\in S^2$ that is loxodromic in $T$ with axis sharing an arc with a $G$--translate of $\beta$.
\end{lem}
\begin{proof}
    Let $\mc{M}(s)$ denote the fixed set of $s$ if $s$ is elliptic, or the axis of $s$ if $s$ is loxodromic. For $s,s'\in S$ with $\mc{M}(s)\cap\mc{M}(s')=\emptyset$, let $\delta_{s,s'}$ be the shortest arc from $\mc{M}(s)$ to $\mc{M}(s')$. For each loxodromic element $s\in S$, pick an arc $\g_s\sq{\rm Ax}(s;T)$ of length $\ell(s;T)$. We choose $\gamma_s$ so that it intersects at least one arc $\delta_{s',s''}$ with $s',s''\in S$; if this is not possible, then $\bigcap_{s\in S}\mc{M}(s)$ is nonempty and we pick $\gamma_s$ containing a chosen basepoint in this intersection (the same for all $s$).
    
    Let $\Delta\sq T$ be the union of the arcs $\delta_{s,s'}$ for $s,s'\in S$ with $\mc{M}(s)\cap\mc{M}(s')=\emptyset$, also together with the arcs $\gamma_s$ with $s\in S$ loxodromic. Note that $\Delta$ is connected and $s\Delta\cap\Delta\neq\emptyset$ for all $s\in S$. It follows that $T=\bigcup_{g\in G}g\Delta$ and so there exists an arc $\beta_0\sq\Delta$ that is contained in a $G$--translate of $\beta$. Up to shrinking $\beta_0$, we have either $\beta_0\sq\gamma_s$ or $\beta_0\sq\delta_{s,s'}$ for some $s,s'\in S$. In the former case, $\beta_0$ is contained in the axis of $s\in S\sq S^2$, while in the latter $\beta_0$ is contained in the axis of $ss'\in S^2$. This proves the lemma.
\end{proof}

\subsection{The shortening theorem}\label{sub:shortening_theorem}

Let $G$ be special. Realise $G$ as a convex-cocompact subgroup of a RAAG $\mc{A}_{\G}$, and a fix a reference system $\mf{R}=(\mc{H},\ll,\{S_H\}_{H\in\mc{H}})$. As in \Cref{sub:reference_systems}, this determines an equivalence relation $\asymp$ on $\Out(G)$ all of whose equivalence classes are finite, and a well-ordering $\sqsubseteq$ on $\Out(G)/\asymp$. For each $\varphi\in\Aut(G)$, we denote by $[\varphi]$ its projection to $\Out(G)/\asymp$. We introduced centraliser twists and ascetic twists in Definitions~\ref{defn:DT_main_types} and~\ref{defn:ascetic_twist}, respectively. The family $\Sal(G)$ was discussed in \Cref{subsub:salient}, and IOM-lines in \Cref{sec:degenerations}.

\begin{defn}
    A degeneration $G\acts\X_{\om}$ is \emph{decentralised} if, for every $v\in\G$ such that $G$ is non-elliptic in $T^v_{\om}$ and for every IOM-line $\alpha\sq T^v_{\om}$, the centre of the stabiliser $G_{\alpha}$ is elliptic in $T^v_{\om}$.
\end{defn}

One source of decentralised degenerations are sequences of automorphisms relative to a system of complements, see \Cref{prop:salient_vs_degenerations}(2). Our hierarchical shortening argument only works for decentralised degenerations (some motivation was given in \Cref{rmk:shortening_motile+}), but this will suffice.

\begin{thm}\label{thm:shortening}
    Let $G\acts\X_{\om}$ be the degeneration arising from a sequence $\varphi_n\in\Aut(G)$. Let $\mscr{E}$ be the family of all subgroups of $G$ elliptic in $\X_{\om}$, and let $\mscr{F}\sq\mscr{E}$ be a finite subset of finitely generated subgroups. If $\X_{\om}$ is decentralised, there is a subgroup $\Delta(\X_{\om})\leq\Aut(G)$ with the following properties.
    \begin{enumerate}
        \item There are elements $\delta_n\in\Delta(\X_{\om})$ with $[\varphi_n\delta_n]\sqsubset[\varphi_n]$ for $\om$--all $n$.
        \item The group $\Delta(\X_{\om})$ is either generated by finitely many centraliser twists relative to $\mscr{E}\cup\Sal(G)$, or by finitely many ascetic twists relative to $\mscr{F}$.
        \item If the $\varphi_n$ are coarse-median preserving, then $\Delta(\X_{\om})$ is generated by finitely many coarse-median preserving centraliser twists relative to $\mscr{E}\cup\Sal(G)$.
    \end{enumerate}
\end{thm}
\begin{proof}
The proof of the theorem is long, so we subdivide it into eight steps. Throughout, let $H^*$ be the minimum element of $(\mc{H},\ll)$ that is not elliptic in $\X_{\om}$. Recall that $H^*$ lies in the enlarged hierarchy $\mf{H}^*(G)$ (\Cref{defn:enlarged_hierarchy}) and thus there exists $H\in\mf{H}(G)$ such that either $H^*=H$ or $H^*/Z_H(H)$ is a $1$--ended free factor of $H/Z_H(H)$. By \Cref{defn:reference_system}(b), we also have that $H$ is the minimum element of $(\mc{H}\cap\mf{H}(G),\ll)$ that is not elliptic in $\X_{\om}$.

\smallskip
{\bf Step~1.} \emph{Consider a vertex $v\in\G$ such that $H^*$ is non-elliptic in $T^v_{\om}$. We can assume that the subtree $\Min(H^*;T^v_{\om})$ does not share any arcs with IOM-lines of $G\acts T^v_{\om}$. In particular, we have $G_{\beta}\in\mc{Z}(G)$ for every arc $\beta\sq\Min(H^*;T^v_{\om})$, and the group $H$ is non-abelian.}

\smallskip\noindent
Suppose that $\Min(H^*;T^v_{\om})$ shares an arc with an IOM-line $\alpha\sq T^v_{\om}$. Let $A$ be the centre of $G_{\alpha}$, which is elliptic in $T^v_{\om}$ by the hypothesis that $\X_{\om}$ is decentralised. Let $S_{H^*}\sq H^*$ be the finite generating set chosen by the reference system $\mf{R}$. By \Cref{lem:axis_in_S^2}, there exists an element $h\in S_{H^*}^2$ that is loxodromic in $T^v_{\om}$ with axis sharing an arc with a $G$--translate of $\alpha$. Now, \Cref{prop:shortening_motile} provides an HNN splitting $G\acts T$ and a sequence of $(A,\mscr{F})$--ascetic twists $\tau_n$ with respect $T$ such that $\ell(\varphi_n\tau_n(S_{H^*}^2);\mf{E}_G)<\ell(\varphi_n(S_{H^*}^2);\mf{E}_G)$ for $\om$--all $n$. Every $K\in\mc{H}$ with $K\ll H^*$ is elliptic in $\X_{\om}$ and hence in $T$ (without loss of generality), 
so the $\tau_n$ are inner on $K$, and we have $\ell(\varphi_n\tau_n(S_K^2);\mf{E}_G)=\ell(\varphi_n(S_K^2);\mf{E}_G)$. This shows that we have $[\varphi_n\tau_n]\sqsubset[\varphi_n]$ for $\om$--all $n$. We can then define $\Delta(\X_{\om})\leq\Aut(G)$ as the group generated by all $(A,\mscr{F})$--ascetic twists with respect to $T$. Since $A$ is finitely generated, $\Delta(\X_{\om})$ is finitely generated.

This proves the theorem in this case, so we can assume that $\Min(H^*;T^v_{\om})$ does not share arcs with IOM-lines. \Cref{thm:inert_vs_motile} then implies that the $G$--stabiliser of each arc of $\Min(H^*;T^v_{\om})$ lies in $\mc{Z}(G)$. Since $H$ lies in $\mf{H}(G)$, it is staunchly $G$--parabolic. Thus, $H$ must be non-abelian, as otherwise $\Min(H^*;T^v_{\om})=\Min(H;T^v_{\om})$ would be an IOM-line by \Cref{rmk:abelians_give_IOM}. 

Note that Step~1 can be skipped if the $\varphi_n$ are coarse-median preserving, since then there are no motile lines by \Cref{thm:inert_vs_motile}, and no indiscrete lines by \Cref{lem:line_addenda}(4).

\smallskip
{\bf Step~2.} \emph{We can assume that $\Perp_H(G_{\beta})$ is non-elliptic in $T^v_{\om}$ for each arc $\beta\sq\Min(H^*;T^v_{\om})$.}

\smallskip\noindent
Say that $\beta\sq\Min(H^*;T^v_{\om})$ is an arc with $\Perp_H(G_{\beta})$ elliptic in $T^v_{\om}$. By \Cref{rmk:BF-stable}, there exists a $G$--stable arc $\beta'\sq\beta$. Since we have $G_{\beta'}\geq G_{\beta}$ and $\Perp_H(G_{\beta'})\leq\Perp_H(G_{\beta})$, the group $\Perp_H(G_{\beta'})$ is also elliptic in $T^v_{\om}$. Thus, we can assume without loss of generality that the arc $\beta$ is itself $G$--stable. Moreover, the arc $\beta$ is $G$--inert by Step~1. The normaliser $N_H(G_{\beta})$ virtually splits as $H_{\beta}\x\Perp_H(G_{\beta})$ by \Cref{lem:cc_basics}(2), and so it is elliptic in $T^v_{\om}$. 

Now, invoking \Cref{prop:shortening_inert} (and again \Cref{lem:axis_in_S^2}), we obtain a $(\mc{Z}(G),\mscr{E})$--splitting 
$G\acts T$ and Dehn twists $\tau_n\in\Aut(G)$ with respect to $T$ such that $[\varphi_n\tau_n]\sqsubset[\varphi_n]$ for $\om$--all $n$. Note that all elements of $\Sal(G)$ are elliptic in $T$ by \Cref{prop:shortening_inert}(2), because their minimal subtrees in $T^v_{\om}$ are IOM-lines by \Cref{rmk:abelians_give_IOM} (if they are non-elliptic), and so they cannot share any arcs with $\beta$ by Step~1. We define $\Delta(\X_{\om})$ as the group generated by Dehn twists with respect to $T$. This is finitely generated because, for every edge group $E$ and vertex group $V$, the centraliser $Z_V(E)$ is convex-cocompact in $G$ (see \cite[Proposition~2.30]{Fio11}). If the $\varphi_n$ are coarse-median preserving, \Cref{prop:shortening_inert}(3) guarantees that it suffices to define $\Delta(\X_{\om})$ as the group generated by folds or partial conjugations with respect to $T$, which are coarse-median preserving by \Cref{thm:cmp_DT}. This proves the theorem in this case.

\smallskip
{\bf Step~3.} \emph{The group $H$ virtually splits as $H'\x Z_H(H)$, with $H'$ strongly irreducible, $G$-parabolic and with $N_G(\Perp_G(H'))=N_G(H')$. The centre $Z_H(H)$ is elliptic in $\X_{\om}$.}

\smallskip\noindent
If $H$ had at least two extended factors, then $\ll$--minimality of $H$ would imply that each of these factors is elliptic in $\X_{\om}$ and so $H$ would itself be elliptic in $\X_{\om}$, a contradiction. Thus, $H$ is the only extended factor of itself, and this means that $H$ virtually splits as $H'\x Z_H(H)$ for a strongly irreducible, $H$--parabolic subgroup $H'$. Since $H$ is $G$--parabolic and non-abelian (by Step~1), then so is $H'$. Invoking again $\ll$--minimality of $H$, we also see that $Z_H(H)$ is elliptic in $\X_{\om}$.

We are left to prove the equality $N_G(\Perp_G(H'))=N_G(H')$. Since conjugations respect orthogonals, we have $N_G(H')\leq N_G(\Perp_G(H'))\leq N_G(\Perp_G\Perp_G(H'))$, and so it suffices to show that $\Perp_G\Perp_G(H')=H'$. Recall that $H=Z_GZ_G(H)$ by \Cref{prop:hierarchy_properties}(2). Since $H'$ has trivial centre, \Cref{lem:cc_basics}(2) implies that we have $\Perp_G(H')=Z_G(H')\geq Z_G(H)$ and hence $\Perp_G\Perp_G(H')\leq Z_GZ_G(H)=H$. Now, observing that $Z_H(H)\leq\Perp_G(H')$, we conclude that $\Perp_G\Perp_G(H')=H'$, completing Step~3.

\smallskip
{\bf Step~4.} \emph{For each arc $\beta\sq\Min(H^*;T^v_{\om})$, we have $G_{\beta}=\Perp_G(H')$.}

\smallskip\noindent
First, we show that $\Perp_G(H')\leq G_{\beta}$ for all arcs $\beta\sq\Min(H^*;T^v_{\om})$. By Step~1, we have $G_{\beta}\in\mc{Z}(G)$, so $G_{\beta}$ is convex-cocompact. If $\Perp_G(H')$ were to contain an element $g$ that is loxodromic in $T^v_{\om}$, then the line $\g:={\rm Ax}(g;T^v_{\om})$ would be $H'$--invariant, as $H'$ commutes with $\Perp_G(H')$. The action $H'\acts\g$ would be nontrivial because $Z_H(H)$ is elliptic in $T^v_{\om}$ and $H$ is not. Thus, the kernel of the action $H'\acts\g$ would be a proper, convex-cocompact, co-abelian subgroup of $H'$, violating strong irreducibility of $H'$ by \Cref{lem:cc_basics}(2). This shows that $\Perp_G(H')$ is elliptic in $T^v_{\om}$, 
and so it fixes $\Min(H';T^v_{\om})=\Min(H;T^v_{\om})$ pointwise, which contains $\Min(H^*;T^v_{\om})$.

Next, we claim that $G_{\beta}\leq N_G(H)$. Otherwise, there would exist an element $g\in G_{\beta}\setminus N_G(H)$ and we would have $\Perp_H(G_{\beta})\leq H\cap Z_G(g)$. The subgroup $H\cap Z_G(g)$ lies in $\mc{J}_G'(H)$ (see \Cref{sub:junctures}) and so it is contained in an element of $\mc{J}_G(H)$, which is elliptic in $\X_{\om}$ by $\ll$--minimality of $H$ (and by \Cref{defn:hierarchy}(H3)). This contradicts the fact that $\Perp_H(G_{\beta})$ is not elliptic in $T^v_{\om}$, which we assumed in Step~2, and so our claim is proven.

Now, the inclusions $\Perp_G(H')\leq G_{\beta}\leq N_G(H)$ and parts~(2) and~(4) of \Cref{lem:cc_basics} show that $G_{\beta}$ virtually splits as $H'_{\beta}\x\Perp_G(H')$. Moreover, since $Z_H(H)\leq G_{\beta}$, we have $\Perp_{H'}(G_{\beta})=\Perp_H(G_{\beta})$, and this subgroup is not elliptic in $T^v_{\om}$ by Step~2. Thus, if $H'_{\beta}$ were nontrivial, then the subgroup $H'_{\beta}\x\Perp_{H'}(G_{\beta})$ would be a nontrivial product, and so it would be contained in a singular subgroup $S\in\mc{S}(H')$. However, the $G$--parabolic closure of $S\x Z_H(H)$ is elliptic in $T^v_{\om}$ by $\ll$--minimality of $H$ (and by \Cref{defn:hierarchy}(H2)). This would contradict the fact that $\Perp_{H'}(G_{\beta})$ is not elliptic.

This shows that $H'_{\beta}=\{1\}$ and hence $G_{\beta}=\Perp_G(H')$, since $G_{\beta}$ is a centraliser and thus root-closed. 

\smallskip
{\bf Step~5.} \emph{We can assume that $H'$ has infinitely many ends, and that the $1$--ended free factors of $H'$ are elliptic in $\X_{\om}$. In particular, we have $H^*=H$.}

\smallskip\noindent
The group $H^*$ virtually splits as $J\x Z_H(H)$, where $J$ is either $H'$ or a $1$--ended free factor of $H'$. Suppose for a moment that $J$ is $1$--ended.

Consider a vertex $v\in\G$ such that $J$ is non-elliptic in $T^v_{\om}$. Let $G\acts\mc{G}$ be a geometric approximation of $\Min(G;T^v_{\om})$; we can use \cite[Corollary~B.3]{Fio11} to ensure that all elements of $\mscr{E}\cup\Sal(G)$ are elliptic in $\mc{G}$. 
By Step~4, the action $H'\acts\Min(H;T^v_{\om})$ has trivial arc-stabilisers. Since $J$ is $1$--ended, the action $J\acts\Min(J;\mc{G})$ has no edge-like arcs or exotic components; it also has no axial components by Step~1. Thus, the tree $\Min(J;\mc{G})$ is covered by indecomposable components of surface type (with respect to the $J$--action), and each of these is contained in an indecomposable component of surface type of $\mc{G}$ (with respect to the $G$--action),
whose arc-stabilisers all equal $\Perp_G(H')$.
Essential simple closed curves on the associated surfaces give cyclic splittings of $J$ that extend to $(\mc{Z}(G),\mscr{E}\cup\Sal(G))$--splittings of $G$, whose edge groups contain $\Perp_G(H')$ as a co-cyclic subgroup.

Now, let $U_1,\dots,U_k$ be representatives of the $J$--orbits of surface type components that arise as we vary $v\in\G$ among vertices with $J\acts T^v_{\om}$ non-elliptic. The $J$--stabiliser of each $U_i$ is (contained in) a quadratically hanging vertex group of the JSJ decomposition of $J$ over cyclic subgroups. Let $\Theta\leq\Aut(J)$ be the group generated by Dehn twists in the cyclic splittings of $J$ corresponding to essential simple closed curves on the surfaces associated to the $U_i$. Finitely many of these twists suffice to generate $\Theta$, and each of them extends to a Dehn twist of $G$ with respect to a $(\mc{Z}(G),\mscr{E}\cup\Sal(G))$--splitting. Define $\Delta(\X_{\om})\leq\Aut(G)$ as the subgroup generated by these finitely many Dehn twists. This guarantees that there are elements $\delta_n\in\Delta(\X_{\om})$ such that $J$ is elliptic in the degeneration given by the sequence $\varphi_n\delta_n\in\Aut(G)$ (if we use the same scaling factors as for $\X_{\om}$),
and so we have $[\varphi_n\delta_n]\sqsubset[\varphi_n]$ for $\om$--all $n$. Moreover, the elements of $\Delta(\X_{\om})$ are coarse-median preserving by \Cref{thm:cmp_DT}(2) (regardless of the $\varphi_n$). This proves the theorem in this case.

Thus, we can assume in the rest of the proof that $J$ is not $1$--ended. This means that $H=H^*$ and that $H/Z_H(H)$ is not $1$--ended. By $\ll$--minimality, we also have that all subgroups in $\mc{F}(H)$ are elliptic in $\X_{\om}$. This implies that $H'$ is not $1$--ended, and that its $1$--ended free factors are elliptic in $\X_{\om}$, completing Step~5.

To state Step~6, we need some terminology: say that a finitely generated subgroup $K\leq H$ is \emph{surface-covered} if, for every $v\in\G$ such that $K$ is non-elliptic in $T^v_{\om}$, each arc of the subtree $\Min(K;T^v_{\om})$ shares an arc with a surface-type component of $\Min(H;T^v_{\om})$.
Here, with an abuse, we speak of surface-type components of $\Min(H;T^v_{\om})$ referring to images of surface-type indecomposable components of geometric approximations of $H\acts\Min(H;T^v_{\om})$.

\smallskip
{\bf Step~6.} \emph{We can assume that all surface-covered subgroups $K\leq H$ are elliptic in $\mc{X}_{\om}$.}

\smallskip\noindent
Step~6 will take some work, which we split into claims. To begin with, let $\mscr{Y}\sq\X_{\om}$ be a minimal $H'$--invariant (or, equivalently, $H$--invariant) convex subset, 
which exists by \cite[Theorem~C]{Fio10b}. Without loss of generality, $\mscr{Y}$ is contained in the $\om$--limit of copies of an $H'$--essential convex subcomplex $\mf{E}_{H'}\sq\X_{\G}$, which is quasi-isometric to $H'$. 
Thus, Step~5 implies that $\mscr{Y}$ is tree-graded, in the sense of \cite{Drutu-Sapir05}. Pieces correspond to ultralimits of left cosets of $1$--ended free factors of $H'$, and so they are preserved by the $H$--action (as a collection). The $H'$--stabiliser of each piece is contained in a $1$--ended free factor of $H'$, and so it is elliptic in $\mscr{Y}$ by Step~5.

Let $d_p$ be the pseudo-distance on $\mscr{Y}$ defined as follows: given points $x,y\in\mscr{Y}$, pick any geodesic $\gamma$ between them, and let $d_p(x,y)$ be the sum of the lengths of the 
arcs that $\gamma$ spends inside the pieces of $\mscr{Y}$. We also consider the pseudo-distance $d_*:=d-d_p$, where $d$ is the median metric on $\mscr{Y}$ induced by $\X_{\om}$. Let $\mc{T}$ be the quotient metric space associated with $(\mscr{Y},d_*)$. Note that $\mc{T}$ is an $\R$--tree and the quotient map $\pi_{\mc{T}}\colon\mscr{Y}\ra\mc{T}$ is $H$--equivariant, $1$--Lipschitz and median-preserving. Each piece of $\mscr{Y}$ projects to a single point of $\mc{T}$, but it is important to remember that point-preimages $\pi_{\mc{T}}^{-1}(x)$ are not pieces of $\mscr{Y}$ in general (they are not even unions of pieces).
We call a subset $Y\sq\mscr{Y}$ \emph{piece-averse} if it shares at most one point with each piece. 

Now, suppose that there exists a finitely generated, surface-covered subgroup $K\leq H$ that is not elliptic in $\mscr{Y}$. Let $\mscr{Y}^K\sq\mscr{Y}$ be the minimal $K$--invariant convex subset.

\smallskip
{\bf Claim~6.1.} \emph{The subset $\mscr{Y}^K\sq\mscr{Y}$ is piece-averse.}

\smallskip\noindent
\emph{Proof of Claim~6.1.}
Suppose that there is a piece $\mc{P}\sq\mscr{Y}$ such that $\mc{P}\cap \mscr{Y}^K$ contains at least two points $x,y$. Pick a vertex $v\in\G$ such that $\pi^v_{\om}(x)\neq\pi^v_{\om}(y)$. We then have a nontrivial arc in the projection $\pi^v_{\om}(\mc{P}\cap \mscr{Y}^K)$, which is contained in $\pi^v_{\om}(\mscr{Y}^K)=\Min(K;T^v_{\om})$. The fact that $K$ is surface-covered then implies that $\pi^v_{\om}(\mc{P})$ shares an arc with some indecomposable component $U$ of the action $H\acts\Min(H;T^v_{\om})$. Now, the action $H_U\acts U$ is indecomposable by \cite[Lemma~1.19(1)]{Guir-Fourier},
and so, since distinct pieces of $\mscr{Y}$ cannot share walls, we have that $H_U$ leaves $\mc{P}$ invariant. This contradicts the fact, seen above, that $H$--stabilisers of pieces are elliptic in $\mscr{Y}$ and hence in $T^v_{\om}$.
\hfill$\blacksquare$

\smallskip
By Claim~6.1, the set $\mscr{Y}^K$ is an $\R$--tree and the metrics $d$ and $d_*$ coincide on $\mscr{Y}^K$.

\smallskip
{\bf Claim~6.2.} \emph{There exist $v\in\G$ and a surface-type component $U\sq\Min(H;T^v_{\om})$ such that the minimal $H_U$--invariant convex subset $\mscr{Y}^U\sq\mscr{Y}$ is piece-averse and $\pi^v_{\om}$ is injective on $\mscr{Y}^U$.}

\smallskip\noindent
\emph{Proof of Claim~6.2.}
Choose an arc $\beta\sq \mscr{Y}^K$ maximising the set $V(\beta)$ of vertices $v\in\G$ such that $\pi^v_{\om}(\beta)$ is a single point. It follows that, for all $v\in\G\setminus V(\beta)$, the projection $\pi^v_{\om}$ is injective on $\beta$.
Fix some $v\in\G\setminus V(\beta)$ and pick an indecomposable component $U\sq\Min(H;T^v_{\om})$ sharing an arc with $\pi^v_{\om}(\beta)$. In fact, up to shrinking $\beta$, we can assume that $\pi^v_{\om}(\beta)\sq U$. Let $\Om\sq\mscr{Y}$ be the union of all $H_U$--translates of the arc $\beta$. The arc $\beta$ is piece-averse because it is contained in $\mscr{Y}^K$. Thus, if two $H_U$--translates of $\pi^v_{\om}(\beta)$ share an arc, then so do the corresponding translates of $\beta$; in particular, they have to intersect. Along with indecomposability of $U$, this shows that $\Om$ is connected. Adding the fact that $\Om$ is piece-averse, we see that $\Om$ is convex within $\mscr{Y}$, and hence $\mscr{Y}^U=\Om$.

This shows that $\mscr{Y}^U$ is piece-averse. Since $\mscr{Y}^U$ is covered by translates of $\beta$, we also get that $\pi^v_{\om}$ is injective on $\mscr{Y}^U$: indeed, since $\mscr{Y}^U$ is a convex tree and $\pi^v_{\om}$ is median-preserving, if $\pi^v_{\om}$ were not injective on $\mscr{Y}^U$ then it would collapse an arc of $\mscr{Y}^U$ to a point, and so it would collapse to a point a sub-arc of $\beta$, violating maximality of the set $V(\beta)$.
\hfill$\blacksquare$

\smallskip
We can now forget about the subgroup $K$ entirely. We will use the component $U$ provided by Claim~6.2 to construct a coarse-median preserving Dehn twist $\tau\in\Aut(G)$ with respect to a $(\mc{Z}(G),\mscr{E}\cup\Sal(G))$--splitting of $G$ such that $\ell(\tau(S_H^2);\X_{\om})<\ell(S_H^2;\X_{\om})$. These quantities are the $\om$--limits of $\ell(\varphi_n\tau(S_H^2);\mf{E}_G)$ and $\ell(\varphi_n(S_H^2);\mf{E}_G)$, so this will prove the theorem in this case.

By Claim~6.2, we have $d=d_*$ on $\mscr{Y}^U$, and so the projection $\pi_{\mc{T}}\colon\mscr{Y}\ra\mc{T}$ is isometric on $\mscr{Y}^U$. In particular, this shows that the tree $\mc{T}$ is nontrivial. Also by Claim~6.2, the projection $\pi^v_{\om}\colon \mscr{Y}^U\ra U$ is a homeomorphism, which implies that $\pi_{\mc{T}}(\mscr{Y}^U)$ is a surface-type component of the tree $H\acts\mc{T}$.

We can now argue exactly as in \cite[Section~5]{RS94} to shorten the action $H\acts\mc{T}$. The action $H'_U\acts U$ has trivial arc-stabilisers. For any finite subset $F\sq H'_U$, any $\eps>0$ and any point $y\in \pi_{\mc{T}}(\mscr{Y}^U)$, there exists a Dehn twist $\overline\tau_{F,\eps,y}\in\Aut(H'_U)$ such that the elements of $\overline\tau_{F,\eps,y}(F)$ all move the point $y$ by less than $\eps$; see \cite[Proposition~5.2]{RS94}. This Dehn twist originates from a splitting of $H'_U$ over a non-peripheral cyclic subgroup, and thus it extends to a Dehn twist $\tau_{F,\eps,y}\in\Aut(G)$ with respect to a $(\mc{Z}(G),\mscr{E}\cup\Sal(G))$--splitting of $G$, which is coarse-median preserving by \Cref{thm:cmp_DT}(2) (this is all identical to the argument in Step~5). Now, by \Cref{lem:axis_in_S^2}, there are elements of $S_H^2$ that are loxodromic in $\mc{T}$ with axis sharing an arc with $\pi_{\mc{T}}(\mscr{Y}^U)$. By \cite[Theorem~5.1]{RS94} (and its proof), we thus have that, for a suitable choice of $F,\eps,y$, the automorphism $\tau=\tau_{F,\eps,y}$ satisfies $\ell(\tau(S_H^2);\mc{T})<\ell(S_H^2;\mc{T})$. (Note that $H'$ is freely decomposable here, but Rips and Sela do not use their free indecomposability assumption in the parts of the argument that are relevant here.)

So far, we have only described how to shorten the action $H\acts\mc{T}$. However, due to the tree-graded structure of $\mscr{Y}$ and the fact that $\mscr{Y}^U$ is piece-averse, the Dehn twist $\tau$ actually satisfies 
\[ \ell(s;\mc{T})-\ell(\tau(s);\mc{T})=\ell(s;\mscr{Y})-\ell(\tau(s);\mscr{Y}) \]
for all $s\in S_H^2$. This proves the theorem under the assumption that there exists a finitely generated, surface-covered subgroup $K\leq H$ that is not elliptic in $\X_{\om}$. We can thus assume that no such subgroup exists, completing Step~6.

\smallskip
{\bf Step~7.} \emph{There exists a $(\mc{Z}(G),\mscr{E}\cup\Sal(G))$--splitting $G\acts T$ such that:
    \begin{enumerate}
        \item[(a)] the edge-stabilisers of $T$ are all $G$--conjugate to $\Perp_G(H')$;
        \item[(b)] the actions $H'\acts T$ and $H'\acts\X_{\om}$ have the same elliptic subgroups.
    \end{enumerate}
}

\smallskip\noindent
Consider a vertex $v\in\G$ such that $H'$ is not elliptic in $T^v_{\om}$. Let $G\acts\mc{G}$ be a geometric approximation of $\Min(G;T^v_{\om})$ with the same elliptic subgroups as $\Min(G;T^v_{\om})$, which exists by \cite[Corollary~B.3]{Fio11}. 
We pick $\mc{G}$ also ensuring that the approximating morphism $f\colon\mc{G}\ra\Min(G;T^v_{\om})$ maps $\Min(H';\mc{G})$ into $\Min(H';T^v_{\om})$; this is possible because there is a finite set of loxodromics in $H'$ whose axes and their $G$--translates cover $\Min(H';\mc{G})$ (e.g.\ by \Cref{lem:axis_in_S^2}), and so it suffices to ensure that $f$ is isometric along these finitely many axes.

Now, it follows from Steps~4 and 5 
that $\Perp_G(H')$ is the $G$--stabiliser of every arc of $\Min(H';\mc{G})$; in particular, the action $H'\acts\Min(H';\mc{G})$ has trivial arc-stabilisers. There are no axial components in $\mc{G}$ by Step~1. Thus, $\Min(H';\mc{G})$ has a transverse covering by maximal edge-like arcs and indecomposable components of exotic and surface types; let $H'\acts\mc{D}$ be the splitting dual to this transverse covering (see \cite[Section~1.3]{Guir-Fourier}). Note that the $H'$--stabiliser of each exotic component $U$ has a free splitting whose vertex groups are precisely the nontrivial point-stabilisers of the action $H'_U\acts U$ (see \cite[Proposition~7.2]{Guir-CMH}). Using these splittings to refine $\mc{D}$ and then collapsing all edges with nontrivial stabiliser, we obtain a splitting $H'\acts\mc{D}'$ with the following properties:
\begin{itemize}
    \item all edge-stabilisers are trivial;
    \item vertex-stabilisers are maximal subgroups $L\leq H'$ such that either $L$ is elliptic in $\mc{G}$ or $\Min(L;\mc{G})$ is a union of surface-type components of $H'\acts\Min(H';\mc{G})$.
\end{itemize}
Since $N_G(H)$ leaves $\Min(H';\mc{G})$ invariant, the action $H'\acts\mc{D}'$ naturally extends to a splitting $N_G(H)\acts\mc{D}'$ with $\Perp_G(H')$ acting trivially and equalling all edge-stabilisers.

Recall from Step~3, that we have $N_G(\Perp_G(H'))=N_G(H')$. Since $\Perp_G(H')$ is the $G$--stabiliser of each arc of $\Min(H';\mc{G})$, it follows that distinct $G$--translates of the subtree $\Min(H';\mc{G})$ share at most one point. Thus, we can extend the family of $G$--translates of $\Min(H';\mc{G})$ to a transverse covering of $\mc{G}$ and, considering the dual tree, we obtain a splitting $G\acts\mc{D}''$ with a vertex $x\in\mc{D}''$ whose $G$--stabiliser is $N_G(H)$ and such that $G$--stabilisers of edges incident to $x$ are point-stabilisers of the action $N_G(H)\acts\Min(H';\mc{G})$. Now, we can use the splitting $N_G(H)\acts\mc{D}'$ to refine $G\acts\mc{D}''$, then collapse all the old edges of $\mc{D}''$. The result is a splitting $G\acts\mf{D}^{(v)}$ with the following properties:
\begin{itemize}
    \item all edge-stabilisers of $\mf{D}^{(v)}$ are $G$--conjugate to $\Perp_G(H')$;
    \item the subgroups of $G$ elliptic in $T^v_{\om}$ are also elliptic in $\mf{D}^{(v)}$;
    \item if $L$ is a vertex-stabiliser of $H'\acts\mf{D}^{(v)}$, then either $L$ is elliptic in $T^v_{\om}$, or $\Min(L;T^v_{\om})$ is a union of surface-type components of $\Min(H';T^v_{\om})$.
\end{itemize}

Finally, we can consider all the splittings $G\acts\mf{D}^{(v)}$, as $v\in\G$ varies through the vertices such that $H'$ is not elliptic in $T^v_{\om}$. We can combine these into a splitting $G\acts\mf{D}$ whose elliptic subgroups are precisely the subgroups elliptic in all $\mf{D}^{(v)}$, and whose edge-stabilisers are intersections of $G$--conjugates of $\Perp_G(H')$ (see e.g.\ \cite[Lemma~4.15]{Fio11}). Let $G\acts T$ be the collapse of $\mf{D}$ where we kill all edges having no $G$--translate contained in $\Min(H';\mc{D})$. 

Let us check that $T$ is the splitting that we are looking for. The subgroup $\Perp_G(H')$ is elliptic in all $\mf{D}^{(v)}$, so it is elliptic in $\mf{D}$, and so it fixes $\Min(H';\mf{D})$ pointwise. In particular, the $G$--stabiliser of each edge of $\Min(H';\mf{D})$ equals $\Perp_G(H')$,
and hence all edge-stabilisers of $T$ are $G$--conjugate to $\Perp_G(H')$. Note that $\Perp_G(H')\in\mc{Z}(G)$ because, since $H'$ has trivial centre, we have $\Perp_G(H')=Z_G(H')$. The subgroups in $\mscr{E}\cup\Sal(G)$ are elliptic in all $\mf{D}^{(v)}$, and hence elliptic in $\mf{D}$ and $T$. Finally, consider a subgroup $L\leq H'$ that is elliptic in $T$. Since $\Min(H';\mf{D})$ projects injectively to $T$, the group $L$ is also elliptic in $\mf{D}$, and so it is elliptic in all $\mf{D}^{(v)}$. This means that $L$ is ``surface-covered'', as defined at the end of Step~5, and so Step~6 implies that $L$ is elliptic in $\X_{\om}$. This completes Step~7.

\smallskip
{\bf Step~8.} \emph{Proof of the theorem in the last case.}

\smallskip\noindent
Let $G\acts T$ be the splitting constructed in Step~7. The action $N_G(H')/\Perp_G(H')\acts\Min(H';T)$ is a free splitting; its nontrivial vertex-stabilisers form a factor system $\mscr{F}$ for $Q:=N_G(H')/\Perp_G(H')$. Let $\mscr{F}'$ be the family of nontrivial point-stabilisers of the action $H'\acts\X_{\om}$. Note that $H'$ projects injectively to a finite-index subgroup of $Q$ and, under this identification, $\mscr{F}'$ consists precisely of the intersections with $H'$ of the subgroups in $\mscr{F}$. 

Since $N_G(\Perp_G(H'))=N_G(H')$, two edges of $\Min(H';T)$ are in the same $G$--orbit if and only if they are in the same $N_G(H')$--orbit. Thus, every Dehn twist with respect to a $1$--edge collapse of the free splitting $Q\acts\Min(H';T)$ extends to a Dehn twist with respect to a $1$--edge collapse of $G\acts T$. Moreover, these extensions are coarse-median preserving by \Cref{thm:cmp_DT}(1). Let $\Delta(\X_{\om})\leq\Aut(G)$ be subgroup generated by these extensions, noting that $\Delta(\X_{\om})$ is finitely generated. This shows that all elements of the pure Fuchs-Rabinowitz group ${\rm FR}^0(\mscr{F})\leq\Aut(Q)$ extend to elements of $\Delta(\X_{\om})\leq\Aut(G)$ (see \Cref{rmk:FR_generated_by_DTs}).

Finally, \Cref{cor:shortening_free_products} implies that there are elements $\delta_n\in\Delta(\X_{\om})$ such that $H'$ is elliptic in the degeneration given by the sequence $\varphi_n\delta_n\in\Aut(G)$ (if we use the same scaling factors as for $\X_{\om}$).
This implies that $[\varphi_n\delta_n]\sqsubset[\varphi_n]$ for $\om$--all $n$, concluding the entire proof of \Cref{thm:shortening}.
\end{proof}

\section{Proof of the main results}\label{sect:proofs}

We can now use the shortening theorem (\Cref{thm:shortening}) as in \cite{RS94} to prove all our main results (after reducing to the decentralised case by \Cref{prop:reduction_to_preserving_complements}). Recall that, given a family of subgroups $\mscr{E}$, we denote by $\Out(G;\mscr{E}^t)$ the group of outer automorphisms that are trivial on each element of $\mscr{E}$. Given a coarse median structure $[\mu]$ on $G$, we denote by $\Out(G,[\mu])$ the group of outer automorphisms preserving $[\mu]$. We also define $\Out(G,[\mu];\mscr{E}^t):=\Out(G,[\mu])\cap\Out(G;\mscr{E}^t)$.

Items~(1), (4) and~(5) in \Cref{thm:main} prove, respectively, Theorems~\ref{thmnewintro:fg}, \ref{thmnewintro:poison} and~\ref{thmintro:cmp} from the Introduction. \Cref{propnewintro:bad_examples} and \Cref{thmnewintro:fi_no_poison} were shown in \Cref{ex:poisonous_centre} and \Cref{cor:poison_obstructs_DT_generation}(2).

\begin{thm}\label{thm:main}
    Let $G$ be a special group. Let $\mscr{E}$ be any family of subgroups of $G$, and let $\mscr{F}$ be a finite set of finitely generated subgroups. Then all the following statements hold.
    \begin{enumerate}
    \setlength\itemsep{.2em}
        \item The group $\Out(G;\mscr{F}^t)$ is finitely generated.
        \item The group $\Out(G;\mscr{F}^t)$ is virtually generated by centraliser twists relative to $\mscr{F}\cup\Sal(G)$, ascetic twists relative to $\mscr{F}$, and pseudo-twists relative to $\mscr{F}$.
        \item If $G$ has no $\mscr{F}$--poison subgroups, then $\Out(G;\mscr{F}^t)$ is virtually generated by centraliser twists relative to $\mscr{F}\cup\Sal(G)$ and ascetic twists relative to $\mscr{F}$.
        \item The group $\Out(G;\mscr{F}^t)$ is virtually generated by Dehn twists with respect to (arbitrary) splittings of $G$ relative to $\mscr{F}$ if and only if $G$ has no $\mscr{F}$--poison subgroups. 
        \item Let $[\mu]$ be a coarse median structure on $G$ induced by a convex-cocompact embedding into a RAAG. The group $\Out(G,[\mu];\mscr{E}^t)$ is virtually generated by finitely many coarse-median preserving centraliser twists relative to $\mscr{E}\cup\Sal(G)$.
    \end{enumerate}
\end{thm}
\begin{proof}
    Realise $G$ as a convex-cocompact subgroup of a RAAG $\mc{A}_{\G}$ and fix a reference system $\mf{R}$ (\Cref{defn:reference_system}). This determines an equivalence relation $\asymp$ on $\Out(G)$, whose equivalence classes are finite, 
    and also a well-ordering $\sqsubseteq$ on the quotient $\Out(G)/\asymp$. Choose an $\mscr{F}$--system of complements $\mf{C}=\{\mf{C}_A\}_A$ for the subgroups $A\in\xSal(G)$ (\Cref{defn:SOC}). 
    
    The core of the proof lies in the following fact.

    \smallskip
    {\bf Claim.} \emph{The group $\Out(G;(\mscr{F}\cup\mf{C})^t)$ is virtually generated by finitely many centraliser twists relative to $\mscr{F}\cup\Sal(G)$ and finitely many ascetic twists relative to $\mscr{F}\cup\mf{C}$.}

    \smallskip\noindent
    \emph{Proof of claim.}
    Let $\Delta\leq\Out(G;(\mscr{F}\cup\mf{C})^t)$ be the group generated by all centraliser twists relative to $\mscr{F}\cup\Sal(G)$ and all ascetic twists relative to $\mscr{F}\cup\mf{C}$. 
    
    We first show that $\Delta$ has finite index in $\Out(G;(\mscr{F}\cup\mf{C})^t)$. Thus, suppose for the sake of contradiction that there exists a sequence $\phi_n\in\Out(G;(\mscr{F}\cup\mf{C})^t)$ such that the cosets $\phi_n\Delta$ are pairwise distinct. Since $\sqsubseteq$ is a well-ordering, each coset $\phi_n\Delta$ has $\sqsubseteq$--minima 
    and we can assume that $\phi_n$ is one of them. Let $G\acts\X_{\om}$ be the degeneration determined by the sequence $\phi_n$. This degeneration is decentralised by \Cref{prop:salient_vs_degenerations}(2), and the elements of $\mscr{F}$ are elliptic in it by \Cref{rmk:elliptic_in_degeneration}. Thus, we can apply the shortening theorem (\Cref{thm:shortening}) to $\X_{\om}$. The result is a sequence of elements $\delta_n\in\Delta$ giving the strict inequality $\phi_n\delta_n\sqsubset\phi_n$ for $\om$--all $n$. In particular, this inequality holds for at least one index $n$, violating minimality of the $\phi_n$. Thus, $\Delta$ must have finite index in $\Out(G;(\mscr{F}\cup\mf{C})^t)$.

    We are left to show that $\Delta$ is finitely generated. For this, choose an enumeration $\tau_n$ of the countably many twists generating $\Delta$, and let $\Delta(n)$ be the subgroup generated by $\tau_1,\dots,\tau_n$. Suppose for the sake of contradiction that $\Delta(n)$ is a proper subgroup of $\Delta$ for all $n$. Up to discarding some $\tau_n$, we can assume that $\tau_{n+1}\not\in\Delta(n)$ for all $n$. Let $\phi_n\in\Out(G)$ be a $\sqsubseteq$--minimum of the coset $\tau_{n+1}\Delta(n)$, which exists because $\sqsubseteq$ is a well-ordering. Note that the $\phi_n$ are pairwise distinct,
    and so they determine a degeneration $G\acts\X_{\om}$, which is again decentralised with all elements of $\mscr{F}$ elliptic. By \Cref{thm:shortening}, there exist a \emph{finitely generated} subgroup $\Delta(\X_{\om})\leq\Delta$ and elements $\delta_n\in\Delta(\X_{\om})$ such that the strict inequality $\phi_n\delta_n\sqsubset\phi_n$ holds for $\om$--all $n$; in particular, this inequality holds for arbitrarily large values of $n$. Now, since $\Delta(\X_{\om})$ is finitely generated, we have $\Delta(\X_{\om})\leq\Delta(n)$ for all large $n$, violating the fact that $\phi_n$ is a $\sqsubseteq$--minimum in the coset $\phi_n\Delta(n)$. In conclusion, $\Delta$ is finitely generated and has finite index in $\Out(G;(\mscr{F}\cup\mf{C})^t)$, proving the claim.
    \hfill$\blacksquare$

    \smallskip
    Now, recall from \Cref{prop:reduction_to_preserving_complements} that we can write $\Out(G;\mscr{F}^t)$ as a product $\mc{U}_{\mf{C}}\cdot\mc{V}_{\mf{C}}$, where:
    \begin{itemize}
        \item $\mc{U}_{\mf{C}}\sq\Out(G;\mscr{F}^t)$ is a finite union of left cosets of $\Out(G;(\mscr{F}\cup\mf{C})^t)$;
        \item $\mc{V}_{\mf{C}}\leq\Out(G;\mscr{F}^t)$ is generated by finitely many pseudo-twists relative to $\mscr{F}$ or, in the case that $G$ has no $\mscr{F}$--poison subgroups, by finitely many ascetic twists relative to $\mscr{F}$.
    \end{itemize}
    By the claim, $\mc{U}_{\mf{C}}$ is a finite union of left cosets of a subgroup $\mc{U}_{\mf{C}}'$ generated by finitely many centraliser twists relative to $\mscr{F}\cup\Sal(G)$ and finitely many ascetic twists relative to $\mscr{F}\cup\mf{C}$. Thus, the group $\langle \mc{U}_{\mf{C}}',\mc{V}_{\mf{C}}\rangle$ has finite index in $\Out(G;\mscr{F}^t)$. This proves Items~(1)--(3) of the theorem, as well as the backward 
    arrow of Item~(4). The forward arrow of Item~(4) was shown in \Cref{cor:poison_obstructs_DT_generation}(2).

    We are left to show Item~(5). Here the proof is almost identical to that of the claim, replacing $\Out(G;(\mscr{F}\cup\mf{C})^t)$ with $\Out(G,[\mu];\mscr{E}^t)$. In fact, we are in a simpler situation because degenerations arising from coarse-median preserving automorphisms never have any IOM-lines: there are no indiscrete lines by \Cref{lem:line_addenda}(4), and no motile lines by \Cref{thm:inert_vs_motile}. Thus, these degenerations are always decentralised, we only need centraliser twists in order to shorten (see \Cref{thm:shortening}(3)), and we can work relative to an arbitrary family $\mscr{E}$ of arbitrary subgroups (the elements of $\mscr{E}$ are still all elliptic in the degeneration by \Cref{rmk:elliptic_in_degeneration}). This concludes the proof of the theorem.
\end{proof}

Recall from \Cref{subsub:singular} that $\mc{S}(G)$ is the family of \emph{singular subgroups} of $G$, that is, the maximal subgroups virtually splitting as a direct product.

\begin{rmk}
    If $\mc{S}(G)\sq\mscr{E}$, then $\Out(G;\mscr{E}^t)$ is generated by finitely many centraliser twists relative to $\mscr{E}$ (that is, one can avoid ascetic twists). For instance, it is not hard to deduce this from the existence of the canonical $(\mc{ZZ}(G),\mc{S}(G))$--splitting constructed in \cite[Theorem~4.28]{Fio11}, arguing as in \cite{Levitt-GD}. 
    Alternatively, one can prove this as in \Cref{thm:main}: one just needs to observe that, if $G\acts\X_{\om}$ is a degeneration in which the elements of $\mc{S}(G)$ are elliptic, then there are no IOM-lines, and so the proof of \Cref{thm:shortening} only requires centraliser twists in order to shorten.

    All required centralised twists are relative to $\mc{S}(G)$ here, and hence coarse-median preserving by \Cref{thm:cmp_DT}. This shows that $\Out(G;\mc{S}(G)^t)$ is virtually contained in $\Out(G,[\mu])$ for all $[\mu]$.
\end{rmk}

We conclude with two examples showing that our results --- particularly the existence of finite-index subgroups $G_0\leq G$ such that $\Out(G_0)$ is virtually generated by Dehn twists --- badly fail for other natural classes of groups.

\begin{ex}\label{ex:IMM}
    \Cref{thmnewintro:fi_no_poison} can fail if $G$ is a non-convex-cocompact subgroup of a RAAG, that is, if $G$ is the fundamental group of a non-compact special cube complex. The theorem can also fail if $G$ is a general $\CAT$ group. Indeed, in both of these classes, there are examples of groups $G$ for which all finite-index subgroups $G_0\leq G$ have infinite $\Out(G_0)$ and still admit no Dehn twists whatsoever. In particular, Swarup's question mentioned in the Introduction \cite[Q2.1]{BesQ} has a negative answer (and this does not require the work in the present article).

    The examples arise from the work of Italiano--Migliorini--Martelli and Martelli, as we now explain. \cite[Section~3]{IMM} constructs\footnote{In \cite{IMM}, the groups $G$ and $H$ are swapped compared to our notation.} a torsion-free hyperbolic group $H$ with a type--$F$ subgroup $G\leq H$ such that $\Out(G)$ is infinite and $G$ admits no splittings over cyclic subgroups. (The same is true of finite-index subgroups of $G$.) As shown in \cite{Groves-Manning-IMM}, this construction can be performed so that $G$ and $H$ embed in a RAAG (here $H$ is convex-cocompact, but $G$ is not). Later work of Martelli \cite[Section~5]{Martelli-closing} shows that we can alternatively ensure that $G$ is $\CAT$.
    
    Now, we claim that $G$ admits no Dehn twists with respect to any splittings (the same argument applies to any finite-index subgroup of $G$). Indeed, for a splitting $G\acts T$ to give rise to a (non-identity) Dehn twist, it is necessary that we have an edge group $E$ and a vertex group $V$ such that $Z_V(E)\neq\{1\}$. Since $G$ is torsion-free and admits no cyclic splittings, the group $E$ must be a non-elementary subgroup of the hyperbolic group $H$. However, this implies that the centraliser $Z_H(E)$ is trivial (since $H$ is torsion-free), and so $Z_V(E)$ is also trivial.
\end{ex}

\begin{ex}\label{ex:nilpotent}
    For a finitely generated nilpotent group $N$, the group $\Out(N)$ is always finitely generated \cite{Auslander-Baumslag}, and even arithmetic over $\Q$ \cite{Baues-Grunewald}. However, $\Out(N)$ is not virtually generated by Dehn twists with respect to splittings of $N$, in general. 

    Indeed, consider $N$ nilpotent and let $\gamma_k(N)$ be the terms in its lower central series, that is, $\gamma_2(N)=[N,N]$ and $\gamma_k(N)=[N,\gamma_{k-1}(N)]$ for $k\geq 3$. Let $n$ be such that $\gamma_n(N)\neq\{1\}$ and $\gamma_{n+1}(N)=\{1\}$, and suppose that $n\geq 2$. 
        
    Let $N\acts T$ be a splitting. Note that $\gamma_n(N)$ is contained in the centre of $N$. Thus, if $\gamma_n(N)$ were to contain a loxodromic element, then $T$ would equal its axis, and so the $N$--action on $T$ would be given by a homomorphism $\rho\colon N\ra\Z\rtimes\Z/2\Z$. Since $\gamma_n(N)$ is central and contained in the commutator subgroup $\gamma_2(N)$, it would follow that $\rho$ vanishes on $\gamma_n(N)$, a contradiction. This shows that $\gamma_n(N)$ is necessarily elliptic in $T$, and so it must actually fix $T$ pointwise, and the action $N\acts T$ factors through an action $N/\gamma_n(T)$. Iterating the argument, we see that any splitting $N\acts T$ factors through an action of the abelianisation $N/[N,N]\acts T$. In particular, all Dehn twists of $N$ are inner on the commutator subgroup $[N,N]$.

    Now, consider the free group $F_n$ and the non-abelian nilpotent group $N_{n,c}=F_n/\gamma_c(F_n)$ for $n,c\geq 3$. 
    We have a natural homomorphism $\rho_{n,c}\colon\Aut(F_n)\ra\Aut(N_{n,c})$ and we claim that the image of $\rho_{n,c}$ always contains an element none of whose nontrivial powers is inner on $[N_{n,c},N_{n,c}]$. We show this assuming for simplicity that $n=3$ and letting $x,y,z$ be a free basis of $F_3$. Consider the automorphism $\varphi\in\Aut(F_3)$ that maps $x\mapsto xz$ while fixing $y$ and $z$. Observe that $\varphi^m([x,y])=x[z^m,y]x^{-1}\cdot[x,y]$. If $\rho_{3,c}(\varphi^m)$ were inner on $[N_{3,c},N_{3,c}]$, then there would exist elements $g\in F_3$ and $u\in\gamma_c(F_n)\leq\gamma_3(F_n)$ such that $\varphi^m([x,y])=u\cdot g[x,y]g^{-1}$. We would then have
    \[ x[z^m,y]x^{-1}\cdot[x,y] = u\cdot g[x,y]g^{-1} \quad\Rightarrow\quad [z^m,y] = x^{-1}\cdot u[g,[x,y]]\cdot x \leq\gamma_3(F_n) . \]
    This is impossible for $m\neq 0$. (It would give a law not satisfied by the Heisenberg group.)
\end{ex}

\appendix
\section{Shortening automorphisms of free products}\label{app:A}

This appendix is devoted to the proof of a shortening theorem for automorphisms of free products (\Cref{cor:shortening_free_products}). Roughly, given a finitely generated group $G$ with a factor system $\mc{F}$, and given an automorphism $\varphi\in\Aut(G)$ that is completely unrelated to $\mc{F}$, we wish to describe how much we can shorten $\varphi$ by means of automorphisms of $G$ that are inner on the elements of $\mc{F}$. The main difficulties arise when the elements of $\mc{F}$ are not freely indecomposable.

\subsection{Generalities}\label{sub_app:generalities}

Let $G$ be a finitely generated group. A \emph{factor system} for $G$ is a collection $\mc{F}$ consisting of the $G$--conjugates of the subgroups $G_1,\dots,G_k$ appearing in some decomposition
\[ G=G_1\ast\dots\ast G_k\ast F_m \]
with $k,m\geq 0$ and $k+m\geq 2$. We require the $G_i$ to be nontrivial, but not that they be freely indecomposable; we allow $\mc{F}$ to be empty if $G\cong F_m$ for $m\geq 2$. The group $G$ admits a factor system whenever it is freely decomposable and not isomorphic to $\Z$.

The \emph{Fuchs-Rabinowitz group} ${\rm FR}(\mc{F})$ is the subgroup\footnote{Note that we have ${\rm FR}(\mc{F})=\Aut(G;\mc{F}^t)$, but we prefer a lighter notation in this appendix.} of $\Aut(G)$ consisting of those automorphisms that on each subgroup $H\in\mc{F}$ coincide with some inner automorphism of $G$. (Note that these automorphisms are not required to preserve any element of $\mc{F}$, and they can equal different inner automorphisms of $G$ on different elements of $\mc{F}$.)
Fix for a moment a writing $G=G_1\ast\dots\ast G_k\ast F_m$ such that $\mc{F}$ consists of the conjugates of the $G_i$, and fix a free generating set $t_1,\dots,t_m$ for $F_m$. It is convenient to name the following three kinds of elementary automorphisms in ${\rm FR}(\mc{F})$.
\begin{enumerate}
    \setlength\itemsep{.2em}
    \item Choosing an index $1\leq i\leq k$ and any element $x\in G_1\cup\dots\cup G_k\cup\{t_1^{\pm},\dots,t_m^{\pm}\}$, there is an automorphism $\kappa_{i,x}\in{\rm FR}(\mc{F})$ that satisfies $\kappa_{i,x}(u)=xux^{-1}$ for all $u\in G_i$ and is the identity on $\bigcup_{j\neq i}G_j\cup\{t_1,\dots,t_m\}$. We refer to $\kappa_{i,x}$ as a \emph{partial conjugation}.
    \item Choosing an index $1\leq i\leq m$ and an element $x\in G_1\cup\dots\cup G_k\cup\{t_1^{\pm},\dots,t_m^{\pm}\}$ with $x\neq t_i^{\pm}$, there are automorphisms $\lambda_{i,x},\rho_{i,x}\in {\rm FR}(\mc{F})$ that satisfy $\lambda_{i,x}(t_i)=xt_i$ and $\rho_{i,x}(t_i)=t_ix$ and equal the identity on all $G_j$ and all $t_j\neq t_i$. We refer to $\lambda_{i,x}$ and $\rho_{i,x}$ as \emph{transvections}.
    \item Choosing an index $1\leq i\leq m$, there is an automorphism $\iota_i\in {\rm FR}(\mc{F})$ that maps $t_i\mapsto t_i^{-1}$ and is the identity on all $G_j$ and all $t_j\neq t_i$. We refer to $\iota_i$ as an \emph{inversion}.
\end{enumerate}
These three kinds elementary automorphisms generate ${\rm FR}(\mc{F})$ \cite{FR1,FR2,McCullough-Miller}, 
while partial conjugations and transvections alone generate a subgroup that we denote by ${\rm FR}^0(\mc{F})$ and refer to as the \emph{pure Fuchs-Rabinowitz group}.
The group ${\rm FR}^0(\mc{F})$ has finite index in ${\rm FR}(\mc{F})$,
and it is independent of the choice of the subgroups $G_j$ and elements $t_j$ made above. 
Note that, since $G$ is finitely generated, so are the groups ${\rm FR}(\mc{F})$ and ${\rm FR}^0(\mc{F})$. 

\begin{rmk}\label{rmk:FR_generated_by_DTs}
    Conjugations and transvections represent the same outer classes in $\Out(G)$ as certain Dehn twists with respect to $1$--edge free splittings of $G$ relative to $\mc{F}$. Indeed, for a partial conjugation $\kappa_{i,x}$, one can consider the amalgamated-product splitting $T$ in which $G_i$ is a vertex group, and in which the other vertex group contains all $G_j\neq G_i$ and all $t_j$. If $x\in G_i$, then $\kappa_{i,x}$ is already a Dehn twist with respect to $T$; if instead $x\not\in G_i$, then $\kappa_{i,x}$ becomes a Dehn twist with respect to $T$ after we compose it with the conjugation by $x^{-1}$ on the whole $G$. The transvections $\lambda_{i,x}$ and $\rho_{i,x}$ are Dehn twists with respect to the HNN splitting in which $t_i$ is a stable letter, while all $G_j$ and all $t_j\neq t_i$ are contained in the same vertex group. Summing up, ${\rm FR}^0(\mc{F})$ is generated by inner automorphisms and by finitely many Dehn twists with respect to $1$--edge free splittings of $G$ relative to $\mc{F}$.
\end{rmk}

The goal of this appendix is to prove the following theorem. Recall the notation $\mf{t}(\cdot;\cdot)$ from \Cref{sub:reference_systems}. We also use the notation $a\vee b$ for $\max\{a,b\}$.

\begin{thm}\label{thm:shortening_free_products} 
    Let $G$ be a (discrete) group with a proper cocompact action $G\acts X$ on a geodesic metric space. There exists a function $\vartheta_X\colon\N_{\geq 1}\ra\N_{\geq 1}$, depending only on $G$ and its action on $X$, such that the following holds. Consider any writing $G=G_1\ast\dots\ast G_k\ast\langle t_1\rangle\ast\dots\ast\langle t_m\rangle$ where $k,m\geq 0$, the $G_i$ are nontrivial subgroups, and the $t_i$ are infinite-order elements. Let $\mc{F}$ be the factor system formed by the $G$--conjugates of $G_1,\dots,G_k$. Choosing any finite generating sets $S_i\sq G_i$ and defining $S:=S_1\cup\dots\cup S_k\cup\{t_1,\dots,t_m\}$, there exists an automorphism $\varphi\in{\rm FR}^0(\mc{F})$ such that
    \[ \mf{t}(\varphi(S);X)\leq \vartheta_X(|S|)\cdot\big(1\vee \max_{1\leq i\leq k}\mf{t}(S_i;X)\big).\]
\end{thm}

We immediately deduce the following consequence.

\begin{cor}\label{cor:shortening_free_products}
    Let $G$ be a finitely generated group with a factor system $\mc{F}$. Fix a (locally finite) Cayley graph $\G$ for $G$. Let $S\sq G$ be any finite generating set and let $S_1,\dots,S_k$ be finite generating sets of representatives of the subgroups in $\mc{F}$. Then there exists a constant $K\geq 1$ such that, for every automorphism $\psi\in\Aut(G)$, there exists an automorphism $\varphi\in{\rm FR}^0(\mc{F})\leq\Aut(G)$ with
    \[ \mf{t}(\psi\varphi(S);\G)\leq K\cdot\big(1\vee\max_{1\leq i\leq k}\mf{t}(\psi(S_i);\G)\big) .\]
\end{cor}
\begin{proof}
    Choose a writing $G=G_1\ast\dots\ast G_k\ast\langle t_1\rangle\ast\dots\ast\langle t_m\rangle$ such that $\mc{F}$ consists of the conjugates of the $G_i$. Up to conjugating the sets $S_i$ by elements of $G$, which does not alter the value of $\mf{t}(\psi(S_i);\G)$, we can assume that $S_i$ is a finite generating set of the chosen $G_i$. Set $S':=S_1\cup\dots\cup S_k\cup\{t_1,\dots,t_m\}$. Let $\vartheta_{\G}$ be the function provided by \Cref{thm:shortening_free_products} in relation to the $G$--action on the Cayley graph $\G$ mentioned in the statement of the corollary. For every $\psi\in\Aut(G)$, we can then apply \Cref{thm:shortening_free_products} to the generating set $\psi(S')$ to obtain an automorphism $\varphi\in{\rm FR}^0(\psi(\mc{F}))$ such that
    \[ \mf{t}(\varphi\psi(S');\G) \leq \vartheta_{\G}(|S'|)\cdot\big(1\vee\max_{1\leq i\leq k} \mf{t}(\psi(S_i);\G) \big) .\]
    Note that the automorphism $\varphi\psi$ can be rewritten as $\psi\varphi'$ for $\varphi':=\psi^{-1}\varphi\psi\in{\rm FR}^0(\mc{F})$. Finally, observe that, if $S$ is the generating set appearing in the statement of the corollary, we have
    \[ \mf{t}(\psi\varphi'(S);\G)\leq \max_{s\in S}|s|_{S'}\cdot \mf{t}(\psi\varphi'(S');\G) ,\]
    where $|\cdot|_{S'}$ denotes word lengths with respect to $S'$. Combining this with the previous inequality proves the corollary for $K:=\vartheta_{\G}(|S'|)\cdot \max_{s\in S}|s|_{S'}$.
\end{proof}

\begin{rmk}
    It is easy to prove \Cref{thm:shortening_free_products} if we allow the function $\vartheta$ to depend on the factor system $\mc{F}$, as well as on the action $G\acts X$: indeed, it suffices to consider a quasi-isometry between $X$ and a tree of spaces in which the vertex spaces are Cayley graphs of the subgroups in $\mc{F}$, and then use the fact that ${\rm FR}(\mc{F})$ acts cofinitely on the space of free splittings of $G$ relative to $\mc{F}$. Unfortunately, it does not seem possible to readily deduce \Cref{thm:shortening_free_products} from this observation: although there are only finitely many $\Aut(G)$--orbits of factor systems, the function $\vartheta_{X,\mc{F}}(\cdot)$ gets worse and worse as we replace $\mc{F}$ by its $\Aut(G)$--translates.
\end{rmk}

The rest of the appendix is devoted to the proof of \Cref{thm:shortening_free_products}. Rather than with the parameter $\mf{t}(\Om;X)$ appearing in the theorem, it will be more convenient to work with
\begin{equation}\label{eq:mf(s)}
    \mf{s}(\Om;X):= \inf_{x\in X}\sum_{s\in S} d(x,sx) ,
\end{equation}
for a finite set $\Om\sq G$ and an action $G\acts X$. We clearly have $\mf{t}(\Om;X)\leq \mf{s}(\Om;X)\leq |\Om|\cdot\mf{t}(\Om;X)$.

\subsection{Grushko triples}

Let $G$ be finitely generated. We define a \emph{Grushko triple} of $G$ as a triple $\mscr{G}=(\mc{G},\mc{C},\mc{D})$ where $\mc{G}=\{G_1,\dots,G_k\}$ is a set of nontrivial subgroups with 
\[ G=G_1\ast\dots\ast G_k , \]
and where $\{1,\dots,k\}=\mc{C}\sqcup\mc{D}$ is a partition such that $G_i\cong\Z$ for all $i\in\mc{C}$. We allow $k=1$. Note that $G_i$ is allowed to be infinite cyclic also for $i\in\mc{D}$, and that none of the $G_i$ is required to be freely indecomposable. We denote by $t_i\in G_i$ a choice of generator for each $i\in\mc{C}$.

The \emph{Fuchs-Rabinowitz group} ${\rm FR}(\mscr{G})\leq\Aut(G)$ is defined simply as the Fuchs-Rabinowitz group of the factor system $\{gG_ig^{-1}\mid g\in G,\ i\in\mc{D}\}$; we define ${\rm FR}^0(\mscr{G})$ analogously. As such, note that ${\rm FR}(\mscr{G})$ and ${\rm FR}^0(\mscr{G})$ are completely unaffected by the groups $G_i$ with $i\in\mc{C}$. Given any automorphism $\varphi\in\Aut(G)$, we define the Grushko triple $\varphi(\mscr{G})$ as the one given by the subgroups $\varphi(G_i)$ with $G_i\in\mc{G}$, keeping the same partition of the index set. 

A Grushko triple $\mscr{G}$ \emph{dominates} another (written $\mscr{G}\succ\mscr{G}'$) if $\mscr{G}'=(\mc{G}',\mc{C}',\mc{D}')$ is obtained from $\mscr{G}$ by a (nontrivial) sequence of the following two moves.
\begin{enumerate}
\item[(M1)] Making a $\mc{C}$--indexed subgroup $\mc{D}$--indexed: pick some $i\in\mc{C}$ and set $\mc{G}':=\mc{G}$, $\mc{D}'=\mc{D}\cup\{i\}$, $\mc{C}'=\mc{C}\setminus\{i\}$.
\item[(M2)] Merging two $\mc{D}$--indexed subgroups: pick distinct indices $i,j\in\mc{D}$ and set $\mc{G}':=\mc{G}\setminus\{G_i,G_j\}\cup\{\langle G_i,G_j\rangle\}$, with $\mc{D}'$ being the union of $\mc{D}\setminus\{i,j\}$ and a new element indexing $\langle G_i,G_j\rangle$.
\end{enumerate}
Note that domination $\mscr{G}\succ\mscr{G}'$ yields inclusions ${\rm FR}(\mscr{G}')\leq{\rm FR}(\mscr{G})$ and ${\rm FR}^0(\mscr{G}')\leq{\rm FR}^0(\mscr{G})$.

We define the \emph{complexity} $\kappa(\mscr{G})$ of a Grushko triple as the lexicographic pair of cardinalities $(|\mc{G}|,|\mc{C}|)$, with the leftmost coordinate taking precedence. With this in mind, $\mscr{G}\succ\mscr{G}'$ implies that $\kappa(\mscr{G})>\kappa(\mscr{G}')$. Note that only a uniformly bounded number of Grushko triples are dominated by any given one.

Given a free splitting $G\acts T$ and a subgroup $H\leq G$, we define $\mc{M}_T(H):=\Min(H;T)$ if $H$ is not elliptic in $T$, and $\mc{M}_T(H):=\Fix(H;T)$ if it is. In particular, if $H\cong\Z$ and $H$ is not elliptic, we have $\mc{M}_T(H)\cong\R$. Fixed sets are singletons, since $T$ has trivial edge-stabilisers.

The following measures the complexity of a Grushko triple with respect to a free splitting of $G$.

\begin{defn}
    Consider a free splitting $G\acts T$ and an (unrelated) Grushko triple $\mscr{G}$. A \emph{$T$--spanning tree} for $\mscr{G}$ is a subtree $\Sigma\sq T$ such that:
    \begin{itemize}
        \item $\Sigma$ intersects $\mc{M}_T(G_i)$ for all $i\in\mc{C}\cup\mc{D}$;
        \item for each $i\in\mc{C}$ with $\ell_T(t_i)\geq 2$, the arc $\Sigma\cap\mc{M}_T(G_i)$ has length $\geq \ell_T(t_i)-1$.
    \end{itemize}
    We denote by $\s_T(\mscr{G})$ the minimum number of edges in a $T$--spanning tree for $\mscr{G}$.
\end{defn}

\subsection{Simplifying Grushko triples}

In view of the Milnor--Schwarz lemma, it suffices to prove \Cref{thm:shortening_free_products} for a single proper cocompact action $G\acts X$. We will use the $G$--action on the graph $\G$ constructed as follows.

To begin with, let $G\acts T$ be the Bass--Serre tree of a splitting of $G$ of the form
    
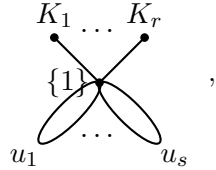
\[
    \begin{tikzpicture}[baseline=(current bounding box.center)]
        \draw[fill] (0,0) circle [radius=0.05cm];
        \draw[fill] (-.6,.6) circle [radius=0.05cm];
        \draw[fill] (.6,.6) circle [radius=0.05cm];
        \draw[thick] (0,0) -- (-.6,.6);
        \draw[thick] (0,0) -- (.6,.6);
        \node[above] at (-.6,.6) {$K_1$};
        \node[above] at (.6,.6) {$K_r$};
        \node[left] at (0,0) {$\{1\}$};
        \draw[thick,rotate around={45:(-.4,-.4)}] (-.4,-.4) ellipse (0.564cm and 0.15cm);
        \draw[thick,rotate around={-45:(.4,-.4)}] (.4,-.4) ellipse (0.564cm and 0.15cm);
        \node at (0,.7) {$\dots$};
        \node at (0,-.7) {$\dots$};
        \node at (-1,-1) {$u_1$};
        \node at (1,-1) {$u_s$};
        \node at (1.5,0) {,};
    \end{tikzpicture}
    \]
where the $K_i$ are freely indecomposable subgroups of $G$, and the $u_i$ are some stable letters. Let $\G$ be the graph, equipped with a proper cocompact $G$--action, that is obtained from $T$ by (equivariantly) blowing up each vertex with nontrivial $G$--stabiliser to a Cayley graph of said stabiliser (with respect to some finite generating set). In particular, $\G$ is a tree of Cayley graphs, and there is a natural $1$--Lipschitz $G$--equivariant projection $\pi\colon\G\ra T$. Edges of $\G$ come in two kinds: edges that $\pi$ collapses to a vertex of $T$, and edges that $\pi$ maps isometrically onto an edge of $T$. 

Given two edges $e,f\sq T$, we denote by $d_{\G}(e,f)$ the distance in the graph $\G$ of the two edges of $\G$ to which $e$ and $f$ lift. If $g\in G$ is loxodromic in $T$, we define $\ell_{\G}(g):=d_{\G}(e,ge)+1$ for any edge $e\sq\mc{M}_T(g)$; note that this quantity is independent of the choice of $e$.

\begin{ass}\label{label:setup}
    Throughout the following discussion, we consider the following fixed objects.
    \begin{itemize}
        \item The tree $T$ and graph $\G$ are the ones defined just above.
        \item The Grushko triple $\mscr{G}=(\mc{G},\mc{C},\mc{D})$ is completely arbitrary, that is, unrelated to $T$.
        \item We fix finite generating sets $S_i\sq G_i$ for all $i\in\mc{D}$ and set 
        \[ D:= \max_{i\in\mc{D}}\mf{s}(S_i;\G) ,\]
        using the notation introduced in \Cref{eq:mf(s)}.
        \item For each $i\in\mc{C}\cup\mc{D}$, we write $C_i:=\mc{M}_T(G_i)$. In particular, $C_i\cong\R$ for all $i\in\mc{C}$, in view of \Cref{lem:using_free_indecomposability}(1) just below.
        \item For each $i\in\mc{C}\cup\mc{D}$, we denote by $\mc{E}_i$ the set of edges $e\sq T$ such that $e\cap C_i\neq\emptyset$ and ${\rm int}(e)$ separates $C_i$ from a point in the union $\bigcup_{j\in\mc{C}\cup\mc{D}}C_j$. In particular, each edge in $\mc{E}_i$ shares exactly one vertex with $C_i$. 
        \item For each $i\in\mc{C}\cup\mc{D}$, we denote by $\mc{V}_i$ the set of vertices of $C_i$ at which the edges in $\mc{E}_i$ are based. In other words, the points in $\mc{V}_i$ are the projections to $C_i$ of the points in $\bigcup_{j\in\mc{C}\cup\mc{D}}C_j\setminus C_i$.
    \end{itemize}
\end{ass}

Our goal is now to modify $\mscr{G}$ by applying an element of ${\rm FR}^0(\mscr{G})$ so as to ``simplify'' its aspect in relation to the action $G\acts T$. Note that applying such an automorphism does not alter the quantities $\mf{s}(S_i;\G)$, nor the constant $D$.

The simplification of $\mscr{G}$ will occur progressively over the next few lemmas, beginning with \Cref{lem:spanning_tree_long}. Before starting, however, we need to collect some preliminary observations in the next result. We denote by $\Gr(\cdot)$ the Grushko rank of a group.

\begin{lem}\label{lem:using_free_indecomposability}
    Under \Cref{label:setup}, the following statements hold for all $\varphi\in{\rm FR}(\mscr{G})$.
    \begin{enumerate}
    \setlength\itemsep{.2em}
        \item The elements $\varphi(t_i)$ with $i\in\mc{C}$ are loxodromic in $T$.
        \item For each $i\in\mc{D}$ and each vertex $v\in T$, either the $\varphi(G_i)$--stabiliser of $v$ is trivial, or $\varphi(G_i)$ contains the entire $G$--stabiliser of $v$. In the latter case, the $\varphi(G_i)$--stabiliser of $v$ acts transitively on edges of $T$ incident to $v$.
        \item For all $i\in\mc{D}$, there is a connected $\varphi(G_i)$--invariant subgraph of $\G$ with $\leq D$ $\varphi(G_i)$--orbits of edges. In particular, the action $\varphi(G_i)\acts\mc{M}_T(\varphi(G_i))$ has $\leq D$ orbits of edges.
        \item Consider $i\in\mc{D}$ and an arc $\beta\sq\mc{M}_T(\varphi(G_i))$ whose interior vertices have trivial $\varphi(G_i)$--stabiliser and lie in pairwise distinct $\varphi(G_i)$--orbits. Then, denoting by $e$ and $f$ the first and last edge of $\beta$, we have $d_{\G}(e,f)<6\Gr(G_i)^2 D$.
        \item If $i,j\in\mc{D}$ are distinct and the intersection $\mc{M}_T(\varphi(G_i))\cap\mc{M}_T(\varphi(G_j))$ contains either an edge or a vertex with nontrivial $\varphi(G_i)$--stabiliser, then there exists an automorphism $\varphi'\in{\rm FR}^0(\mscr{G})$ such that $\mf{s}(\varphi'(S_i)\cup\varphi'(S_j);\G)\leq 2D(1+|S_i|+|S_j|)$.
    \end{enumerate}
\end{lem}
\begin{proof}
    Let $G\acts T'$ be a free splitting having the conjugates of the $G_i$ with $i\in\mc{D}$ as its nontrivial vertex groups. Instead, if $K$ is a nontrivial vertex group of $T$, then $K$ is freely indecomposable. Thus, for every automorphism $\varphi\in\Aut(G)$, the subgroup $\varphi^{-1}(K)$ is elliptic in $T'$. For each index $i\in\mc{C}$, the element $t_i$ is loxodromic in $T'$, and hence $t_i\not\in\varphi^{-1}(K)$ and $\varphi(t_i)\not\in K$. This proves part~(1). Similarly, for each index $i\in\mc{D}$, we have either $\varphi^{-1}(K)\cap G_i=\{1\}$ or $\varphi^{-1}(K)\leq G_i$, and hence either $K\cap\varphi(G_i)=\{1\}$ or $K\leq\varphi(G_i)$. By construction, each nontrivial vertex-stabiliser of $T$ acts transitively on the incident edges of $T$, and so part~(2) follows.

    Regarding parts~(3) and~(4), note that it suffices to prove them for $\varphi=\id_G$, since $\varphi(G_i)$ is conjugate to $G_i$ for all $\varphi\in{\rm FR}(\mscr{G})$ and $i\in\mc{D}$.
    
    Thus, consider an index $i\in\mc{D}$ and pick a vertex $x\in\G$ realising $\mf{s}(S_i;\G)$. It follows that $\sum_{s\in S_i}d(x,sx)\leq D$, and hence there exists a connected subgraph $A_i^0\sq\G$ with at most $D$ edges containing the set $\{x\}\cup\{sx\mid s\in S_i\}$. Define $A_i$ as the union of all $G_i$--translates of $A_i^0$. Note that $A_i$ is a connected $G_i$--invariant subgraph of $\G$ with at most $D$ $G_i$--orbits of edges. The same holds for the projection $\pi(A_i)\sq T$, which must contain the subtree $\mc{M}_T(G_i)$. This proves part~(3).

    For part~(4), consider an arc $\beta\sq\mc{M}_T(G_i)$ as in the statement. Let $\beta_1,\dots,\beta_s$ be the maximal sub-arcs of $\beta$ all of whose interior vertices have degree $2$ within $\mc{M}_T(G_i)$. Choose an edge $\eps_i$ within each $\beta_i$. Let $E$ be the the union of $\{\eps_1,\dots,\eps_s\}$ with the the set of edges of $\beta\setminus(\beta_1\cup\dots\cup\beta_s)$, also adding $e$ and $f$ to $E$ if necessary. Finally, let $E'$ be the union of $E$ with the set of edges of $\mc{M}_T(G_i)$ that are based at interior vertices of $\beta$ without being contained in $\beta$.

    \smallskip
    {\bf Claim.} \emph{We have $|E'|\leq 6\Gr(G_i)^2$.}

    \smallskip\noindent
    \emph{Proof of claim.}
    First, we show that $|E|\leq 3\Gr(G_i)$. For this, note that no two edges of $\beta$ are in the same $G_i$--orbit, except possibly $e$ and $f$. Consider the graph $\Delta$ obtained from the quotient $\mc{M}_T(G_i)/G_i$ by replacing with a single edge each maximal path all of whose interior vertices have degree $2$ and trivial vertex group. A standard argument shows that $\Delta$ has at most $3\Gr(G_i)-3$ edges. 
    Since the edges in $E\setminus\{e,f\}$ give distinct edges of $\Delta$, we obtain the bound on $|E|$.

    Now, consider any vertex $v$ in the interior of $\beta$. We claim that $v$ has degree $\leq 2\Gr(G_i)$ in $\mc{M}_T(G_i)$. For this, construct a splitting $G_i\acts T'$ by starting with $\mc{M}_T(G_i)$ and collapsing all edges that do not have a $G_i$--translate incident to $v$. Since the $G_i$--stabiliser of $v$ is trivial by hypothesis, the edges of $\mc{M}_T(G_i)$ incident to $v$ project to pairwise distinct edges of $T'/G_i$, possibly except for some pairs of edges projecting to the same loops. The bound on the degree of $v$ then follows from standard bounds on the number of edges of free splittings.
    
    Finally, combining the previous two inequalities, we immediately get $|E'|\leq 6\Gr(G_i)^2$.
    \hfill$\blacksquare$

    \smallskip
    Now, let $\tilde e,\tilde f$ be the unique lifts of $e,f\sq\beta$ to edges of $\G$. Let again $x\in\G$ be the vertex considered in part~(3), and pick elements $g,g'\in G_i$ such that the path $\beta$ separates $\pi(gx)$ from $\pi(g'x)$ in $T$. Write $g'=gs_1\dots s_r$ for elements $s_r\in S_i^{\pm}$ and choose a geodesic $\gamma_i\sq\G$ from $gs_1\dots s_{i-1}x$ to $gs_1\dots s_i x$ for each $i$. The concatenation $\gamma_1\dots\gamma_r$ is a path from $gx$ to $g'x$ in $\G$, and so it contains both $\tilde e$ and $\tilde f$. Each segment $\gamma_i$ has length $\leq D$. Moreover, no segment $\gamma_i$ is entirely contained in the preimage $\pi^{-1}({\rm int}(\beta))$, because the endpoints of $\gamma_i$ are in the same $G_i$--orbit, while no two vertices in ${\rm int}(\beta)$ are, by hypothesis. As a consequence, we obtain a sequence of distinct edges $e_1,\dots,e_u\in E'$ such that $e_1=e$, $e_u=f$, and for each $i$ the lifts to $\G$ of the edges $e_i,e_{i+1}$ are contained in some union $\gamma_j\cup\gamma_{j+1}$. In particular, $d_{\G}(e_i,e_{i+1})<2D$. Combined with the claim, this yields
    \[ d_{\G}(e,f)< 2D|E'|\leq 6\Gr(G_i)^2 D,\]
    completing the proof of part~(4).

    Finally, we prove part~(5). Since $\varphi(G_i)$ and $\varphi(G_j)$ are conjugate to $G_i$ and $G_j$, respectively, we can argue as in part~(3) to find connected subgraphs $A_i,A_j\sq\G$ on which $\varphi(G_i)$ and $\varphi(G_j)$ act with at most $D$ orbits of edges. Any edge of $\mc{M}_T(\varphi(G_i))$ lifts to a unique edge of $\G$, and this lift lies in $A_i$. Similarly, for any vertex $v\in\mc{M}_T(\varphi(G_i))$ with nontrivial $\varphi(G_i)$--stabiliser, the entire (vertex set of the) preimage $\pi^{-1}(v)$ is contained in $A_i$, by part~(2). In conclusion, if the intersection $\mc{M}_T(\varphi(G_i))\cap\mc{M}_T(\varphi(G_j))$ contains either an edge or a vertex with nontrivial $\varphi(G_i)$--stabiliser, we have $A_i\cap A_j\neq\emptyset$.

    Now, pick a vertex $y\in A_i\cap A_j$. Also pick vertices $x_i\in A_i$ and $x_j\in A_j$ realising $\mf{s}(\varphi(S_i);\G)$ and $\mf{s}(\varphi(S_j);\G)$, respectively. Choose elements $g_i\in G_i$ and $g_j\in G_j$ such that $d(\varphi(g_i)x_i,y)\leq D$ and $d(\varphi(g_j)x_j,y)\leq D$. We then have
    \begin{align*}
        \mf{s}(\varphi(g_iS_ig_i^{-1})\cup \varphi(g_jS_jg_j^{-1});\G) &\leq \mf{s}(\varphi(S_i);\G)+2|S_i|d(\varphi(g_i)x_i,y)+\mf{s}(\varphi(S_j);\G)+2|S_j|d(\varphi(g_j)x_j,y) \\
        &\leq 2D +2|S_i|D+2|S_j|D .
    \end{align*}
    Finally, there exists an automorphism $\varphi'\in{\rm FR}^0(\mscr{G})$ such that $\varphi'(S_i)=\varphi(g_iS_ig_i^{-1})$ and $\varphi'(S_j)=\varphi(g_jS_jg_j^{-1})$: indeed $\varphi$ can be replaced by an element of ${\rm FR}^0(\mscr{G})$ without altering $\varphi(S_i)$ and $\varphi(S_j)$, after which we can right-multiply by an element of ${\rm FR}^0(\mscr{G})$ equaling the conjugations by $g_i$ and $g_j$ on $G_i$ and $G_j$, respectively. This concludes the proof of part~(5) and of the entire lemma.
\end{proof}

We now make additional assumptions, to rule out ``easy'' ways of simplifying the triple $\mscr{G}$.

\begin{ass}\label{ass:assumptions}
    In addition to \Cref{label:setup}, suppose that all $\varphi\in{\rm FR}^0(\mscr{G})$ satisfy:
\begin{enumerate}
    \item[(i)] for all $i\in\mc{C}$, we have $\ell_{\G}(\varphi(t_i))>24\Gr(G)^3 D$;
    \item[(ii)] for all $i,j\in\mc{D}$ with $i\neq j$, the intersection $\mc{M}_T(\varphi(G_i))\cap\mc{M}_T(\varphi(G_j))$ is either empty, or a single vertex with trivial $\varphi(G_i)$-- and $\varphi(G_j)$--stabiliser;
    \item[(iii)] we have $\s_T(\varphi(\mscr{G}))\geq\s_T(\mscr{G})$.
\end{enumerate}
\end{ass}

Roughly, Items~(i) and~(ii) can always be ensured by enlarging the constant $D$ and replacing the Grushko triple $\mscr{G}$ by a triple $\mscr{G}'\prec\mscr{G}$. Item~(iii) can then always be ensured by replacing $\mscr{G}$ by $\varphi(\mscr{G})$ for some $\varphi\in{\rm FR}^0(\mscr{G})$. We will explain this in detail in the proof of \Cref{thm:shortening_free_products} below.

\begin{lem}\label{lem:spanning_tree_long}
Under \Cref{ass:assumptions}, let $\Sigma\sq T$ be a $T$--spanning tree for $\mscr{G}$ with exactly $\s_T(\mscr{G})$ edges. Then, the following are satisfied.
\begin{enumerate}
\item For each $i\in\mc{C}\cup\mc{D}$, no two points of $\Sigma$ lie in the same $G_i$--orbit.
\item For each $i\in\mc{C}$, the intersection $\gamma_i:=\Sigma\cap C_i$ is an arc of length exactly $\ell_T(t_i)-1$. For every $j\in\mc{C}\setminus\{i\}$, we have $C_i\cap C_j\sq\gamma_i\cap\gamma_j$. 
\item For all $i\in\mc{C}$ and $j\in\mc{D}$, the $G_j$--stabiliser of each vertex of $\gamma_i$ is trivial.
\item For all $i\in\mc{C}$ and $j\in\mc{D}$, the arc $C_i\cap C_j$ has length $\leq \min\{D+1,\ell_T(t_i)\}$.
\item For each $i\in\mc{D}$, no two vertices in $\mc{V}_i$ lie in the same $G_i$--orbit. Similarly, no two edges in $\mc{E}_i$ lie in the same $G_i$--orbit.
\end{enumerate}
\end{lem}
\begin{proof}
The proof of all parts of the lemma is based on roughly the same idea: if they fail, one can construct an automorphism $\varphi\in{\rm FR}^0(\mscr{G})$ violating one of our assumptions (usually Item~(iii) in \Cref{ass:assumptions}). Details vary, so we treat each part separately.

\smallskip
{\bf Part~(1).} Suppose for the sake of contradiction that there exist an index $i\in\mc{C}\sqcup\mc{D}$, an element $g\in G_i$, and a point $x\in\Sigma$ with $gx\in\Sigma\setminus\{x\}$. For simplicity, say that $i=1$.

We first treat the case when $g$ is loxodromic in $T$. Here it is not restrictive to assume that $x,gx\in\Sigma\cap\mc{M}_T(g)$. Note that $[x,gx]\sq C_1$, and let $e$ be the edge of the arc $[x,gx]$ that is incident to $gx$. Let $A,B$ be the connected components of $\Sigma\setminus{\rm int}(e)$ containing, respectively, $x$ and $gx$. Set $\Sigma':=A\cup g^{-1}B$ and note that $\Sigma'$ is a tree with strictly fewer edges than $\Sigma$. We define an automorphism $\varphi\in{\rm FR}^0(\mscr{G})$ as follows:
\begin{enumerate}
\item[(a)] for $j=1$, the automorphism $\varphi$ is the identity on $G_j$;
\item[(b)] if $C_j\cap B=\emptyset$, then $\varphi$ is the identity on $G_j$;
\item[(c)] if $C_j\cap A=\emptyset$, then $\varphi(u)=g^{-1}ug$ for all $u\in G_j$;
\item[(d)] if $e\sq C_j$ and $j\in\mc{D}$, then $\varphi$ is the identity on $G_j$;
\item[(e)] if $e\sq C_j$ and $j\in\mc{C}\setminus\{1\}$, then $\varphi(t_j)=g^{-1}t_j$ if $t_je$ is on the same side of $e$ as $B$, and $\varphi(t_j^{-1})=g^{-1}t_j^{-1}$ otherwise.
\end{enumerate}
We claim that $\Sigma'$ is a $T$--spanning tree for $\varphi(\mscr{G})$ and hence $\s_T(\varphi(\mscr{G}))<\s_T(\mscr{G})$, which will contradict Item~(iii) of \Cref{ass:assumptions}.

We now prove the claim. Recall that the elements $\varphi(t_i)$ with $i\in\mc{C}$ are all loxodromic in $T$ by \Cref{lem:using_free_indecomposability}(1). When the index $j$ falls into Cases~(b)--(d), it is straightforward to check that the intersection $\Sigma'\cap\mc{M}_T(\varphi(G_j))$ is nonempty, and that it has length $\geq\ell_T(\varphi(t_j))-1$ if $j\in\mc{C}$. In Case~(a), it is clear that $\Sigma'\cap C_1\neq\emptyset$ and, if $1\in\mc{C}$, the arc $[x,gx]\setminus e$ is contained in $\Sigma'$ and has length $\geq\ell_T(g)-1\geq\ell_T(t_1)-1$. Finally, we discuss Case~(e), supposing without loss of generality that $\varphi(t_j)=g^{-1}t_j$. Since $\Sigma$ is a $T$--spanning tree for $\mscr{G}$, there exists an arc $\eta\sq\Sigma\cap C_j$ of length $\ell_T(t_j)-1$ and, since $e\sq\Sigma\cap C_j$, we can choose the arc $\eta$ so that $e\sq\eta$. Letting $a,b$ be the endpoints of $\eta$, respectively in $A$ and $B$, we have $d(t_ja,b)=1$. Note that the arc $[a,g^{-1}b]$ is entirely contained in $\Sigma'$, and that the point $\varphi(t_j)a=g^{-1}t_ja$ is adjacent to $g^{-1}b$. Now, since $\varphi(t_j)$ takes a point of $\Sigma'$ to a point at distance $\leq 1$ from $\Sigma'$, it follows that $\Sigma'$ contains an arc of the axis of $\varphi(t_j)$ of length $\geq\ell_T(\varphi(t_j))-1$ (since $\varphi(t_j)$ cannot be elliptic). This proves the claim, and yields the required contradiction when $g$ is loxodromic.

We now treat the case when $g$ is elliptic in $T$. Here we necessarily have $1\in\mc{D}$ and, denoting by $p\in C_1$ the unique fixed point of $g$, there exists an edge $f$ incident to $p$ such that $f\cup gf\sq\Sigma$. We can assume that $x$ is the vertex in $f\setminus\{p\}$ and, setting $e:=gf$, we can then define exactly as above the sets $A,B,\Sigma'$ and the automorphism $\varphi\in{\rm FR}^0(\mscr{G})$. The same claim stands, with the same proof, and it again yields a contradiction. This completes the proof of part~(1).

\smallskip
{\bf Part~(2).} Consider $i\in\mc{C}$ and suppose again for simplicity that $i=1$. The arc $\gamma_1:=\Sigma\cap C_1$ has length at least $\ell_T(t_1)-1$ because $\Sigma$ is a $T$--spanning tree for $\mscr{G}$, and it also has length at most this quantity by part~(1). Hence the length of $\gamma_1$ is exactly $\ell_T(t_1)-1$. Considering $j\in\mc{C}$ such that $C_1\cap C_j\neq\emptyset$, we are left to show that $C_1\cap C_j\sq\Sigma$. 

Note that $C_1\cap C_j$ intersects $\Sigma$ by Helly's lemma, since $\Sigma$ intersects both $C_1$ and $C_j$ by definition. Supposing for the sake of contradiction that the arc $C_1\cap C_j$ is not contained in $\Sigma$, there exists an edge $e\sq C_1\cap C_j$ sharing a single vertex with $\Sigma$; this vertex is also the only point of intersection of $e$ with the arcs $\gamma_1$ and $\gamma_j$. Up to replacing $t_1$ and $t_j$ with their inverses, we can assume that $e$ comes after the arcs $\gamma_1$ and $\gamma_j$ along the lines $C_1$ and $C_j$ in the direction of translation of $t_1$ and $t_j$. Letting $e_1\sq\gamma_1$ and $e_j\sq\gamma_j$ be the edges that are farthest from $e$, we have $t_1e_1=e=t_je_j$. Thus, the element $t_j^{-1}t_1$ takes $e_1$ to $e_j$, and hence either $t_j^{-1}t_1$ is elliptic or the union $\gamma_1\cup\gamma_j\sq\Sigma$ contains an arc of $\mc{M}_T(t_j^{-1}t_1)$ of length $\ell_T(t_j^{-1}t_1)$. 

Letting $\varphi\in{\rm FR}^0(\mscr{G})$ be the automorphism that satisfies $\varphi(t_1)=t_j^{-1}t_1$ and equals the identity on $\bigcup_{i\neq 1}G_i$, it follows that $\varphi(t_1)$ is not elliptic (because of \Cref{lem:using_free_indecomposability}(1)), and therefore $\Sigma$ is a $T$--spanning tree for $\varphi(\mscr{G})$. This contradicts part~(1) of the lemma applied to the Grushko triple $\varphi(\mscr{G})$ (which still satisfies \Cref{ass:assumptions}), since the two edges $e_1,e_j\sq\Sigma$ are in the same $\varphi_1(G_1)$--orbit. This proves part~(2).

\smallskip
{\bf Part~(3).} Consider $i\in\mc{C}$, $j\in\mc{D}$, and a point $x\in\gamma_i$. Say for simplicity that $i=1$ and $j=2$. Let $e,f$ be the two edges of $C_1$ incident to $x$. Suppose for the sake of contradiction that the $G_2$--stabiliser of $x$ is nontrivial. Then $e$ and $f$ are in the same $G_2$--orbit, by \Cref{lem:using_free_indecomposability}(2); pick an element $g\in G_2$ with $ge=f$. By part~(1), the edges $e$ and $f$ do not both lie in $\Sigma$. At the same time, since $x\in\gamma_1$, at least one among $e$ and $f$ does lie in $\Sigma$ (note that we cannot have $\gamma_1=\{x\}$, as otherwise we would have $\ell_T(t_1)=1$ and either $gt_1$ or $g^{-1}t_1$ would give an edge inversion). 

Thus, we can suppose without loss of generality that $e\not\sq\gamma_1$ and $f\sq\gamma_1$, and that $e$ follows $f$ along $C_1$ in the direction of translation of $t_1$. Consider the automorphism $\varphi\in{\rm FR}^0(\mscr{G})$ mapping $t_1\mapsto g^{-1}t_1$ and equaling the identity on all $G_j\in\mc{G}$ with $j\neq 1$. The element $\varphi(t_1)$ maps a vertex of $\gamma_1$ (its initial endpoint) to another vertex of $\gamma_1$ (the one preceding the terminal endpoint $x$). Thus, $\Sigma$ is still a $T$--spanning tree for $\varphi(\mscr{G})$, and therefore $\varphi(\mscr{G})$ still satisfies \Cref{ass:assumptions}. The fact that two points of $\gamma_1=C_1\cap\Sigma$ are in the same $\varphi(G_1)$--orbit now violates part~(1), providing the required contradiction and proving part~(3).

\smallskip
{\bf Part~(4).} We begin by stating the following claim, which allows us to slide $T$--spanning trees along axes, and which will also be needed in the proof of part~(5) below. As usual, for simplicity, we consider the index $i=1$ and suppose that $1\in\mc{C}$. For each index $j$, we denote by $\Pi_j$ the projection of $C_j$ to $C_1$ (this is simply the intersection $C_1\cap C_j$ if the latter is nonempty). 

\smallskip
{\bf Claim.} \emph{Let $e$ be the initial edge of the arc $\gamma_1=\Sigma\cap C_1$ (orienting $\gamma_1$ in the direction of translation of $t_1$). There exist $\varphi\in{\rm FR}^0(\mscr{G})$ and a $T$--spanning tree $\Sigma'$ for $\varphi(\mscr{G})$ such that:
\begin{itemize}
    \item $\Sigma$ and $\Sigma'$ have the same number of edges;
    \item $\Sigma'\cap C_1=\overline{(\gamma_1\setminus e)}\cup f$, where $f$ is the (unique) edge whose interior separates $\gamma_1$ from $t_1e$;
    \item for each index $j$, the automorphism $\varphi$ is equal to the identity on $G_j$ except when $\Pi_j\cap\gamma_1$ equals the initial endpoint of $\gamma_1$, and except when we have $j\in\mc{C}\setminus\{1\}$ and $e\sq C_j$.
\end{itemize}
}

We omit the proof of the claim since it is analogous to the proof of the claim used for part~(1): if $A,B$ are the connected components of $\Sigma\setminus{\rm int}(e)$ containing, respectively, the initial endpoint of $\gamma_1$ and the entirety of $\gamma_1\setminus e$, then one has to consider $\Sigma'=t_1A\cup B\cup f$.

Now, suppose for the sake of contradiction that, for some index $j\in\mc{D}$, the arc $C_1\cap C_j$ has length $|C_1\cap C_j|>\min\{D+1,\ell_T(t_1)\}$. If we had $|C_1\cap C_j|\geq \ell_T(t_1)+1$, then there would exist an edge $e\sq C_j$ such that $t_1e\sq C_j$. By \Cref{lem:using_free_indecomposability}(3), there would exist an element $g\in G_j$ with $d_{\G}(ge,t_1e)\leq D$, and hence we would have $\ell_{\G}(g^{-1}t_1)\leq D$. This would violate Item~(i) of \Cref{ass:assumptions}, since there exists an element of ${\rm FR}^0(\mscr{G})$ taking $t_1$ to $g^{-1}t_1$.

Thus, we must have $D+2\leq|C_1\cap C_j|\leq\ell_T(t_1)$. Hence there exists an arc $\eta\sq C_1\cap C_j$ of length $D+1\leq \ell_T(t_1)-1$. Now, by a repeated application of the above claim, we can find an automorphism $\varphi\in{\rm FR}^0(\mscr{G})$ that equals the identity on $G_1\cup G_j$ and such that $\varphi(\mscr{G})$ admits a $T$--spanning tree $\Sigma'$ with the same number of edges as $\Sigma$ and with $\eta\sq\Sigma'$. Note that $\varphi(\mscr{G})$ still satisfies \Cref{ass:assumptions} 
and so, by part~(1), no two edges of $\Sigma'$ can be in the same $G_j$--orbit. At the same time, the fact that $\eta$ has length $D+1$ and \Cref{lem:using_free_indecomposability}(3) imply that two edges of $\eta\sq\Sigma'$ actually are in the same $G_j$--orbit. This is the required contradiction, proving part~(4).

\smallskip
{\bf Part~(5).} Say as usual that $i=1$ and $1\in\mc{D}$. We first show that no two vertices in $\mc{V}_1$ lie in the same $G_1$--orbit. Thus, suppose for the sake of contradiction that there exist edges $f_2,f_3\in\mc{E}_1$ based at distinct vertices $x_2,x_3\in C_1$, respectively, and that $gx_2=x_3$ for some $g\in G_1$.

Part~(1) shows that $x_2$ and $x_3$ cannot both lie in $\Sigma$; say without loss of generality that $x_2\not\in\Sigma$. This implies that $x_2$ does not lie in the projection to $C_1$ of any set $C_j$ such that $C_j\cap C_1$ has at most one point. In view of Item~(ii) of \Cref{ass:assumptions}, this means that $x_2$ lies in $C_1\cap C_j$ for at least one index $j\in\mc{C}$, and for no index in $\mc{D}\setminus\{1\}$. In fact, part~(2) shows that $x_2$ lies in $C_j$ for a \emph{unique} index $j\in\mc{C}$. Up to re-indexing, we can assume that $j=2$. The edge $f_2$ is then contained in the line $C_2$. We orient the line $C_2$ in the direction of translation of $t_2$ and, up to replacing $t_2$ with its inverse, we can assume that $f_2$ comes after the arc $C_1\cap C_2$ along $C_2$.

Observe that we have $t_2^{-1}x_2\in\Sigma$. Indeed, if this were to fail, we would have $[t_2^{-1}x_2,x_2]\cap\Sigma=\emptyset$, since $\Sigma\cap C_2$ has length $\ell_T(t_2)-1$ and $x_2\not\in\Sigma$. At the same time, we must have $\Sigma\cap C_1\cap C_2\neq\emptyset$ by Helly's lemma, and so the intersection $C_1\cap C_2$ would have to contain at least one edge outside the arc $[t_2^{-1}x_2,x_2]$. However, this would imply that the arc $C_1\cap C_2$ has length $\geq \ell_T(t_2)+1$, contradicting part~(4).

Now, if $x_3\in\Sigma$, we can consider the automorphism $\varphi\in{\rm FR}^0(\mscr{G})$ taking $t_2$ to $gt_2$ and equaling the identity on all $G_j$ with $j\neq 2$. We have that $\varphi(t_2)$ takes the point $t_2^{-1}x_2\in\Sigma$ to the point $x_3\in\Sigma$. These two points are distinct, because $\varphi(t_2)$ cannot be elliptic in $T$ by \Cref{lem:using_free_indecomposability}(1). Thus $\Sigma$ is a $T$--spanning tree for $\varphi(\mscr{G})$, which implies that $\varphi(\mscr{G})$ satisfies \Cref{ass:assumptions}, and finally violates part~(1) since two distinct vertices of $\Sigma$ lie in the same $\varphi(G_2)$--orbit.

Thus, we must have $x_3\not\in\Sigma$. Exactly as above, this implies that $x_3$ lies in $C_j$ for a unique index $j\in\mc{C}\cup\mc{D}$, and that we have $j\in\mc{C}$ and $f_3\sq C_j$. Note that we have $j\neq 2$, since otherwise we would have $d(x_2,x_3)\leq \ell_T(t_2)$ by part~(4), and this would contradict the fact that $\Sigma\cap\{x_2,x_3\}=\emptyset$ while $\Sigma\cap C_2$ has length $\ell_T(t_2)-1$. Thus, we can suppose that $j=3$ up to reindexing, and that the arc $C_1\cap C_3$ precedes the edge $f_3$ along $C_3$, in the direction of translation of $t_3$. Again as above, we see that $t_3^{-1}x_3\in\Sigma$. We can then consider the automorphism $\varphi\in{\rm FR}^0(\mscr{G})$ taking $t_2$ to $t_3^{-1}gt_2$ and equaling the identity on all $G_j$ with $j\neq 2$. As in the previous paragraph, this yields a contradiction, since $\varphi(t_2)$ takes the point $t_2^{-1}x_2\in\Sigma$ to the point $t_3^{-1}x_3\in\Sigma$. This the final contradiction, proving that no two vertices in $\mc{V}_1$ can lie in the same $G_1$--orbit.

We are left to show that no two edges in $\mc{E}_1$ lie in the same $G_1$--orbit. Suppose again for the sake of contradiction that there exist distinct edges $f_2,f_3\in\mc{E}_1$ and an element $g\in G_1$ with $gf_2=f_3$. By the previous discussion, $f_2$ and $f_3$ must be based at the same vertex of $C_1$. Call $x$ this vertex and note that $gx=x$. \Cref{lem:using_free_indecomposability}(2) then implies that $G_1$ contains the $G$--stabiliser of $x$, and that all edges of $T$ incident to $x$ are in the same $G_1$--orbit. Since $f_2\not\sq C_1$ by definition, the latter implies that $C_1=\{x\}$ and that $G_1$ equals the $G$--stabiliser of $x$. In view of Item~(ii) of \Cref{ass:assumptions},
$C_1$ is disjoint from all sets $C_j$ with $j\in\mc{D}\setminus\{1\}$. 
We can then once more construct an automorphism $\varphi\in{\rm FR}^0(\mscr{G})$ such that $\s_T(\varphi(\mscr{G}))<\s_T(\mscr{G})$, arguing as in the last paragraph of the proof of part~(1) (the elliptic case). This contradicts Item~(iii) of \Cref{ass:assumptions}, proving part~(5). 

This concludes the proof of the entire lemma.
\end{proof}

\begin{rmk}
    \Cref{lem:spanning_tree_long}(5) implies that, whenever $C_i$ is a single point, the set $\mc{E}_i$ consists of a single edge. Also note that, if $i\in\mc{D}$ and $C_i$ is not a single point, then the $G_i$--stabiliser of each point $x\in\mc{V}_i$ is trivial (otherwise the entire $1$--neighbourhood of $x$ would be contained in $C_i$, by \Cref{lem:using_free_indecomposability}(2), contradicting the fact that $x\in\mc{V}_i$).
\end{rmk}

Now comes the most technical part of the proof of \Cref{thm:shortening_free_products}. The main difficulty is that, for $i\in\mc{C}$ and $j\in\mc{D}$, the arc $\gamma_i=C_i\cap\Sigma$ does not always contain the entire projection of $C_j$ to $C_i$ (unlike when $j\in\mc{C}$). Our goal will be to find spanning trees with the following properties.

\begin{defn}\label{defn:perfect+admissible}
    Let $\Sigma\sq T$ be a $T$--spanning tree for $\mscr{G}$ with exactly $\s_T(\mscr{G})$ edges.
    \begin{itemize}
        \item We say that $\Sigma$ is \emph{perfect} if, for all $i\in\mc{C}$ and $j\in\mc{D}$, the intersection $\gamma_i=\Sigma\cap C_i$ 
        contains the entire projection of $C_j$ to $C_i$.
        \item We say that $\Sigma$ is \emph{admissible} if, for all $i\in\mc{C}$, there exists an arc $\gamma_i^*$ such that $\gamma_i\sq\gamma_i^*\sq C_i$ and such that all of the following hold:
        \begin{itemize}
            \item $\gamma_i^*$ contains the projection of $C_j$ to $C_i$ for all $j\in\mc{D}$;
            \item if $\gamma_i^*\neq\gamma_i$, the difference $\gamma_i^*\setminus\gamma_i$ consists of a single (half open) edge;
            \item if $\gamma_i^*\neq\gamma_i$, there exists an index $k(i)\in\mc{D}$ such that $C_{k(i)}$ contains the edge $\gamma_i^*\setminus\gamma_i$, and the projection of $C_j$ to $C_i$ is actually contained in $\gamma_i$ for all $j\in\mc{D}\setminus\{k(i)\}$. 
        \end{itemize}
    \end{itemize}
\end{defn}

We construct admissible spanning trees in the next result, and perfect ones in \Cref{lem:small_gamma} below. It is convenient to introduce now the following notion, which is used in several proofs.

\begin{defn}\label{defn:boring}
    An arc $\beta\sq T$ is \emph{boring} if, for each pair of indices $i,j\in\mc{C}\cup\mc{D}$, the open arc ${\rm int}(\beta)$ is either disjoint from, or entirely contained in, the projection of $C_i$ to $C_j$.
\end{defn}

Note that we allow $i=j$ in the previous definition, and so boring arcs $\beta$ have the property that, for each index $i\in\mc{C}\cup\mc{D}$, either $\beta\sq C_i$ or ${\rm int}(\beta)\cap C_i=\emptyset$. 

\begin{lem}\label{lem:smallish_gamma}
    Suppose that \Cref{ass:assumptions} holds. Up to replacing $\mscr{G}$ with $\psi(\mscr{G})$ for some $\psi\in{\rm FR}^0(\mscr{G})$ with $\s_T(\psi(\mscr{G}))=\s_T(\mscr{G})$, there exists an admissible $T$--spanning tree for $\mscr{G}$.
\end{lem}
\begin{proof}
    Consider a $T$--spanning tree $\Sigma$ for $\mscr{G}$ with exactly $\s_T(\mscr{G})$ edges, and set as usual $\gamma_i:=\Sigma\cap C_i$ for all $i\in\mc{C}$. Note that there exists an edge $e_i\sq C_i$ such that $e_i$ and $t_ie_i$ are precisely the two edges of $C_i$ that share a single vertex with $\gamma_i$. For every $j\in\mc{C}\setminus\{i\}$, the projection of $C_j$ to $C_i$ is contained in $\gamma_i$ by \Cref{lem:spanning_tree_long}(2). If $j\in\mc{D}$, then the projection of $C_j$ to $C_i$ must intersect $\gamma_i$ (by the definition of spanning tree). Finally, by Item~(ii) in \Cref{ass:assumptions}, there exists at most one index $k(i)\in\mc{D}$ such that $C_j$ contains one among $e_i$ and $t_ie_i$. 
    Up to inverting $t_i$, it is not restrictive to assume in what follows that, whenever $k(i)$ is defined, it is the edge $t_ie_i$ to be contained in $C_{k(i)}$. 
    
    Summing up, for every $j\in(\mc{C}\cup\mc{D})\setminus\{i,k(i)\}$, the projection of $C_j$ to $C_i$ is contained in $\gamma_i$. Moreover, we have $t_ie_i\sq (C_i\cap C_{k(i)})\setminus\gamma_i$.
    
    In order to prove the lemma, we are left to apply an element of ${\rm FR}^0(\mscr{G})$ so as to ensure that, whenever $k(i)$ is defined, we have $C_{k(i)}\cap C_i\sq\gamma_i\cup t_ie_i$ (without increasing $\s_T(\mscr{G})$). For this, we define the complexity $\mf{k}(\Sigma;\mscr{G})\geq 0$ as the number of indices $i\in\mc{C}$ such that $k(i)$ is defined and $C_{k(i)}\cap C_i\not\sq\gamma_i\cup t_ie_i$. We then set $\mf{k}(\mscr{G})$ as the minimum of the integers $\mf{k}(\Sigma;\mscr{G})$ as $\Sigma$ varies among $T$--spanning trees for $\mscr{G}$ with exactly $\s_T(\mscr{G})$ edges. We can assume that there does not exist any $\psi\in{\rm FR}^0(\mscr{G})$ with $\s_T(\psi(\mscr{G}))=\s_T(\mscr{G})$ and $\mf{k}(\psi(\mscr{G}))<\mf{k}(\mscr{G})$. Our goal is then to show that $\mf{k}(\mscr{G})=0$.

    Suppose for the sake of contradiction that $\mf{k}(\mscr{G})>0$. Fix a $T$--spanning tree $\Sigma$ realising $\mf{k}(\mscr{G})$, and define $\gamma_i,e_i$ as above. Consider an index $i\in\mc{C}$ such that $k(i)$ is defined and $C_{k(i)}\cap C_i\not\sq\gamma_i\cup t_ie_i$. Say for simplicity that $i=1$. For each $j$, denote by $\Pi_j$ the projection of $C_j$ to $C_1$. We also orient each line $C_i$ with $i\in\mc{C}$ in the direction of translation of $t_i$. We complete the proof by distinguishing three cases requiring slightly different arguments.

    \smallskip
    {\bf Case~1.} \emph{There exists a boring arc $\beta\sq\gamma_1$ (\Cref{defn:boring}) such that, denoting by $f$ and $f'$ the first and last edges of $\beta$, we have $d_{\G}(f,f')\geq 6\Gr(G)^2 D$.}

    \noindent
    Note that we cannot have any inclusions $\beta\sq\Pi_j$ with $j\in\mc{D}$. Indeed, parts~(1) and~(3) of \Cref{lem:spanning_tree_long} imply that the interior vertices of $\beta$ have trivial $G_j$--stabiliser and lie in pairwise distinct $G_j$--orbits.
    Thus, the inclusion $\beta\sq C_j$ would contradict \Cref{lem:using_free_indecomposability}(4).
    
    Now, write $\gamma_1=[a,b]$ with $a$ preceding $b$ along $C_1$. Note that $\beta$ is contained in $\Sigma$ and does not contain any branch points of the finite tree $\Sigma$. Thus, we can write $\Sigma=A\cup\beta\cup B$, where $A$ and $B$ are the connected components of $\Sigma\setminus{\rm int}(\beta)$ containing $a$ and $b$, respectively. Set $\Sigma':=B\cup t_1e_1\cup t_1 A\cup t_1\overline{(\beta\setminus f')}$ 
    and observe that $\Sigma'$ is connected and has the same number of edges as $\Sigma$. Define $\psi\in{\rm FR}^0(\mscr{G})$ as follows:
    \begin{enumerate}
        \item[(a)] $\psi$ is the identity on $G_1$;
        \item[(b)] if $\Pi_j\sq B$, then $\psi$ is the identity on $G_j$;
        \item[(c)] if $\Pi_j\sq A$, then $\psi(u)=t_1ut_1^{-1}$ for all $u\in G_j$;
        \item[(d)] if $\beta\sq\Pi_j$ (which can only happen for $j\in\mc{C}$) and $j\neq 1$, then $\psi(t_j)=t_jt_1^{-1}$ if $t_1$ and $t_j$ translate in the same direction along $\beta$, and $\psi(t_j^{-1})=t_j^{-1}t_1^{-1}$ otherwise.
    \end{enumerate}
    It is straightforward to check that $\Sigma'$ is a $T$--spanning tree for $\psi(\mscr{G})$ (in fact, one can view this as an iterated application of the claim used in the proof of \Cref{lem:spanning_tree_long}(4)). Thus, we have $\s_T(\psi(\mscr{G}))=\s_T(\mscr{G})$. It is also straightforward to check that indices $j\in\mc{C}\setminus\{1\}$ give a nonzero contribution to the complexity $\mf{k}(\psi(\mscr{G}))$ if and only if they do to the complexity $\mf{k}(\mscr{G})$: the most important observation is that, if $\psi(t_j)=t_jt_1^{-1}$, then $\psi(t_j)$ takes the oriented edge $t_1e_j$ to the oriented edge $t_je_j$, and we have $t_1e_j\cap t_1A\neq\emptyset$ and $t_je_j\cap B\neq\emptyset$; since we have $C_1\cap C_j\sq\gamma_1\cap\gamma_j$ by \Cref{lem:spanning_tree_long}(2), the oriented edges $t_1e_j$ and $t_je_j$ have compatible orientations, and so they lie on the axis of $\varphi(t_j)$.

    Finally, the index $i=1$ contributes to the complexity $\mf{k}(\mscr{G})$, but not to the complexity $\mf{k}(\psi(\mscr{G}))$: namely, the projection of $\mc{M}_T(\varphi(G_j))$ to $C_1$ is contained in $\Sigma'\cap C_1$ for all $j\in(\mc{C}\cup\mc{D})\setminus\{1\}$. For $j=k(1)$, this uses the fact that $t_1\beta\not\sq C_{k(1)}$ because of \Cref{lem:using_free_indecomposability}(4). In conclusion, we obtain $\mf{k}(\psi(\mscr{G}))<\mf{k}(\mscr{G})$, which is the required contradiction in Case~1.

    \smallskip
    {\bf Case~2.} \emph{There exist two consecutive edges $h,h'\sq C_1$ such that $d_{\G}(h,h')\geq 6\Gr(G)^2 D$.}

    \noindent
    Let $v$ be the shared vertex of $h,h'$. By a repeated application of the claim in the proof of \Cref{lem:spanning_tree_long}(4), we can apply an element of ${\rm FR}^0(\mscr{G})$ to slide $\Sigma$ along $C_1$ until $\gamma_1$ starts with the vertex $v$. As in Case~1, this does not alter which indices $i\in\mc{C}\setminus\{1\}$ contribute to the complexity $\mf{k}(\mscr{G})$. At the same time, if $\gamma_1$ starts with $v$, this means that $C_1\cap C_{k(1)}$ is contained in $\gamma_1\cup t_1e_1$, as otherwise $C_{k(1)}$ would contain a $G_1$--translate of $h\cup h'$, violating \Cref{lem:using_free_indecomposability}(4). Again, this violates minimality of $\mf{k}(\mscr{G})$.

    \smallskip
    {\bf Case~3.} \emph{There exists no arc $\beta$ as in Case~1 and no pair of edges $h,h'$ as in Case~2.}

    \noindent
    Consider a subset $\Delta\sq\gamma_1$ of cardinality $\leq 2(|\mc{G}|-1)$ that contains each $\Pi_j\cap\gamma_1$ that is a singleton, as well as the endpoints of each $\Pi_j\cap\gamma_1$ that is a nontrivial arc. The complement $\gamma_1\setminus\Delta$ consists of at most $2|\mc{G}|-1$ open arcs. 
    
    Let $\delta$ be one of the maximal open arcs in $\gamma_1\setminus\Delta$. For each $j\in(\mc{C}\cup\mc{D})\setminus\{1\}$, the open arc $\delta$ is either disjoint from or contained in the projection $\Pi_j$. Thus, if $d,d'$ are the first and last edge of $\delta$, we must have $d_{\G}(d,d')<6\Gr(G)^2 D$, as otherwise we would fall into Case~1.

    At the same time, we have $d_{\G}(h,h')\leq 6\Gr(G)^2 D$ for any pair of consecutive edges $h,h'$ of the arc $e_1\cup\gamma_1\cup t_1e_1$, otherwise we would fall into Case~2. As a consequence, we obtain
    \[ \ell_{\G}(t_1)=d_{\G}(e_1,t_1e_1)< 6\Gr(G)^2 D\cdot \big(1 + 2(|\mc{G}|-1) + (2|\mc{G}|-1) + 1\big) \leq 24\Gr(G)^3 D . \]
    This contradicts Item~(i) of \Cref{ass:assumptions}, concluding the proof of the lemma.
\end{proof}

The next lemma is the main step in upgrading admissible spanning trees to perfect ones.

\begin{lem}\label{lem:no_barriers}
    Under \Cref{ass:assumptions}, suppose that there exists an admissible $T$--spanning tree $\Sigma$ for $\mscr{G}$. Then, for every $i\in\mc{C}$, every vertex of $C_i$ has trivial $G$--stabiliser.
\end{lem}
\begin{proof}
    Let $\Sigma\sq T$ be an admissible $T$--spanning tree. Set as usual $\gamma_i:=\Sigma\cap C_i$ for all $i\in\mc{C}$, and let $\gamma_i^*$ and $k(i)$ be as provided by admissibility. We orient each line $C_i$ with $i\in\mc{C}$ in the direction of translation of $t_i$. Up to replacing some $t_i$ with $t_i^{-1}$, we can assume that, whenever $\gamma_i^*\neq\gamma_i$, the edge in $\gamma_i^*\setminus\gamma_i$ follows $\gamma_i$ along $C_i$.
    
    Consider some index $i_0\in\mc{C}$ and a vertex $x\in C_{i_0}$. We will show that the $G$--stabiliser of $x$ is trivial. Without loss of generality, we can suppose that $i_0=1$. Up to replacing $x$ with a $G_1$--translate, it is also not restrictive to assume that $x\in\gamma_1$. \Cref{lem:spanning_tree_long}(3) then shows that, for each $j\in\mc{C}\cup\mc{D}$, the $G_j$--stabiliser of $x$ is trivial.
    
    For those indices $i\in\mc{C}$ for which $k(i)$ is defined, we define $\mc{V}_i'$ as the set of vertices $v\in t_i^{-1}C_{k(i)}$ such that $v$ is the projection to $t_i^{-1}C_{k(i)}$ of a point in $\bigcup_{j\in\mc{C}\cup\mc{D}} C_j\setminus t_i^{-1}C_{k(i)}$. Note that, by admissibility, the intersection $\gamma_i\cap t_i^{-1}C_{k(i)}$ is precisely the initial endpoint of $\gamma_i$.

    \smallskip
    {\bf Claim~1.} \emph{The following holds for all $i\in\mc{C}$ with $\gamma_i^*\neq\gamma_i$: if we have $g\in G_{k(i)}$ and $v\in t_i\mc{V}_i'$ with $gv\in\mc{V}_{k(i)}$, then $g=1$ and $v$ is an endpoint of the arc $C_i\cap C_{k(i)}$.}

    \smallskip\noindent
    \emph{Proof of Claim~1.}
    Consider points $y\in\mc{V}_i'$ and $x\in\mc{V}_{k(i)}$, and suppose that there exists an element $g\in G_{k(i)}$ such that $x=gt_i\cdot y$. If $y$ is an endpoint of the arc $t_i^{-1}C_{k(i)}\cap C_i$, then $t_iy\in\mc{V}_{k(i)}$. Thus, we cannot have $g\neq 1$ because of \Cref{lem:spanning_tree_long}(5) and \Cref{lem:using_free_indecomposability}(2), and hence $g=1$ and we are in the one allowed exception. In the rest of the proof of the claim, we therefore assume that $y$ is not an endpoint of the arc $t_i^{-1}C_{k(i)}\cap C_i$. By the same argument, writing $gt_i=t_i(t_i^{-1}gt_i)$, we also assume that $x$ is not an endpoint of the arc $C_{k(i)}\cap C_i$.

    Let $x',y'$ be the projections of $x,y$ to $C_i$. Choose indices $k_x,k_y\in\mc{C}\cup\mc{D}$ and points $z_x\in C_{k_x}\setminus C_{k(i)}$ and $y\in C_{k_y}\setminus t_i^{-1}C_{k(i)}$ whose projections to $C_{k(i)}$ and $t_i^{-1}C_{k(i)}$ are $x$ and $y$, respectively. Since $x$ and $y$ are not endpoints of $C_{k(i)}\cap C_i$ and $t_i^{-1}C_{k(i)}\cap C_i$, respectively, we have that $x'$ and $y'$ are also the projections to $C_i$ of $z_x$ and $z_y$. For the same reason, we have that $\{k_x,k_y\}\cap\{i,k(i)\}=\emptyset$. The combination of the last two observations implies that $\{x',y'\}\sq\gamma_i$. As a consequence, we have that $g\neq 1$, as no two points of $\gamma_i$ are in the same $G_i$--orbit.
    
    Now, let $m$ be the second endpoint of the arc $t_i^{-1}C_{k(i)}\cap C_i$ (with respect to the chosen orientation on $C_i$). 
    Let $e\sq C_i$ be the edge containing $m$ and following it along $C_i$. 
    Before continuing, we need to formulate and prove a sub-claim.
    
    Consider for a moment an index $i'\in\mc{C}\setminus\{i\}$ such that we have $k(i')=k(i)$ and $e\sq C_i\cap C_{i'}$ (in particular, $t_i$ and $t_{i'}$ translate in the same direction along $C_i\cap C_{i'}$, by our conventions). Letting $q$ be the point where the line $C_{i'}$ exits the set $C_{k(i')}$, we claim that $q$ is also the point where the axis of $t_{i'}t_i^{-1}$ exits the set $C_{k(i)}=C_{k(i')}$ (recall that this element is indeed loxodromic by \Cref{lem:using_free_indecomposability}(1)). In order to prove this sub-claim, denote by $f$ the one edge in the difference $\gamma_{i'}^*\setminus\gamma_{i'}$ and note that $q$ is its terminal endpoint; also let $f'$ be the edge of $C_{i'}$ right after $f$. The point $t_{i'}^{-1}q$ is the initial endpoint of $\gamma_{i'}$ and hence the edge $t_{i'}^{-1}f$ is not contained in $C_i$. Thus, the oriented edges $f$ and $t_it_{i'}^{-1}f$ have compatible orientations, and so they lie on the axis of $t_{i'}t_i^{-1}$. 
    If we also have that $t_{i'}^{-1}f'\not\sq C_i$, then the same argument shows that $f'$ lies on the axis of $t_{i'}t_i^{-1}$, and so $q$ is indeed the exit point of this axis from $C_{k(i')}$. If instead $t_{i'}^{-1}f'\sq C_i$, denote by $\eps$ the edge of $C_i$ preceding $t_{i'}^{-1}f'$ along $C_i$. The edge $t_{i'}\eps$ is then the edge of the axis of $t_{i'}t_i^{-1}$ following $f$, so our goal is to show that $t_{i'}\eps\not\sq C_{k(i')}$. For this, note that, since $e\sq C_i\cap C_{i'}$ and $t_{i'}^{-1}f\not\sq C_i$, 
    we have $\eps\sq t_i^{-1}C_{k(i)}$ (recall that $e$ was defined as the first edge after $t_i^{-1}C_{k(i)}$ along $C_i$). Thus, if we were to have $t_{i'}\eps\sq C_{k(i')}$, then the element $t_{i'}t_i^{-1}$ would take an edge of $C_{k(i')}$ to an edge of $C_{k(i')}$, which would violate \Cref{lem:spanning_tree_long}(4) (after applying an element of ${\rm FR}^0(\mscr{G})$). This proves the sub-claim.

    We are now ready to complete the proof of Claim~1. Write $\Sigma=A\cup e\cup B$, where $A$ and $B$ are the connected components of $\Sigma\setminus{\rm int}(e)$ containing, respectively, the negative and positive semi-axis of $t_i$. By the claim used in the proof of \Cref{lem:spanning_tree_long}(4), we can apply an automorphism $\varphi\in{\rm FR}^0(\mscr{G})$ so as to slide $\Sigma$ along $C_i$ until $e$ is the last edge of $C_i$ before the beginning of the arc $\Sigma\cap C_i$. 
    The restriction of $\varphi$ to $G_{k(i)}$ is the identity, and one can check that the points $x$ and $t_iy$ now both lie in $\mc{V}_{k(i)}$: if $k_x$ (resp.\ $k_y$) lies in $\mc{C}$, this uses the sub-claim; if instead $k_x\in\mc{D}$, this uses the observation that $x'$ is the projection to $C_i$ of the entire set $C_{k_x}$ and so $\varphi$ is the identity on $G_{k_x}$ (resp.\ the conjugation by $t_i$ on $G_{k_y}$).

    Finally, if $t_iy\neq x$, we can invoke \Cref{lem:spanning_tree_long}(5) to conclude that $g$ cannot map one $t_iy$ to $x$. If instead $t_iy=x$, we can use the fact that $g\neq 1$ and \Cref{lem:using_free_indecomposability}(2) to reach the same conclusion. This proves Claim~1.
    \hfill$\blacksquare$

    \smallskip
    Now that Claim~1 is proven, we will use a delicate ping-pong argument to show that the $G$--stabiliser of the vertex $x$ must be trivial. Suppose for the sake of contradiction that there exists an element $g\in G\setminus\{1\}$ such that $gx=x$. Since $g\not\in G_1$, we can write $g=g_1g_2\dots g_nh$ with $n\geq 1$, $g_j\in G_{i_j}\setminus\{1\}$, $i_n\neq 1$ and $h\in G_1$. Here we assume that $i_j\neq i_{j+1}$ for all $j$. For $1\leq j\leq n$, define the point $x_j:=g_jg_{j+1}\dots g_nh\cdot x$. We say that $j$ is a \emph{good index} if one of the following holds.
    \begin{enumerate}
        \item[(d1)] We have $i_j\in\mc{D}$, $x_j\not\in C_{i_j}$, and the projection of $x_j$ to $C_{i_j}$ does not lie in $\mc{V}_{i_j}$.
        \item[(d2)] We have $i_j\in\mc{D}$, the set $C_{i_j}$ is a singleton, and $x_j$ is separated from $C_{i_j}$ by the interior of an edge $e_j$ containing $C_{i_j}$ and with $e_j\not\in\mc{E}_{i_j}$.
        \item[(c1)] We have $i_j\in\mc{C}$ and the projection of $x_j$ to $C_{i_j}$ lies outside $\gamma_{i_j}^*$.
        \item[(c2)] We have $i_j\in\mc{C}$, $\gamma_{i_j}^*\neq\gamma_{i_j}$, $x_j\not\in C_{k(i_j)}$, the projection of $x_j$ to $C_{k(i_j)}$ lies in $t_{i_j}\mc{V}_{i_j}'$, and this projection is not the initial endpoint of the arc $C_{i_j}\cap C_{k(i_j)}$.
        \item[(c3)] We have $i_j\in\mc{C}$, $\gamma_{i_j}^*\neq\gamma_{i_j}$, $x_j\not\in t_{i_j}^{-1}C_{k(i_j)}$, and the projection of $x_j$ to $t_{i_j}^{-1}C_{k(i_j)}$ is not in $\mc{V}_{i_j}'$.
    \end{enumerate}
    With this definition, we have the following. 

    \smallskip
    {\bf Claim~2.} \emph{If some index $j$ with $2\leq j\leq n$ is good, then $j-1$ is good as well.}

    \smallskip\noindent
    \emph{Proof of Claim~2.}
    Suppose first that $i_{j-1}\in\mc{D}$. If we have $i_j\in\mc{C}$ and $i_{j-1}=k(i_j)$, then $j-1$ is good by Claim~1 and \Cref{lem:spanning_tree_long}(5). Otherwise, goodness of $j$ implies that either $x_j\not\in C_{i_{j-1}}$ and the projection of $x_j$ to $C_{i_{j-1}}$ lies in $\mc{V}_{i_{j-1}}$, or there exists an edge $\eps_{j-1}\in\mc{E}_{i_{j-1}}$ whose interior separates $x_j$ from $C_{i_{j-1}}$. (If $i_j\in\mc{C}$, the proof of the latter fact again uses part of Claim~1 in the following form: for each $i\in\mc{C}$, the intersection $t_i\mc{V}_i'\cap\mc{V}_{k(i)}$ consists precisely of the endpoints of the arc $C_i\cap C_{k(i)}$.) Once the previous property is obtained, the fact that $j-1$ is good follows from \Cref{lem:spanning_tree_long}(5).

    Suppose now that $i_{j-1}\in\mc{C}$. Let $x_j'$ be the projection of $x_j$ to $C_{i_{j-1}}$. Goodness of $j$ implies that $x_j'\in\gamma_{i_{j-1}}^*$. If $x_j'$ is not an endpoint of $\gamma_{i_{j-1}}^*$, then $g_{j-1}x_j'\not\in\gamma_{i_{j-1}}^*$, and since this is the projection of $x_{j-1}=g_{j-1}x_j$, we have that $j-1$ is good. If $x_j'$ is the initial endpoint of $\gamma_{i_{j-1}}^*$, then $x_j\not\in t_{i_{j-1}}^{-1}C_{k(i_{j-1})}$, and the projection of $x_j$ to $t_{i_{j-1}}^{-1}C_{k(i_{j-1})}$ lies in $\mc{V}_{i_{j-1}}'$ and is not the initial endpoint of the arc $C_{i_{j-1}}\cap t_{i_{j-1}}^{-1}C_{k(i_{j-1})}$; considering $x_{j-1}=g_{j-1}x_j$, it is then immediate that $j-1$ is good. Finally, if $x_j'$ is the terminal endpoint of $\gamma_{i_{j-1}}^*$, then we must have $i_j=k(i_{j-1})$, and it easily follows from Claim~1 that $j-1$ is good. This proves Claim~2.
    \hfill$\blacksquare$

    \smallskip

    Claim~2 shows that, if any index $j\leq n$ is good, then $1$ is good. In turn, if $1$ is good, then the point $x_1=gx$ cannot coincide with the point $x\in\gamma_1\sq C_1$. In view of this, we conclude the proof of the lemma by showing that the index $n$ is good (or, possibly, that $n-1$ or $n-2$ is good). For this, recall that $hx\in C_1$ and let $x'$ denote the projection of $hx$ to $C_{i_n}$.
    
    Suppose first that $i_n\in\mc{D}$. Note that $x'\in\mc{V}_{i_n}$. If $C_{i_n}$ is not a singleton, then \Cref{lem:using_free_indecomposability}(2) and \Cref{lem:spanning_tree_long}(5) show that $g_nx'\not\in\mc{V}_{i_n}$. Therefore, if $C_{i_n}$ is not a singleton and if $hx\not\in C_{i_n}$, then $n$ is good. If $C_{i_n}$ is a singleton, then $C_{i_n}$ is necessarily disjoint from $C_1$, and hence the interior of the one edge in $\mc{E}_{i_n}$ separates $C_{i_n}$ from $hx$; thus, \Cref{lem:spanning_tree_long}(5) shows once more that $n$ is good. Finally, suppose that $C_{i_n}$ is not a singleton and that $hx\in C_{i_n}$. In this case, $x_n=g_nhx$ is a point of $C_{i_n}$ not lying in $\mc{V}_{i_n}$; in particular, we have $x_n\neq x$. If $n=1$, we are done. If instead $n>1$, we will show that $n-1$ is good. If $x_n\not\in C_{i_{n-1}}$, then we can argue as in the previous cases to deduce that $n-1$ is good. If instead $x_n\in C_{i_{n-1}}$, then the fact that $x_n\not\in\mc{V}_{i_n}$ implies that $i_{n-1}\in\mc{C}$ and that $x_n$ is not an endpoint of $C_{i_n}\cap C_{i_{n-1}}$. In particular, $x_n$ is not an endpoint of the arc $\gamma_{i_{n-1}}^*$, and hence $x_{n-1}=g_{n-1}x_n\not\in\gamma_{i_{n-1}}^*$, showing again that $n-1$ is good.

    Suppose now that $i_n\in\mc{C}$. Note that $x'\in\gamma_{i_n}$ by \Cref{lem:spanning_tree_long}(2). If $x'$ is not the initial endpoint of $\gamma_{i_n}$, or if $\gamma_{i_n}^*=\gamma_{i_n}$, or again if $g_n\neq t_{i_n}$, then we have $g_nx'\not\in\gamma_{i_n}^*$. Since $g_nx'$ is the projection of $x_n$ to $C_{i_n}$, it then follows that $n$ is good. Suppose instead that $g_n=t_{i_n}$, that $\gamma_{i_n}^*\neq\gamma_{i_n}$, and that $x'$ is indeed the shared endpoint of $\gamma_{i_n}$ and $\gamma_{i_n}^*$. In this case, let $x''$ be the projection of $hx$ to $t_{i_n}^{-1}C_{k(i_n)}$; note that $x''\in\mc{V}_{i_n}'$ and $x''$ is not the initial endpoint of the arc $C_{i_n}\cap t_{i_n}^{-1}C_{k(i_n)}$. Thus, if $hx\not\in t_{i_n}^{-1}C_{k(i_n)}$, then $n$ is good. Suppose that we instead have $hx\in t_{i_n}^{-1}C_{k(i_n)}$, so that $x_n=t_{i_n}hx$ lies in $C_{k(i_n)}$; note that we have $x_n\in t_{i_n}\mc{V}_{i_n}'$, since $x=x''$ in this case. The projection of $x_n$ to $C_{i_n}$ is the vertex in $\gamma_{i_n}^*\setminus\gamma_{i_n}$, so we have $x_n\neq x$ and we are done if $n=1$. If instead $n>1$, we aim to show that $n-1$ is good. Note that either $x_n$ is the final endpoint of $\gamma_{i_n}^*$, or we have $x_n\not\in\mc{V}_{k(i_n)}$ by Claim~1. Therefore, if $i_{n-1}\neq k(i_n)$, we have $x_n\not\in C_{i_{n-1}}$
    and this implies that $n-1$ is good arguing as above. Finally, suppose that $i_{n-1}=k(i_n)$. The point $x_{n-1}=g_{n-1}x_n$ then satisfies $x_{n-1}\in C_{i_{n-1}}$ and $x_{n-1}\not\in\mc{V}_{i_{n-1}}$ by Claim 1. If $n=2$, we again have $gx=x_{n-1}\neq x$ and we are done. If instead $n>2$, then we certainly have $x_{n-1}\not\in C_{i_{n-2}}$ and hence $n-2$ is good, arguing as above. This concludes the proof of the lemma.
\end{proof}

Now, that \Cref{lem:no_barriers} is proved, we can use it to obtain perfect spanning trees.

\begin{lem}\label{lem:small_gamma}
    Suppose that \Cref{ass:assumptions} holds. Up to replacing $\mscr{G}$ with $\psi(\mscr{G})$ for some $\psi\in{\rm FR}^0(\mscr{G})$ with $\s_T(\psi(\mscr{G}))=\s_T(\mscr{G})$, there exists a perfect $T$--spanning tree for $\mscr{G}$.
\end{lem}
\begin{proof}
    In view of \Cref{lem:smallish_gamma}, we can assume that there exist admissible spanning trees for $\mscr{G}$. Given an admissible spanning tree $\Sigma$, we define the integer $\mf{h}(\Sigma;\mscr{G})$ as the number of indices $i\in\mc{C}$ for which $\gamma_i^*\neq\gamma_i$. We then define $\mf{h}(\mscr{G})$ as the minimum of $\mf{h}(\Sigma;\mscr{G})$ as $\Sigma$ varies among admissible spanning trees for $\mscr{G}$. Assume that there does not exist any $\psi\in{\rm FR}^0(\mscr{G})$ such that, simultaneously, we have $\s_T(\psi(\mscr{G}))=\s_T(\mscr{G})$, the triple $\psi(\mscr{G})$ has admissible spanning trees, and we have $\mf{h}(\psi(\mscr{G}))<\mf{h}(\mscr{G})$. Our goal is then to show that $\mf{h}(\mscr{G})=0$.

    Suppose for the sake of contradiction that $\mf{h}(\mscr{G})>0$. Fix an admissible spanning tree $\Sigma$ realising $\mf{h}(\mscr{G})$, and define the arcs $\gamma_i,\gamma_i^*$ accordingly. Up to reindexing $\mc{G}$, we can assume that $1\in\mc{C}$ and $\gamma_1^*\neq\gamma_1$. \Cref{lem:no_barriers} implies that $\ell_{\G}(t_1)=\ell_T(t_1)$. In particular, in view of Item~(i) of \Cref{ass:assumptions}, we have $\ell_T(t_1)\geq 24\Gr(G)^3 D$. As a consequence, we can argue as in Case~3 of the proof of \Cref{lem:smallish_gamma} to show that there exists a boring arc $\beta\sq\gamma_1$ of length $\geq D+2$ (now computing lengths in $T$). By \Cref{lem:spanning_tree_long}(4), $\beta$ is not contained in $C_1\cap C_j$ for any $j\in\mc{D}$. We can then argue as in Case~1 of the proof of \Cref{lem:smallish_gamma} to produce an automorphism $\psi\in{\rm FR}^0(\mscr{G})$ such that $\s_T(\psi(\mscr{G}))=\s_T(\mscr{G})$, such that $\psi(\mscr{G})$ has admissible spanning trees, and such that $\mf{h}(\psi(\mscr{G}))<\mf{h}(\mscr{G})$. This is the required contradiction.
\end{proof}

Now that we have perfect spanning trees, we can once more use a ping-pong argument to show that the subtrees $C_i$ must be pairwise close in $T$ (and moreover have pairwise close lifts to $\G$). The main ingredient for this is the following.

\begin{lem}\label{lem:ping_pong}
    Under \Cref{ass:assumptions}, suppose that there exists a perfect spanning tree $\Sigma$ for $\mscr{G}$. Let $\beta\sq T$ be an arc (possibly degenerated to a point) satisfying {\bf one} of the following:
    \begin{enumerate}
        \item[(a)] the subtrees $C_1,C_2$ are disjoint and $\beta$ is the shortest arc connecting $C_1$ to $C_2$;
        \item[(b)] the intersection $C_1\cap C_2$ is a single point and $\beta=C_1\cap C_2$;
        \item[(c)] we have $1\in\mc{C}$ and $\beta\sq\gamma_1$.
    \end{enumerate}
    If $\beta$ is non-degenerate, assume in addition that $\beta$ is boring and that ${\rm int}(\beta)\cap C_j=\emptyset$ for all $j\in\mc{D}$. Then, $g\beta\cap\beta=\emptyset$ for all $g\in G\setminus\{1\}$, possibly except if we are in Case~(a) and $g\in G_1\cup G_2$.
\end{lem}
\begin{proof}
    The proof is based on a similar --- and much simpler --- version of the ping-pong argument used in the proof of \Cref{lem:no_barriers}. The lemma is clear if $g\in G_1$ (recalling Item~(ii) of \Cref{ass:assumptions} for Case~(b)). Thus, consider an element $g\in G\setminus G_1$ and write $g=g_1g_2\dots g_nh$ with $n\geq 1$, $g_j\in G_{i_j}\setminus\{1\}$, $i_n\neq 1$ and $h\in G_1$. Here we assume that $i_j\neq i_{j+1}$ for all $j$. For $1\leq j\leq n$, define the arc $\beta_j:=g_jg_{j+1}\dots g_nh\cdot\beta$. We say that $j$ is a \emph{good index} if one of the following holds.
    \begin{enumerate}
        \item[(d1)] We have $i_j\in\mc{D}$, $\beta_j\cap C_{i_j}=\emptyset$, and the projection of $\beta_j$ to $C_{i_j}$ does not lie in $\mc{V}_{i_j}$.
        \item[(d2)] We have $i_j\in\mc{D}$, $\beta_j\cap C_{i_j}=\emptyset$, the set $C_{i_j}$ is a singleton, and the interior of the one edge in $\mc{E}_{i_j}$ does not separate $\beta_j$ from $C_{i_j}$.
        \item[(c)] We have $i_j\in\mc{C}$ and the projection of $\beta_j$ to $C_{i_j}$ is disjoint from the arc $\gamma_{i_j}$.
    \end{enumerate}
    It is straightforward to check that, if any index $j\leq n$ is good, then $1$ is good. Moreover, if $1$ is good, then $g\beta\cap\beta=\emptyset$. 

    Our goal is thus to show that $n$ is good (modulo exceptions). If $h\beta\cap C_{i_n}=\emptyset$, then $n$ is good: this is immediate if $i_n\in\mc{C}$ and uses \Cref{lem:spanning_tree_long}(5) if $i_n\in\mc{D}$. Note that we can only have $h\beta\cap C_{i_n}\neq\emptyset$ if $h=1$: this is immediate in Case~(c), while it uses \Cref{lem:using_free_indecomposability}(2) and \Cref{lem:spanning_tree_long}(5) in Cases~(a) and~(b). Finally, if we have $h\beta\sq C_{i_n}$, then $i_n\in\mc{C}$ (by our hypotheses) and we have $h\beta\sq\gamma_{i_n}$; hence $\beta_n\cap\gamma_{i_n}=\emptyset$ and $n$ is good. 

    In conclusion, we can suppose in the rest of the proof that $h=1$, that $\beta\cap C_{i_n}\neq\emptyset$, and that $C_{i_n}\cap\beta$ is an endpoint of $\beta$ (since $\beta$ is either boring or degenerate). Calling $p$ this endpoint, we have that $p$ is also the projection of $\beta$ to $C_{i_n}$. If $i_n\in\mc{C}$, then $p\in\gamma_{i_n}$ and $g_np\not\in\gamma_{i_n}$, so $n$ is good. If $i_n\in\mc{D}$, we need to distinguish two possibilities. First, if $g_np=p$, then \Cref{lem:using_free_indecomposability}(2) and Item~(ii) of \Cref{ass:assumptions} imply that we are in Case~(a) with $i_n=2$ and $C_2$ is a singleton. If $n=1$, this is one of the allowed exceptions; if $n>1$, we have $\beta_n\cap C_{i_{n-1}}=\emptyset$ by \Cref{lem:spanning_tree_long}(5), which implies that $n-1$ is good as above. Finally, if we instead have $g_np\neq p$, then $g_np\not\in\mc{V}_{i_n}$ by \Cref{lem:spanning_tree_long}(5). If $n=1$, this implies that $g\beta\cap\beta=\emptyset$ and we are done. Suppose instead that $n>1$. If $\beta_n\cap C_{i_{n-1}}=\emptyset$, we again have that $n-1$ is good. If instead $\beta_n\cap C_{i_{n-1}}\neq\emptyset$, then we must have $i_{n-1}\in\mc{C}$ and the intersection $\beta_n\cap C_{i_{n-1}}$ is a single point of $\gamma_{i_{n-1}}$, which again implies that $n-1$ is good. This concludes the proof of the lemma.
\end{proof}

\begin{cor}\label{cor:C_is_empty}
    Under \Cref{label:setup}, suppose that $\mc{D}\neq\emptyset$ and that Items~(i) and~(ii) in \Cref{ass:assumptions} hold. Then we have $\mc{C}=\emptyset$.
\end{cor}
\begin{proof}
    After applying an element of ${\rm FR}^0(\mscr{G})$, \Cref{lem:small_gamma} guarantees that there exists a perfect spanning tree $\Sigma$ for $\mscr{G}$. Suppose for the sake of contradiction that $1\in\mc{C}$. Denote by $\Pi_j$ the projection of $C_j$ to $C_1$. We now argue as in Case~(iii) of the proof of \Cref{lem:smallish_gamma}.

    Consider a subset $\Delta\sq\gamma_1$ of cardinality $\leq 2(|\mc{G}|-1)$ that contains each $\Pi_j\cap\gamma_1$ that is a singleton, as well as the endpoints of each $\Pi_j\cap\gamma_1$ that is a nontrivial arc. The complement $\gamma_1\setminus\Delta$ consists of at most $2|\mc{G}|-1$ open arcs. Let $\delta$ be the closure of one of these arcs. If $\delta\sq C_j$ for some $j\in\mc{D}$, then $\delta$ has length $\leq D+1$ by \Cref{lem:spanning_tree_long}(4). Otherwise, $\delta$ projects injectively to the quotient $T/G$, by \Cref{lem:ping_pong}; in the latter case, $\delta$ contains at most $2$ edges. 
    This shows that 
    \[ \ell_T(t_1)\leq 1+(2|\mc{G}|-1)(D+1)\leq 24\Gr(G)^3 D ,\] 
    since $D\geq 1$ because of the fact that $\mc{D}\neq\emptyset$. At the same time, \Cref{lem:no_barriers} implies that $\ell_{\G}(t_1)=\ell_T(t_1)$. This violates Item~(i) of \Cref{ass:assumptions}, providing the required contradiction.
\end{proof}

\begin{cor}\label{cor:if_empty_C}
    Under \Cref{label:setup}, suppose that $\mc{C}=\emptyset$ and that Item~(ii) in \Cref{ass:assumptions} holds. Then, setting $S:=\bigcup_iS_i$, there exists an automorphism $\psi\in{\rm FR}^0(\mscr{G})$ such that
    \[ \mf{s}(\psi(S);\G)\leq 5\Gr(G)|S|D .\]
\end{cor}
\begin{proof}
    After applying an element of ${\rm FR}^0(\mscr{G})$, we can assume that Item~(iii) of \Cref{ass:assumptions} also holds. Let $\Sigma\sq T$ be a $T$--spanning tree with exactly $\s_T(\mscr{G})$ edges. For each $i\in\mc{D}$, choose a connected $G_i$--invariant subgraph $\G_i\sq\G$ such that the action $G_i\acts\G_i$ has $\leq D$ orbits of edges; this is possible by \Cref{lem:using_free_indecomposability}(3).

    Say that a pair of indices $i,j\in\mc{D}$ is \emph{tight} if we have $C_i\cap C_j=\emptyset$ and, denoting by $\beta_{i,j}\sq T$ the shortest arc joining $C_i$ to $C_j$, we have ${\rm int}(\beta_{i,j})\cap C_k=\emptyset$ for all $k\in\mc{D}$. If $i,j$ are a tight pair of indices, then \Cref{lem:ping_pong} shows that $\beta_{i,j}$ projects injectively to the quotient $T/G$, and so $\beta_{i,j}$ consists of at most two edges. \Cref{lem:ping_pong} also shows that the $G$--stabiliser of the vertex in $\beta_{i,j}\setminus(C_i\cup C_j)$ is trivial (if this vertex exists), and that, for each $h\in\{i,j\}$, either $C_h$ is a singleton or the $G$--stabiliser of the vertex in $\beta_{i,j}\cap C_h$ is trivial. In conclusion, we have $d_{\G}(\G_i,\G_j)\leq 2$ for any tight pair of indices $i,j$.

    Consider now the finite tree $\Sigma$. The combination of \Cref{lem:using_free_indecomposability}(2) and \Cref{lem:spanning_tree_long}(5) (and a straightforward ping-pong argument) show that vertices of $\Sigma$ have trivial $G$--stabiliser, with the possible exception of degree--$1$ vertices of $\Sigma$. In turn, for each degree--$1$ vertex $v\in\Sigma$, we have that either the $G$--stabiliser of $v$ is trivial, or there exists $i\in\mc{D}$ such that $C_i=\{v\}$. Now, consider for a moment two degree--$1$ vertices $a,b\in\Sigma$. There exist two indices $i,j\in\mc{D}$ such that $C_i\cap\Sigma=\{a\}$ and $C_j\cap\Sigma=\{b\}$, and there also exists a sequence of indices $i=:i_1,\dots,i_k:=j$ in $\mc{D}$ such that consecutive indices form a tight pair. The above observation about tight pairs and the fact the interior vertices of $\Sigma$ have trivial $G$--stabiliser then imply that $d_T(a,b)\leq 2k\leq 2\Gr(G)$. Given that $a$ and $b$ were arbitrary, this shows that $\diam(\Sigma)\leq 2\Gr(G)$ and, as a consequence, $\Sigma$ is contained in the ball of radius $\Gr(G)$ around a point $\overline x\in\Sigma$. Since we have seen that $\overline x$ has trivial $G$--stabiliser, it admits a unique lift $x\in\G$, and we have $d_{\G}(x,\G_i)\leq\Gr(G)$ for all $i\in\mc{D}$.

    Pick for each $i\in\mc{D}$ a vertex $x_i\in\G_i$ with $d_{\G}(x,x_i)\leq\Gr(G)$. Also pick a vertex $x_i'\in\G_i$ realising $\mf{s}(S_i,\G)$, that is, such that $\sum_{s\in S_i}d_{\G}(sx_i',x_i')\leq D$. Finally, pick an element $g_i\in G_i$ such that $d_{\G}(g_ix_i',x_i)\leq D$. For each $i\in\mc{D}$, we have
    \begin{align*}
        \sum_{s\in S_i}d_{\G}(g_isg_i^{-1}\cdot x,x)&\leq 2\Gr(G)|S_i| + \sum_{s\in S_i}d_{\G}(g_isg_i^{-1}\cdot x_i,x_i) \\
        &\leq 2(\Gr(G)+D)|S_i| + \sum_{s\in S_i}d_{\G}(s\cdot x_i',x_i') \\ 
        &\leq 2(\Gr(G)+D)|S_i|+D \leq 5\Gr(G)|S_i|D .
    \end{align*}
    Now, considering the element $\psi\in{\rm FR}^0(\mscr{G})$ that equals the conjugation by $g_i$ on each $G_i$ with $i\in\mc{D}$, we finally obtain
    \[ \mf{s}(\psi(S);\G)\leq \sum_{i\in\mc{D}}\sum_{s\in S_i}d_{\G}(g_isg_i^{-1}\cdot x,x) \leq 5\Gr(G)|S|D ,\]
    as required.
\end{proof}

\begin{proof}[Proof of \Cref{thm:shortening_free_products}]
    By the Milnor--Schwarz lemma, it suffices to prove the theorem for a specific proper cocompact action $G\acts X$, and we choose this to be the action $G\acts\G$ from \Cref{label:setup}. Let $\mscr{G}$ be the Grushko triple corresponding to the writing $G=G_1\ast\dots\ast G_k\ast\langle t_1\rangle\ast\dots\ast\langle t_m\rangle$ appearing in the statement of the theorem (for clarity, we have $|\mc{D}|=k$ and $|\mc{C}|=m$). Let $S_i\sq G_i$ be the finite generating sets appearing in the theorem statement, and set $S:=\bigcup_{i\in\mc{D}}S_i\cup\{t_1,\dots,t_m\}$.
    
    We will prove the theorem by showing that there exists $\varphi\in{\rm FR}^0(\mscr{G})$ such that 
    \[ \mf{s}(\varphi(S);\G)\leq K_{k,m}\cdot\big(1\vee D) = K_{k,m}\cdot\big(1\vee\max_{i\in\mc{D}}\mf{s}(S_i;\G)\big) ,\]
    where:
    \[ K_{k,m}:=(24\Gr(G)^3)^{m+1}(2+2|S|)^{k+m+1}\leq \big[48\Gr(G)^3(1+|S|)\big]^{\Gr(G)+1}.\]
    We prove this latter statement by induction on the complexity pair $\kappa(\mscr{G})=(k+m,m)$. The base case when $\kappa(\mscr{G})=(1,0)$ is clear, since we have $\mf{s}(S;X)\leq D$ there. For the inductive step, we distinguish three cases, based on whether \Cref{ass:assumptions}(i) fails, \Cref{ass:assumptions}(ii) fails, or neither of the two does.

    Suppose first that Item~(i) of \Cref{ass:assumptions} fails. Thus, there exist an index $i\in\mc{C}$ and an automorphism $\psi\in{\rm FR}^0(\mscr{G})$ such that $\ell_{\G}(\psi(t_i))\leq 24\Gr(G)^3 D$. We can then move the index $i$ into $\mc{D}$ to form a new Grushko triple $\mscr{G}'\prec\psi(\mscr{G})$ with $\kappa(\mscr{G}')=(k+m,m-1)$. Defining $S_j':=\psi(S_j)$ for all $j\in\mc{D}$ and $S_i':=\{\psi(t_i)\}$, we have $\max_{j\in\mc{D}'}\mf{s}(S_j';\G)\leq 24\Gr(G)^3 D$. Therefore, the inductive assumption yields an automorphism $\psi'\in{\rm FR}^0(\psi(\mscr{G}))\leq{\rm FR}^0(\mscr{G})$ such that 
    \[ \mf{s}(\psi'\psi(S);\G)\leq K_{k+1,m-1}\cdot 24\Gr(G)^3 D\leq K_{k,m}D, \]
    as required.

    Suppose now that Item~(ii) of \Cref{ass:assumptions} fails. Thus, there exist a pair of indices $i,j\in\mc{D}$ and an automorphism $\psi\in{\rm FR}^0(\mscr{G})$ such that $\mc{M}_T(\psi(G_i))\cap\mc{M}_T(\psi(G_i))$ either contains an edge or equals a vertex with nontrivial $\psi(G_i)$--stabiliser. We can then merge $\psi(G_i)$ and $\psi(G_j)$ to form a Grushko triple $\mscr{G}'\prec\psi(\mscr{G})$ with $\kappa(\mscr{G}')=(k+m-1,m)$. Defining $S_k':=\psi(S_k)$ for all $k\in\mc{D}$ and choosing $\psi(S_i)\cup\psi(S_j)$ as the generating set for $\langle\psi(G_i),\psi(G_j)\rangle$, we can invoke \Cref{lem:using_free_indecomposability}(5) to obtain
    \[ \max_{k\in\mc{D}'}\mf{s}(S_k';\G)\leq 2(1+|S|)D. \]
    Therefore, the inductive assumption yields an automorphism $\psi'\in{\rm FR}^0(\psi(\mscr{G}))\leq{\rm FR}^0(\mscr{G})$ such that 
    \[ \mf{s}(\psi'\psi(S);\G)\leq K_{k-1,m}\cdot (2+2|S|) D\leq K_{k,m}D , \]
    as required.

    Finally, suppose that Items~(i) and~(ii) of \Cref{ass:assumptions} both hold. If $\mc{D}\neq\emptyset$, then \Cref{cor:C_is_empty} implies that $\mc{C}=\emptyset$, and \Cref{cor:if_empty_C} yields an automorphism $\psi\in{\rm FR}^0(\mscr{G})$ with 
    \[ \mf{s}(\psi(S);\G)\leq 5\Gr(G)|S|D \leq K_{k,m}D .\] 
    If instead $\mc{D}=\emptyset$ (so that $D=0$), let $\psi\in{\rm FR}^0(\mscr{G})$ be an automorphism such that $\psi(\mscr{G})$ satisfies Item~(iii) of \Cref{ass:assumptions}. We can then use \Cref{lem:no_barriers} and \Cref{lem:ping_pong} (similarly to the proof of \Cref{cor:if_empty_C}) to conclude that $\mf{s}(\psi(S);\G)\leq 1$.
    This concludes the proof of the theorem.
\end{proof}

\bibliography{./mybib}
\bibliographystyle{alpha}

\end{document}